\documentclass[12pt,leqno,a4paper]{article}
\usepackage{amsthm}
\usepackage{indentfirst}
\usepackage{amsmath}
\usepackage{mathrsfs}
\usepackage{txfonts}
\usepackage{paralist}
\usepackage{amssymb,enumerate}
\usepackage{hyperref}
\usepackage{titlesec}
\usepackage{url}
\usepackage{extarrows}
\usepackage{booktabs}
\usepackage{threeparttable}
\usepackage{color}
\usepackage{geometry}
\renewcommand\thesubsection{\thesection.\Alph{subsection}}
\newcommand{\sbnal}{0.3em}   %space between number and label for subsections
\titleformat{\subsection}[runin]{\bfseries}{\thesubsection}{\sbnal}{}[]

\theoremstyle{theorem}
\newtheorem{mainthm}{Theorem}

\newtheorem{thm}{Theorem}[section]
\newtheorem{prop}[thm]{Proposition}
\newtheorem{lem}[thm]{Lemma}
\newtheorem{cor}[thm]{Corollary}

\theoremstyle{definition}

\newtheorem{defn}[thm]{Definition}
\newtheorem{rem}[thm]{Remark}
\newtheorem{notation}[thm]{Notation}
\newtheorem*{cau}{Caution}
\newtheorem*{rmk}{Remark}
\newtheorem*{notation*}{Notation}

\newcommand{\GL}{\operatorname{GL}}
\newcommand{\GU}{\operatorname{GU}}
\newcommand{\U}{\operatorname{U}}
\newcommand{\SL}{\operatorname{SL}}
\newcommand{\SU}{\operatorname{SU}}
\newcommand{\PSL}{\operatorname{PSL}}
\newcommand{\PSU}{\operatorname{PSU}}
\newcommand{\Sp}{\operatorname{Sp}}
\newcommand{\PSp}{\operatorname{PSp}}

\newcommand{\GO}{\operatorname{GO}}
\newcommand{\Aut}{\operatorname{Aut}}
\newcommand{\Inn}{\operatorname{Inn}}
\newcommand{\Out}{\operatorname{Out}}
\newcommand{\Ind}{\operatorname{Ind}}
\newcommand{\Res}{\operatorname{Res}}
\newcommand{\Irr}{\operatorname{Irr}}
\newcommand{\IBr}{\operatorname{IBr}}
\newcommand{\Rad}{\operatorname{Rad}}
\newcommand{\dz}{\operatorname{dz}}
\newcommand{\Alp}{\operatorname{Alp}}
\newcommand{\diag}{\operatorname{diag}}

\newcommand{\mrO}{\mathrm{O}}
\newcommand{\J}{\operatorname{J}}
\newcommand{\Lin}{\operatorname{Lin}}

\newcommand{\grp}[1]{\langle#1\rangle}
\newcommand{\Grp}[1]{\left\langle#1\right\rangle}

\newcommand{\set}[1]{\{#1\}}

\newcommand{\F}{\mathbb{F}}
\newcommand{\barF}{\overline{\mathbb{F}}}
\newcommand{\Z}{\mathbb{Z}}

\newcommand{\bG}{\mathbf{G}}
\newcommand{\bL}{\mathbf{L}}
\newcommand{\tG}{\tilde{G}}
\newcommand{\tH}{\tilde{H}}
\newcommand{\tbG}{\tilde{\mathbf{G}}}
\newcommand{\hG}{\hat{G}}

\newcommand{\tbL}{\tilde{\mathbf{L}}}
\newcommand{\tR}{\tilde{R}}
\newcommand{\tS}{\tilde{S}}
\newcommand{\hR}{\hat{R}}
\newcommand{\tC}{\tilde{C}}
\newcommand{\hC}{\hat{C}}

\newcommand{\tN}{\tilde{N}}
\newcommand{\tM}{\tilde{M}}
\newcommand{\hN}{\hat{N}}

\newcommand{\tZ}{\tilde{Z}}
\newcommand{\hZ}{\hat{Z}}
\newcommand{\tB}{\tilde{B}}

\newcommand{\tD}{\tilde{D}}
\newcommand{\tW}{\tilde{W}}

\newcommand{\zero}{\mathbf{0}}
\newcommand{\one}{\mathbf{1}}
\newcommand{\two}{\mathbf{2}}

\newcommand{\bc}{\mathbf{c}}
\newcommand{\bbeta}{\bar{\beta}}

\newcommand{\tfb}{\tilde{\mathfrak{b}}}

\newcommand{\cF}{\mathcal{F}}
\newcommand{\cM}{\mathcal{M}}
\newcommand{\cE}{\mathcal{E}}

\newcommand{\cR}{\mathcal{R}}
\newcommand{\cN}{\mathcal{N}}
\newcommand{\tcR}{\tilde{\mathcal{R}}}
\newcommand{\fZ}{\mathfrak{Z}}
\newcommand{\sC}{\mathscr{C}}
\newcommand{\tsC}{\tilde{\mathscr{C}}}

\newcommand{\fA}{\mathfrak{A}}
\newcommand{\fS}{\mathfrak{S}}

\newcommand{\tg}{\tilde{g}}
\newcommand{\tn}{\tilde{n}}

\newcommand{\ttheta}{\tilde{\theta}}

\newcommand{\tDelta}{\tilde{\Delta}}

\newcommand{\tvarphi}{\tilde{\varphi}}
\newcommand{\tchi}{\tilde{\chi}}
\newcommand{\tpsi}{\tilde{\psi}}

\newcommand{\hs}{\hat{s}}
\newcommand{\hz}{\hat{z}}

\newcommand{\tOmega}{\tilde{\Omega}}

\begin{document}

\title{Equivariant correspondences and the inductive Alperin weight condition for type $\mathsf A$
\footnote{Supported by the NSFC (No. 11631001, No. 11901028, No. 11901478).}}

\author{Zhicheng Feng
\footnote{School of Mathematics and Physics, University of Science and Technology Beijing, Beijing 100083, China. E-mail address: zfeng@pku.edu.cn} \quad Conghui Li
\footnote{School of Mathematics, Southwest Jiaotong University, Chengdu 611756, China. E-mail address: liconghui@swjtu.edu.cn}
\quad Jiping Zhang
\footnote{School of Mathematical Sciences, Peking University, Beijing 100871, China. E-mail address: jzhang@pku.edu.cn}
}

\maketitle

\begin{abstract}
In this paper, we establish the inductive Alperin weight condition for the finite simple groups of Lie type $\mathsf A$, contributing to the program to prove the Alperin weight conjecture by checking the inductive condition for all finite simple groups.
\end{abstract}

\tableofcontents

%%%%%%%%%%%%%%%%%%%%%%%%%%%%%%%%%%%%%%%%%%%

\section{Introduction}\label{sect:intro}

One of the most important local-global conjectures in representations of finite groups is the Alperin weight conjecture \cite{Al87}, claiming that
$$|\Alp(G)|=|\IBr(G)|$$
for every finite group $G$ and every prime number $\ell$,
where $\Alp(G)$ denotes the set of all $G$-conjugacy classes of $\ell$-weights  of $G$.

This conjecture has a blockwise refinement,
asserting that when partitioned into blocks, the numbers of irreducible Brauer characters and weights are also equal.
These conjectures have been verified for symmetric groups, sporadic simple groups and many finite groups of Lie type;
see \cite{AF90,An92,An93b,An93,An94}.

The Alperin weight conjecture and its blockwise version seem extremely difficult to prove in general,
so efforts are made to reduce it to simple groups,
especially after the McKay conjecture had been reduced to simple groups successfully by Isaacs, Malle and Navarro \cite{IMN07}.
First, Navarro and Tiep reduced the Alperin weight conjecture to simple groups in \cite{NT11},
and then Sp\"ath achieved the reduction for the blockwise version in \cite{Sp13}.
Thus it is an important task to verify for all simple groups the \emph{inductive Alperin weight (AW) condition} defined in \cite{NT11}
and the \emph{inductive blockwise Alperin weight (BAW) condition} defined in \cite{Sp13}. 
In the  paper \cite{NT11}, the authors verified their inductive (AW) condition for groups of Lie type in defining characteristic,
as well as for all simple groups with abelian Sylow $2$-subgroups,
while An and Dietrich \cite{AD12} dealt with sporadic groups.
The inductive (BAW) condition has been verified for several families of simple groups so far:
groups of Lie type in their defining characteristic (Sp\"ath \cite{Sp13});
simple alternating groups, Suzuki groups and Ree groups (Malle \cite{Ma14});
many of the $26$ sporadic groups (Breuer \cite{Br});
blocks with cyclic defect groups (Koshitani--Sp\"ath \cite{KS16a,KS16b});
groups of Lie type $G_2$ and $^3D_4$ (Schulte \cite{Sch16});
a special case of type $\mathsf A$ (C. Li--Zhang \cite{LZ18,LZ19});
unipotent blocks of type $\mathsf A$ (Feng \cite{Feng19});
blocks  of type $\mathsf A$ with abelian defect (Brough--Sp\"ath \cite{BS19b});
certain cases of classical groups (Feng--Z. Li--Zhang \cite{FLZ19a});
groups of type $\mathsf C$ (C. Li \cite{Li19,Li19b}, Feng--Malle \cite{FM20});
some more particular simple classical groups of small rank (\cite{BSF19,Du19,Du20,FLL17a,LL19,SF14}).
We mention that the inductive (AW) condition holds for all of the above cases since the inductive (BAW) condition implies the inductive (AW) condition.

The inductive (AW) condition for a simple group  is much stronger than just the validity of the  Alperin weight conjecture,
and in particular also involves the covering groups and the automorphism group of the simple group in question.
These conditions are highly complicated and are essentially two-fold.
The first half requires for each finite quasi-simple group $G$ a bijection $\IBr(G)\to\Alp(G)$
that should be equivariant for all automorphisms of $G$.
The second requirement is of a cohomological nature and relates to cohomology elements on subgroups of $\Out(G)$ (see \cite[\S 3]{NT11}).
Special linear and special unitary groups are therefore seemingly the most difficult cases with respect to this cohomological criterion,
and the classes of cohomology elements involved seem to be of unbounded order.

The main goal of this paper is to show the following theorem.

\begin{mainthm}\label{mainthm-1}
The simple groups $\PSL_n(q)$ and  $\PSU_n(q)$ satisfy the inductive Alperin weight condition.
\end{mainthm}

Let $p$ be a prime, $q=p^f$ and $n$ a positive integer.
Using the convention of putting $\SU_n(q)=\SL_n(-q)$ and $\GU_n(q)=\GL_n(-q)$
allows us to write $\SL_n(\eta q)$ and $\GL_n(\eta q)$ for $\eta=\pm 1$.

Our main statements are proven by applying a recent criterion given by Brough and Sp\"ath  \cite{BS19},
which allows for a slightly bigger group to be considered in place of the universal covering group of a simple group;
see Theorem \ref{thm:criterion}.
This criterion can be divided into three parts:
a global statement about the stabilizers and extendibility of irreducible Brauer characters of $\SL_n(\eta q)$,
a local statement analogous to the first about weights,
and an equivariant bijection between the irreducible Brauer characters and weights of $\GL_n(\eta q)$.
The global statement is easy here
by a result of Cabanes and Sp\"ath \cite{CS17} about the stabilizers and extendibility of irreducible characters of $\SL_n(\eta q)$
and the unitriangular shape of the decomposition matrix of $\SL_n(\eta q)$ given by Kleshchev--Tiep \cite{KT09} and Denoncin \cite{De17}.
The  bijection for $\GL_n(\eta q)$ given by Alperin--Fong \cite{AF90} and An \cite{An92,An93, An94}
is already verified in \cite{LZ18} to be equivariant
but we strengthen the results there in order to describe in addition the irreducible Brauer characters and the weights of $\SL_n(\eta q)$.

The main ingredient for the local statement
is to classify the weights of $\SL_n(\eta q)$ using the construction of the weights of $\GL_n(\eta q)$ in \cite{AF90,An92,An93,An94}.
Although weights of many classical groups have been classified by works of several authors,
the case of $\SL_n(\eta q)$ has been left unsolved for long time.
In \cite{BS19},
Brough and Sp\"ath give a relationship ``cover" between weights of a finite group and its normal subgroups.
The weights of $\SL_n(\eta q)$ were classified in \cite{Feng19} if $\ell\nmid \gcd(n,q-\eta)$;
we improve the methods and generalize the results there
and determine the normalizers and centralizers of radical subgroups
in order to obtain the weights of  $\SL_n(\eta q)$ covered by a given weight of  $\GL_n(\eta q)$.
Moreover, based on the methods for the stabilizers and extendibility of characters in the local case of \cite{CS17,MS16}
and transferring to the twisted radical subgroups,
we give a parametrization of certain characters of local subgroups and
finally establish the stabilizers and extendibility  properties of weight characters. 

\vspace{2ex}

This paper is structured in the following way.
After introducing some notation in \S \ref{sect:notation-preliminaries},
we recall in \S \ref{sect:weights-general-gp} the weights of general linear and unitary groups using the twisted version
and give slight generalizations for some results in \cite{AF90,An92,An93,An94}.
Then in \S \ref{subsect:srs} we determine which radical subgroups of general linear and unitary groups are ``special".
The weights of special linear and unitary groups are classified in \S \ref{subsect:wos},
while in \S \ref{An-equivariant-bijection} we construct a bijection,
coming from \cite{AF90,An92,An93,An94} but with additional equivariance properties, see Theorem \ref{thm-equ-bijection}.
Then we consider the local situation in \S\ref{sec-Stab-ext-wei} and
check the required condition on the stabilizers and extendibility properties of weight characters.
Finally we complete the proof of our Main Theorem \ref{mainthm-1} in \S \ref{proof-main-thm}.

%%%%%%%%%%%%%%%%%%%%%%%%%%%%%%%%%%%%%%%%%%%

\section{Notation and preliminaries }\label{sect:notation-preliminaries}

The notation for representations of finite groups we use can be found in \cite{NT89},
except that $\Ind$ and $\Res$ are used for induction and restriction.
Throughout we consider the modular representations with respect to a fixed prime $\ell$.
For a finite group $G$ and $\chi \in \Irr(G)$, $\chi^0$ is used for the restriction to $\ell$-regular elements of $G$.
We will abbreviate $\ell$-Brauer characters, $\ell$-blocks, etc. as Brauer characters, blocks, etc.
For an integer $m$, we write $m_\ell$ and $m_{\ell'}$ for the $\ell$-, $\ell'$-part of $m$ respectively.

For $K\unlhd G$ and $\theta\in\Irr(K)\cup\IBr(K)$, we denote by $G_\theta$ the stabilizer of $\theta$ in $G$,
and for $H\le G$ and $\theta\in\Irr(H)\cup\IBr(H)$, we denote by $G_{H,\theta}$ the  stabilizer of $\theta$ in $N_G(H)$.
Additionally, $G_\theta$ is also denoted by $G(\theta)$ in this paper since the notation for $G$ and $\theta$ in the following is often cumbersome.
For $K\unlhd G$ we sometimes identify the characters of $G/K$ with the chracters of $G$ whose kernel contains $K$.
For $K\unlhd G$ and $\chi\in\Irr(G)\cup\IBr(G)$,
let $\kappa_K^G(\chi)$ be the number of irreducible constituents of $\Res^{G}_{K}(\chi)$ forgetting multiplicities.

For $K\le G$, $\mathcal K\subseteq \Irr(K)$ and $\mathcal G\subseteq \Irr(G)$,
we define $\Irr(G\mid \mathcal K):=\bigcup\limits_{\theta\in\mathcal K}\Irr(G\mid \theta)$ and
$\Irr(K\mid \mathcal G):=\bigcup\limits_{\chi\in\mathcal G}\Irr(K\mid \chi)$.

If $G$ is abelian, we also write $\Lin(G)=\Irr(G)$ since all irreducible characters of $G$ are linear.
Let $\Lin_{\ell'}(G)$ denote the subgroup of $\Lin(G)$ of all irreducible characters of $G$ of $\ell'$-order.
Then the map $\Lin_{\ell'}(G)\to \IBr(G)$, $\chi\mapsto \chi^0$ is bijective.
From this, we always identify $\IBr(G)$ with $\Lin_{\ell'}(G)$ when $G$ is abelian.

We denote by $\nu$ the exponential valuation associated to the prime $\ell$,
normalised so that  $\nu(\ell)=1$.  \label{def-valuation}
For a finite group $G$, we  abbreviate $\nu(|G|)$ as $\nu(G)$.
For a positive integer $n$, we denote by $\fS_n$ (or $\fS(n)$) the symmetric group on $n$ symbols, while $\fA_n$ denotes the alternating group on $n$ symbols.

\subsection{Clifford theory between weights}\label{subsec:general-criterion}

For a finite group $G$, we denote by $\dz(G)$ the set of all irreducible defect zero characters of $G$.
Set $\dz(G \mid \theta) = \dz(G) \cap \Irr(G \mid \theta)$ for $K \unlhd G$ and $\theta\in\Irr(K)$.
A \emph{weight} of $G$ means a pair $(R,\varphi)$ where $R$ is an $\ell$-subgroup of $G$ and $\varphi\in\dz(N_G(R)/R)$.
In this case, $R$ is necessarily a radical subgroup of $G$, \emph{i.e.} $R=\mrO_\ell(N_G(R))$, and $\varphi$ is called a \emph{weight character}.
We denote by $\Rad(G)$ the set of all radical subgroups of $G$.
We often identify characters in $\dz(N_G(R)/R)$ with their inflations to $N_G(R)$.
Denote the set of all $G$-conjugacy classes of  weights of $G$ by $\Alp(G)$.
For a weight $(R,\varphi)$ of $G$, denote by $\overline{(R,\varphi)}$ the $G$-conjugacy class containing $(R,\varphi)$.
Sometimes we also write $\overline{(R,\varphi)}$ simply as $(R,\varphi)$ when no confusion can arise.

In order to introduce the relationship ``cover" and Clifford theory for weights from \cite{BS19},
we recall the Dade--Glauberman--Nagao correspondence.
Let $K\unlhd M$ be such that $M/K$ is an $\ell$-group.
Let $D_0$ be a normal $\ell$-subgroup of $K$ contained in $Z(M)$.
Suppose that $b$ is an $M$-invariant block of $K$
and has defect group $D_0$.
Assume that $B$, the only block of $M$ covering $b$, has defect group $D$.
Let $L=N_K(D)$ and $b'$ the only block of $L$ covered by $B'$, the Brauer  correspondent of $B$.
Then by \cite[Thm.~5.2]{NS14}, there is a natural bijection $\pi_D:\Irr_D(b)\to \Irr_D(b')$,
called the \emph{Dade--Glauberman--Nagao correspondence}.
Here, $\Irr_D(b)$ denotes the set of $D$-invariant characters in $b$.

Let $G\unlhd \tG$ and let $(R,\varphi)$ be a weight of $G$.
Set $N=N_G(R)$.
Take $N\le M\le N_{\tG}(R)_\varphi$ such that $M/N$ is an $\ell$-group.
Fix a defect group $\tR/R$ of the unique block of $M/R$ which covers the block of $N/R$ containing the character $\varphi$.
Then $\tR\cap N=R$, $M=N\tR$ and 
 $N_{M/R}(\tR/R)=(\tR/R)C_{N/R}(\tR/R)$.
Note that $\pi_{\tR/R}(\varphi)\in\dz(C_{N/R}(\tR/R))$
and denote by $\bar\pi_{\tR/R}(\varphi)$ the associated character in $\Irr(N_{M/R}(\tR/R)/(\tR/R))$ which lifts to $\pi_{\tR/R}(\varphi)\times 1_{\tR/R}\in\Irr(N_{M/R}(\tR/R))$.
On the other hand,  $N_{M/R}(\tR/R)/(\tR/R)\cong N_G(\tR)\tR/\tR$ is normal in $N_{\tG}(\tR)/\tR$.
Following \cite{BS19}, a weight $(\tR,\tvarphi)$ of $\tG$ with $\tvarphi\in\dz(N_{\tG}(\tR)/\tR\mid \bar\pi_{\tR/R}(\varphi))$ is said to \emph{cover} $(R,\varphi)$.
Using this, some Clifford-like properties for weights can be established;
see \cite[\S 2]{BS19}.

\begin{lem}\label{spec-rad-cover}
Let $(R,\varphi)$ be a weight of $G$ and 
$(\tR,\tilde\varphi)$ be a weight of $\tG$ covering $(R,\varphi)$ such that $R=\tR\cap G$.
If $N_{\tG}(R)=N_{\tG}(\tR)$, then $\varphi\in\Irr(N_G(R)\mid \tilde\varphi)$.
\end{lem}

\begin{proof}
In the argument above,  if $N_{\tG}(R)=N_{\tG}(\tR)$, then $N_{M/R}(\tR/R)=M/R$.
Thus $\pi_{\tR/R}$ is the identity map and the lemma follows.
\end{proof}

\subsection{The inductive Alperin weight condition.}

Recently, Brough and Sp\"ath gave a new criterion for the inductive (AW) condition.
We first give some notation.

Let $G\unlhd \tG$ such that $\tG/G$ is abelian.
We denote by $\tG_{\ell'}$ the normal subgroup of $\tG$ with $G\le \tG_{\ell'}$
such that $\tG/\tG_{\ell'}$ is isomorphic to a Sylow $\ell$-subgroup of $\tG/G$.
For $\tilde\psi\in\IBr(\tG)$, 
we define $\J_{G}(\tpsi)=\tG_{\ell'}\tG_{\psi}$ for some $\psi\in\IBr(G\mid\tilde\psi)$.
For a weight $(\tR,\tvarphi)$, we denote by $\J_{G}(\overline{(\tR,\tvarphi)})=\tG_{\ell'}\tG_{\overline{(R,\varphi)}}$,
where $(R,\varphi)$ is a weight of $G$ covered by $(\tR,\tvarphi)$ and
$\tG_{\overline{(R,\varphi)}}$ is the stabilizer of $\overline{(R,\varphi)}$ in $\tG$.
Note that $\J_{G}(\overline{(\tR,\tvarphi)})=\tG_{\ell'} N_{\tG}(R)_\varphi$.

For $\nu \in \Lin_{\ell'}(Z(\tG))$,
we denote by $\Alp(\tG\mid\nu)$ the subset of $\Alp(\tG)$ consisting of $\overline{(\tR,\tvarphi)}$ such that $\nu\in\Irr(Z(\tG)\mid \tvarphi)$.

\begin{thm}[{\cite[Thm.~3.3]{BS19}}] \label{thm:criterion}
	Let $S$ be a finite non-abelian simple group and $\ell$ a prime dividing $|S|$.
	Let $G$ be the universal covering group of $S$ and
	assume there are finite groups $\tG$, $D$ such that $G \unlhd \tG D=\tG \rtimes D$ and the following hold:
	\begin{enumerate}[\rm(i)]\setlength{\itemsep}{0pt}
		\item \begin{enumerate}[\rm(a)]\setlength{\itemsep}{0pt}
			\item $G=[\tG,\tG]$ and $D$ is abelian,
			\item $C_{\tG D}(G)=Z(\tG)$ and $\tG D/Z(\tG) \cong \Aut(G)$,
			\item any element of $\IBr(G)$ extends to its stabilizer in $\tG$,
			\item for any $R \in \Rad(G)$, any element of $\dz(N_G(R)/R)$ extends to its stabilizer in $N_{\tG}(R)/R$.
			\end{enumerate}
		\item There exists a $\Lin_{\ell'}(\tG/G) \rtimes D$-equivariant bijection $\tOmega:\ \IBr(\tG) \to \Alp(\tG)$ such that
			\begin{enumerate}[\rm(a)]\setlength{\itemsep}{0pt}
			\item $\tOmega(\IBr(\tG\mid\nu)) = \Alp(\tG\mid\nu)$ for every $\nu \in \Lin_{\ell'}(Z(\tG))$,
			\item $\J_G(\tpsi) = \J_G(\tOmega(\tpsi))$ for every $\tpsi \in \IBr(\tG)$.
			\end{enumerate}
		\item For every $\tpsi\in\IBr(\tG)$, there exists some $\psi_0\in\IBr(G\mid\tpsi)$ such that
			\begin{enumerate}[\rm(a)]\setlength{\itemsep}{0pt}
			\item $(\tG\rtimes D)_{\psi_0}=\tG_{\psi_0}\rtimes D_{\psi_0}$,
			\item $\psi_0$ extends to $G\rtimes D_{\psi_0}$.
			\end{enumerate}
		\item In every $\tG$-orbit on $\Alp(G)$, there is an $(R,\varphi_0)$ such that
			\begin{enumerate}[\rm(a)]\setlength{\itemsep}{0pt}
			\item $(\tG D)_{R,\varphi_0} = \tG_{R,\varphi_0} (GD)_{R,\varphi_0}$,
			\item $\varphi_0$ extends to $(G\rtimes D)_{R,\varphi_0}$.
			\end{enumerate}
	\end{enumerate}
	Then the inductive Alperin weight (AW) condition from \cite[\S 3]{NT11} holds for $S$ and $\ell$.
\end{thm}

\begin{rem}
In \cite{BS19},
the group $G$ in the criterion for the inductive (AW) condition is set up to be the universal $\ell'$-covering group of the simple group $S$.
We assume $G$ in Theorem \ref{thm:criterion} to be the universal covering group
and note that our conditions here imply the conditions in \cite{BS19},
since $\mrO_\ell(Z(G))$ is contained in the kernel of every Brauer character, in every radical subgroup and in the kernel of every weight characters.	
\end{rem}

Let $S$ be a non-abelian finite simple group,
then we say \emph{the inductive (AW) condition holds for $S$} if the inductive (AW) condition holds for $S$ and any prime $\ell$ dividing $|S|$.

\subsection{Linear and unitary groups.}\label{subsec:Linear and unitary groups}

Let $p$ be a prime, $q=p^f$  a power of $p$ and denote by $\F_q$ the finite field of $q$ elements.
Set $\tbG=\GL_n(\barF_q)$ and $\bG=\SL_n(\barF_q)$ with $n\ge 2$.
Denote by $F_q$ the standard Frobenius endomorphism on $\tbG$ and $\bG$ and by
$\sigma_{it}$ the automorphism of $\tbG$ and $\bG$ defined by $\sigma_{it}(A)=(A^{-1})^t$,
where $^t$ denotes the transpose of matrices.
Set $\eta=\pm1$ and $F_{\eta{q}}=\sigma_{it}^{\frac{1-\eta}{2}}F_q$, then $F_{\eta{q}}$ is a Frobenius endomorphism on $\tbG$ and $\bG$. 
If there is no danger of confusion, we abbreviate $F_{\eta{q}}$ as $F$.
From now on, we set $\tG=\tbG^F$ and $G=\bG^F$, then $\tG=\GL_n(\eta q)$ ($G=\SL_n(\eta q)$) is
the general (special) linear group over the field $\F_q$ or
the general (special) unitary group over the field $\F_{q^2}$ according to $\eta=1$ or $-1$.
Here $\GL_n(-q):=\GU_n(q)=\{A\in\GL_n(q^2)\mid F_q(A)^tA=I_n\}$; similarly for $\SL_n(-q)$ and $\PSL_n(-q)$.
Note that we write $I_n$ for the identity matrix of degree $n$.
Set $D=\langle \sigma_{it},F_p\rangle$ or $\langle F_p\rangle$ according as $n\ge 3$ or $n=2$,
then by \cite[Thm.~2.5.1]{GLS98}, $\tG\rtimes{D}$ induces all automorphisms of $G$; explicitly, $(\tG\rtimes{D})/Z(\tG)\cong\Aut(G)$.

For any subgroup $H$ of $\tG$, we set $\det(H) = \{ \det(h) \mid h\in H \}$, which is a cyclic subgroup of $\barF_q^\times$.
For an integer $k$,
denote by $\fZ_k$ the cyclic subgroup of $\barF_q^\times$ of order $k$.
We identify $\fZ_{q-\eta}$ with $Z(\tG)$ via the natural isomorphism $\fZ_{q-\eta} \to Z(\tG), \zeta \mapsto z_\zeta = \zeta I_n$.

Throughout we assume that the prime $\ell$ is different from $p$.
Denote by $e$ the multiplicative order of $\eta q$ modulo $\ell$.
Note that $e=1$ when $\ell=2$.
The norm map $\cN_\alpha: \fZ_{q^{e\ell^\alpha}-\eta^{e\ell^\alpha}} \to \fZ_{q-\eta}$ is defined as follows:
\begin{equation}\label{eq:norm}
\cN_\alpha: \xi \mapsto \xi F_{\eta q}(\xi) \cdots F_{\eta q}^{e\ell^\alpha-1}(\xi) = \xi^{1+\eta q+\cdots+(\eta q)^{e\ell^\alpha-1}}.
\addtocounter{thm}{1}\tag{\thethm}
\end{equation}

We follow the notation of \cite[\S 1]{FS82}.
Denote by $\F_q[X]$ ($\Irr(\F_q[X])$, resp.) the set of all polynomials (all monic irreducible polynomials, resp.) over the field $\F_q$.
Denote by $d_\Gamma$ the degree of any polynomial $\Gamma$.
For $\Delta(X)=X^m+a_{m-1}X^{m-1}+\cdots+a_0$ in $\F_{q^2}[X]$,
we define $\tDelta(X)=X^ma_0^{-q}\Delta^q(X^{-1})$,
where $\Delta^q(X)$ means the polynomial in $X$ whose coefficients are the $q$-th powers of the corresponding coefficients of $\Delta(X)$.
Then $\zeta$ is a root of $\Delta$ if and only if $\zeta^{-q}$ is a root of $\tDelta$.
Now, we set
\begin{align*}
\cF_0 &= \left\{~ \Delta ~|~ \Delta\in\Irr(\F_q[X]),\Delta\neq X ~\right\},\\
\cF_1 &= \left\{~ \Delta ~|~ \Delta\in\Irr(\F_{q^2}[X]),\Delta\neq X,\Delta=\tDelta ~\right\},\\
\cF_2 &= \left\{~ \Delta\tDelta ~|~ \Delta\in\Irr(\F_{q^2}[X]),\Delta\neq X,\Delta\neq\tDelta ~\right\}
\end{align*}
and define $\cF:=\cF_0$ or $\cF_1\cup\cF_2$ according to $\eta=1$ or $-1$.

Let the Frobenius endomorphism $F=F_{\eta q}$ act on $\barF_q^\times$ as $F (\xi) = \xi^{\eta q}$.
A polynomial $\Gamma\in\cF$ can be identified with the set of roots of $\Gamma$,
which can be again identified with an $F$-orbit $\grp{F}\cdot\xi$ of this action,
where $\xi$ is a root of $\Gamma$; see for example \cite[\S3.1]{De17}.
For any $\zeta\in\fZ_{q-\eta}$ and $\Gamma\in\cF$,
$\zeta.\Gamma$ is defined to be the polynomial in $\cF$ whose roots are the roots of $\Gamma$ multiplied by $\zeta$,
defining an action of $\fZ_{q-\eta}$ on $\cF$.

For a prime $\ell$ different from $p$, we let $\cF'$ be the subset of $\cF$ of those polynomials whose roots are of $\ell'$-order. For a natural number $\alpha$,
the sets $\cF_{\alpha,0},\cF_{\alpha,1},\cF_{\alpha,2},\cF_\alpha,\cF'_\alpha$
are defined similarly with $q,\eta$ replaced by $q^{e\ell^\alpha},\eta^{e\ell^\alpha}$ respectively.
Then we have a surjective map
\begin{equation}\label{eq:Phi_alpha}
\Phi_\alpha:\ \cF_\alpha \to \cF,\ \grp{F_{\eta{q}}^{e\ell^\alpha}}\cdot\xi \mapsto \grp{F_{\eta{q}}}\cdot\xi.
\addtocounter{thm}{1}\tag{\thethm}
\end{equation}
The inverse images of $\grp{F_{\eta{q}}}\cdot\xi$ under $\Phi_\alpha$ are
the $\grp{F_{\eta{q}}}$-orbits of the set $\grp{F_{\eta{q}}^{e\ell^\alpha}}\cdot\xi$.

For any semisimple element $s$ of $\tG$, we let $s=\prod_\Gamma s_\Gamma$ be its primary decomposition.
We denote by $m_\Gamma(s)$ the multiplicity of $\Gamma$ in $s_\Gamma$.
If $m_\Gamma(s)$  is not zero, we call $\Gamma$ an \emph{elementary divisor} of $s$.
For $\zeta\in\fZ_{q-\eta}$, define $\zeta s=z_\zeta s$, then $m_{\zeta.\Gamma}(\zeta s)=m_\Gamma(s)$.

\paragraph{}
For the representations of finite groups of Lie type, see for example \cite{CE04,DM91}.
In particular, if $\tbL$ is a Levi subgroup of $\tbG$, $R_{\tbL}^{\tbG}$ denotes the Lusztig induction.
Note that we always omit the parabolic subgroups when considering Lusztig inductions in this paper,
since the Mackey formula holds for groups of type $\mathsf A$ by  \cite{Bon00}.
Since general linear and unitary groups are self-dual and $Z(\tbL)$ is connected, there is an isomorphism (see for example \cite[(8.19)]{CE04})
\[ Z(\tbL)^F \to \Irr(\tbL^F/[\tbL^F,\tbL^F]),\quad z\mapsto\hz. \]
If $s$ is a semisimple element of $\tG$, then $C_{\tbG}(s)$ is a Levi subgroup of $\tbG$.

By Lusztig’s Jordan decomposition,
the characters of $\tG$ are parameterized by the set $i\Irr(\tG)$ of $\tG$-conjugacy classes of pairs $(s,\mu)$,
where $\mu=\prod_\Gamma\mu_\Gamma$ with $\mu_\Gamma\vdash m_\Gamma(s)$;
see for example \cite[\S1]{FS82} or \cite[\S3.1]{De17}.
The character corresponding to $(s,\mu)^{\tG}$ is denoted as  $\tchi_{s,\mu}$ if no confusion exists.
The Lusztig series $\cE(\tG,s)$ is then the set of all characters of the form $\tchi_{s,\mu}$.
If $s$ is a semisimple $\ell'$-element,
then we denote by $\cE_\ell(\tG,s)$ the union of Lusztig series $\cE(\tG,st)$,
where $t$ runs through the semisimple $\ell$-elements of $\tG$ commutative with $s$.

The group $\fZ_{q-\eta}$ acts on $i\Irr(\tG)$.
For any $(s,\mu)^{\tG}\in i\Irr(\tG)$ and $\zeta\in\fZ_{q-\eta}$,
define $\zeta.(s,\mu)^{\tG}=(\zeta s,\zeta.\mu)^{\tG}$,
where $(\zeta.\mu)_{\zeta.\Gamma}=\mu_\Gamma$.
By \cite[Thm.~7.1]{DM90},
$\widehat{z_\zeta}\tchi_{s,\mu}=\tchi_{\zeta s,\zeta.\mu}$,
thus the action of $\Irr(\tG/G)$ on $\Irr(\tG)$ is induced by the action of $Z(\tG)$ (identifying with $\fZ_{q-\eta}$) on $i\Irr(\tG)$.
See also \cite[\S 3]{De17}.

Let $G\leqslant\hG\leqslant\tG$.
Since $\tG/G$ is cyclic, we have by Clifford theory (see for example \cite[Chap. 3, \S3]{NT89}) that:
$\Irr(\hG)$ is the disjoint union of $\Irr(\hG\mid\tchi_{s,\mu})$,
where $(s,\mu)^{\tG}$ runs over a set of representatives of $\fZ_{|\tG/\hG|}$-orbits on $i\Irr(\tG)$ and
$$ \kappa_{\hG}^{\tG}(\tchi_{s,\mu}) = |\{ \zeta\in\fZ_{|\tG/\hG|} \mid \zeta.(s,\mu)^{\tG}=(s,\mu)^{\tG} \}|. $$

%%%%%%%%%%%%%%%%%%%%%%%%%%%%%%%%%%%%%%%%%%%%
\section{Weights for general linear and unitary groups}\label{sect:weights-general-gp}

In this section, we recall the results in \cite{AF90,An92,An93,An94} for the weights of general linear and unitary groups
with slightly different constructions using the twisted basic subgroups.
The method  to transfer to twisted groups is used in the verification of the inductive McKay condition for type $\mathsf A$ by Cabanes and Sp\"ath \cite{CS17}.
The radical subgroups of general linear and unitary groups are products of some basic subgroups as in \cite{AF90,An92,An93,An94}.
The notation for the basic subgroups of general linear and unitary groups in this paper will be different from those in \cite{AF90,An92,An93,An94},
which are obtained from those in \cite{AF90,An92,An93,An94} by adding~ $\tilde{}$~ on the top,
since notation without~ $\tilde{}$~ are reserved for special linear and unitary groups in this paper.
The essential parts of the basic subgroups are groups of symplectic type.
To give the twisted version, we start with the embedding of groups of symplectic type in general linear and unitary groups; see \cite[\S 4.1]{LZ19}.
We also give a slight improvement of the results in \cite{AF90,An92,An93,An94} about the normalizers of groups of symplectic type;
see Proposition \ref{prop:tN-m,alpha,gamma-odd} and Propositions \ref{prop:tN-m,alpha,gamma-2-linear}--\ref{prop:tN-2-unitary-E-m,gamma}.

\subsection{Groups of symplectic type for odd primes.}
\label{subsect:gst-odd}

Assume $\ell$ is an odd prime different from $p$ and recall that $e$ is the multiplicative order of $\eta q$ modulo $\ell$.
Set $a=\nu((\eta q)^e-1)$, where the discrete valuation $\nu$ is defined as on page \pageref{def-valuation}.
% For any integer $k$, we set $k=k'\ell^{v(k)}$.
For natural numbers $\alpha$ and $\gamma$, let $Z_\alpha$ be the cyclic group of order $\ell^{a+\alpha}$ and $E_\gamma$  the extraspecial group of order $\ell^{2\gamma+1}$ and of exponent $\ell$.
Denote by $Z_\alpha E_\gamma$ the central product of $Z_\alpha$ and $E_\gamma$ over $\Omega_1(Z_\alpha)=Z(E_\gamma)$.
The groups of the form $Z_\alpha E_\gamma$ are called groups of symplectic type.

Assume $m$ is a positive integer and set $\tG_{m,\alpha,\gamma}^0=\GL_{m\ell^\gamma}((\eta q)^{e\ell^\alpha})$.
Let $\zeta_{\ell^{a+\alpha}}\in\barF^\times$ be such that $o(\zeta_{\ell^{a+\alpha}})=\ell^{a+\alpha}$ and set $z_{m,\alpha,\gamma}^0=\zeta_{\ell^{a+\alpha}} I_{m\ell^\gamma}$.
Let $\zeta_\ell=\zeta_{\ell^{a+\alpha}}^{\ell^{a+\alpha-1}}$ and set
$$x^0=\diag\{1,\zeta_\ell,\ldots,\zeta_\ell^{\ell-1}\},\quad
y^0=\begin{bmatrix} \zero & 1\\ I_{\ell-1} & \zero \end{bmatrix}.$$
Then set
$
x_{\gamma,j}^0 = I_\ell \otimes \cdots \otimes x^0 \otimes \cdots \otimes I_\ell$ ($\gamma$ terms) and
$y_{\gamma,j}^0 = I_\ell \otimes \cdots \otimes y^0 \otimes \cdots \otimes I_\ell$ ($\gamma$ terms),
where $x^0$ and $y^0$ appear as the $j$-th components.
Finally set $x_{m,\gamma,j}^0=I_m\otimes x_{\gamma,j}^0$, $y_{m,\gamma,j}^0=I_m\otimes y_{\gamma,j}^0$
and $\tR_{m,\alpha,\gamma}^0 = \grp{ z_{m,\alpha,\gamma}^0, x_{m,\gamma,j}^0, y_{m,\gamma,j}^0 \mid j=1,\dots,\gamma }$.
Then $\tR_{m,\alpha,\gamma}^0$ is a subgroup of $\tG_{m,\alpha,\gamma}^0$ isomorphic to $Z_\alpha E_\gamma$.

Assume $v_\alpha = \begin{bmatrix} \zero&(-1)^{e-1}\\ I_{e\ell^\alpha-1}&\zero \end{bmatrix}$
and $v_{m,\alpha,\gamma}=I_{m\ell^\gamma}\otimes v_\alpha$.
Here, we use the convention for tensor products of matrices that $A \otimes B = (b_{ij}A)_{mn \times mn}$ for $A=(a_{ij})_{m \times m}$ and $B=(b_{ij})_{n \times n}$.
Set $\tbG_{m,\alpha,\gamma}=\GL_{me\ell^{\alpha+\gamma}}(\barF)$ and
$\tG_{m,\alpha,\gamma}^{tw}=\tbG_{m,\alpha,\gamma}^{v_{m,\alpha,\gamma}F_{\eta q}}$.
Then there is an embedding:
\begin{align*}
\hbar:\quad \tG_{m,\alpha,\gamma}^0 &\hookrightarrow \tG_{m,\alpha,\gamma}^{tw}=\tbG_{m,\alpha,\gamma}^{v_{m,\alpha,\gamma}F_{\eta q}}\\
A &\mapsto \diag\left\{A,F_{\eta q}(A),\dots,F_{\eta q}^{e\ell^\alpha-1}(A)\right\}.
\end{align*}
We denote $\tR_{m,\alpha,\gamma}^{tw}=\hbar(\tR_{m,\alpha,\gamma}^0)$, where ``tw'' represents ``twisted''.

Here, by definition, $v_{m,\alpha,\gamma} \in G_{m,\alpha,\gamma} = \SL_{me\ell^{\alpha+\gamma}}(\eta q)$.
By the Lang--Steinberg theorem (see for example \cite[Thm.~21.7]{MT11}),
there exists an element $g_\alpha\in\SL_{e\ell^\alpha}(\barF)$ such that $g_\alpha^{-1}F_{\eta q}(g_\alpha) = v_\alpha$.
Set $g_{m,\alpha,\gamma}=I_{m\ell^\gamma} \otimes g_\alpha$,
then $g_{m,\alpha,\gamma} \in \bG_{m,\alpha,\gamma} := \SL_{me\ell^{\alpha+\gamma}}(\barF)$
and $g_{m,\alpha,\gamma}^{-1}F_{\eta q}(g_{m,\alpha,\gamma})=v_{m,\alpha,\gamma}$.
Denote by $\iota$ the map $x\mapsto g_{m,\alpha,\gamma}xg_{m,\alpha,\gamma}^{-1}$.
Then $\iota$ induces an isomorphism
$$\iota:\quad \tG_{m,\alpha,\gamma}^{tw}=\tbG_{m,\alpha,\gamma}^{v_{m,\alpha,\gamma}F_{\eta q}} \to \tG_{m,\alpha,\gamma}=\tbG_{m,\alpha,\gamma}^{F_{\eta q}}.$$
Set $\tR_{m,\alpha,\gamma}=\iota(\tR_{m,\alpha,\gamma}^{tw})$.

The main constructions above can be summarized as follows
$$\tG_{m,\alpha,\gamma}^0=\GL_{m\ell^\gamma}((\eta{q})^{e\ell^\alpha})
\xlongrightarrow{\hbar} \tG_{m,\alpha,\gamma}^{tw}=\tbG_{m,\alpha,\gamma}^{v_{m,\alpha,\gamma}F_{\eta q}}
\xlongrightarrow{\iota} \tG_{m,\alpha,\gamma}=\tbG_{m,\alpha,\gamma}^{F_{\eta q}}.$$
In the sequel, when we consider constructions in $\tG_{m,\alpha,\gamma}$, we will often transfer to twisted groups, which means to transfer to constructions in $\tG_{m,\alpha,\gamma}^{tw}$ and $\tG_{m,\alpha,\gamma}^0$.

\begin{notation*}\label{notation:tZ-E}
	Set $\tZ_{m,\alpha,\gamma}^0=\grp{z_{m,\alpha,\gamma}^0}$,
	$\tZ_{m,\alpha,\gamma}^{tw}=\hbar(\tZ_{m,\alpha,\gamma}^0)$
	and $\tZ_{m,\alpha,\gamma}=\iota(\tZ_{m,\alpha,\gamma}^{tw})$;
	$E_{m,\alpha,\gamma}^0=\grp{x_{m,\gamma,j}^0,y_{m,\gamma,j}^0\mid j=1,\cdots,\gamma}$,
	$E_{m,\alpha,\gamma}^{tw}=\hbar(E_{m,\alpha,\gamma}^0)$
	and $E_{m,\alpha,\gamma}=\iota(E_{m,\alpha,\gamma}^{tw})$.
	For $*\in\{0,tw,\varnothing\}$, denote by $\tC_{m,\alpha,\gamma}^*$ and $\tN_{m,\alpha,\gamma}^*$
	the centralizer and normalizer respectively of $\tR_{m,\alpha,\gamma}^*$ in $\tG_{m,\alpha,\gamma}^*$.
\end{notation*}

\begin{lem}\label{lem:tC-m,alpha,gamma-odd}
	With the above notation.
	\vspace{-0.5ex}
	\begin{enumerate}[\rm(1)]\setlength{\itemsep}{-0.5ex}
		\item For $*\in\{0,tw,\varnothing\}$, $\det(E_{m,\alpha,\gamma}^*)=1$,
		$\tR_{m,\alpha,\gamma}^* = \tZ_{m,\alpha,\gamma}^*E_{m,\alpha,\gamma}^*$
		is the central product over $Z(E_{m,\alpha,\gamma}^*)$.
		\item $\tC_{m,\alpha,\gamma}^0 = \tC_{m,\alpha}^0 \otimes I_{\ell^\gamma}$,
		$\tC_{m,\alpha,\gamma}^{tw}=\hbar(\tC_{m,\alpha,\gamma}^0)$
		and $\tC_{m,\alpha,\gamma}=\iota(\tC_{m,\alpha,\gamma}^{tw})$.
		\item For $*\in\{0,tw,\varnothing\}$, $\tZ_{m,\alpha,\gamma}^*\leqslant\tC_{m,\alpha,\gamma}^*$
		and thus $\tC_{m,\alpha,\gamma}^*\tR_{m,\alpha,\gamma}^*
		=\tC_{m,\alpha,\gamma}^*E_{m,\alpha,\gamma}^*$ is the central product
		over $Z(E_{m,\alpha,\gamma}^*)$.
	\end{enumerate}
\end{lem}

Let $\Aut^0(Z_\alpha E_\gamma)$ be the group of automorphisms of $Z_\alpha E_\gamma$ acting trivially on $Z_\alpha$.
Set $\Out^0(Z_\alpha E_\gamma)=\Aut^0(Z_\alpha E_\gamma)/\Inn(Z_\alpha E_\gamma)$.
Then by \cite{Win72}, $\Out^0(Z_\alpha E_\gamma) \cong \Sp_{2\gamma}(\ell)$.
To consider the normalizers of subgroups of symplectic type, J. An (\cite[(1A), (1B)]{An94}) considers
realization of $\Aut^0(Z_\alpha E_\gamma)$ in $\GL_{\ell^\gamma}(\eta q)$ by giving a set of generators.
We have given the generators of the group $Z_\alpha E_\gamma$ explicitly in matrix form as above.
Now, we specify the generators of the group $\Aut^0(Z_\alpha E_\gamma)$ modulo $\Inn(Z_\alpha E_\gamma)$ in matrix form
such that most of these matrices have determinant one.

Before giving the list, we first give some elementary properties about the Vandermonde matrix
$$V(\zeta_\ell)= \begin{bmatrix} 1&1&\cdots&1\\ 1&\zeta_\ell&\cdots&\zeta_\ell^{\ell-1}\\ \vdots&\vdots&&\vdots\\ 1&\zeta_\ell^{\ell-1}&\cdots&\zeta_\ell^{(\ell-1)(\ell-1)} \end{bmatrix}.$$
Since $\zeta_\ell$ is a primitive $\ell$-th root of unity, we have
\begin{equation}\label{eq:sum-ell-th-roots}
1+\zeta_\ell+\cdots+\zeta_\ell^{\ell-1}=0.
\addtocounter{thm}{1}\tag{\thethm}
\end{equation}
We consider the determinant of $V(\zeta_\ell)$.
First, assume $\ell\geqslant5$, then
\begin{align*}
\det(V(\zeta_\ell)) = &(\zeta_\ell-1)\cdots(\zeta_\ell^{\ell-1}-1)\cdot(\zeta_\ell^2-\zeta_\ell)\cdots(\zeta_\ell^{\ell-1}-\zeta_\ell)\cdots(\zeta_\ell^{\ell-1}-\zeta_\ell^{\ell-2})\\
=& \zeta_\ell^{\sum_{i=1}^{\ell-2}i(\ell-i-1)} (\zeta_\ell-1)\cdots(\zeta_\ell^{\ell-1}-1)\cdot(\zeta_\ell-1)\cdots(\zeta_\ell^{\ell-2}-1)\cdots(\zeta_\ell-1)\\
=& \zeta_\ell^{\frac{\ell(\ell-1)(\ell-2)}{6}} (\zeta_\ell-1)^{\ell-1}(\zeta_\ell^2-1)^{\ell-2}\cdots(\zeta_\ell^{\ell-1}-1)\\
=& (\zeta_\ell-1)^{\sum_{i=1}^{\ell-1}i} (\zeta_\ell+1)^{\ell-2} \cdots (\zeta_\ell^{\ell-2}+\cdots+1)\\
=& (\zeta_\ell-1)^{\frac{\ell(\ell-1)}{2}} [(\zeta_\ell+1)^{\ell-2}(\zeta_\ell^{\ell-3}+\cdots+1)^2]\\
&\cdots [(\zeta_\ell^{\frac{\ell-3}{2}}+\cdots+1)^{\frac{\ell+1}{2}}(\zeta_\ell^{\frac{\ell-1}{2}}+\cdots+1)^{\frac{\ell-1}{2}}] \cdot (\zeta_\ell^{\ell-2}+\cdots+1)\\
\xlongequal{(\ref{eq:sum-ell-th-roots})} & (\zeta_\ell-1)^{\frac{\ell(\ell-1)}{2}} [(-\zeta_\ell^{\ell-2})^2(\zeta_\ell+1)^\ell] \cdots [(-\zeta_\ell^{\frac{\ell+1}{2}})^\frac{\ell-1}{2}(\zeta_\ell^{\frac{\ell-3}{2}}+\cdots+1)^\ell] \cdot (-\zeta_\ell^{\ell-1})\\
=& (-1)^{\sum_{i=1}^{\frac{\ell-1}{2}}i} (\zeta_\ell-1)^{\frac{\ell(\ell-1)}{2}} \zeta_\ell^{\sum_{i=1}^{\frac{\ell-1}{2}}i(\ell-i)} (\zeta_\ell+1)^\ell \cdots (\zeta_\ell^{\frac{\ell-3}{2}}+\cdots+1)^\ell\\
=& \left( (-1)^{\frac{(\ell-1)(\ell+1)}{8}} (\zeta_\ell-1)^{\frac{\ell-1}{2}} (\zeta_\ell+1) \cdots (\zeta_\ell^{\frac{\ell-3}{2}}+\cdots+1) \right)^\ell\\
=& \left( (-1)^{\frac{(\ell-1)(\ell+1)}{8}} (\zeta_\ell-1) (\zeta_\ell^2-1) \cdots (\zeta_\ell^{\frac{\ell-1}{2}}-1) \right)^\ell.
\end{align*}
Note that $\zeta_\ell^{\frac{\ell(\ell-1)(\ell-2)}{6}}=1$ and $\zeta_\ell^{\sum_{i=1}^{\frac{\ell-1}{2}}i(\ell-i)}=1$ hold only for $\ell\geq5$.
Let $\mu^{-1}=(-1)^{\frac{(\ell-1)(\ell+1)}{8}} (\zeta_\ell-1) (\zeta_\ell^2-1) \cdots (\zeta_\ell^{\frac{\ell-1}{2}}-1)$ and $V_0(\zeta_\ell)=\mu V(\zeta_\ell)$, then $\det(V_0(\zeta_\ell))=1$.
Thus $V_0(\zeta_\ell) \in \SL_\gamma(q^{e\ell^\alpha})$ if $\eta=1$ or $\eta=-1$ and $e$ is even.
If $\eta=-1$ and $e$ is odd, $\zeta_\ell^{q^{e\ell^\alpha}}= \zeta_\ell^{-1}$ since $\ell\mid((\eta q)^{e\ell^\alpha}-1)$, thus
$$\mu^{-(q^{e\ell^\alpha}+1)}=(1-\zeta_\ell)(1-\zeta_\ell^2)\cdots(1-\zeta_\ell^{\ell-1}).$$
From $(X-\zeta_\ell)(X-\zeta_\ell^2)\cdots(X-\zeta_\ell^{\ell-1})=X^{\ell-1}+\cdots+1$, we have $\mu^{-(q^{e\ell^\alpha}+1)}=\ell$.
So
\begin{align*}
F_{q^{e\ell^\alpha}}(V_0(\zeta_\ell))^tV_0(\zeta_\ell)
&=\mu^{q^{e\ell^\alpha}+1}
\begin{bmatrix} 1&1&\cdots&1\\ 1&\zeta_\ell^{-1}&\cdots&\zeta_\ell^{-(\ell-1)}\\ \vdots&\vdots&&\vdots\\ 1&\zeta_\ell^{-(\ell-1)}&\cdots&\zeta_\ell^{-(\ell-1)(\ell-1)} \end{bmatrix}
\begin{bmatrix} 1&1&\cdots&1\\ 1&\zeta_\ell&\cdots&\zeta_\ell^{\ell-1}\\ \vdots&\vdots&&\vdots\\ 1&\zeta_\ell^{\ell-1}&\cdots&\zeta_\ell^{(\ell-1)(\ell-1)} \end{bmatrix}\\
&=\mu^{q^{e\ell^\alpha}+1}\ell I_\ell=I_\ell.
\end{align*}
Thus $V_0(\zeta_\ell) \in \SU_\ell(q^{e\ell^\alpha})$ when $\eta=-1$ and $e$ is odd.

Assume then $\ell=3$.
Thus $\det(V(\zeta_3)) = (\zeta_3-1)(\zeta_3^2-1)(\zeta_3^2-\zeta_3) = -3(1+2\zeta_3) = (1-\zeta_3)^3$.
Set $V_0(\zeta_3) = (1-\zeta_3)^{-1} V(\zeta_3)$,
then it follows that $V_0(\zeta_3) \in \SL_3(q^{e\ell^\alpha})$ when $\eta=1$ or $\eta=-1$ and $e$ is even,
and $V_0(\zeta_3) \in \SU_3(q^{e\ell^\alpha})$ when $\eta=-1$ and $e$ is odd.

Let $W_\alpha$ be a linear space of dimension $\ell$ over the field $\F_{q^{e\ell^\alpha}}$ when $\eta=1$ or $\eta=-1$ and $e$ is even
or a unitary space of dimension $\ell$ over $\F_{q^{2e\ell^\alpha}}$ when $\eta=-1$ and $e$ is odd
such that $W_{\alpha,\gamma}=W_\alpha \otimes\cdots\otimes W_\alpha$ ($\gamma$ terms)
is the underlying space of $\tG_{1,\alpha,\gamma}^0$.
Let $\varepsilon_0,\cdots,\varepsilon_{\ell-1}$ be a basis of $W_\alpha$
such that $\GL(W_{\alpha,\gamma}) \cong \tG_{1,\alpha,\gamma}^0 = \GL_{\ell^\gamma}((\eta q)^{e\ell^\alpha})$
under the basis $\varepsilon_{j_1}\otimes\cdots\otimes\varepsilon_{j_\gamma}$, $j_i=0,\ldots,\ell-1$.
In the sequel, we view the subscript $i$ of $\varepsilon_i$ as element in $\Z/\ell\Z$.
Now, we begin our list of generators.

\begin{enumerate}[(i)]
	\item Set $n_{m,\gamma}^{(i)} = I_m \otimes V_0(\zeta_\ell)\otimes I_{\ell^{\gamma-1}}$,
	then $n_{m,\gamma}^{(i)} \in G_{m,\alpha,\gamma}^0=\SL_{m\ell^\gamma}((\eta q)^{e\ell^\alpha})$ satisfying
	$$n_{m,\gamma}^{(i)} x_{m,\gamma,1}^0 (n_{m,\gamma}^{(i)})^{-1} = (y_{m,\gamma,1}^0)^{-1},\quad
	n_{m,\gamma}^{(i)} y_{m,\gamma,1}^0 (n_{m,\gamma}^{(i)})^{-1}=x_{m,\gamma,1}^0$$
	and fixing other $x_{m,\gamma,k}^0,y_{m,\gamma,k}^0$. \label{n-i}
	
	\item In the situation
	\begin{equation}\label{eq:special-case-3-2}
	\ell=3, \nu(m)=\alpha=0, \gamma=1, a\geq2,
	\addtocounter{thm}{1}\tag{\thethm}
	\end{equation}
	we let $\zeta_9 = \zeta_{3^{a+\alpha}}^{3^{a+\alpha-2}}$ and
	define $n_\gamma^{(ii)}=\diag\{\zeta_9^{-1},\zeta_9^2,\zeta_9^{-1}\}$;
	see the proof of Proposition~4.8 in \cite[p.140]{FLL17a}.
	When we are not in the case (\ref{eq:special-case-3-2}), let $n_\gamma^{(ii)}$ be such that
	$n_\gamma^{(ii)}(\varepsilon_{j_1}\otimes\cdots\otimes\varepsilon_{j_\gamma})
	= \zeta_\ell^{\frac{j_1(j_1+1)}{2}}(\varepsilon_{j_1}\otimes\cdots\otimes\varepsilon_{j_\gamma})$.
	Set $n_{m,\gamma}^{(ii)} = I_m \otimes n_\gamma^{(ii)}$.
	Then $n_{m,\gamma}^{(ii)} \in \tG_{m,\alpha,\gamma}^0$ satisfying
	$n_{m,\gamma}^{(ii)} y_{m,\gamma,1}^0 (n_{m,\gamma}^{(ii)})^{-1} = x_{m,\gamma,1}^0y_{m,\gamma,1}^0$
	and fixing other $x_{m,\gamma,k}^0,y_{m,\gamma,k}^0$.
	Note that $n_{m,\gamma}^{(ii)}$ is a diagonal matrix.
	In the case (\ref{eq:special-case-3-2}),
	$\hbar(n_{m,\gamma}^{(ii)})=n_{m,\gamma}^{(ii)}$ and $\det(n_{m,\gamma}^{(ii)})=1$.
	When we are not in the case (\ref{eq:special-case-3-2}), 
	$$\det(\hbar(n_{m,\gamma}^{(ii)})) 
	= \cN_\alpha\left(\left(\zeta_\ell^{\sum_{j=0}^{\ell-1}\frac{j(j+1)}{2}}\right)^{m\ell^{\gamma-1}}\right)
	= \cN_\alpha\left(\zeta_\ell^{\frac{m\ell^\gamma(\ell-1)(\ell+1)}{6}}\right) .$$
	Thus $\det(\hbar(n_{m,\gamma}^{(ii)}))=1$ unless
	\begin{equation}\label{eq:special-case-3-1}
	\ell=3, a=1, \nu(m)=\alpha=0, \gamma=1,
	\addtocounter{thm}{1}\tag{\thethm}
	\end{equation}
	in which case, $\det\grp{\hbar(n_{m,\gamma}^{(ii)})}=\fZ_3$. \label{detM}
	
	\item Assume $\gamma\geqslant2$.
	Let $P_{\gamma,i}^{(iii)}$ ($i\geqslant2$) be the permutation matrix transposing the pairs
	$(\varepsilon_{j_1} \otimes\cdots\otimes \varepsilon_{j_i} \otimes\cdots\otimes \varepsilon_{j_\gamma},
	\varepsilon_{j_i} \otimes\cdots\otimes \varepsilon_{j_1} \otimes\cdots\otimes \varepsilon_{j_\gamma})$.
	Set $n_{m,\gamma,i}^{(iii)} = I_m \otimes \det(P_{\gamma,i}^{(iii)}) P_{\gamma,i}^{(iii)}$,
	then $n_{m,\gamma,i}^{(iii)}$ satisfies
	$$n_{m,\gamma,i}^{(iii)} x_{m,\gamma,1}^0 (n_{m,\gamma,i}^{(iii)})^{-1}=x_{m,\gamma,i}^0,\quad
	n_{m,\gamma,i}^{(iii)} y_{m,\gamma,1}^0 (n_{m,\gamma,i}^{(iii)})^{-1}=y_{m,\gamma,i}^0$$
	and fixes other $x_{m,\gamma,k}^0,y_{m,\gamma,k}^0$.
	Since $\ell$ is odd, $\det(n_{m,\gamma,i}^{(iii)})=1$ and $n_{m,\gamma,i}^{(iii)} \in G_{m,\alpha,\gamma}^0$.
	
	\item Assume $\gamma\geqslant2$.
	Let $P_{\gamma,1,2}^{(iv)}$ be the permutation matrix such that
	$P_{\gamma,1,2}^{(iv)}(\varepsilon_{j_1}\otimes\varepsilon_{j_2}\otimes\cdots\otimes\varepsilon_{j_\gamma})
	= \varepsilon_{j_1-j_2+1}\otimes\varepsilon_{j_2}\otimes\cdots\otimes\varepsilon_{j_\gamma}$.
	Set $n_{m,\gamma}^{(iv)} = I_m \otimes \det(P_{\gamma,1,2}^{(iv)})P_{\gamma,1,2}^{(iv)}$,
	then $n_{m,\gamma}^{(iv)}$ satisfies
	$$n_{m,\gamma}^{(iv)} x_{m,\gamma,1}^0 (n_{m,\gamma}^{(iv)})^{-1}=x_{m,\gamma,1}^0x_{m,\gamma,2}^0,\quad
	n_{m,\gamma}^{(iv)} y_{m,\gamma,2}^0 (n_{m,\gamma}^{(iv)})^{-1}=(y_{m,\gamma,1}^0)^{-1}y_{m,\gamma,2}^0$$
	and fixes other $x_{m,\gamma,k}^0,y_{m,\gamma,k}^0$.
	Since $\ell$ is odd, $\det(n_{m,\gamma}^{(iv)})=1$ and $n_{m,\gamma}^{(iv)} \in G_{m,\alpha,\gamma}^0$.
	
	\item
	Let $\mu$ be a nonzero element in $\Z/\ell\Z$.
	Assume $P_{\gamma,\mu}^{(v)}$ be the permutation matrix such that
	$P_{\gamma,\mu}^{(v)}(\varepsilon_{j_1}\otimes\varepsilon_{j_2}\otimes\cdots\otimes\varepsilon_{j_\gamma})
	= \varepsilon_{\mu^{-1}j_1}\otimes\varepsilon_{j_2}\otimes\cdots\otimes\varepsilon_{j_\gamma}$.
	Set $n_{m,\gamma,\mu}^{(v)} = I_m \otimes \det(P_{\gamma,\mu}^{(v)}) P_{\gamma,\mu}^{(v)}$,
	then $n_{m,\gamma,\mu}^{(v)}$ satisfies
	$$n_{m,\gamma,\mu}^{(v)} x_{m,\gamma,1}^0 (n_{m,\gamma,\mu}^{(v)})^{-1}=(x_{m,\gamma,1}^0)^\mu,\quad
	n_{m,\gamma,\mu}^{(v)} y_{m,\gamma,1}^0 (n_{m,\gamma,\mu}^{(v)})^{-1}=(y_{m,\gamma,1}^0)^{\mu^{-1}}$$
	and fixes other $x_{m,\gamma,k}^0,y_{m,\gamma,k}^0$.
	Since $\ell$ is odd,
	$\det(n_{m,\gamma,\mu}^{(v)})=1$ and $n_{m,\gamma,\mu}^{(v)} \in G_{m,\alpha,\gamma}^0$.	
\end{enumerate}

\begin{notation}\label{notation:M-m,alpha,gamma-odd}
	Let $\tM_{m,\alpha,\gamma}^0$ be the subgroup of $\tG_{m,\alpha,\gamma}^0$
	generated by $E_{m,\alpha,\gamma}^0$, $n_{m,\gamma}^{(i)}$, $n_{m,\gamma}^{(ii)}$,
	$n_{m,\gamma,i}^{(iii)}$, $n_{m,\gamma}^{(iv)}$, $n_{m,\gamma,\mu}^{(v)}$.
	Set $\tM_{m,\alpha,\gamma}^{tw}=\hbar(\tM_{m,\alpha,\gamma}^0)$.
\end{notation}

\begin{prop}\label{prop:tN-m,alpha,gamma-odd}
	With the Notation~\ref{notation:M-m,alpha,gamma-odd}, we have the following.
	\vspace{-0.5ex}
	\begin{enumerate}[\rm(1)]\setlength{\itemsep}{-0.5ex}
		\item $\tN_{m,\alpha,\gamma}^0$ is the central product of
		$\tC_{m,\alpha,\gamma}^0$ and $\tM_{m,\alpha,\gamma}^0$ over $Z(E_{m,\alpha,\gamma}^0)$;
		similarly for $\hbar(\tN_{m,\alpha,\gamma}^0) = \tC_{m,\alpha,\gamma}^{tw} \tM_{m,\alpha,\gamma}^{tw}$.
		\item $\tM_{m,\alpha,\gamma}^0/E_{m,\alpha,\gamma}^0
		\cong \tM_{m,\alpha,\gamma}^{tw}/E_{m,\alpha,\gamma}^{tw}
		\cong \Sp_{2\gamma}(\ell)$.
		\item $\tN_{m,\alpha,\gamma}^0/\tR_{m,\alpha,\gamma}^0 \cong
		\Sp_{2\gamma}(\ell) \times \tC_{m,\alpha,\gamma}^0/\tZ_{m,\alpha,\gamma}^0$;
		similarly for $\hbar(\tN_{m,\alpha,\gamma}^0)/\tR_{m,\alpha,\gamma}^{tw}$.
		\item $\tN_{m,\alpha,\gamma}^{tw} = \hbar(\tN_{m,\alpha,\gamma}^0) V_{m,\alpha,\gamma}$
		with $V_{m,\alpha,\gamma}=\grp{v_{m,\alpha,\gamma}}$,
		and $\tN_{m,\alpha,\gamma}=\iota(\tN_{m,\alpha,\gamma}^{tw})$.
		\item $\det(\tM_{m,\alpha,\gamma}^{tw})=1$ unless the case (\ref{eq:special-case-3-1}),
		in which case, $\det(\tM_{m,0,1})=\fZ_3$.
	\end{enumerate}
\end{prop}
\begin{proof}
	It is easy to see that $[\tC_{m,\alpha,\gamma}^*,\tM_{m,\alpha,\gamma}^*]=1$.
	By the proof of \cite[(1B)]{An94},
	the map $\tM_{m,\alpha,\gamma}^0 \to \Out^0(Z_\alpha E_\gamma) \cong \Sp_{2\gamma}(\ell)$ is surjective.
	By the construction, $\tC_{m,\alpha,\gamma}^0 \cap \tM_{m,\alpha,\gamma}^0 = Z(E_{m,\alpha,\gamma}^0)$,
	thus we have an isomorphism $\tM_{m,\alpha,\gamma}^0/E_{m,\alpha,\gamma}^0 \cong \Sp_{2\gamma}(\ell)$.
	Then (1) $\sim$ (3) follow easily.
	(4) is the twisted version of \cite[(1C)]{An94}.
	Finally, (5) follows from the definition of $\tM_{m,\alpha,\gamma}^{tw}$,
	noting that  in the case (\ref{eq:special-case-3-1}), $\tM_{m,0,1}$ can be generated by $E_{m,0,1}$ and 
	\[ n_{m,1}^{(i)} = I_m \otimes (1-\zeta_3)^{-1}\begin{bmatrix} 1&1&1\\ 1&\zeta_3&\zeta_3^2\\ 1&\zeta_3^2&\zeta_3 \end{bmatrix},\quad
	n_{m,1}^{(ii)} = I_m \otimes \diag\{ 1,\zeta_3,1 \}. \]
\end{proof}

\begin{rem}\label{M-is-stable-odd}
It can be proved by direct calculation that $\tM_{m,\alpha,\gamma}^{tw}\unlhd \tN_{m,\alpha,\gamma}^{tw}$
and the subgroups $\tR_{m,\alpha,\gamma}^{tw}$,
$\tN_{m,\alpha,\gamma}^{tw}$,
$\tC_{m,\alpha,\gamma}^{tw}$ and 
$\tM_{m,\alpha,\gamma}^{tw}$ are stable under the action of field and graph automorphisms.
\end{rem}

\begin{rmk}
	Part (3) of the Proposition \ref{prop:tN-m,alpha,gamma-odd} improves \cite[(1) of (1C)]{An94} concerning the structure of $N^0$.
	Consequently, the last assertion of \cite[(1) of (1C)]{An94}, which is used in the proof of \cite[(3A)]{An94}, becomes trivial.
	The same remark can be made for the case when $\ell=2$ in the following two subsections.
\end{rmk}

%%%%%%%%%%%%%%%%%%%%%%%%%%%%%%%%%%%%%%%%%%%%%%%%%%%%%%%%%%%%

\subsection{Groups of symplectic type for $\ell=2$ and $4 \mid q-\eta$.}
\label{subsect:gst-2-linear}

Now, assume $\ell=2$, $p$ is an odd prime,
and let $a$ be the positive integer such that $2^{a+1}=(q^2-1)_2$.
Assume $4 \mid q-\eta$ throughout this subsection \S\ref{subsect:gst-2-linear}, then $2^a=(q-\eta)_2$.
Note that $a\geq2$.
For a natural number $\gamma$,
denote by $E_\gamma^+,E_\gamma^-$ the extraspecial group of order $2^{2\gamma+1}$ and of type plus and minus respectively.
In particular, $E_1^+ \cong D_8$ the dihedral group of order $8$ and $E_1^- \cong Q_8$ the quaternion group of order $8$.
For any natural number $\alpha$, denote by $Z_\alpha$ the cyclic group of order $2^{a+\alpha}$.
Note that since $|Z_\alpha|\geq4$, $Z_\alpha E_\gamma^+ \cong Z_\alpha E_\gamma^-$,
which we will denote as $Z_\alpha E_\gamma$ if no confusion is caused.

For a positive integer $m$, let $\tG_{m,\alpha,\gamma}^0 = \GL_{m2^\gamma}((\eta q)^{2^\alpha})$.
Fix a $\zeta_{2^{a+\alpha}} \in \barF_q^\times$ with $o(\zeta_{2^{a+\alpha}}) = 2^{a+\alpha}$
and set $z_{m,\alpha,\gamma}^0 = \zeta_{2^{a+\alpha}} I_{m2^\gamma}$.
Set $\zeta_4 = \zeta_{2^{a+\alpha}}^{2^{a+\alpha-2}}$ and
$\zeta_8=\zeta_{2^{a+\alpha}}^{2^{a+\alpha-3}}$ if $a+\alpha>2$.
Note that $\zeta_{2^{a+\alpha}}^{2^{a+\alpha-1}}=-1$ and set
\[ x^0 = \begin{bmatrix} 1&0\\0&-1 \end{bmatrix},\quad y^0 = \begin{bmatrix} 0&1\\1&0 \end{bmatrix}. \]
When $\gamma\geq2$, set
$x_{\gamma,j}^0 = I_2 \otimes \cdots \otimes x^0 \otimes \cdots \otimes I_2$ ($\gamma$ terms) and
$y_{\gamma,j}^0 = I_2 \otimes \cdots \otimes y^0 \otimes \cdots \otimes I_2$ ($\gamma$ terms),
where $x^0,y^0$ appear as the $j$-th components;
while when $\gamma=1$, set % and $\nu(m)=\alpha=0$, in which case, set
\[ x_{1,1}^0 = \begin{bmatrix} \zeta_4&0\\0&-\zeta_4 \end{bmatrix},\quad y_{1,1}^0 = \begin{bmatrix} 0&1\\-1&0 \end{bmatrix}. \]
Note that $\zeta_4^{-1}=-\zeta_4$.
Finally, set $x_{m,\gamma,j}^0 = I_m \otimes x_{\gamma,j}^0$ and $y_{m,\gamma,j}^0 = I_m \otimes y_{\gamma,j}^0$ as before.
Define $E_{m,\alpha,\gamma}^0, \tZ_{m,\alpha,\gamma}^0, \tR_{m,\alpha,\gamma}^0$ as before.
Then $E_{m,\alpha,\gamma}^0 \cong E_\gamma^+$ when $\gamma\geq2$ and $E_{m,\alpha,1}^0 \cong Q_8$.
Set $\tR_{m,\alpha,\gamma}^0 \cong Z_\alpha E_\gamma$.

Set $\tbG_{m,\alpha,\gamma} = \GL_{m2^{\alpha+\gamma}}(\barF_q)$
and $\bG_{m,\alpha,\gamma} = \SL_{m2^{\alpha+\gamma}}(\barF_q)$.
Set $ v_\alpha = \begin{bmatrix} \zero&-1\\ I_{2^\alpha-1}&\zero \end{bmatrix}$
%\[ v_\alpha =
%\begin{cases}
%	\begin{bmatrix} \zero&-1\\ I_{2^\alpha-1}&\zero \end{bmatrix}, & \textrm{if $\alpha>0$ and $\nu(m)=\gamma=0$},\\
%	\begin{bmatrix} \zero&1\\ I_{2^\alpha-1}&\zero \end{bmatrix}, & \textrm{otherwise},
%\end{cases}\]
and $v_{m,\alpha,\gamma}=I_{m2^\gamma}\otimes v_\alpha$.
Define $\hbar,g_{m,\alpha,\gamma},\iota$ as before,
and let the notation $\tG_{m,\alpha,\gamma}^*$, $G_{m,\alpha,\gamma}^*$,
$E_{m,\alpha,\gamma}^*$, $\tZ_{m,\alpha,\gamma}^*$, $\tR_{m,\alpha,\gamma}^*$,
$\tC_{m,\alpha,\gamma}^*$, $\tN_{m,\alpha,\gamma}^*$, etc. for $*\in\set{0,tw,\varnothing}$ be as before.
The conclusions of  Lemma \ref{lem:tC-m,alpha,gamma-odd} hold.

Let $\Aut^0(Z_\alpha E_\gamma)$ and $\Out^0(Z_\alpha E_\gamma)$ be as before.
Then by \cite{Win72}, $\Out^0(Z_\alpha E_\gamma) \cong \Sp_{2\gamma}(2)$.
To consider the normalizers of subgroups of symplectic type, J. An (\cite{An92,An93}) considers
realization of $\Aut^0(Z_\alpha E_\gamma)$ in $\GL_{2^\gamma}(\eta q)$ by giving a set of generators.
We have given the generators of the group $Z_\alpha E_\gamma$ explicitly in matrix form as above.
Now, we specify the generators of the group $\Aut^0(Z_\alpha E_\gamma)$ modulo $\Inn(Z_\alpha E_\gamma)$ in matrix form
such that most of these matrices have determinant one.

Let $W_\alpha$ be a linear space of dimension $2$ over the field $\F_{q^{2^\alpha}}$ when $\eta=1$ or $\eta=-1$ and $\alpha>0$
or a unitary space over $\F_{q^2}$ when $\eta=-1$ and $\alpha=0$ such that
$W_{\alpha,\gamma}=W_\alpha \otimes\cdots\otimes W_\alpha$ ($\gamma$ terms) is the underlying space of $\tG_{1,\alpha,\gamma}$.
Let $\varepsilon_0,\varepsilon_1$ be a basis of $W_\alpha$
such that $\GL(W_{\alpha,\gamma}) \cong \tG_{1,\alpha,\gamma}^0 = \GL_{2^\gamma}((\eta q)^{2^\alpha})$
with respect to the basis $\varepsilon_{j_1}\otimes\cdots\otimes\varepsilon_{j_\gamma}$, $j_i=0,1$.

Now, we consider the normalizer of $\tR_{m,\alpha,\gamma}^0$ in $\tG_{m,\alpha,\gamma}^0$.
Assume $\gamma=1$ first.
Let $\tn^{(i)} = \begin{bmatrix} -1&\zeta_4\\-1&-\zeta_4 \end{bmatrix}$, then
$ \tn^{(i)}x_{1,1}^0(\tn^{(i)})^{-1}=x_{1,1}^0y_{1,1}^0$ and $\tn^{(i)}y_{1,1}^0(\tn^{(i)})^{-1}=x_{1,1}^0$.
Note that $\det\tn^{(i)} = 2\zeta_4$.

\begin{lem}
	$2\zeta_4$ is a square in $\F_q^\times$ if $4 \mid q-1$,
	and $2\zeta_4$ is a square in $\F_{q^2}^\times$.
\end{lem}
\begin{proof}
	By Euler's quadratic reciprocity law, $2$ is a square in $\F_q^\times$ if and only if $q\equiv\pm1 \mod 8$.
	Then if $4 \mid q-1$, $2$ and $\zeta_4$ are both squares or both non-squares, so $2\zeta_4$ is a square in $\F_q^\times$.
	Since $q^2\equiv1\mod8$, both $2$ and $\zeta_4$ are squares in $\F_{q^2}^\times$ and so is $2\zeta_4$.
\end{proof}

Assume $4 \mid q-\eta$ and
take $\lambda\in\F_{q^{2^\alpha}}^\times$ when $\eta=1$ or $\eta=-1$ and $\alpha>0$
or take $\lambda\in\F_{q^2}^\times$ when $\eta=-1$ and $\alpha=0$
such that $\lambda^2=2\zeta_4$.

\begin{lem}
	Assume $\eta=-1$, $\alpha=0$ and let $\lambda$ be as above.
	\vspace{-0.5ex}
	\begin{enumerate}[\rm(1)]\setlength{\itemsep}{-0.5ex}
		\item $2^{\frac{q-1}{2}}=1$ if $a>2$ while $2^{\frac{q-1}{2}}=-1$ if $a=2$.
		\item $\lambda^{q+1}=2$.
	\end{enumerate}
\end{lem}
\begin{proof}
	Assume $a>2$, then $q\equiv-1\mod8$,
	so $2$ is square in $\F_q^\times$,
	thus $2^{\frac{q-1}{2}}=1$.
	Now, assume $a=2$.
	By assumption, $4 \mid q+1$, then $q\equiv3\mod8$.
	So $2$ is again a non-square in $\F_q^\times$.
	Since $\frac{q-1}{2}$ is odd, $2$ would be a square if $2^{\frac{q-1}{2}}=1$,
	thus $2^{\frac{q-1}{2}}=-1$.
	Recall that $\lambda^2=2\zeta_4$, then (2) follows from (1) by easy calculation. 
	%If $8 \mid q+1$, then by (1),
	%\[ \lambda^{q+1} = (\lambda^2)^{\frac{q+1}{2}} = (2\zeta_4)^{\frac{q+1}{2}} = 2^{\frac{q+1}{2}} = 2^{q-\frac{q-1}{2}} = 2. \]
	%If $8 \mid q+1$, then by (1),
	%\[ \lambda^{q+1} = (\lambda^2)^{\frac{q+1}{2}} = (2\zeta_4)^{\frac{q+1}{2}} = -2^{\frac{q+1}{2}} = -2^{q-\frac{q-1}{2}} = 2. \]
\end{proof}

We now consider the normalizer of $\tR_{m,\alpha,\gamma}^*$ in $\tG_{m,\alpha,\gamma}^*$ for $*\in\set{0,tw,\varnothing}$. 

\begin{enumerate}[(1)]
	
	\item Assume $\gamma=1$.
	
	\begin{enumerate}[(i)]
		\item Set $n_{m,1}^{(i)} = I_m \otimes \frac{1}{\lambda}\tn^{(i)}$,
		then by the above two lemmas, $n_{m,1}^{(i)} \in \SU_{2m}((\eta q)^{2^\alpha})$,
		$ n_{m,1}^{(i)}x_{m,1,1}^0(n_{m,1}^{(i)})^{-1}=x_{m,1,1}^0y_{m,1,1}^0$ and  $n_{m,1}^{(i)}y_{m,1,1}^0(n_{m,1}^{(i)})^{-1}=x_{m,1,1}^0$. 
		\item Set
		\[ n_{m,1}^{(ii)} = \begin{cases}
		I_m \otimes \begin{bmatrix} \zeta_8&0\\0&\zeta_8^{-1} \end{bmatrix}, & \textrm{if $\alpha\geq1$ or $a>2$},\\
		\diag\{ \zeta_4, I_{m-1} \} \otimes \begin{bmatrix} \zeta_4&0\\0&1 \end{bmatrix},
		& \textrm{if $\alpha=0$, $a=2$ and $\nu(m)=1$},\\
		I_m \otimes \begin{bmatrix} \zeta_4&0\\0&1 \end{bmatrix}, & \textrm{if $\alpha=0$, $a=2$ and $\nu(m)\neq1$}.
		\end{cases}\]
		Then $ n_{m,1}^{(ii)}x_{m,1,1}^0(n_{m,1}^{(ii)})^{-1}=x_{m,1,1}^0$ and
		$n_{m,1}^{(ii)}y_{m,1,1}^0(n_{m,1}^{(ii)})^{-1}=x_{m,1,1}^0y_{m,1,1}^0$. 
		Note that $\det \hbar(n_{m,1}^{(ii)}) = 1$ unless $\alpha=\nu(m)=0$ and $a=2$,
		in which case, $\det \hbar(n_{m,\gamma}^{(ii)}) = \zeta_4^m$.
	\end{enumerate}
	
	\item Assume $\gamma\geq2$.
	
	\begin{enumerate}[(i)]
		\item Set $n_{m,\gamma}^{(i)}
		= I_m \otimes \frac{1}{\lambda}\begin{bmatrix} 1&1\\-\zeta_4&\zeta_4 \end{bmatrix} \otimes I_{2^{\gamma-1}}$,
		where $\lambda$ is as above,
		then similar as above, $n_{m,\gamma}^{(i)} \in \SU_{m2^\gamma}((\eta q)^{2^\alpha})$ satisfying
		$ n_{m,\gamma}^{(i)}x_{m,\gamma,1}^0(n_{m,\gamma}^{(i)})^{-1}=\zeta_4x_{m,\gamma,1}^0y_{m,\gamma,1}^0$,
		$n_{m,\gamma}^{(i)}y_{m,\gamma,1}^0(n_{m,\gamma}^{(i)})^{-1}=x_{m,\gamma,1}^0$
		and fixing all other $x_{m,\gamma,k}^0, y_{m,\gamma,k}^0$.
		\item Set
		\[ n_{m,\gamma}^{(ii)} = \begin{cases}
		I_m \otimes \begin{bmatrix} \zeta_8&0\\0&\zeta_8^{-1} \end{bmatrix} \otimes I_{2^{\gamma-1}},
		& \textrm{if $\alpha\geq1$ or $a>2$},\\
		I_m \otimes \begin{bmatrix} \zeta_4&0\\0&1 \end{bmatrix} \otimes I_{2^{\gamma-1}},
		& \textrm{if $\alpha=0$ and $a=2$}.
		\end{cases}\]
		Then $n_{m,\gamma}^{(ii)}$ satisfies
		$n_{m,\gamma}^{(ii)}y_{m,\gamma,1}^0(n_{m,\gamma}^{(ii)})^{-1}=\zeta_4x_{m,\gamma,1}^0y_{m,\gamma,1}^0$
		and fixes all other $x_{m,\gamma,k}^0, y_{m,\gamma,k}^0$.
		Note that $\det \hbar(n_{m,\gamma}^{(ii)}) = 1$ unless $\alpha=\nu(m)=0$, $\gamma=2$ and $a=2$,
		in which case $\det \hbar(n_{m,\gamma}^{(ii)})=-1$.
		\item Let $P_{\gamma,i}^{(iii)}$ be the permutation matrix transposing all the pairs
		$(\varepsilon_{j_1}\otimes\cdots\otimes\varepsilon_{j_i}\otimes\cdots\otimes\varepsilon_{j_\gamma},
		\varepsilon_{j_i}\otimes\cdots\otimes\varepsilon_{j_1}\otimes\cdots\otimes\varepsilon_{j_\gamma})$.
		Note that $\det P_{\gamma,i}^{(iii)} = (-1)^{2^{\gamma-2}}$.
		Set \[ n_{m,\gamma,i}^{(iii)} = \begin{cases}
		\diag\{ \zeta_8, I_{m-1} \} \otimes P_{2,i}^{(iii)}, & \textrm{if $\alpha=\nu(m)=0, \gamma=2, a>2$},\\
		I_m \otimes P_{\gamma,i}^{(iii)}, & \textrm{otherwise}.
		\end{cases}\]
		Then $n_{m,\gamma,i}^{(iii)}$ transposes pairs $(x_{m,\gamma,1}^0,x_{m,\gamma,i}^0)$, $(y_{m,\gamma,1}^0,y_{m,\gamma,i}^0)$
		and fixes all other $x_{m,\gamma,k}^0,y_{m,\gamma,k}^0$.
		Note that $\det \hbar(n_{m,\gamma}^{(iii)}) = 1$ unless $\alpha=\nu(m)=0$, $\gamma=2$ and $a=2$,
		in which case $\det \hbar(n_{m,2}^{(iii)}) = -1$.
		\item Let $P_{\gamma,1,2}^{(iv)}$ be the permutation matrix transposing all the pairs
		$(\varepsilon_{j_1}\otimes\varepsilon_{j_2}\otimes\cdots\otimes\varepsilon_{j_\gamma},
		\varepsilon_{j_1+j_2+1}\otimes\varepsilon_{j_2}\otimes\cdots\otimes\varepsilon_{j_\gamma})$.
		Note that $\det P_{\gamma,1,2}^{(iv)} = (-1)^{2^{\gamma-2}}$.
		Set \[ n_{m,\gamma}^{(iv)} = \begin{cases}
		\diag\{ \zeta_8, I_{m-1} \} \otimes P_{\gamma,1,2}^{(iv)}, & \textrm{if $\alpha=\nu(m)=0, \gamma=2, a>2$},\\
		I_m \otimes P_{\gamma,1,2}^{(iv)}, & \textrm{otherwise}. \end{cases}\]
		Then $n_{m,\gamma}^{(iv)}$ satisfies
		$ n_{m,\gamma}^{(iv)} x_{m,\gamma,1}^0(n_{m,\gamma}^{(iv)})^{-1}=x_{m,\gamma,1}^0x_{m,\gamma,2}^0$,
		$n_{m,\gamma}^{(iv)} y_{m,\gamma,2}^0(n_{m,\gamma}^{(iv)})^{-1}=y_{m,\gamma,1}^0y_{m,\gamma,2}^0$
		and fixes all other $x_{m,\gamma,k}^0,y_{m,\gamma,k}^0$.
		Note that $\det \hbar(n_{m,\gamma}^{(iv)}) = 1$ unless $\alpha=\nu(m)=0$, $\gamma=2$ and $a=2$,
		in which case $\det \hbar(n_{m,2}^{(iv)}) = -1$.
	\end{enumerate}
	
\end{enumerate}

\begin{notation*}
	Let $\tM_{m,\alpha,\gamma}^0$ be the subgroup of $\tG_{m,\alpha,\gamma}^0$ generated by $n_{m,\gamma,i}^{(i)}$, $n_{m,\gamma}^{(ii)}$, 	$n_{m,\gamma,\nu}^{(iii)}$, $n_{m,\gamma}^{(iv)}$ and $E_{m,\alpha,\gamma}^0$.
	Set $\tM_{m,\alpha,\gamma}^{tw}=\hbar(\tM_{m,\alpha,\gamma}^0)$ and $\tM_{m,\alpha,\gamma}=\iota(\tM_{m,\alpha,\gamma}^{tw})$.
	We also set
	\begin{gather}
	\label{eq:special-case-2-linear-0}
	\ell=2, 4 \mid q-\eta, \alpha=0 ~\textrm{and}~ \gamma=1, \nu(m)=1, a=2~\textrm{or}~\gamma=2,\nu(m)=0,a>2;
	\addtocounter{thm}{1}\tag{\thethm}\\
	\label{eq:special-case-2-linear-1} \ell=2, 4 \mid q-\eta, a=2, \nu(m)=\alpha=0, 1\leq\gamma\leq2, \addtocounter{thm}{1}\tag{\thethm}
	\end{gather}
\end{notation*}

Then similar as in subsection \S \ref{subsect:gst-odd}, we have the following from \cite{An92,An93}.

\begin{prop}\label{prop:tN-m,alpha,gamma-2-linear}
	Assume $\ell=2$, $4 \mid q-\eta$ and keep the above notation.
	\vspace{-0.5ex}
	\begin{enumerate}[\rm(1)]\setlength{\itemsep}{-0.5ex}
		\item $\tC_{m,\alpha,\gamma}^0 = \GL_m((\eta q)^{2^\alpha}) \otimes I_{2^\gamma}$.
		\item We have $C_{\tM_{m,\alpha,\gamma}^0}(E_{m,\alpha,\gamma}^0)=Z(E_{m,\alpha,\gamma})$ and
		$\tM_{m,\alpha,\gamma}^0/E_{m,\alpha,\gamma}^0 \cong \Sp_{2\gamma}(2)$,
		unless (\ref{eq:special-case-2-linear-0}) holds, in which case,
		$\tM_{m,\alpha,\gamma}^0/C_{\tM_{m,\alpha,\gamma}^0}(E_{m,\alpha,\gamma}^0)E_{m,\alpha,\gamma}^0 \cong \Sp_{2\gamma}(2)$.
		Similar results hold for $\tM_{m,\alpha,\gamma}^{tw}$.
		\item $\tN_{m,\alpha,\gamma}^0 = \tC_{m,\alpha,\gamma}^0\tM_{m,\alpha,\gamma}^0$ is
		the central product of $\tC_{m,\alpha,\gamma}^0$ and $\tM_{m,\alpha,\gamma}^0$ over $Z(\tR_{m,\alpha,\gamma}^0)$,
		unless (\ref{eq:special-case-2-linear-0}) holds, in which case,
		$\tN_{m,\alpha,\gamma}^0 = \tC_{m,\alpha,\gamma}^0\tM_{m,\alpha,\gamma}^0$
		with $\tC_{m,\alpha,\gamma}^0 \cap \tM_{m,\alpha,\gamma}^0$ being a $2$-group.
		Similar results hold for $\hbar(\tN_{m,\alpha,\gamma}^0) = \tC_{m,\alpha,\gamma}^{tw} \tM_{m,\alpha,\gamma}^{tw}$.
		\item $\tN_{m,\alpha,\gamma}^0/\tR_{m,\alpha,\gamma}^0 \cong
		\Sp_{2\gamma}(2) \times \tC_{m,\alpha,\gamma}^0/\tZ_{m,\alpha,\gamma}^0$,
		unless (\ref{eq:special-case-2-linear-0}) holds, in which case,
		$\tN_{m,\alpha,\gamma}^0/\tC_{m,\alpha,\gamma}^0\tR_{m,\alpha,\gamma}^0 \cong \Sp_{2\gamma}(2)$.
		Similar results hold for $\hbar(\tN_{m,\alpha,\gamma}^0)$.
		\item $\tN_{m,\alpha,\gamma}^{tw} = \hbar(\tN_{m,\alpha,\gamma}^0) V_{m,\alpha,\gamma}$
		with $V_{m,\alpha,\gamma}=\grp{v_{m,\alpha,\gamma}}$,
		and $\tN_{m,\alpha,\gamma}=\iota(\tN_{m,\alpha,\gamma}^{tw})$.
		\item $\det(\tM_{m,\alpha,\gamma}^{tw})=1$ unless (\ref{eq:special-case-2-linear-1}) holds,
		in which case, $\det(\tM_{m,0,1}^{tw})=\fZ_4^{2^{\gamma-1}}$.
	\end{enumerate}
\end{prop}

\begin{rem}\label{not-cent}
If we change the definitions slightly as follows:
\begin{enumerate}[(1)]
	\item When $\gamma=1,\alpha=0,a=2,\nu(m)=1$, set $n_{m,1}^{(ii)}=I_m\otimes\begin{bmatrix}\zeta_4&0\\0&1\end{bmatrix}$,
	\item When $\gamma=2,\alpha=\nu(m)=0,a>2$,
		set $n_{m,2,i}^{(iii)} = I_m \otimes P_{2,i}^{(iii)}$ and $n_{m,2}^{(iv)} = I_m \otimes P_{2,1,2}^{(iv)}$,
\end{enumerate}
and define $\tM_{m,\alpha,\gamma}^0,\tM_{m,\alpha,\gamma}^{tw},\tM_{m,\alpha,\gamma}$ as above,
then the statements in (2) (3) (4) can be simplified by deleting ``unless (\ref{eq:special-case-2-linear-0}), \ldots''.
Note that with this new definition, $\det (\tM_{m,\alpha,\gamma}^{tw})=\{\pm1\}$ in case (\ref{eq:special-case-2-linear-0}).
\end{rem}

\begin{rem}\label{M-is-stable-2-1}
	It can be proved by direct calculation that $\tM_{m,\alpha,\gamma}^{tw}\unlhd \tN_{m,\alpha,\gamma}^{tw}$
	and the groups $\tR_{m,\alpha,\gamma}^{tw}$,
	$\tN_{m,\alpha,\gamma}^{tw}$,
	$\tC_{m,\alpha,\gamma}^{tw}$ and 
	$\tM_{m,\alpha,\gamma}^{tw}$ are stable under the action of field and graph automorphisms.
\end{rem}

%%%%%%%%%%%%%%%%%%%%%%%%%%%%%%%%%%%%%%%%%%%%%%%%%%%%%%%%%%%%

\subsection{Groups of symplectic type for $\ell=2$ and $4 \mid q+\eta$.}
\label{subsect:gst-2-unitary}

Assume $\ell=2$, $p$ is an odd prime, and let $a$ be as in \S\ref{subsect:gst-2-linear}.
Assume $4 \mid q+\eta$ throughout this subsection \S\ref{subsect:gst-2-unitary}, then $2^a=(q+\eta)_2$ and $(q-\eta)_2=2$.
Let $E_\gamma^+,E_\gamma^-$ for a natural integer $\gamma$ be as in \S\ref{subsect:gst-2-linear},
and let $Z_\alpha$ for a positive integer $\alpha$ be as in \S\ref{subsect:gst-2-linear},
but let $Z_0$ denote the cyclic group of order $2$.

First, for $\alpha>0$, define $\tR_{m,\alpha,\gamma}$ and all related notation as in \S\ref{subsect:gst-2-linear}.
Then all the results in \S\ref{subsect:gst-2-linear} continue to hold.
Note that since $\alpha>0$, the exceptional cases (\ref{eq:special-case-2-linear-0}) and (\ref{eq:special-case-2-linear-1}) can not happen.
For convenience, we restate these results as follows.

\begin{prop}\label{prop:tN-m,alpha,gamma-2-unitary}
	Assume $\ell=2$, $4 \mid q+\eta$ and $\alpha>0$.
	\vspace{-0.5ex}
	\begin{enumerate}[\rm(1)]\setlength{\itemsep}{-0.5ex}
		\item $\tC_{m,\alpha,\gamma}^0 = \GL_m(q^{2^\alpha}) \otimes I_{2^\gamma}$.
		$\tN_{m,\alpha,\gamma}^0 = \tC_{m,\alpha,\gamma}^0\tM_{m,\alpha,\gamma}^0$
		is the central product over $Z(E_{m,\alpha,\gamma}^0)$,
		and similarly for $\hbar(\tN_{m,\alpha,\gamma}^0) = \tC_{m,\alpha,\gamma}^{tw} \tM_{m,\alpha,\gamma}^{tw}$.
		\item $\tM_{m,\alpha,\gamma}^0/E_{m,\alpha,\gamma}^0
		\cong \tM_{m,\alpha,\gamma}^{tw}/E_{m,\alpha,\gamma}^{tw} \cong \Sp_{2\gamma}(2)$.
		\item $\tN_{m,\alpha,\gamma}^0/\tR_{m,\alpha,\gamma}^0 \cong
		\Sp_{2\gamma}(2) \times \tC_{m,\alpha,\gamma}^0/\tZ_{m,\alpha,\gamma}^0$;
		similarly for $\hbar(\tN_{m,\alpha,\gamma}^0)/\tR_{m,\alpha,\gamma}^{tw}$.
		\item $\tN_{m,\alpha,\gamma}^{tw} = \hbar(\tN_{m,\alpha,\gamma}^0) \rtimes V_{m,\alpha,\gamma}$
		with $V_{m,\alpha,\gamma}=\grp{v_{m,\alpha,\gamma}}$,
		and $\tN_{m,\alpha,\gamma} = \iota(\tN_{m,\alpha,\gamma}^{tw})$.
		\item $\det(\tM_{m,\alpha,\gamma}^{tw})=1$.
	\end{enumerate}
\end{prop}

Then assume $\gamma\geq1$.
Let $E_{m,1,\gamma-1}^{tw}$ and $z_{m,1,\gamma-1}^0$ be as in \S\ref{subsect:gst-2-linear},
and set $z_{m,1,\gamma-1}^{tw} = \hbar(z_{m,1,\gamma-1}^0)$,
$\tau_{m,1,\gamma-1}^{tw} = v_{m,1,\gamma-1} = I_{m2^{\gamma-1}} \otimes \begin{bmatrix} 0&1\\ -1&0 \end{bmatrix}$.
Set $\tS_{m,1,\gamma-1}^{tw} = \grp{E_{m,1,\gamma-1}^{tw},z_{m,1,\gamma-1}^{tw},\tau_{m,1,\gamma-1}^{tw}}$,
then $\tS_{m,1,\gamma-1}^{tw} \leq \tG_{m,1,\gamma-1}^{tw}$.
Note that $\grp{z_{m,1,\gamma-1}^{tw},\tau_{m,1,\gamma-1}^{tw}}$ is isomorphic to the semidihedral group $S_{a+2}$ of order $2^{a+2}$.
By the definition, $E_{m,1,\gamma-1}^{tw} = E_{m,1,\gamma-1}^0 \otimes I_2$,
$z_{m,1,\gamma-1}^{tw} = I_{m2^{\gamma-1}} \otimes \diag\{\zeta_{2^{a+1}},\zeta_{2^{a+1}}^{\eta q}\}$,
thus $\tS_{m,1,\gamma-1}^{tw} = E_{m,1,\gamma-1}^{tw}\grp{z_{m,1,\gamma-1}^{tw},\tau_{m,1,\gamma-1}^{tw}}$
is a central product of $E_{m,1,\gamma-1}^{tw}$ and $\grp{z_{m,1,\gamma-1}^{tw},\tau_{m,1,\gamma-1}^{tw}}$.
Note that $\tbG_{m,0,\gamma} = \tbG_{m,1,\gamma-1} = \GL_{m2^\gamma}(\barF_q)$,
$\bG_{m,0,\gamma} = \bG_{m,1,\gamma-1} = \SL_{m2^\gamma}(\barF_q)$.
Let $g_{m,1,\gamma-1}$ and $\iota$ be as in \S\ref{subsect:gst-2-linear},
and set $\tS_{m,1,\gamma-1} = \iota(\tS_{m,1,\gamma-1}^{tw})$.
Then $\tS_{m,1,\gamma-1}$ is a subgroup of $\tG_{m,0,\gamma}=\tG_{m,1,\gamma-1}$
isomorphic to the central product $E_{\gamma-1}S_{a+2}$,
which is independent of the type of $E_{\gamma-1}$ by \cite[(1F)]{An92}.
Set $z_{m,1,\gamma-1}=\iota(z_{m,1,\gamma-1}^{tw})$, $\tau_{m,1,\gamma-1}=\iota(\tau_{m,1,\gamma-1}^{tw})$.
Note that $\det z_{m,1,\gamma-1} = \det z_{m,1,\gamma-1}^{tw} = (-1)^{m2^{\gamma-1}}$
and $\det \tau_{m,1,\gamma-1} = \det \tau_{m,1,\gamma-1}^{tw} = 1$.

\begin{prop}\label{prop:tN-2-unitary-S-m,gamma}
	Assume $\ell=2$, $4 \mid q+\eta$, $\gamma>0$ and keep the above notation.
	Set $\tC^{tw} = C_{\tG_{m,1,\gamma-1}^{tw}}(\tS_{m,1,\gamma-1}^{tw})$,
	$\tN^{tw} = N_{\tG_{m,1,\gamma-1}^{tw}}(\tS_{m,1,\gamma-1}^{tw})$
	and $\tN_0^{tw} = C_{\tN^{tw}}(z_{m,1,\gamma-1}^{tw})$.
	\begin{enumerate}[\rm(1)]\setlength{\itemsep}{-2pt}
		\item $\tC^{tw} = \GL_m(\eta q) \otimes I_{2^\gamma}$.
		\item $\tN_0^{tw} = \tC^{tw} \grp{\tM_{m,1,\gamma-1}^{tw},z_{m,1,\gamma-1}^{tw}}$
		with $\tM_{m,1,\gamma-1}^{tw}$ as in \S\ref{subsect:gst-2-linear}.		
		\item $\tN^{tw} = \grp{\tN_0^{tw},\tau_{m,1,\gamma-1}^{tw}} = \tC^{tw} \tS_{m,1,\gamma-1} \tM_{m,1,\gamma-1}^{tw}$ and
		$\tN^{tw}/\tS_{m,1,\gamma}^{tw} \cong \Sp_{2(\gamma-1)}(2) \times \GL_m(\eta q)/\mrO_2(Z(\GL_m(\eta q)))$.
	\end{enumerate}
\end{prop}

\begin{proof}
	These are twisted versions of refinements of results in \cite{An92,An93}, which can be proved similarly as before,
	noting that $\tC^{tw} = C_{\tC_{m,1,\gamma-1}^{tw}}(\tau_{m,1,\gamma-1}^{tw})$
	with $\tC_{m,1,\gamma-1}^{tw}$ being as in \S\ref{subsect:gst-2-linear}.
\end{proof}

\begin{rem}\label{M-is-stable-2-2}
	It can be proved by direct calculation that $\tM_{m,\alpha,\gamma}^{tw}\unlhd \tN_{m,\alpha,\gamma}^{tw}$
	and the groups $\tR_{m,\alpha,\gamma}^{tw}$,
	$\tN_{m,\alpha,\gamma}^{tw}$,
	$\tC_{m,\alpha,\gamma}^{tw}$ and 
	$\tM_{m,\alpha,\gamma}^{tw}$ are stable under the action of field and graph automorphisms.
\end{rem}

Now, set
\[ x^+ = \begin{bmatrix} 1&0\\0&-1 \end{bmatrix},\quad y^+ = \begin{bmatrix} 0&1\\1&0 \end{bmatrix}, \]
and set
$
x_{\gamma,j}^+ = I_2 \otimes \cdots \otimes x^+ \otimes \cdots \otimes I_2$ ($\gamma$ terms) and
$y_{\gamma,j}^+ = I_2 \otimes \cdots \otimes y^+ \otimes \cdots \otimes I_2$ ($\gamma$ terms),
where $x^+,y^+$ appear as the $j$-th component.
Set $x_{m,\gamma,j}^+ = I_m \otimes x_{\gamma,j}^+$, $y_{m,\gamma,j}^+ = I_m \otimes y_{\gamma,j}^+$
and $E_{m,\gamma}^+ = \grp{ x_{m,\gamma,j}^+,y_{m,\gamma,j}^+ \mid j=1,\ldots,\gamma }$.
Then $E_{m,\gamma}^+$ is isomorphic to the extra-special $2$-group of plus type.
Set $x_{m,\gamma,j}^- = x_{m,\gamma,j}^+$, $y_{m,\gamma,j}^- = y_{m,\gamma,j}^+$ for $j\leq\gamma-1$
and $x_{m,\gamma,\gamma}^- = z_{m,1,\gamma}^{2^{a-1}}$, $y_{m,\gamma,\gamma}^- = \tau_{m,1,\gamma-1}$.
Set $E_{m,\gamma}^- = \grp{ x_{m,\gamma,j}^-,y_{m,\gamma,j}^- \mid j=1,\ldots,\gamma }$,
then $E_{m,\gamma}^-$ is isomorphic to the extra-special $2$-group of minus type.
We then consider the normalizers of $E_{m,\gamma}^{\pm}$ in $\tG_{m,0,\gamma}=\GL_{m2^\gamma}(\eta q)$.
Note that $\det E_{m,\gamma}^\pm=1$ except that $\det E_{m,1}^+=-1$ when $\nu(m)=0$.

\begin{prop}\label{prop:tN-2-unitary-E-m,gamma}
	Assume $\ell=2$, $4 \mid q+\eta$, $\gamma>0$ and keep the above notation.
	Set $\tC_{m,0,\gamma}^\pm = C_{\tG_{m,0,\gamma}}(E_{m,\gamma}^\pm)$,
	$\tN_{m,0,\gamma}^\pm = N_{\tG_{m,0,\gamma}}(E_{m,\gamma}^\pm)$.
	There is a subgroup $\tM_{m,0,\gamma}^\pm$ of $\tN_{m,0,\gamma}^\pm$
	containing $E_{m,\gamma}^\pm$ such that the following hold.
	\vspace{-0.5ex}
	\begin{enumerate}[\rm(1)]\setlength{\itemsep}{-0.5ex}
		\item $\tC_{m,0,\gamma}^\pm = \GL_m(\eta q) \otimes I_{2^\gamma}$ and
		$\tN_{m,0,\gamma}^\pm = \tC_{m,0,\gamma}^\pm \tM_{m,0,\gamma}^\pm$ is the central product over $Z(E_{m,\gamma}^\pm)$.
		\item $\tN_{m,0,\gamma}^\pm/E_{m,\gamma}^\pm \cong
		\GL_m(\eta q)/\mrO_2(Z(\GL_m(\eta q))) \times \tM_{m,0,\gamma}^\pm/E_{m,\gamma}^\pm$
		with $\tM_{m,0,\gamma}^\pm/E_{m,\gamma}^\pm \cong \GO_{2\gamma}^\pm(2)$.
		\item $\det(\tM_{m,0,\gamma}^\pm)=1$ unless
		{\setlength\abovedisplayskip{0.5ex}\setlength\belowdisplayskip{0.5ex}
			\begin{equation}\label{special-case-2-uni-1}
			\begin{aligned}
			&\textrm{the type is plus}, \gamma=1,2, \nu(m)=0,~\textrm{or}\\
			&\textrm{the type is minus}, \gamma=1, \nu(m)=0, a=2,
			\end{aligned}
			\addtocounter{thm}{1}\tag{\thethm}
			\end{equation}}
		in which case, $\det(\tM_{m,0,\gamma}^+)=\set{\pm1}$.
	\end{enumerate}
\end{prop}

\begin{proof}
	The assertion for $\tC_{m,0,\gamma}^\pm$ is clear.
	For $\tN_{m,0,\gamma}^\pm$, \cite{An93} has shown that
	$\tN_{m,0,\gamma}^\pm/E_{m,\gamma}^{\pm}\tC_{m,0,\gamma}^\pm \cong \Out(E_{m,\gamma}^{\pm}) \cong \GO_{2\gamma}^\pm(2)$
	by exhibiting a set of generators of $\tN_{m,0,\gamma}^\pm$.
	We will consider the determinants of these generators by exhibiting explicitly the matrices of these generators.
	
	First, consider the case when $\gamma=1$ and the type is minus.
	To do this, we use the twisted version $E_1^{-,tw}$ of $E_1^-$.
	Let $x_{m,1}^{-,tw} = I_m \otimes \begin{bmatrix} \zeta_4&0\\0&-\zeta_4 \end{bmatrix}$,
	$y_{m,1}^{-,tw} = I_m \otimes \begin{bmatrix} 0&1\\-1&0 \end{bmatrix}$
	and $E_1^{-,tw} = \Grp{x_{m,1}^{-,tw},y_{m,1}^{-,tw}}$.
	Let $g_{m,1} \in \bG_{m,0,1}=\SL_{2m}(\barF_q)$ such that $(g_{m,1})^{-1}F_{\eta q}(g_{m,1})=y_{m,1}^{-,tw}$
	and denote by $\iota: \tG_{m,0,1}^{tw}=\tbG_{m,0,1}^{y_{m,1}^{-,tw}F_{\eta q}} \to \tbG_{m,0,1}^{F_{\eta q}}$
	the isomorphism induced by the conjugation by $g_{m,1}$,
	then $E_1^- = \iota(E_1^{-,tw})$.
	Thus it suffices to prove the assertion for $E_1^{-,tw}$.
	Denote by $\tC_{m,0,1}^{-,tw},\tN_{m,0,1}^{-,tw}$ the centralizer and normalizer of $E_1^{-,tw}$ in $\tG_{m,0,1}^{tw}$.
	Set $n_{m,1}^{(i)} = I_m \otimes \frac{1}{\lambda}\begin{bmatrix} -1&\zeta_4\\-1&-\zeta_4 \end{bmatrix}$,
	$n_{m,1}^{(ii)} = I_m \otimes \begin{bmatrix} \zeta_8&0\\0&\zeta_8^{\eta q} \end{bmatrix}$,
	where $\lambda\in\F_{q^2}^\times$ such that $\lambda^2=2\zeta_4$.
	Similar as in \S\ref{subsect:gst-2-linear}, $n_{m,1}^{(i)} \in \tN_{m,0,1}^{-,tw}$ and $\det n_{m,1}^{(i)} = 1$.
	Direct calculation shows also that $n_{m,1}^{(ii)} \in \tN_{m,0,1}^{-,tw}$ and
	$\det n_{m,1}^{(ii)} = 1$ unless $\nu(m)=0$ and $a=2$, in which case, $\det n_{m,1}^{(ii)} = -1$.
	Set $\tM_{m,0,1}^{-,tw} = \Grp{E_{m,1}^{-,tw},n_{m,1}^{(i)},n_{m,1}^{(ii)}}$,
	then all the assertions hold in this case.
	
	Next, assume $\gamma\neq2$ when the type is plus and $\gamma\geq2$ when the type is minus.
	In this case, $\GO_{2\gamma}^\pm(2)$ is generated by orthogonal transvections of elements of order $4$ in $E_{m,\gamma}^{\pm}$;
	see \cite[p.258]{An93}.
	Note that one of $2$ and $-2$ is a square in $\F_q$ since $-1$ is not a square in $\F_q$,
	and $2$ is a square in $\F_q^\times$ if and only if $q\equiv\pm1 \mod 8$.
	When $\eta=1$, let $\mu\in\F_q$ such that $\mu^2=2$ or $\mu^2=-2$ according to that $2$ or $-2$ is a square in $\F_q$.
	When $\eta=-1$, let $\mu\in\F_q$ such that $\mu^2=2$ if $q\equiv1\mod8$
	and let $\mu\in\F_{q^2}-\F_q$ such that $\mu^2=-2$ if $q\equiv5\mod8$, then $\mu^{q+1}=2$
	(this is obvious when $q\equiv1\mod8$, while when $q\equiv5\mod8$, note that $\mu^{2(q+1)}=2^2$ and $\mu^{q+1}\neq-2$).
	Set \[ n_{m,\gamma}^{(i)} =
	\begin{cases}
	I_m \otimes \mu^{-1}\begin{bmatrix} 1&-1\\1&1 \end{bmatrix} \otimes I_{2^{\gamma-1}}, & \textrm{if $\mu^2=2$},\\
	I_m \otimes \mu^{-1}\begin{bmatrix} 1&1\\1&-1 \end{bmatrix} \otimes I_{2^{\gamma-1}}, & \textrm{if $\mu^2=-2$}.
	\end{cases}\]
	Then $n_{m,\gamma}^{(i)} \in \SU_{m2^\gamma}(\eta q)$ satisfying
	\begin{align*}
	n_{m,\gamma}^{(i)} x_{m,\gamma,1}^\pm (n_{m,\gamma}^{(i)})^{-1} = y_{m,\gamma,1}^\pm,\
	n_{m,\gamma}^{(i)} y_{m,\gamma,1}^\pm (n_{m,\gamma}^{(i)})^{-1} = -x_{m,\gamma,1}^\pm,\quad
	& \textrm{if $\mu^2=2$},\\
	n_{m,\gamma}^{(i)} x_{m,\gamma,1}^\pm (n_{m,\gamma}^{(i)})^{-1} = y_{m,\gamma,1}^\pm,\
	n_{m,\gamma}^{(i)} y_{m,\gamma,1}^\pm (n_{m,\gamma}^{(i)})^{-1} = x_{m,\gamma,1}^\pm,\quad
	& \textrm{if $\mu^2=-2$},
	\end{align*}
	and fixing all other $x_{m,\gamma,k}^\pm, y_{m,\gamma,k}^\pm$.
	By \cite[p.258 (2)]{An93}, $n_{m,\gamma}^{(i)}$ is the transvection for $x_{m,\gamma,1}^\pm y_{m,\gamma,1}^\pm$.
	By \cite[p.259--260]{An93},
	all elements in $E_{m,\gamma}^\pm$ of order $4$ are conjugate to $x_{m,\gamma,1}^\pm y_{m,\gamma,1}^\pm$ in $\tN_{m,0,\gamma}^\pm$,
	thus all transvections have determinant $1$.
	Let $\tM_{m,0,\gamma}^\pm$ be the subgroup of $\tN_{m,0,\gamma}^\pm$
	generated by $E_{m,\gamma}^\pm$ and all $\tN_{m,0,\gamma}^\pm$-conjugates of $n_{m,\gamma}^{(i)}$.
	Then all the assertions follow.
	
	Finally, consider the case when the type is plus and $\gamma=2$.
	Let $n_{m,2}^{(i)}$ be as above and
	\vspace{-0.5ex}
	\begin{enumerate}\setlength{\itemsep}{-0.5ex}
		\item[(ii)] Let $P^{(ii)}$ be the permutation matrix transposing all the pairs
		$(\varepsilon_{j_1}\otimes\varepsilon_{j_2},\varepsilon_{j_2}\otimes\varepsilon_{j_1})$.
		Set $n_{m,2}^{(ii)} = I_m \otimes P^{(ii)}$;
		\item[(iii)]	Let $P^{(iii)}$ be the permutation matrix transposing all the pairs
		$(\varepsilon_{j_1}\otimes\varepsilon_{j_2},\varepsilon_{j_1+j_2+1}\otimes\varepsilon_{j_2})$.
		Set $n_{m,2}^{(iii)} = I_m \otimes P^{(iii)}$;
		\item[(iv)] $n_{m,2}^{(iv)} = I_m \otimes \diag\{1,1,1,-1\}$.
	\end{enumerate}
	\vspace{-0.5ex}
	Let $\tM_{m,0,2}^+$ be the subgroup of $\tN_{m,0,2}^\pm$
	generated by $E_{m,2}^\pm$, $n_{m,2}^{(i)}$, $n_{m,2}^{(ii)}$, $n_{m,2}^{(iii)}$, $n_{m,2}^{(iv)}$,
	then by \cite[p.260]{An93} the assertions hold.
\end{proof}

\begin{rem}\label{M-is-stable-2-3}
	It can be proved by direct calculation that $\tM_{m,\alpha,\gamma}^{\pm,tw}\unlhd \tN_{m,\alpha,\gamma}^{\pm,tw}$
	and the groups $\tR_{m,\alpha,\gamma}^{\pm,tw}$,
	$\tN_{m,\alpha,\gamma}^{\pm,tw}$,
	$\tC_{m,\alpha,\gamma}^{\pm,tw}$ and 
	$\tM_{m,\alpha,\gamma}^{\pm,tw}$ are stable under the action of field and graph automorphisms.
\end{rem}

Finally, we set $\tR_m = \set{\pm I_m}$.
The subgroups
$\tR_{m,\alpha,\gamma}(\alpha>0,\gamma\geq0)$, $\tS_{m,1,\gamma-1}(\gamma\geq1)$,  $E_{m,\gamma}^\pm(\gamma\geq1)$, $\tR_m$
of $\tG_{m,\alpha,\gamma}$ are called the \emph{subgroups of symplectic type} when $\ell=2$ and $4 \mid q+\eta$.
For convenience, we also set $\tR_{m,0,\gamma}=\tS_{m,1,\gamma-1}$ and $\tR_{m,0,\gamma}^\pm=E_{m,\gamma}^\pm$ for $\gamma>0$.
Note that by the above Proposition, \emph{$E_{m,1}^+$ is not a radical subgroup}.

%%%%%%%%%%%%%%%%%%%%%%%%%%
\subsection{Basic subgroups.}
\label{subsect:basic-subgps}

First, assume $\ell$ is an odd prime.
Denote by $A_c$ the elementary abelian group of order $\ell^c$ for any natural number $c$.
For any sequence $\bc=(c_1,\dots,c_r)$ of natural numbers, set $A_{\bc}=A_{c_1}\wr\cdots\wr A_{c_r}$ and $|\bc|=c_1+\cdots+c_r$.
We write $\bc=\zero$ or $\one$ if $|\bc|=0$ or $1$.
Set $\tR_{m,\alpha,\gamma,\bc}=\tR_{m,\alpha,\gamma} \wr A_{\bc}$;
here $A_{c_i}$ denotes the regular representation of itself.
We will call $\tR_{m,\alpha,\gamma,\bc}$ a basic subgroup, which is conjugate to the basic subgroup defined in \cite{AF90} and \cite{An94}.
By \cite{AF90} and \cite{An94},
any radical subgroup of $\GL_n(\eta q)$ is conjugate to a subgroup of the form $\tR_0\times \tR_1\times\cdots\times \tR_u$,
where $\tR_0$ is a trivial group and $\tR_i=\tR_{m_i,\alpha_i,\gamma_i,\bc_i}$ ($i\geqslant1$) is a basic subgroup;
note that $\nu(m_i)=0$ for each $i\geqslant1$ if $\tR$ provides a weight.

Set $\tbG_{m,\alpha,\gamma,\bc}=\GL_{me\ell^{\alpha+\gamma+|\bc|}}(\barF)$,
then $\tR_{m,\alpha,\gamma,\bc}$ is a subgroup of $\tbG_{m,\alpha,\gamma,\bc}$ and
we say that $\tR_{m,\alpha,\gamma,\bc}$ has degree $me\ell^{\alpha+\gamma+|\bc|}$.
Set $v_{m,\alpha,\gamma,\bc}=v_{m,\alpha,\gamma}\otimes I_\bc$ and $g_{m,\alpha,\gamma,\bc}=g_{m,\alpha,\gamma}\otimes I_\bc$,
where $I_\bc$ is the identity matrix of degree $\ell^{|\bc|}$.
Then $v_{m,\alpha,\gamma,\bc},g_{m,\alpha,\gamma,\bc}\in\tG_{m,\alpha,\gamma,\bc}$
and $g_{m,\alpha,\gamma,\bc}^{-1}F_{\eta q}(g_{m,\alpha,\gamma,\bc})=v_{m,\alpha,\gamma,\bc}$.
Denote again by $\iota$ the isomorphism (and all the similar isomorphism)
$$\iota:\quad \tG_{m,\alpha,\gamma,\bc}^{tw}=\tbG_{m,\alpha,\gamma,\bc}^{v_{m,\alpha,\gamma,\bc}F_{\eta q}} \to \tG_{m,\alpha,\gamma,\bc}=\tbG_{m,\alpha,\gamma,\bc}^{F_{\eta q}}$$
induced by conjugation by $g_{m,\alpha,\gamma,\bc}$.
Set $\tR_{m,\alpha,\gamma,\bc}^{tw}= \tR_{m,\alpha,\gamma}^{tw} \wr A_\bc$.
Then $\tR_{m,\alpha,\gamma,\bc}=\iota(\tR_{m,\alpha,\gamma,\bc}^{tw})$.
We call $\tR_{m,\alpha,\gamma,\bc}^{tw}$ the twisted basic subgroups.

\begin{notation*}
	Denote $\tC_{m,\alpha,\gamma,\bc}^*:=C_{\tG_{m,\alpha,\gamma,\bc}^*}(\tR_{m,\alpha,\gamma,\bc}^*)$ and
	$\tN_{m,\alpha,\gamma,\bc}^*:=N_{\tG_{m,\alpha,\gamma,\bc}^*}(\tR_{m,\alpha,\gamma,\bc}^*)$ for $*\in\{tw,\O\}$.
	We will use some obvious abbreviation,
	such as $\tR_m=\tR_{m,0,0,\zero}$, $\tR_{m,\alpha}=\tR_{m,\alpha,0,\zero}$,
	$\tR_{m,\alpha,\gamma}=\tR_{m,\alpha,\gamma,\zero}$, etc.
\end{notation*}

\begin{cau}
	The subgroup $N_{m,\alpha,\gamma}^0$ in \cite{AF90} and \cite{An94} is $(\iota\circ\hbar)(\tN_{m,\alpha,\gamma}^0)$,
	where $\tN_{m,\alpha,\gamma}^0$ is the group defined in \S\ref{subsect:gst-odd}.
\end{cau}

\begin{lem}\label{lem:tCtN}
	With the above notation, 
	$\tC_{m,\alpha,\gamma,\bc}^{tw} = \tC_{m,\alpha,\gamma}^{tw} \otimes I_\bc$ and
	$\tN_{m,\alpha,\gamma,\bc}^{tw} = (\tN_{m,\alpha,\gamma}^{tw}/\tR_{m,\alpha,\gamma}^{tw}) \otimes N_{\fS(\ell^{|\bc|})}(A_\bc)$
	with $\otimes$ defined as \cite[(1.5)]{AF90}.
\end{lem}
\begin{proof}
	This is just the twisted version of results in \cite{AF90} and \cite{An94}.
\end{proof}

\begin{notation*}
	Replacing $\eta q$ by $(\eta q)^{e\ell^\alpha}$, the corresponding constructions will be denoted by adding a superscript $^{(\alpha)}$.
	For example, $\tR_{m,\beta,\gamma,\bc}^{(\alpha)}$ denotes a basic subgroup of $\tG_{m,\beta,\gamma,\bc}^{(\alpha)}=\GL_{m\ell^{\beta+\gamma+|\bc|}}((\eta q)^{e\ell^\alpha})$.
\end{notation*}

\begin{notation*}
	For any natural number $\beta$, set $\bbeta=(1,\ldots,1)$ with $\beta$ one's.
	Set $\tR_{m,\alpha,\bbeta}=\tR_{m,\alpha,0,\bbeta}$.
\end{notation*}

$\tR_{m,\alpha,\bbeta}$ is conjugate to the subgroup $R^{m,\alpha,\beta}$ in \cite[\S3]{FS82}.
By \cite{FS82}, the defect groups of blocks of $\tG$ are conjugate to subgroups of the form $\tR_0\times \tR_1\times\cdots\times \tR_u$,
where $\tR_0$ is the trivial group, $\tR_i=\tR_{m_i,\alpha_i,\bbeta_i}$ ($i\geqslant1$) and $\nu(m_i)=0$.

For $\ell=2$ and $4 \mid q-\eta$,
let $\tR_{m,\alpha,\gamma,\bc}=\tR_{m,\alpha,\gamma}\wr A_{\bc}$,
then the similar results hold.
For $\ell=2$ and $4 \mid q+\eta$,
set $\tR_{m,\alpha,\gamma,\bc} = \tR_{m,\alpha,\gamma} \wr A_{\bc}$
and $\tR_{m,0,\gamma,\bc}^\pm = \tR_{m,0,\gamma}^\pm \wr A_{\bc}(\gamma>0)$.
Then the similar results hold.
Since $\tR_{m,0,0,\one}=\tR_{m,0,1}^+=E_{m,1}^+$ is not a radical subgroup,
we call $\tR_{m,\alpha,\gamma,\bc}$(except when $\alpha=\gamma=0, c_1=1$) and
$\tR_{m,0,\gamma,\bc}^\pm$($\gamma>1$ if the type is plus, $\gamma>0$ if the type is minus) the basic subgroups.

\subsection{Parametrization of weights of $\GL_n(q)$ and $\GU_n(q)$.}\label{subsect:weights-general}

The weights of general linear and unitary groups have been classified by Alperin--Fong \cite{AF90} and An \cite{An92,An93,An94}.
An explicit labelling of this classification has been given in \cite[\S5]{LZ18}.
Since we have used different notation for the radical subgroups, we repeat the construction in \cite{LZ18} to fix the notation.

Recall that $\cF'$ is the subset of $\cF$ of those polynomials whose roots are of $\ell'$-order.
For $\Gamma\in\cF'$, we let  $e_\Gamma$ be the multiplicative order of $(\eta q)^{d_\Gamma}$ modulo $\ell$.
Note that $e_\Gamma=1$ for any $\Gamma\in\cF'$ when $\ell=2$.
Given any $\Gamma\in\cF'$, let $m_\Gamma,\alpha_\Gamma$ be non-negative integers determined by $m_\Gamma e\ell^{\alpha_\Gamma}=e_\Gamma d_\Gamma$ and $(m_\Gamma,\ell)=1$.
Note that there is no direct connection between $m_\Gamma$ and $m_\Gamma(s)$.

First, assume $\ell$ is an odd prime.
By \cite[(5A)]{FS82}, given any $\Gamma\in\cF'$,
there is a unique block $\tB_\Gamma$ of $\tG_\Gamma=\tG_{m_\Gamma,\alpha_\Gamma}$
with $\tR_\Gamma=\tR_{m_\Gamma,\alpha_\Gamma}$ as a defect group.
Let $\tC_\Gamma=C_{\tG_\Gamma}(\tR_\Gamma)$ and $\tN_\Gamma=N_{\tG_\Gamma}(\tR_\Gamma)$.
$\Gamma$ also determines a unique $\tN_\Gamma$-conjugacy classes of pairs $(\tfb_\Gamma,\ttheta_\Gamma)$,
where $\tfb_\Gamma$ is a root block of $\tC_\Gamma\tR_\Gamma=\tC_\Gamma$ with defect group $\tR_\Gamma$
and $\ttheta_\Gamma$ is the canonical character of $\tfb_\Gamma$.
This block is $\cE_\ell(\tG_\Gamma,s_\Gamma))$,
where $s_\Gamma$ is conjugate to $e_\Gamma (\Gamma)$ in $\tG_\Gamma$,
and when viewed as an element of $C_{\tG_\Gamma}(\tR_\Gamma) \cong \GL(m_\Gamma,(\eta q)^{e\ell^{\alpha_\Gamma}})$,
$s_\Gamma$ has a unique elementary divisor which is in the inverse image of $\Gamma$ under $\Phi_{\alpha_\Gamma}$,
where $\Phi_{\alpha_\Gamma}$ is defined as in (\ref{eq:Phi_alpha}).
%where $s_\Gamma$ is in the inverse image of $e_\Gamma (\Gamma)$ under $\Phi_{\alpha_\Gamma}$,

Let $\tR_{\Gamma,\gamma,\bc}=\tR_{m_\Gamma,\alpha_\Gamma,\gamma,\bc}$ be a basic subgroup and $\tG_{\Gamma,\gamma,\bc},\tC_{\Gamma,\gamma,\bc},\tN_{\Gamma,\gamma,\bc}$ defined similarly.
Then $\tC_{\Gamma,\gamma,\bc} \cong \tC_\Gamma\otimes I_{\ell^\gamma} \otimes I_{\bc}$.
Let $\ttheta_{\Gamma,\gamma,\bc}=\ttheta_\Gamma\otimes I_{\ell^\gamma} \otimes I_{\bc}$, then $\ttheta_{\Gamma,\gamma,\bc}$ can be viewed as a canonical character of $\tC_{\Gamma,\gamma,\bc}\tR_{\Gamma,\gamma,\bc}$ with $\tR_{\Gamma,\gamma,\bc}$ in the kernel and all canonical characters are of this form.

Let $\tcR_{\Gamma,\delta}$ be the set of all the basic subgroups of the form $\tR_{\Gamma,\gamma,\bc}$ with $\gamma+|\bc|=\delta$.
Label the basic subgroups in $\tcR_{\Gamma,\delta}$ as $\tR_{\Gamma,\delta,1}$, $\tR_{\Gamma,\delta,2}$, $\cdots$ and let $\tC_{\Gamma,\delta,i},\tN_{\Gamma,\delta,i}$ be defined similarly.
Denote the canonical character associated to $\tR_{\Gamma,\delta,i}$ by $\ttheta_{\Gamma,\delta,i}$.
In the case when $m_{\Gamma'}=m_\Gamma=:m$ and $\alpha_{\Gamma'}=\alpha_\Gamma=:\alpha$ for another $\Gamma'\in\cF'$, we may choose the labeling of $\tcR_{\Gamma,\delta}=\tcR_{\Gamma',\delta}$ such that $\tR_{\Gamma,\delta,i}=\tR_{\Gamma',\delta,i}$ for $i=1,2,\cdots$.
We denote $\tR_{m,\alpha,\gamma,\bc}$ as $\tR_{\Gamma,\delta,i}$ or $\tR_{\Gamma',\delta,i}$ depending on that the related canonical character of $\tC_{m,\alpha}\tR_{m,\alpha}=\tC_{m,\alpha}$ considered is $\ttheta_\Gamma$ or $\ttheta_{\Gamma'}$.

Let 
\begin{equation}\label{def-tsC-Gamma,delta}
\tsC_{\Gamma,\delta}=\bigcup_i \dz(\tN_{\Gamma,\delta,i}(\ttheta_{\Gamma,\delta,i})/\tR_{\Gamma,\delta,i}\mid\ttheta_{\Gamma,\delta,i}).
\addtocounter{thm}{1}\tag{\thethm}
\end{equation}
We assume $\tsC_{\Gamma,\delta}=\{\tpsi_{\Gamma,\delta,i,j}\}$ with $\tpsi_{\Gamma,\delta,i,j}$ a character of $\tN_{\Gamma,\delta,i}$.

We now define $i\Alp(G)$ in the exactly same way as on \cite[p.145]{LZ18}.
\begin{equation*}%\label{eq:iW}
i\Alp(G)=\left\{~(s,\lambda,K)^{\tG}~\middle|~
\begin{array}{c}
s~\textrm{is a  semisimple $\ell'$-element of}~\tG,\\
\lambda=\prod_\Gamma\lambda_\Gamma,~\lambda_\Gamma~\textrm{is the $e_\Gamma$-core of a partition of}~m_\Gamma(s),\\
K=K_\Gamma,~K_\Gamma:\bigcup_\delta\tsC_{\Gamma,\delta}\to\{~\ell\textrm{-cores}~\}~\textrm{s.t.}~\\
\sum_{\delta,i,j}\ell^\delta |K_\Gamma(\psi_{\Gamma,\delta,i,j})|=w_\Gamma,
m_\Gamma(s)=|\lambda_\Gamma|+e_\Gamma w_\Gamma.
\end{array}~\right\}
%\addtocounter{thm}{1}\tag{\thethm}
\end{equation*}
A bijection between $\Alp(G)$ and $i\Alp(G)$ constructed implicitly in \cite{AF90} and \cite{An94} can be described as follows.

Let $(\tR,\tvarphi)$ be a weight.
Then we may assume $\tR=\tR_0\times \tR_1\times\cdots\times \tR_u$ with $\tR_0$ the trivial group of degree $n_0$,
and we denote $\tR_+=\tR_1\times\cdots\times \tR_u$.
Then $\tC=C_{\tG}(\tR)$ and $\tN=N_{\tG}(\tR)$ have corresponding decompositions: $\tC=\tC_0\times \tC_+$, $\tN=\tN_0\times \tN_+$ with $\tC_0=\tN_0=\GL_{n_0}(\eta q)$.
$\tvarphi$ lies over a canonical character $\ttheta$ of $\tC\tR$, which can be decomposed as $\ttheta=\ttheta_0\times\ttheta_+$.
Then $\tvarphi=\tvarphi_0\times\tvarphi_+$ with $\tvarphi_0=\ttheta_0$ and $\tvarphi_+\in\Irr(\tN_+\mid\ttheta_+)$.

So $\tvarphi_0=\ttheta_0$ is a character of $\GL_{n_0}(\eta q)$ of defect zero.
Then $\tvarphi_0=\chi_{s_0,\lambda}$, where $s_0$ a semisimple $\ell'$-element of $\GL_{n_0}(\eta q)$ and $\lambda=\prod_\Gamma \lambda_\Gamma$ with $\lambda_\Gamma$ a partition of $m_\Gamma(s_0)$ without $e_\Gamma$-hook which affords the second component of the triple $(s,\lambda,K)$.

Assume we have the decompositions
$\ttheta_+=\prod_{\Gamma,\delta,i}\ttheta_{\Gamma,\delta,i}^{t_{\Gamma,\delta,i}}$ and
$\tR_+=\prod_{\Gamma,\delta,i}\tR_{\Gamma,\delta,i}^{t_{\Gamma,\delta,i}}$,
then $\ttheta_\Gamma$ determines a semisimple $\ell'$-element with canonical form $e_\Gamma(\Gamma)$ in $\tG_\Gamma$.
Then $s=s_0\prod_{\Gamma,\delta,i}e_\Gamma(\Gamma)\otimes I_{\ell^\delta}\otimes I_{t_{\Gamma,\delta,i}}$ is the first component of the triple $(s,\lambda,K)$.
In particular, $\lambda_\Gamma$ is the $e_\Gamma$-core of a partition of $m_\Gamma(s)$.

It is easy to see that
$\tN_+(\ttheta_+) = \prod\limits_{\Gamma,\delta,i} \tN_{\Gamma,\delta,i}(\ttheta_{\Gamma,\delta,i}) \wr \fS(t_{\Gamma,\delta,i})$.
Thus $\tvarphi_+=\Ind_{\tN_+(\theta_+)}^{\tN_+}\tpsi_+$ with $\tpsi_+\in\dz(\tN_+(\ttheta_+)/\tR_+\mid\ttheta_+)$.
Then $\tpsi_+=\prod\limits_{\Gamma,\delta,i}\tpsi_{\Gamma,\delta,i}$,
where $\tpsi_{\Gamma,\delta,i}$ is a character of $\tN_{\Gamma,\delta,i}(\ttheta_{\Gamma,\delta,i})\wr\fS(t_{\Gamma,\delta,i})$.
By Clifford theory, $\tpsi_{\Gamma,\delta,i}$ is of the form
\begin{equation}\label{equation:tpsi}
\Ind_{\tN_{\Gamma,\delta,i}(\ttheta_{\Gamma,\delta,i})\wr\prod_j\fS(t_{\Gamma,\delta,i,j})}^{\tN_{\Gamma,\delta,i}(\ttheta_{\Gamma,\delta,i})\wr\fS(t_{\Gamma,\delta,i})}
\overline{\prod_j\tpsi_{\Gamma,\delta,i,j}^{t_{\Gamma,\delta,i,j}}} \cdot \prod_j\phi_{\lambda_{\Gamma,\delta,i,j}},
\addtocounter{thm}{1}\tag{\thethm}
\end{equation}
where $t_{\Gamma,\delta,i}=\sum_jt_{\Gamma,\delta,i,j}$, $\overline{\prod_j\tpsi_{\Gamma,\delta,i,j}^{t_{\Gamma,\delta,i,j}}}$ is the canonical extension of $\prod_j\tpsi_{\Gamma,\delta,i,j}^{t_{\Gamma,\delta,i,j}}\in\Irr(\tN_{\Gamma,\delta,i}(\ttheta_{\Gamma,\delta,i})^{t_{\Gamma,\delta,i}})$ to $\tN_{\Gamma,\delta,i}(\ttheta_{\Gamma,\delta,i})\wr\prod_j\fS(t_{\Gamma,\delta,i,j})$ as in the proof of \cite[Prop.~2.3.1]{Bon99b}, $\lambda_{\Gamma,\delta,i,j}\vdash t_{\Gamma,\delta,i,j}$ without $\ell$-hook and $\phi_{\lambda_{\Gamma,\delta,i,j}}$ is the character of $\fS(t_{\Gamma,\delta,i,j})$ corresponding to $\lambda_{\Gamma,\delta,i,j}$.
Note that (\ref{equation:tpsi}) is slightly different from \cite[(5.1)]{LZ18}; see \cite[\S4.2]{LZ19} for the reason.
Now, define $K_\Gamma:\cup_\delta\sC_{\Gamma,\delta}\to\{\ell\textrm{-cores}\}$, $\psi_{\Gamma,\delta,i,j}\mapsto\lambda_{\Gamma,\delta,i,j}$.
Then we can associate the conjugacy class $\overline{(\tR,\tvarphi)}$ of weights with the label $(s,\lambda,K)$.

\vspace{1.6ex}

When $\ell=2$ and $4 \mid q-\eta$, the exactly same constructions apply.
Assume $\ell=2$ and $4 \mid q+\eta$ now.
Since some radical subgroups will not provide weights, set
\begin{equation}\label{def-D-2-uni}
\tD_{m,\alpha,\gamma,\bc} =
\left\{\begin{array}{ll}
\tR_{m,\alpha,\gamma,\bc} & \textrm{if}~\alpha>0,\\
\tS_{m,1,\gamma-1,\bc} & \textrm{if}~\alpha=0, \gamma>1,\\
\tR_{m,0,1,\bc}^- &  
\textrm{if}~\alpha=0, \gamma=1,\\
\tR_{m,0,0,\bc} & \textrm{if}~\alpha=\gamma=0,c_1\neq1\\
\tS_{m,1,0,\bc'},  & \textrm{if}~\alpha=\gamma=0,c_1=1,\\
\end{array}\right.
\addtocounter{thm}{1}\tag{\thethm}
\end{equation}
where $\bc'=(c_2,\ldots,c_r)$ for $\bc=(c_1,c_2,\ldots,c_r)$.
Using $\tD_{m,\alpha,\gamma,\bc}$ in place of $\tR_{m,\alpha,\gamma,\bc}$,
we can define $i\Alp(\tG)$ analogously,
 then the similar constructions apply again.

The definition of $\tD_{m,\alpha,\gamma,\bc}$ here is slightly different from that given in \cite{An92, An93},
as well as the one used in \cite{LZ18,Feng19}.
However, we should mention that the arguments in \cite{LZ18} and \cite{Feng19} also apply if we change the definition of $\tD_{m,\alpha,\gamma,\bc}$ there to be as in (\ref{def-D-2-uni}).

For the twisted basic subgroups, we also write $\tD^{tw}_{m,\alpha,\gamma,\bc}$ in the obvious sense.

\vspace{2ex}

The paper \cite{Feng19} defines an action of the group $\mrO_{\ell'}(\fZ_{q-\eta})$ on $i\Alp(\tG)$ which we recall as follows.
Let $\overline{(\tR,\tvarphi)}$ be a weight of $\tG$ with label $(s,\lambda,K)$.
Note that for  $\zeta\in\mrO_{\ell'}(\fZ_{q-\eta})$,
then $\hat\zeta\in\Lin_{\ell'}(\tG/G)$, and we may identify $\hat \zeta$ with its restriction to $N_{\tG}(\tR)$.
Define $\zeta.\lambda$ and $\zeta.K$ such that $(\zeta.\lambda)_{\zeta.\Gamma}=\lambda_\Gamma$ and $(\zeta.K)_{\zeta.\Gamma}=K_\Gamma$.
Then the action of $\Lin_{\ell'}(\tG/G)$ on $\Alp(\tG)$ is induced by the action of $\mrO_{\ell'}(\fZ_{q-\eta})$ (identified with $\mrO_{\ell'}(Z(\tG))$) on $i\Alp(\tG)$.

\begin{thm}[{\cite[Prop. 5.12]{Feng19}}]\label{action-z-onwei}
Assume that  $\overline{(\tR,\tvarphi)}$ is a weight of $\tG$ with label $(s,\lambda,K)$, then the weight
$\overline{(\tR,\hat \zeta\tvarphi)}$ has label $(\zeta s,\zeta.\lambda,\zeta.K)$.
\end{thm}

%%%%%%%%%%%%%%%%%%%%%%%%%%%%%%%%%%%

\section{Special radical subgroups}\label{subsect:srs}

In this section, we consider the radical subgroups of special linear and unitary groups.

\subsection{Some general lemmas.}\label{subsect:srs-general}
We start with some general observations about the radical subgroups of a finite group and its normal subgroups.
In this subsection \S \ref{subsect:srs-general}, let $H$ be a normal subgroup of $\tH$.
We first recall a general result in \cite[\S 2.1]{Feng19}.

\begin{lem}\label{lem:motivation-special-radical}
	\begin{enumerate}[\rm(1)]\setlength{\itemsep}{-2pt}
		\item If $\tR$ is a radical subgroup of $\tH$,
		then $R := \tR\cap H$ is a radical subgroup of $H$. 
		\item The map $\Rad(\tH) \to \Rad(H),\ \tR \mapsto \tR \cap H$ is surjective.
		In particular, if $R$ is a radical subgroup of $H$,
		then $\tR:=\mrO_\ell(N_{\tH}(R))$ is a radical subgroup of $\tH$
		such that $\tR \cap H = R$ and $N_{\tH}(R)=N_{\tH}(\tR)$.
	\end{enumerate}
\end{lem}
\begin{proof}
	(1) is just \cite[(2.1)]{OU95} and (2) is from \cite[Lemma~2.2]{Feng19} and its proof.
\end{proof}

From Lemma \ref{lem:motivation-special-radical}, we introduce the following definition.

\begin{defn}
	Assume $H$ is a normal subgroup of a finite $\tH$.
	A radical subgroup $\tR$ of $\tH$ is called \emph{special} (with respect to $H$) if $\tR= \mrO_\ell (N_{\tH}(\tR\cap H))$.
\end{defn}

Given any radical subgroup $\tR$ of $\tH$, $\tR$ may be not special;
but there must be a special radical subgroup $\tR_s$ of $\tH$ such that $\tR \cap H = \tR_s \cap H$ by  Lemma \ref{lem:motivation-special-radical} (2).
Assume $\tR$ is a special radical subgroup of $\tH$ and denote $R = \tR \cap H$, then $N_{\tH}(R) = N_{\tH}(\tR)$ and thus $N_H(R)$ can be calculated by $N_H(R) = N_{\tH}(\tR) \cap H$.

\begin{lem}\label{lem:split-classes-radical}
	Assume $\tR$ is a special radical subgroup of $\tH$, then the number of $H$-conjugacy classes of radical subgroups of $H$ corresponding to the $\tH$-conjugacy class of $\tR$ via the intersection with $H$ (i.e. those $H$-conjugacy classes of radical subgroups contained in $\{H\cap{^{\tg}\tR} \mid \tg\in\tH\}$) is $|\tH:HN_{\tH}(\tR)|$.
\end{lem}
\begin{proof}
	Set $R := H\cap\tR$, then $H\cap{^{\tg}\tR}={^{\tg}R}$.
	Since $\tR$ is special, $N_{\tH}(R)=N_{\tH}(\tR)$, so $|\{H\cap{^{\tg}\tR} \mid \tg\in\tH\}|=|\tH:N_{\tH}(\tR)|$.
	All the $H$-conjugacy classes in $\{H\cap{^{\tg}\tR} \mid \tg\in\tH\}$ have the same cardinality $H/N_H(R) \cong HN_{\tH}(\tR)/N_{\tH}(\tR)$.
	So the assertion follows.
\end{proof}

The term ``special'' comes from the case of the general linear (or unitary) group $\tG$ and the special linear (or unitary) group $G$ respectively,
which indicates the importance of the special radical subgroups of $\tG$ to the study of the weights of $G$.
But note that even for a special radical subgroup $\tR$ of $\tG$, it may happen that $\tR$ does not provide any weight for $\tG$ while $\tR\cap G$ provides some weights for $G$.

\subsection{Special radical subgroups for odd primes.}\label{subsect:srs-odd}
Now, we classify the special radical subgroups of $\tG = \GL_n(\eta q)$ with respect to $G = \SL_n(\eta q)$ for an odd prime $\ell$.
Recall that $\nu$ is the discrete valuation such that $\nu(\ell)=1$.
The radical $\ell$-subgroups of $G$ are classified in \cite{Feng19} when $\ell \nmid \gcd(n,q-\eta)$; in this case every radical $\ell$-subgroup of $\tG$ is special.
So we may assume $\ell \mid q-\eta$ throughout \S \ref{subsect:srs-odd}.
Then $e=1$ and $a=\nu(q-\eta)$.

We start with the radical subgroups of the form $\tR_{m,\alpha,\gamma}$.
Let $R_{m,\alpha,\gamma} := \tR_{m,\alpha,\gamma} \cap G_{m,\alpha,\gamma}$.
Transfer to the twisted groups.
Recall that $\tR_{m,\alpha,\gamma}^{tw}=\tZ_{m,\alpha,\gamma}^{tw}E_{m,\alpha,\gamma}^{tw}$ and $\det(E_{m,\alpha,\gamma}^{tw})=1$, thus $R_{m,\alpha,\gamma}^{tw}=Z_{m,\alpha,\gamma}^{tw}E_{m,\alpha,\gamma}^{tw}$ with $Z_{m,\alpha,\gamma}^{tw}:=Z(R_{m,\alpha,\gamma}^{tw})=\tZ_{m,\alpha,\gamma}^{tw}\cap G_{m,\alpha,\gamma}^{tw}$.
Since
\begin{equation}
\tZ_{m,\alpha,\gamma}^{tw}=\hbar(\tZ_{m,\alpha,\gamma}^0),\
\tZ_{m,\alpha,\gamma}^0=\mrO_\ell(Z(\tG_{m,\alpha,\gamma}^0))=\{\zeta I_{m\ell^\gamma} \mid \zeta\in\fZ_{\ell^{a+\alpha}}\},
\addtocounter{thm}{1}\tag{\thethm}
\end{equation}
and $\det\hbar(\zeta I_{m\ell^\gamma}) = \cN_\alpha(\zeta^{m\ell^\gamma}) = \zeta^{k\ell^{\alpha+\nu(m)+\gamma}}$ for some $\ell'$-integer $k$ by the assumption $\ell \mid (q-\eta)$, we have
\begin{equation}\label{eq:ZR_m,alpha,gamma}
Z_{m,\alpha,\gamma}^{tw}=\hbar(\hZ_{m,\alpha,\gamma}^0) ~\textrm{with}~
\hZ_{m,\alpha,\gamma}^0 = \{ \zeta I_{m\ell^\gamma} \mid \zeta\in\fZ_{\ell^{a+\alpha}}\cap\fZ_{\ell^{\alpha+\gamma+\nu(m)}} \}.
\addtocounter{thm}{1}\tag{\thethm}
\end{equation}
We also set $\hR_{m,\alpha,\gamma}^0=\hZ_{m,\alpha,\gamma}^0E_{m,\alpha,\gamma}^0$, then $R_{m,\alpha,\gamma}^{tw}=\hbar(\hR_{m,\alpha,\gamma}^0)$.

\begin{prop}\label{prop:special-odd-1}
	Assume $\ell$ is an odd prime.
	The basic subgroup $\tR_{m,\alpha,\gamma}$ of $\tG_{m,\alpha,\gamma}$ is special with respect to $G_{m,\alpha,\gamma}$
	if and only if $\gamma+\nu(m)\geqslant a$ or $\alpha=0$.
\end{prop}
\begin{proof}
	First, assume $\gamma+\nu(m)\geqslant a$.
	Then $\hZ_{m,\alpha,\gamma}^0=\tZ_{m,\alpha,\gamma}^0$ and $\tR_{m,\alpha,\gamma}=R_{m,\alpha,\gamma}$ is obviously special.
	Assume $a-\alpha<\gamma+\nu(m)<a$.
	Then $\hZ_{m,\alpha,\gamma}^0 = \{ \zeta I_{m\ell^\gamma} \mid \zeta\in\fZ_{\ell^{a+(\alpha+\gamma+\nu(m)-a)}} \}$.
	Thus $R_{m,\alpha,\gamma} = \tR_{m,\alpha,\gamma} \cap G_{m,\alpha,\gamma}$ is $\tG_{m,\alpha,\gamma}$-conjugate to
	$R_{m\ell^{a-\gamma-\nu(m)},\alpha+\gamma+\nu(m)-a,\gamma}
	:= \tR_{m\ell^{a-\gamma-\nu(m)},\alpha+\gamma+\nu(m)-a,\gamma} \cap G_{m,\alpha,\gamma}$
	(since the $Z_{m,\alpha,\gamma}$ and $Z_{m\ell^{a-\gamma-\nu(m)},\alpha+\gamma+\nu(m)-a,\gamma}$ are conjugate in $\tG_{m,\alpha,\gamma}$
	and the extraspecial component of $R_{m,\alpha,\gamma}$ are unique up to conjugacy in $C_{\tG_{m,\alpha,\gamma}}(Z_{m,\alpha,\gamma})$).
	But $\tR_{m\ell^{a-\gamma-\nu(m)},\alpha+\gamma+\nu(m)-a,\gamma}$ is a special radical subgroup of
	$\tG_{m,\alpha,\gamma}=\GL_{m\ell^{\alpha+\gamma}}(\eta{q})=\tG_{m\ell^{a-\gamma-\nu(m)},\alpha+\gamma+\nu(m)-a,\gamma}$.
	Finally, assume $\alpha+\gamma+\nu(m)\leqslant a$.
	Then similarly, $R_{m,\alpha,\gamma}$ is $\tG_{m,\alpha}$-conjugate to
	$R_{m\ell^\alpha,0,\gamma} := \tR_{m\ell^\alpha,0,\gamma} \cap G_{m,\alpha,\gamma}$.
	Since $\tZ_{m\ell^\alpha,0,\gamma}$ is central in
	$\tG_{m,\alpha,\gamma}=\GL_{m\ell^{\alpha+\gamma}}(\eta{q})=\tG_{m\ell^\alpha,0,\gamma}$,
	$\tR_{m\ell^\alpha,0,\gamma}$ is  special by construction.
\end{proof}

\begin{prop}\label{prop:cc-R-m,alpha,gamma-odd}
	Assume $\ell$ is an odd prime.
	\vspace{-0.5ex}
	\begin{enumerate}[\rm(1)]\setlength{\itemsep}{-0.5ex}
		\item $\det(\tR_{m,\alpha,\gamma})=\det(\tR_{m,\alpha,\gamma}^{tw})=\fZ_{\ell^a}^{\ell^{\nu(m)+\gamma}}$;
		$\det(\tC_{m,\alpha,\gamma})=\det(\tC_{m,\alpha,\gamma}^{tw})=\fZ_{q-\eta}^{\ell^\gamma}$.
		\item $\det(\tN_{m,\alpha,\gamma})=\det(\tN_{m,\alpha,\gamma}^{tw})=\fZ_{q-\eta}^{\ell^\gamma}$,
		unless (\ref{eq:special-case-3-1}), in which case,
		$\det(\tN_{m,0,1})=\Grp{\det(\tC_{m,0,1}),\det(\tM_{m,0,1})}=\fZ_{q-\eta}$.
		\item Assume $\tR_{m,\alpha,\gamma}$ is special in $\tG_{m,\alpha,\gamma}$,
		then	the number of $G_{m,\alpha,\gamma}$-conjugacy classes of radical subgroups contained in the set
		$\{ G_{m,\alpha,\gamma}\cap{^{\tg}\tR_{m,\alpha,\gamma}} \mid \tg\in\tG_{m,\alpha,\gamma} \}$ is
		\[\left\{\begin{array}{ll}
		\ell^{\min\{a,\gamma\}} & \textrm{if we are not in the case (\ref{eq:special-case-3-1});}\\
		1, & \textrm{if we are in the case (\ref{eq:special-case-3-1}).}
		\end{array}\right.\]
	\end{enumerate}
\end{prop}
\begin{proof}
	(1) follows easily from Lemma~\ref{lem:tC-m,alpha,gamma-odd}.
	(2) follows from (1) and Proposition \ref{prop:tN-m,alpha,gamma-odd}.
	Then (3) follows from Lemma~\ref{lem:split-classes-radical}.
\end{proof}

Set
\begin{equation}\label{eq:hC-m,alpha,gamma}
\hC_{m,\alpha,\gamma}^0=\hG_{m,\alpha}^{0,\gamma}\otimes I_{\ell^\gamma} ~\textrm{with}~ \hG_{m,\alpha}^{0,\gamma} = \{ A\in\tG_{m,\alpha}^0 \mid \cN_\alpha(\det(A)^{\ell^\gamma})=1 \},
\addtocounter{thm}{1}\tag{\thethm}
\end{equation}
where $\cN_\alpha$ is defined in (\ref{eq:norm}).
We also set
\begin{equation}\label{eq:hz-m,alpha-0,gamma}
\hZ_{m,\alpha}^{0,\gamma} = \{ \zeta I_m \mid \zeta\in\fZ_{\ell^{a+\alpha}}\cap\fZ_{\ell^{\alpha+\gamma+\nu(m)}} \},
\addtocounter{thm}{1}\tag{\thethm}
\end{equation}
then $\hZ_{m,\alpha}^{0,\gamma}=\mrO_\ell(Z(\hG_{m,\alpha}^{0,\gamma}))$, $\hZ_{m,\alpha,\gamma}^0 = \hZ_{m,\alpha}^{0,\gamma}\otimes I_{\ell^\gamma}$ and $|\hZ_{m,\alpha}^{0,\gamma}|=\ell^{\min\{\alpha+a,\alpha+\gamma+\nu(m)\}}$.
Finally, set
\begin{equation}\label{eq:hN-m,alpha,gamma}
\hN_{m,\alpha,\gamma}^0=\hC_{m,\alpha,\gamma}^0\tM_{m,\alpha,\gamma}^0.
\addtocounter{thm}{1}\tag{\thethm}
\end{equation}

Assume that $\tR_{m,\alpha,\gamma}$ is special in $\tG_{m,\alpha,\gamma}$.
Denote by $C_{m,\alpha,\gamma}^*$, $N_{m,\alpha,\gamma}^*$ for the centralizer and normalizer
of $R_{m,\alpha,\gamma}^*$ in $G_{m,\alpha,\gamma}^*$, and $M_{m,\alpha,\gamma}^*=\tM_{m,\alpha,\gamma}^*\cap G_{m,\alpha,\gamma}^*$ for $*\in\set{tw,\varnothing}$.

\begin{prop}\label{prop:CN_m,alpha,gamma-odd}
	Then $C_{m,\alpha,\gamma}^{tw}=\hbar(\hC_{m,\alpha,\gamma}^0)$ and the following holds.
	\vspace{-0.5ex}
	\begin{enumerate}[\rm(1)]\setlength{\itemsep}{-0.5ex}
		\item Assume $\gamma+\nu(m)\geqslant a$.
		\vspace{-0.5ex}
		\begin{enumerate}\setlength{\itemsep}{-0.25ex}
			\item[\rm(1.1)] If we are not in the case (\ref{eq:special-case-3-1}), then $\det(\tM_{m,\alpha,\gamma}^{tw})=1$.
			Thus $\hbar(\tN_{m,\alpha,\gamma}^0) \cap G_{m,\alpha,\gamma}^{tw}
			= \hbar(\hN_{m,\alpha,\gamma}^0) = C_{m,\alpha,\gamma}^{tw}\tM_{m,\alpha,\gamma}^{tw}$,
			$\hbar(\hN_{m,\alpha,\gamma}^0)/R_{m,\alpha,\gamma}^{tw}
			\cong C_{m,\alpha,\gamma}^{tw}/Z(R_{m,\alpha,\gamma}^{tw}) \times \Sp_{2\gamma}(\ell)$
			and $N_{m,\alpha,\gamma}^{tw}=\hbar(\hN_{m,\alpha,\gamma}^0)\rtimes V_{m,\alpha,\gamma}$.
			\item[\rm(1.2)] In the case (\ref{eq:special-case-3-1}),
			 $N_{m,0,1}=C_{m,0,1}Q_{m,0,1}$ is the central product of $C_{m,0,1}$ and $M_{m,0,1}$ over $Z(E_{m,0,1})$
			such that $M_{m,0,1}/E_{m,0,1} \cong Q_8$ with $Q_8$ being the quaternion group
			and $N_{m,0,1}/E_{m,0,1} \cong C_{m,0,1}/Z(R_{m,0,1}) \times Q_8$.
		\end{enumerate}
		\vspace{-0.5ex}
		\item If $\gamma+\nu(m)<a$ and $\alpha=0$,
		$\det(\tM_{m,0,\gamma})=1$, $N_{m,0,\gamma}=C_{m,0,\gamma}\tM_{m,0,\gamma}$
		and $N_{m,0,\gamma}/R_{m,0,\gamma} \cong C_{m,0,\gamma}/Z(R_{m,0,\gamma}) \times \Sp_{2\gamma}(\ell)$.
	\end{enumerate}
\end{prop}
\begin{proof}
	(1) In this case, $R_{m,\alpha,\gamma}^{tw}=\tR_{m,\alpha,\gamma}^{tw}$.
	So $C_{\tG_{m,\alpha,\gamma}^{tw}}(R_{m,\alpha,\gamma}^{tw})=\tC_{m,\alpha,\gamma}^{tw}$
	and $N_{\tG_{m,\alpha,\gamma}^{tw}}(R_{m,\alpha,\gamma}^{tw})=\tN_{m,\alpha,\gamma}^{tw}$,
	which are given in Proposition \ref{prop:tN-m,alpha,gamma-odd}.
	Thus $C_{m,\alpha,\gamma}^{tw} = \tC_{m,\alpha,\gamma}^{tw} \cap G_{m,\alpha,\gamma}^{tw} = \hbar(\hC_{m,\alpha,\gamma}^0)$
	and $N_{m,\alpha,\gamma}^{tw} = \left(\hbar(\tN_{m,\alpha,\gamma}^0)\cap G_{m,\alpha,\gamma}^{tw}\right) \rtimes V_{m,\alpha,\gamma}$
	since $\det(v_{m,\alpha,\gamma})=1$.
	
	(1.1) If we are not in the case (\ref{eq:special-case-3-1}), the assertions follow from Proposition~\ref{prop:tN-m,alpha,gamma-odd}.
	
	(1.2) Now, assume we are in the case (\ref{eq:special-case-3-1}),
	i.e. $\ell=3$, $\nu(m)=\alpha=0$, $\gamma=1$ and $(q-\eta)_3=3$.
	In particular, there is no need to introduce the twisting process.
	Since $\tR_{m,0,1}$ is special, it is readily seen that
	$|\tN_{m,0,1}/\tC_{m,0,1}E_{m,0,1}|/|N_{m,0,1}/C_{m,0,1}E_{m,0,1}| = |\det\tN_{m,0,1}|/|\det\tC_{m,0,1}| = 3$
	by Proposition \ref{prop:cc-R-m,alpha,gamma-odd}.
	Since $\tN_{m,0,1}/\tC_{m,0,1}E_{m,0,1} \cong \SL_2(3)$, $|\SL_2(3)|=24$
	and $\SL_2(3)$ has a Sylow $2$-subgroup isomorphic to the quaternion group $Q_8$,
	$N_{m,0,1}/C_{m,0,1}E_{m,0,1}$ is isomorphic to a subgroup of $Q_8$.
	On the other hand, let
	\[ n_m^{(vi)}=I_m\otimes (1-\zeta_3)^{-1}\begin{bmatrix} \zeta_3&\zeta_3&1\\ 1&\zeta_3^2&1\\ 1&\zeta_3&\zeta_3 \end{bmatrix}, \]
	then it can be verified that $n_m^{(vi)}\in G_{m,0,1}$,
	$$(n_m^{(vi)})^{-1}x_{m,1,1}^0n_m^{(vi)}=x_{m,1,1}^0y_{m,1,1}^0,\ (n_m^{(vi)})^{-1}y_{m,1,1}^0n_m^{(vi)}=x_{m,1,1}^0(y_{m,1,1}^0)^{-1}.$$
	So it is easy to see that $\Grp{n_{m,1}^{(i)},n_m^{(vi)},E_{m,0,1}} \cong Q_8$ and $\det \Grp{n_{m,1}^{(i)},n_m^{(vi)},E_{m,0,1}} = 1$
	(see p.\pageref{n-i} for the definition of $n_{m,1}^{(i)}$),
	Thus $M_{m,0,1}=\Grp{n_{m,1}^{(i)},n_m^{(vi)},E_{m,0,1}}$ and the rest follows easily.
	
	(2) In this case, there is no need to introduce the twisting process.
	Recall that $\tR_{m,0,\gamma}=\tZ_{m,0,\gamma}E_{m,0,\gamma}$, $R_{m,0,\gamma}=Z_{m,0,\gamma}E_{m,0,\gamma}$
	and $Z_{m,0,\gamma}=\tZ_{m,0,\gamma}\cap G_{m,0,\gamma}$.
	Since $\tZ_{m,0,\gamma}$ is central in $\tG_{m,0,\gamma}$ in this case,
	$C_{\tG_{m,0,\gamma}}(R_{m,0,\gamma})=\tC_{m,0,\gamma}$
	and $N_{\tG_{m,0,\gamma}}(R_{m,0,\gamma})=\tN_{m,0,\gamma}$.
	In this case, (\ref{eq:special-case-3-1}) can not happen,
	thus $\det(\tM_{m,0,\gamma})=1$ and the rest follows similarly as the case (1.1).
\end{proof}

\begin{rem}
	\cite[Prop.~4.8]{FLL17a} demonstrates the special case $m=1,\alpha=0,\gamma=1$ of the above propositions.
	The elements $n_1,n_2$ in the proof of \cite[Prop.~4.8]{FLL17a} are different from $n_{m,1}^{(ii)},n_m^{(vi)}$ in the above proof, this is because we choose the following embedding of the symmetric group $\fS(3)$ in $\SL_3(q)$ in \cite{FLL17a}:
	$$\Grp{ \begin{bmatrix} 0&1&0\\1&0&0\\0&0&-1 \end{bmatrix}, \begin{bmatrix} 0&0&-1\\-1&0&0\\0&1&0 \end{bmatrix} }.$$
	If we change the embedding to the following probably better one:
	$$\Grp{ \begin{bmatrix} 0&-1&0\\-1&0&0\\0&0&-1 \end{bmatrix}, \begin{bmatrix} 0&0&1\\1&0&0\\0&1&0 \end{bmatrix} },$$
	then the elements $n_1,n_2$ could be changed to $n_{m,1}^{(ii)},n_m^{(vi)}$ as above.
\end{rem}

Now, assume $\bc=(c_1,\dots,c_r)\neq\zero$ is a
sequence  of natural numbers
and consider the special radical subgroup of the form $\tR_{m,\alpha,\gamma,\bc}$.
Recall that $\tR_{m,\alpha,\gamma,\bc}=\tR_{m,\alpha,\gamma}\wr A_{\bc}$, where $A_{\bc}$ is defined as in \S \ref{subsect:basic-subgps}.
Denote by $\tB_{m,\alpha,\gamma,\bc}$ the base subgroup of $\tR_{m,\alpha,\gamma,\bc}$, then $\tR_{m,\alpha,\gamma,\bc}=\tB_{m,\alpha,\gamma} \rtimes A_{\bc}$.
Set $R_{m,\alpha,\gamma,\bc} := \tR_{m,\alpha,\gamma,\bc} \cap G_{m,\alpha,\gamma,\bc}$
and $B_{m,\alpha,\gamma,\bc} = \tB_{m,\alpha,\gamma,\bc} \cap G_{m,\alpha,\gamma,\bc}$,
then $R_{m,\alpha,\gamma,\bc}=B_{m,\alpha,\gamma} \rtimes A_{\bc}$ since $\det(A_{\bc})=1$.
To consider the condition for a basic subgroup $\tR_{m,\alpha,\gamma,\bc}$ to be special, we begin with two lemmas.

\begin{lem}\label{lem:normal-abelian}
	Assume $\ell$ is an odd prime and $\bc\neq\zero$.
	All normal abelian subgroups of $R_{m,\alpha,\gamma,\bc}$ are contained in $B_{m,\alpha,\gamma,\bc}$
	unless $\ell=3$, $a=1$, $\nu(m)=\alpha=\gamma=0$ and $\bc=\one$.
\end{lem}
\begin{proof}
	We argue by contradiction using the method of \cite[(a) of (2A)]{An92}.
	Assume $A$ is a normal abelian subgroup of $R_{m,\alpha,\gamma,\bc}$ not contained in $B_{m,\alpha,\gamma,\bc}$.
	First, assume $\bc=(c_1)$.
	Then there is an element $bh$ in $A$ with $b\in B_{m,\alpha,\gamma,\bc}$ and $1\neq h\in A_{\bc}$.
	So $b=\diag\{b_1,\ldots,b_{\ell^{c_1}}\}$ with $b_i \in \tR_{m,\alpha,\gamma}$ and $\det(b_1\cdots b_{\ell^{c_1}})=1$.
	Without loss of generality, we may assume $h$ permutes cyclically $1,\ldots,\ell$.
	Take $y=\diag\{y_1,1,\ldots,1\}$ with $y_1 \in E_{m,\alpha,\gamma}$, then $y \in B_{m,\alpha,\gamma,\bc}$.
	Note that $[y,bh]=ybhy^{-1}h^{-1}b^{-1}=\diag\{y_1,\ldots\} \in A$.
	If $\gamma>0$, let $y_1$ run over all elements of $E_{m,\alpha,\gamma}$,
	then there would be a contradiction to $A$ being abelian.
	Thus $\gamma=0$.
	Since $\tR_{m,\alpha}$ is abelian, for any $y=\diag\{y_1,y_1^{-1},1,\ldots,1\}$ with $y_1 \in \tR_{m,\alpha}$, we have
	\begin{equation}\label{eq:[y,bh]}
	[y,bh]=[y,h]=yhy^{-1}h^{-1}=\diag\{y_1,y_1^{-2},y_1,\ldots\} \in A.
	\addtocounter{thm}{1}\tag{\thethm}
	\end{equation}
	Since $A$ is abelian and $bh,[y,bh]\in A$, we have $bh[y,bh]h^{-1}b^{-1}=[y,bh]$.
	Consequently,
	\begin{equation}\label{eq:h[y,h]h-1}
	h[y,h]h^{-1}=[y,h].
	\addtocounter{thm}{1}\tag{\thethm}
	\end{equation}
	If $\ell>3$, $h[y,h]h^{-1}=\diag\{1,y_1,y_1^{-2},y_1,\ldots\}$,
	then by (\ref{eq:[y,bh]}) and (\ref{eq:h[y,h]h-1}), we have a contradiction by taking $y_1\neq1$.
	So we assume furthermore that $\ell=3$.
	Then $h[y,h]h^{-1}=\diag\{y_1,y_1,y_1^{-2},\ldots\}$.
	If $a+\alpha>1$, there exists $y_1\in\tR_{m,\alpha}$ such that $o(y_1)=9$,
	thus we have a contradiction again from (\ref{eq:[y,bh]}) and (\ref{eq:h[y,h]h-1}).	
	Now, assume $\ell=3$, $a=1$ and $\alpha=\gamma=0$.
	If $\nu(m)\neq0$, then $\det(\tR_m)=1$ and taking $y=\diag\{y_1,1,\ldots,1\}$ with $y_1 \in \tR_m$ will give a contradiction.
	So assume $\nu(m)=0$.
	If $\bc=(c_1)$ with $c_1\neq1$, we take $y'=\diag\{y_1,1,1,y_1^{-1},1,1,\ldots\}$ with $y_1\neq1$.
	Then $[y',bh]=[y',h]=\diag\{y_1,y_1^{-1},1,\ldots\} \in A.$
	So (\ref{eq:h[y,h]h-1}) also holds with $y'$ replacing $y$, but $h[y',h]h^{-1}=\diag\{1,y_1,y_1^{-1},\ldots\}$, which is a contradiction.
	Now, assume $\bc=(c_1,\ldots,c_r)$ with $r\geq2$.
	Let $\bc'=(c_1,\ldots,c_{r-1})$ and $\tR_{m,\alpha,\gamma,\bc} = \tR_{m,\alpha,\gamma,\bc'} \wr A_{c_r}$.
	Since $\tR_{m,\alpha,\gamma,\bc'}$ has a non-abelian subgroup with determinant $1$,
	an argument as before shows that any normal abelian subgroup of $\tR_{m,\alpha,\gamma,\bc}$
	is contained in the base subgroup of $\tR_{m,\alpha,\gamma,\bc'} \wr A_{c_r}$,
	thus induction on $r$ shows the assertion.
\end{proof}

\begin{rem}\label{rem:back-symplectic}
	When we are in the case
	\begin{equation}\label{eq:special-case-3-3}
	\ell=3, a=1, \nu(m)=\alpha=\gamma=0 ~\textrm{and}~ \bc=\one,
	\addtocounter{thm}{1}\tag{\thethm}
	\end{equation}
	$\tG_{m,0,0,\one}=\tG_{m,0,1}$ and $R_{m,0,0,\one} = \tR_{m,0,0,\one} \cap G_{m,0,0,\one} = \tR_{m,0,1} \cap G_{m,0,1} = R_{m,0,1}$,
	which has been addressed before.
	So from now on we consider the cases other than (\ref{eq:special-case-3-3}).
\end{rem}

\begin{lem}\label{lem:characteristic-subgp}
	Assume $\ell$ is an odd prime, $\bc\neq\zero$ and we are not in the case (\ref{eq:special-case-3-3}).
	Then $Z(B_{m,\alpha,\gamma,\bc})$ is a characteristic subgroup of $R_{m,\alpha,\gamma,\bc}$.
\end{lem}
\begin{proof}
	In fact, we prove that $Z(B_{m,\alpha,\gamma,\bc})$ is
	the intersection of all maximal normal abelian subgroups of $R_{m,\alpha,\gamma,\bc}$ as in \cite[p.14]{AF90}.
	Any maximal normal abelian subgroup of $R_{m,\alpha,\gamma,\bc}$
	contains $Z(B_{m,\alpha,\gamma,\bc})$ by Lemma \ref{lem:normal-abelian}.
	Conversely, given any $x\in \tR_{m,\alpha,\gamma}-Z(\tR_{m,\alpha,\gamma})$,
	there exists a maximal normal abelian subgroup $W$ of $\tR_{m,\alpha,\gamma}$ not containing $x$; see \cite[p.14]{AF90}.
	Set $A=\{ \diag\{w_1,\ldots,w_{\ell^{|\bc|}}\} \mid w_i\in W, \det(w_1\cdots w_{\ell^{|\bc}})=1 \}$.
	Then $A$ is a maximal normal abelian subgroup of $R_{m,\alpha,\gamma,\bc}$
	not containing any element of $B_{m,\alpha,\gamma,\bc}$ with a component $x$; see again \cite[p.14]{AF90}.
	This completes the proof.
\end{proof}

\begin{prop}\label{prop:special-odd-2}
	Assume $\ell$ is an odd prime, $\bc\neq\zero$ and we are not in the case (\ref{eq:special-case-3-3}).
	The basic subgroup $\tR_{m,\alpha,\gamma,\bc}$ of $\tG_{m,\alpha,\gamma,\bc}$ is special.
\end{prop}
\begin{proof}
	Assume $n\in N_{\tG_{m,\alpha,\gamma,\bc}}(R_{m,\alpha,\gamma,\bc})$.
	Then by Lemma \ref{lem:characteristic-subgp}, $n$ normalizes $Z(B_{m,\alpha,\gamma,\bc})$.
	Let $U$ be the underlying space of $\tG_{m,\alpha,\gamma,\bc}$ and consider the following two sets (see \cite[p.14]{AF90}):
	\begin{align*}
	\tilde{\cE} &= \{ [g,U] \mid g\in Z(\tB_{m,\alpha,\gamma,\bc}), g\neq1 \};\\
	\cE &= \{ [g_1,[g_2,U]] \mid g_i\in Z(B_{m,\alpha,\gamma,\bc}), g_i\neq1 \}.
	\end{align*}
	Since $\ell^{|\bc|}>2$, the minimal elements of $\tilde{\cE}$ and $\cE$ are the same,
	each being the underlying space of a factor of the base subgroup $\tB_{m,\alpha,\gamma,\bc}$.
	Also note that the projection from $Z(B_{m,\alpha,\gamma,\bc})$ to each component $Z(\tR_{m,\alpha,\gamma})$ is surjective.
	From these, we can conclude that the normalizers of $Z(B_{m,\alpha,\gamma,\bc})$
	and $Z(\tB_{m,\alpha,\gamma,\bc})$ in $\tG_{m,\alpha,\gamma,\bc}$ are the same,
	since the normalizer permutes the minimal elements of $\tilde{\cE}$ and $\cE$.
	So $n$ normalizes $Z(\tB_{m,\alpha,\gamma,\bc})$.
	But note that $\tR_{m,\alpha,\gamma,\bc}=R_{m,\alpha,\gamma,\bc}Z(\tB_{m,\alpha,\gamma,\bc})$,
	thus $n$ normalizes $\tR_{m,\alpha,\gamma,\bc}$.
	So $N_{\tG_{m,\alpha,\gamma,\bc}}(R_{m,\alpha,\gamma,\bc}) = N_{\tG_{m,\alpha,\gamma,\bc}}(\tR_{m,\alpha,\gamma,\bc})$,
	which means that $\tR_{m,\alpha,\gamma,\bc}$ is special.
\end{proof}

For $*\in\{tw,\varnothing\}$,
denote by $C_{m,\alpha,\gamma,\bc}^*$ and $N_{m,\alpha,\gamma,\bc}^*$ respectively
the centralizer and normalizer of $R_{m,\alpha,\gamma,\bc}^*$ in $G_{m,\alpha,\gamma,\bc}^*$.

\begin{notation}\label{notation:section-Sn}
	For any positive integer $k$,
	denote by $\tW_k$ the subgroup of all permutation matrices in $\GL_k(\eta q)$
	and set $W_k = \{ (\det\sigma)\sigma \mid \sigma\in\tW_k \}$, then $\tW_k \cong W_k \cong \fS_k$.
	Note that if $k$ is odd,
	then $W_k\leqslant\SL_k(\eta q)$,
	which means that $W_k$ is an embedding of $\fS_k$ in $\SL_k(\eta q)$.
	Recall from Lemma~\ref{lem:tCtN} that
	$$\tN_{m,\alpha,\gamma,\bc} = (\tN_{m,\alpha,\gamma}/\tR_{m,\alpha,\gamma}) \otimes N_{\fS(\ell^{|\bc|})}(A_\bc),$$
	where $\otimes$ is defined as \cite[(1.3) (1.5)]{AF90}.
	Now, we change slightly the descrption of the elements on the right hand of the above equation.
	We identify $\fS(\ell^{|\bc|})$ with its embedding $W_{\ell^{|\bc|}}$.
	Since $\ell$ is odd, $\det(W_{\ell^{|\bc|}})=1$.
	Note that $A_{\bc}\leqslant W_{\ell^{|\bc|}}$.
	Let $\sigma$ be a permutation matrix normalizing $A_{\bc}$, then $\det(\sigma)\sigma\in W_{\ell^{|\bc|}}$.
	Replace the non-zero element in $\det(\sigma)\sigma$
	by $\det(\sigma)n_1,\ldots,\det(\sigma)n_{\ell^{|\bc|}}$
	with $n_i\in\tN_{m,\alpha,\gamma}$ lying in the same coset of $\tR_{m,\alpha,\gamma}$
	and denote the resulting element as $\sigma(n_1,\ldots,n_{\ell^{|\bc|}})$,
	then $\sigma(n_1,\ldots,n_{\ell^{|\bc|}}) \in \tN_{m,\alpha,\gamma,\bc}$.
	Conversely, any elements of $\tN_{m,\alpha,\gamma,\bc}$ is of this form.
	In the subsection \S \ref{subsect:srs-odd}, we will always represent the elements in $\tN_{m,\alpha,\gamma,\bc}$ in this way.
	In particular $\det(\sigma(n_1,\ldots,n_{\ell^{|\bc|}}))=\det(n_1)\cdots\det(n_{\ell^{|\bc|}})$.
\end{notation}

\begin{prop}\label{prop:cc-R-m,alpha,gamma,c-odd}
	Assume $\ell$ is an odd prime and $\bc\ne\zero$.
	\vspace{-0.5ex}
	\begin{enumerate}[\rm(1)]\setlength{\itemsep}{-0.5ex}
		\item $\det(\tR_{m,\alpha,\gamma,\bc}) = \det(\tR_{m,\alpha,\gamma}) = \fZ_{\ell^a}^{\ell^{\nu(m)+\gamma}}$,
		$\det(\tC_{m,\alpha,\gamma,\bc}) = (\det\tC_{m,\alpha,\gamma})^{\ell^{|\bc|}}=\fZ_{q-\eta}^{\ell^{\gamma+|\bc|}}$
		and $\det(\tN_{m,\alpha,\gamma,\bc}) = \Grp{\fZ_{q-\eta}^{\ell^{\gamma+|\bc|}},\fZ_{\ell^a}^{\ell^{\nu(m)+\gamma}}}$.
		\item Assume $\tR_{m,\alpha,\gamma,\bc}$ is special in $\tG_{m,\alpha,\gamma,\bc}$,
		then	the number of $G_{m,\alpha,\gamma,\bc}$-conjugacy classes of radical subgroups
		contained in the set $\{ G_{m,\alpha,\gamma,\bc}\cap{^{\tg}\tR_{m,\alpha,\gamma,\bc}} \mid \tg\in\tG_{m,\alpha,\gamma,\bc} \}$
		is $\ell^{\min\{a,\gamma+|\bc|,\gamma+\nu(m)\}}.$
	\end{enumerate}
\end{prop}

\begin{proof}
	The assertions about $\det(\tR_{m,\alpha,\gamma,\bc})$ and $\det(\tC_{m,\alpha,\gamma,\bc})$ is seen by direct computation.
	Assume $n\in\tN_{m,\alpha,\gamma,\bc}$ has non-zero components $\det(\sigma)n_1$, \dots, $\det(\sigma)n_{\ell^{|\bc|}}$ with
	$n_i\in\tN_{m,\alpha,\gamma}$ for some permutation matrix $\sigma$ normalizing $A_{\bc}$,
	then by Notation \ref{notation:section-Sn}, $\det(n)=\det(n_1)\cdots\det(n_{\ell^{|\bc|}})$.
	Recall that $\tN_{m,\alpha,\gamma} = \tC_{m,\alpha,\gamma}\tM_{m,\alpha,\gamma} \rtimes \iota(V_{m,\alpha,\gamma})$.
	Since the $n_i$'s lie in the same $\tR_{m,\alpha,\gamma}$-coset,
	we may assume $n_i=c_0m_0v_0r_i$,
	where $c_0\in\tC_{m,\alpha,\gamma}$, $m_0\in \tM_{m,\alpha,\gamma}$,
	$v_0\in \iota(V_{m,\alpha,\gamma})$ and $r_i\in\tR_{m,\alpha,\gamma}$.
	Note that $\det(\iota(V_{m,\alpha,\gamma}))=\det(V_{m,\alpha,\gamma})=1$,
	$\det(\tM_{m,\alpha,\gamma}) = 1$ or $\fZ_3$ by Proposition \ref{prop:tN-m,alpha,gamma-odd} and $|\bc|>0$.
	Thus $\det(n) = \det(c_0)^{\ell^{|\bc|}}\det(r_1\cdots r_{\ell^{|\bc|}})$.
	So $\det(\tN_{m,\alpha,\gamma,\bc})
	= \Grp{\det(\tC_{m,\alpha,\gamma})^{\ell^{|\bc|}},\det(\tR_{m,\alpha,\gamma})}
	= \Grp{\fZ_{q-\eta}^{\ell^{\gamma+|\bc|}},\fZ_{\ell^a}^{\ell^{\nu(m)+\gamma}}}$.
	Then (2) follows by Lemma~\ref{lem:split-classes-radical}.
\end{proof}

Finally, we consider the special radical subgroups of the form $\tR=\tR_1\times\cdots\times\tR_u$
with $u>1$ and $\tR_i=\tR_{m_i,\alpha_i,\gamma_i,\bc_i}$.
Let $R=G\cap\tR$.
First, we give some notation.
Let $\tR$ and $R$ be as above.
Set $\tC=C_{\tG}(\tR)$, $\tN=N_{\tG}(\tR)$ and $C=C_G(R)$, $N=N_G(R)$.
Denote $\tC_i=\tC_{m_i,\alpha_i,\gamma_i,\bc_i}$, $\tN_i=\tN_{m_i,\alpha_i,\gamma_i,\bc_i}$ for each $i$.
Set
\begin{align*}
a(\tR_i)=\nu(m_i)+\gamma_i,\quad &a(\tR) = \min\{ a(\tR_i) \mid i=1,\dots,u \};\\
\delta(\tR_i)=\gamma_i+|\bc_i|,\quad &\delta(\tR) = \min\{ \delta(\tR_i) \mid i=1,\dots,u \}.
\end{align*}
The structure of $\tN$ is given in \cite[(4B)]{AF90} and \cite[(2C)]{An94}.
Note that the symmetric groups there consist of permutation matrices, which have determinants $\pm 1$, an $\ell'$-element.
By Proposition \ref{prop:cc-R-m,alpha,gamma-odd} and  \ref{prop:cc-R-m,alpha,gamma,c-odd},
we have the following.

\begin{lem}\label{det-radical-N-C-odd}
	Assume $\ell$ is an odd prime.
	Let $\tR$ be a radical subgroup of $\tG$ which can be expressed as 
	$\tR=\tR_1\times\cdots\times \tR_u$, a direct product of basic subgroups.
	\vspace{-0.5ex}
	\begin{enumerate}[\rm(1)]\setlength{\itemsep}{-0.5ex}
		\item $\mrO_{\ell'}(\det(\tN))=\mrO_{\ell'}(\det(\tR\tC))=\mrO_{\ell'}(\fZ_{q-\eta})$.
		In particular, $|\tG:GN_{\tG}(\tR)|$ is an $\ell$-number.
		\item If we are not in the case that
		\begin{equation}\label{special-case-wei-ell3}
		\begin{aligned}
		&\ell=3, a=1, a(\tR)=1  \ \textrm{and}\\
		&\textrm{there exists some $1\le i\le u$ satisfying}\ \tR_{i}=\tR_{m_{i},0,1}\  \textrm{and} \  3\nmid m_{i},
		\end{aligned}
		\addtocounter{thm}{1}\tag{\thethm}
		\end{equation}
		then $\det(\tN)=\det(\tR\tC)$.
		\item If we are in the case (\ref{special-case-wei-ell3}),
		then $\nu(\det(\tR\tC))=0$ and $\nu(\det(\tN))=1$.
	\end{enumerate}
\end{lem}

\begin{prop}\label{prop:special-odd-3}
	Assume $\ell$ is an odd prime and $\tR=\tR_1\times\cdots\times\tR_u$ with $u>1$ and $a(\tR_1) \leq\cdots\leq a(\tR_u)$.
	Then $\tR$ is special if and only if one of the following holds:
	\vspace{-0.5ex}
	\begin{enumerate}[\rm(1)]\setlength{\itemsep}{-0ex}
		\item $a(\tR)\geqslant a$;
		\item $\sum\limits_{i:a(\tR_i)=a(\tR)} \ell^{|\bc_i|} \geqslant2$ except when
		\vspace{-0.5ex}
		\begin{enumerate}[\rm(2a)]\setlength{\itemsep}{-0.5ex}
			\item\label{special-odd-3-1} $\tR_1 = \tR_{m_1}, \tR_2 = \tR_{m_2}$, $\nu(m_1)=\nu(m_2)=0, \nu(m_1+m_2) \geq a$,
			$\det (\tR_3 \times\cdots\times \tR_u)=1$;
			\item\label{special-odd-3-2} $\ell=3, a=1$, $a(\tR_1)=\alpha_1=0, \bc_1=\one$ and $a(\tR_i)>0$ for $i>1$.
		\end{enumerate}
		\item $a(\tR_1)=a(\tR)<\min\{a,a(\tR_2)\}$ with $\bc_1=\zero$ and $\alpha_1=0$.
	\end{enumerate}
	Finally, if $\tR$ is special, then $C_{\tG}(R)=\tC$.
\end{prop}

\begin{proof}
	(1) If $a(\tR)\geqslant a$, then $\det(\tR)=1$ and $R = \tR \cap G = \tR$, thus $\tR$ is special by definition.
	So we assume $a(\tR)<a$ from now on.
	
	(2) Assume $\sum_{i:a(\tR_i)=a(\tR)} \ell^{|\bc_i|} \geqslant2$.
	Let $\tB_i$ be the base subgroup of $\tR_i$, $\tB = \tB_1 \times \cdots \times \tB_u$ and $B = \tB \cap G$.
	Assume we are not in the case (\ref{special-odd-3-1}) or (\ref{special-odd-3-2}),
	then similar arguments as in Lemma \ref{lem:characteristic-subgp} show that $Z(B)$ is a characteristic subgroup of $R$
	and the normalizers of $Z(B)$ and $Z(\tB)$ coincide and
	the same arguments as in Proposition \ref{prop:special-odd-2} show that $\tR$ is special.
	If we are in the case (\ref{special-odd-3-1}),
	direct calculation shows that $R = \tR_{m_1+m_2} \times \tR_3 \times \cdots \times \tR_u$,
	thus $\tR$ is not special.
	If we are in the case (\ref{special-odd-3-1}), $\tR$ is not special by Proposition \ref{prop:special-odd-2}

	(3) Now, assume $a(\tR_1)=a(\tR)<\min\{a,a(\tR_2)\}$ with $\bc_1=\zero$.
	We may replace $\tR_{m_1,\alpha_1,\gamma_1}$ by $\tR_{m_1\ell^{\alpha_1-\alpha'_1},\alpha'_1,\gamma_1}$
	with some $\alpha'_1<\alpha_1$ and we either go to (1) or (2) or go to the case $\alpha'_1=0$.
	
	The last statement follows from case by case verification.
\end{proof}

\subsection{Special radical subgroups for $\ell=2$ and $4 \mid q-\eta$.}
\label{subsect:srs-2-linear}

Now, we classify the special radical subgroups of $\tG = \GL_n(\eta q)$ for $\ell=2$ and $4 \mid q-\eta$.
Recall that $\nu$ is the discrete valuation such that $\nu(2)=1$ and $a=\nu(q-\eta)$.

We start also with the radical subgroup of the form $\tR_{m,\alpha,\gamma}$.
Keep all the notation in \S\ref{subsect:gst-2-linear} and
use notation in \S\ref{subsect:srs-odd} for the case when $\ell=2$ and $4 \mid q-\eta$ with the odd prime $\ell$ replaced by $\ell=2$.

\begin{prop}\label{prop:special-2-linear-1}
	Assume $\ell=2$ and $4 \mid q-\eta$.
	The basic subgroup $\tR_{m,\alpha,\gamma}$ of $\tG_{m,\alpha,\gamma}$ is special
	if and only if $\gamma+\nu(m)\geqslant a$ or $\alpha=0$.
\end{prop}
\begin{proof}
	Similar as for Proposition \ref{prop:special-odd-1}.
\end{proof}

\begin{prop}\label{prop:cc-R-m,alpha,gamma-2-linear}
	Assume $\ell=2$ and $4 \mid q-\eta$.
	\vspace{-0.5ex}
	\begin{enumerate}[\rm(1)]\setlength{\itemsep}{-0.5ex}
		\item $\det(\tR_{m,\alpha,\gamma})=\det(\tR_{m,\alpha,\gamma}^{tw})=\fZ_{2^a}^{2^{\nu(m)+\gamma}}$;
		$\det(\tC_{m,\alpha,\gamma})=\det(\tC_{m,\alpha,\gamma}^{tw})=\fZ_{q-\eta}^{2^\gamma}$.
		\item We have $\det(\tN_{m,\alpha,\gamma})=\det(\tN_{m,\alpha,\gamma}^{tw})=\fZ_{q-\eta}^{2^\gamma}$
		unless in case (\ref{eq:special-case-2-linear-1}), in which case,
		$\det(\tN_{m,0,\gamma}) = \Grp{\det(\tC_{m,0,\gamma}),\det(\tM_{m,0,\gamma})}=\fZ_{q-\eta}^{2^{\gamma-1}}$.
		\item Assume $\tR_{m,\alpha,\gamma}$ is special in $\tG_{m,\alpha,\gamma}$,
		then	the number of $G_{m,\alpha,\gamma}$-conjugacy classes of radical subgroups contained in the set
		$\{ G_{m,\alpha,\gamma}\cap{^{\tg}\tR_{m,\alpha,\gamma}} \mid \tg\in\tG_{m,\alpha,\gamma} \}$ is
		\[\left\{\begin{array}{ll}
		2^{\min\{a,\gamma\}} & \textrm{if (\ref{eq:special-case-2-linear-1}) does not hold;}\\
		2^{\gamma-1}, & \textrm{if (\ref{eq:special-case-2-linear-1}) holds.}
		\end{array}\right.\]
	\end{enumerate}
\end{prop}
\begin{proof}
	Similar as Proposition \ref{prop:cc-R-m,alpha,gamma-odd}.
\end{proof}

We define the notation $\hC_{m,\alpha,\gamma}^0,\hN_{m,\alpha,\gamma}^0$, etc. similarly as in \S\ref{subsect:srs-odd}.
Assume that $\tR_{m,\alpha,\gamma}$ is special in $\tG_{m,\alpha,\gamma}$.
Denote by $C_{m,\alpha,\gamma}^*$, $N_{m,\alpha,\gamma}^*$ for the centralizer and normalizer
of $R_{m,\alpha,\gamma}^*$ in $G_{m,\alpha,\gamma}^*$
and $M_{m,\alpha,\gamma}^*=\tM_{m,\alpha,\gamma}^*\cap G_{m,\alpha,\gamma}^*$		
for $*\in\set{tw,\varnothing}$.

\begin{prop}\label{prop:CN_m,alpha,gamma-2-linear}
We have $C_{m,\alpha,\gamma}^{tw}=\hbar(\hC_{m,\alpha,\gamma}^0)$ and the following holds.
	\vspace{-0.5ex}
	\begin{enumerate}[\rm(1)]\setlength{\itemsep}{-0.5ex}
		\item If we are not in the case (\ref{eq:special-case-2-linear-1}), $\det(\tM_{m,\alpha,\gamma}^{tw})=1$.
		Thus $\hbar(\tN_{m,\alpha,\gamma}^0) \cap G_{m,\alpha,\gamma}^{tw}
		= \hbar(\hN_{m,\alpha,\gamma}^0) = C_{m,\alpha,\gamma}^{tw}\tM_{m,\alpha,\gamma}^{tw}$,
		$\hbar(\hN_{m,\alpha,\gamma}^0)/C_{m,\alpha,\gamma}^{tw}R_{m,\alpha,\gamma}^{tw} \cong \Sp_{2\gamma}(2)$
		and $N_{m,\alpha,\gamma}^{tw}=\hbar(\hN_{m,\alpha,\gamma}^0)V_{m,\alpha,\gamma}$.
		When we are not in the case (\ref{eq:special-case-2-linear-0}), we have furthermore that
		$\hbar(\hN_{m,\alpha,\gamma}^0)/R_{m,\alpha,\gamma}^{tw}
		\cong C_{m,\alpha,\gamma}^{tw}/Z(R_{m,\alpha,\gamma}^{tw}) \times \Sp_{2\gamma}(2)$.
		\item If we are in the case (\ref{eq:special-case-2-linear-1}) and $\gamma=1$,
		$R_{m,0,1} \cong Q_8$ and $N_{m,0,1}/C_{m,0,1}R_{m,0,1} \cong C_3$.
		\item If we are in the case (\ref{eq:special-case-2-linear-1}) and $\gamma=2$,
		$N_{m,0,2}/C_{m,0,2}R_{m,0,2} \cong \fA_6$, the alternating group of $6$ symbols.
	\end{enumerate}
\end{prop}

\begin{proof}
	Similar as in Proposition \ref{prop:CN_m,alpha,gamma-odd}, noting that $\Sp_2(2) \cong \fS_3\cong C_3 \rtimes C_2$ and $\Sp_4(2) \cong \fS_6$.
\end{proof}

\begin{rem}\label{M-for-not-central-product}
If the definition of $\tM_{m,\alpha,\gamma}^0,\tM_{m,\alpha,\gamma}^{tw},\tM_{m,\alpha,\gamma}$ is chosen as in Remark \ref{not-cent},
then the isomorphism $\hbar(\hN_{m,\alpha,\gamma}^0)/R_{m,\alpha,\gamma}^{tw} \cong
C_{m,\alpha,\gamma}^{tw}/Z(R_{m,\alpha,\gamma}^{tw}) \times \Sp_{2\gamma}(2)$ holds without any restriction.

Assume that we are in the case (\ref{eq:special-case-2-linear-0}). Then $\det(\tM_{m,\alpha,\gamma})=\{\pm 1\}$.
In addition, if we let $M_{m,\alpha,\gamma}=\tM_{m,\alpha,\gamma}\cap G_{m,\alpha,\gamma}$, then $R_{m,\alpha,\gamma}M_{m,\alpha,\gamma}/R_{m,\alpha,\gamma}\cong C_3$ or $\fA_6$ according as $\gamma=1$ or $2$.
\end{rem}

Now, assume $\bc\neq\zero$ and consider the special radical subgroup of the form $\tR_{m,\alpha,\gamma,\bc}$.
Let $R_{m,\alpha,\gamma,\bc} := \tR_{m,\alpha,\gamma,\bc} \cap G_{m,\alpha,\gamma,\bc}$.
Recall that $\tR_{m,\alpha,\gamma,\bc} = \tR_{m,\alpha,\gamma}\wr (A_{c_1} \wr\cdots\wr A_{c_r})$ with $\bc=(c_1,\ldots,c_r)$.
Denote by $\tB_{m,\alpha,\gamma,\bc}$ the base subgroup of $\tR_{m,\alpha,\gamma,\bc}$.
For $i\geq2$, the elements in $A_{c_i}$ are represented by matrices of the form
$I_{m2^{\alpha+\gamma+c_1+\cdots+c_{i-1}}} \otimes P$ with $P$ a permutation matrix, thus have determinant $1$.
On the other hand, the elements in $A_{c_1}$ are represented by matrices of the form
$I_{m2^{\alpha+\gamma}} \otimes P$ with $P$ a permutation matrix consisting of $2^{c_1-1}$ transpositions,
thus when $\nu(m)+\alpha+\gamma>0$ or $c_1>1$, $A_{c_1}$ has determinant $1$.
Now, when $\nu(m)+\alpha+\gamma=0$ and $c_1=1$,
we replace $A_{c_1}$ by the group $\Grp{ I_m \otimes \begin{bmatrix} 0&1\\-1&0 \end{bmatrix} }$,
then $\det A_{c_1} =1$,
but note that in this case, $\tB_mA_{c_1}$ is not a semidirect product.
Set $B_{m,\alpha,\gamma,\bc} = \tB_{m,\alpha,\gamma,\bc} \cap G_{m,\alpha,\gamma,\bc}$,
then $R_{m,\alpha,\gamma,\bc}=B_{m,\alpha,\gamma,\bc}A_{\bc}$ since $\det(A_{\bc})=1$.

\begin{lem}\label{lem:normal-abelian-2}
	Assume $\ell=2$, $4 \mid q-\eta$ and $\bc\neq\zero$.
	All normal abelian subgroups of $R_{m,\alpha,\gamma,\bc}$ are contained in $B_{m,\alpha,\gamma,\bc}$
	unless $\alpha=\gamma=0$, $a=2$, $\nu(m)\leq1$ and $\bc=\one$.
\end{lem}
\begin{proof}
	We argue as the proof of Lemma \ref{lem:normal-abelian}.
	Assume $A$ is a normal abelian subgroup of $R_{m,\alpha,\gamma,\bc}$ not contained in $B_{m,\alpha,\gamma,\bc}$,
	and assume first $\bc=(c_1)$.
	Then there is an element $bh$ in $A$ with $b\in B_{m,\alpha,\gamma,\bc}$ and $1\neq h\in A_{\bc}$.
	So $b=\diag\{b_1,\ldots,b_{\ell^{c_1}}\}$ with $b_i \in \tR_{m,\alpha,\gamma}$ and $\det(b_1\cdots b_{\ell^{c_1}})=1$.
	Without loss of generality, we may assume $h$ transposes $1$ and $2$.
	By the same argument as in the proof of Lemma \ref{lem:normal-abelian},
	we have $\gamma=0$,
	and for any $y=\diag\{y_1,y_1^{-1},1,\ldots,1\}$ with $y_1 \in \tR_{m,\alpha}$, we have
	\[ [y,bh]=[y,h]=yhy^{-1}h^{-1}=\diag\{y_1^2,y_1^{-2},1,\ldots,1\} \in A, \]
	\[ h[y,h]h^{-1}=[y,h], \]
	and \[ h[y,h]h^{-1}=\diag\{y_1^{-2},y_1^2,1,\ldots,1\}. \]
	If $a+\alpha>2$, there exists $y_1\in\tR_{m,\alpha}$ such that $o(y_1)=8$,
	then we have a contradiction from the above three equations.
	Now, assume $\gamma=\alpha=0$ and $a=2$.
	If $\nu(m)\geq2$, then $\det(\tR_m)=1$ and taking $y=\diag\{y_1,1,\ldots,1\}$ with $y_1 \in \tR_m$ and $o(y_1)=4$ will give a contradiction.
	So we have $\nu(m)\leq1$.
	The rest can be proved as Lemma \ref{lem:normal-abelian}.
\end{proof}

\begin{rem}\label{rem:back-symplectic-2}
	When we are in the following case:
	\begin{equation}\label{eq:special-case-2-linear-2}
	\ell=2,\ 4 \mid q-\eta,\ a=2,\ \alpha=\gamma=0,\ \nu(m)\leq1,\ \bc=\one,
	\addtocounter{thm}{1}\tag{\thethm}
	\end{equation}
	then $\tG_{m,0,0,\one}=\tG_{m,0,1}$ and $R_{m,0,0,\one} = \tR_{m,0,0,\one} \cap G_{m,0,0,\one} = \tR_{m,0,1} \cap G_{m,0,1} = R_{m,0,1}$,
	which has been addressed before.
	So now we consider the cases other than (\ref{eq:special-case-2-linear-2}).
	Note that $R_{m,0,1} \cong Z_0E_1$ when $\nu(m)=1$ and $R_{m,0,1} \cong Q_8$ when $\nu(m)=0$.
\end{rem}

\begin{prop}\label{prop:special-2-linear-2}
	Assume $\ell=2$, $4 \mid q-\eta$ and $\bc\neq\zero$.
	The basic subgroup $\tR_{m,\alpha,\gamma,\bc}$ of $\tG_{m,\alpha,\gamma,\bc}$ is special
	unless (\ref{eq:special-case-2-linear-2}) holds.
\end{prop}
\begin{proof}
	If $|\bc|>1$, this can be proved by the same arguments as for Proposition \ref{prop:special-odd-2}.
	If $|\bc|=1$ and (\ref{eq:special-case-2-linear-2}) does not hold,
	there are elements $g \in \tR_{m,\alpha,\gamma}$ with $g\neq1$ and $\det g = 1$,
	then the following two sets
	\begin{align*}
	\tilde{\cE} &= \{ [g,U] \mid g\in Z(\tB_{m,\alpha,\gamma,\bc}), g\neq1 \},\\
	\cE &= \{ [g,U] \mid g\in Z(B_{m,\alpha,\gamma,\bc}), g\neq1 \}
	\end{align*}
	are the same,
	then the assertion in this case can be proved again by the same arguments as for Proposition \ref{prop:special-odd-2}.
\end{proof}

For $*\in\{tw,\varnothing\}$,
denote by $C_{m,\alpha,\gamma,\bc}^*$ and $N_{m,\alpha,\gamma,\bc}^*$ respectively
the centralizer and normalizer of $R_{m,\alpha,\gamma,\bc}^*$ in $G_{m,\alpha,\gamma,\bc}^*$.

\begin{rem}
	By results in \cite{An92,An93},
	$$\tN_{m,\alpha,\gamma,\bc} = (\tN_{m,\alpha,\gamma}/\tR_{m,\alpha,\gamma}) \otimes N_{\fS(2^{|\bc|})}(A_\bc),$$
	where $\otimes$ is defined as in \cite[(1.3) (1.5)]{AF90}.
	Identify $\fS(2^{|\bc|})$ with the group of permutation matrices in $\GL_{2^{|\bc|}}(\eta q)$.
	Let $\sigma$ be a permutation matrix normalizing $A_{\bc}$.
	Replace the non-zero element in $\sigma$ by $n_1,n_2,\ldots,n_{2^{|\bc|}}$
	with $n_i\in\tN_{m,\alpha,\gamma}$ lying in the same coset of $\tR_{m,\alpha,\gamma}$
	and denote the resulting element as $\sigma(n_1,n_2,\ldots,n_{2^{|\bc|}})$,
	then $\sigma(n_1,n_2,\ldots,n_{2^{|\bc|}}) \in \tN_{m,\alpha,\gamma,\bc}$.
	Conversely, any element of $\tN_{m,\alpha,\gamma,\bc}$ is of this form.
	Then $\det(\sigma(n_1,n_2,\ldots,n_{2^{|\bc|}}))=\det(n_1)\det(n_2)\cdots\det(n_{2^{|\bc|}})$
	unless $\nu(m)+\alpha+\gamma=0$, in which case,
	$\det(\sigma(n_1,n_2,\ldots,n_{2^{|\bc|}})) = \pm \det(n_1)\det(n_2)\cdots\det(n_{2^{|\bc|}})$.
\end{rem}

\begin{prop}\label{prop:cc-R-m,alpha,gamma,c-2-linear}
	Assume $\ell=2$, $4 \mid q-\eta$ and $\bc\ne \zero$.
	\vspace{-0.5ex}
	\begin{enumerate}[\rm(1)]\setlength{\itemsep}{-0.5ex}
		\item $\det(\tR_{m,\alpha,\gamma,\bc}) = \det(\tR_{m,\alpha,\gamma}) = \fZ_{2^a}^{2^{\nu(m)+\gamma}}$,
		$\det(\tC_{m,\alpha,\gamma,\bc})=\det(\tC_{m,\alpha,\gamma})^{2^{|\bc|}}=\fZ_{q-\eta}^{2^{\gamma+|\bc|}}$
		and $\det(\tN_{m,\alpha,\gamma,\bc}) = \Grp{\fZ_{q-\eta}^{2^{\gamma+|\bc|}},\fZ_{2^a}^{2^{\nu(m)+\gamma}}}$.
		\item Assume $\tR_{m,\alpha,\gamma,\bc}$ is special in $\tG_{m,\alpha,\gamma,\bc}$,
		then	the number of $G_{m,\alpha,\gamma,\bc}$-conjugacy classes of radical subgroups
		contained in the set $\{ G_{m,\alpha,\gamma,\bc}\cap{^{\tg}\tR_{m,\alpha,\gamma,\bc}} \mid \tg\in\tG_{m,\alpha,\gamma,\bc} \}$
		is $2^{\min\{a,\gamma+|\bc|,\gamma+\nu(m)\}}.$
	\end{enumerate}
\end{prop}

\begin{proof}
	This can be proved similarly as Proposition \ref{prop:cc-R-m,alpha,gamma,c-odd}
	using the observation of the above remark and Proposition \ref{prop:tN-m,alpha,gamma-2-linear}.
\end{proof}

Finally, we consider the special radical subgroups of the form $\tR=\tR_1\times\cdots\times\tR_u$
with $u>1$ and $\tR_i=\tR_{m_i,\alpha_i,\gamma_i,\bc_i}$.
We use the same notation preceding Proposition \ref{det-radical-N-C-odd}.

\begin{lem}\label{det-radical-N-C-2-linear}
	Assume $\ell=2$ and $4 \mid q-\eta$.
	Let $\tR$ be a radical subgroup of $\tG$ which can be expressed as 
	$\tR=\tR_1\times\cdots\times \tR_u$, a direct product of basic subgroups.
	Denote $\tC=C_{\tG}(\tR)$, $\tN=N_{\tG}(\tR)$.
	\vspace{-0.5ex}
	\begin{enumerate}[\rm(1)]\setlength{\itemsep}{-0.5ex}
		\item $\mrO_{2'}(\det(\tN))=\mrO_{2'}(\det(\tR\tC))=\mrO_{2'}(\fZ_{q-\eta})$.
		In particular, $|\tG:GN_{\tG}(\tR)|$ is a $2$-number.
		\item We have $\det(\tN)=\det(\tR\tC)$ unless 
		\begin{equation} \label{special-case-wei-ell2-linear}
		\begin{aligned}
		&a=2 \ \text{and there exists some}\ 1\le i\le u\
		\text{such that}\\
		&\tR_i=\tR_{m_{i},0,\gamma_i}\ \text{with}\ \nu(m_i)=0, \gamma_i=1,2 \ \text{and}\ a(\tR_i)=a(\tR)=\delta(\tR),
		\end{aligned}
			\addtocounter{thm}{1}\tag{\thethm}		
		\end{equation}
			in which case, $\nu(\det(\tR\tC))=2-\gamma_i$ and $\nu(\det(\tN))=3-\gamma_i$.
	\end{enumerate}
\end{lem}

\begin{proof}
The structure of $\tN$ is given in 	\cite[2C]{An92} and \cite[2C]{An93}.
We rewrite $\tR=\tR_1^{t_1}\times\cdots\times \tR_s^{t_s}$ where $R_i$'s are distinct basic subgroups.
Then $\tN=\prod_{i=1}^s \tN_i\wr \fS(t_i)$. Here, the symmetric groups consist of permutation matrices, whose determinants are $\pm 1$.
We claim that $\det(\tN)=\grp{\det(\tN_1),\ldots, \det(\tN_s)}$. Then this proposition follows by
 Proposition \ref{prop:cc-R-m,alpha,gamma-2-linear} and  \ref{prop:cc-R-m,alpha,gamma,c-2-linear} immediately.

If $\det(\tN)=\grp{\det(\tN_1),\ldots, \det(\tN_s)}$ does not hold, then $\grp{\det(\tN_1),\ldots, \det(\tN_s)}$ is a $2'$-group and $2\mid|\det(\tN)|$.
Then there exists $i_0$ satisfying that $\tR_{i_0}=\tR_{m,0,0,\zero}$ with odd $m$.
However, in this case, $\det(\tN_{i_0})=\det(\tC_{i_0})=\fZ_{q-\eta}$, and this is a contradiction. Thus the claim holds.
\end{proof}

\begin{prop}\label{prop:special-2-linear-3}
	Assume $\ell=2$, $4 \mid q-\eta$ and $\tR=\tR_1\times\cdots\times\tR_u$ with $u>1$ and $a(\tR_1) \leq\cdots\leq a(\tR_u)$.
	Then $\tR$ is special if and only if one of the following holds:
	\begin{enumerate}[\rm(1)]\setlength{\itemsep}{-2pt}
		\item $a(\tR)\geqslant a$;
		\item $\sum\limits_{i:a(\tR_i)=a(\tR)} 2^{|\bc_i|} \geqslant2$ except when
		\begin{enumerate}[\rm(2a)]\setlength{\itemsep}{-2pt}
			\item $\tR_1 = \tR_{m_1}, \tR_2 = \tR_{m_2}$, $\nu(m_1)=\nu(m_2)=0, \nu(m_1+m_2) \geq a$,
			$\det (\tR_3 \times\cdots\times \tR_u)=1$;
			\item $a=2$, $\alpha_1=\gamma_1=0, v(m_1) \leq 1, \bc_1=\one$
			and $a(\tR_i)>a(\tR_1)$ for $i>1$.
		\end{enumerate}
		\item $a(\tR_1)=a(\tR)<\min\{a,a(\tR_2)\}$ with $\bc_1=\zero$ and $\alpha_1=0$.
	\end{enumerate}
	Finally, if $\tR$ is special, then $C_{\tG}(R)=\tC$.
\end{prop}
\begin{proof}
	Similar as Proposition \ref{prop:special-odd-3},
	using arguments in Lemma \ref{lem:normal-abelian-2}
	and Proposition \ref{prop:special-2-linear-2} with some slight modifications.
\end{proof}

\subsection{Special radical subgroups for $\ell=2$ and $4 \mid q+\eta$.}
\label{subsect:srs-2-unitary}

Now, we classify the special radical subgroups of $\tG = \GL_n(\eta q)$ for $\ell=2$ and $4 \mid q+\eta$.
Note that in this case, for any $2$-subgroup $\tR$ of $\tG$, $\det\tR\leq\set{\pm1}$.
Recall that $\nu$ is the discrete valuation such that $\nu(2)=1$ and $a=\nu(q+\eta)$.

We start with the radical subgroups of the form $\tR_{m,\alpha,\gamma}$
or $\tR_{m,0,\gamma}^\pm$($\gamma>1$ if the type is plus, $\gamma>0$ if the type is minus).
For $\alpha>0$, all notation in \S\ref{subsect:srs-2-linear} will be used without further reference.

\begin{prop}\label{prop:special-2-unitary-1}
	Assume $\ell=2$, $4 \mid q+\eta$, and $\tR=\tR_{m,\alpha,\gamma}$ or
	$\tR=\tR_{m,0,\gamma}^\pm$($\gamma>1$ if the type is plus, $\gamma>0$ if the type is minus)
	is a subgroup of symplectic type of $\tG_{m,\alpha,\gamma}$.
	Then $\tR$ is special if and only if one of the following holds.
	\vspace{-0.5ex}
	\begin{enumerate}[\rm(1)]\setlength{\itemsep}{-0.5ex}
		\item $\tR=\tR_{m,\alpha,\gamma}(\alpha\ge1)$ with $\nu(m)+\gamma>0$ or  $\alpha=1$.
		\item $\tR=\tR_{m,0,\gamma}^\pm=E_{m,\gamma}^\pm$($\gamma>1$ if the type is plus, $\gamma>0$ if the type is minus).
		\item $\tR=\tR_{m,0,\gamma}=\tS_{m,1,\gamma-1}(\gamma\geq1)$ unless $\nu(m)=0,\gamma=1,a=2$.
		\item $\tR=\tR_m$.
	\end{enumerate}
\end{prop}
\begin{proof}
	Set $\tG=\tG_{m,\alpha,\gamma}$, $G=G_{m,\alpha,\gamma}$ and $R= \tR \cap G$.
	
	(1) Assume $\tR=\tR_{m,\alpha,\gamma}$ with $\alpha\geq1$.
	When $\nu(m)+\gamma>0$, $\det\tR=1$, thus $R=\tR$ and $\tR$ is special.
	Assume $\nu(m)=\gamma=0$.
	If $\alpha>1$, then direct calculation shows that
	$\tR_{m,\alpha} \cap G_{m,\alpha}$ is conjugate in $\tG_{m,\alpha}$ to $\tR_{2m,\alpha-1} \cap G_{m,\alpha}$,
	but note that $\tR_{2m,\alpha-1}$ is special.
	Thus we assume $\nu(m)=\gamma=0$ and $\alpha=1$.
	Then $\tR_{m,1}^{tw} = \Grp{I_m\otimes\diag\{\zeta_{2^{a+1}},\zeta_{2^{a+1}}^{\eta q}\}}$,
	and $R^{tw} := \tR_{m,1}^{tw} \cap G_{m,1}^{tw} = \Grp{I_m\otimes\diag\{\zeta_{2^a},\zeta_{2^a}^{-1}\}}$.
	Direct calculation shows that $C_{\tG_{m,1}^{tw}}(\tR_{m,1}^{tw})=C_{\tG_{m,1}^{tw}}(R_{m,1}^{tw})=\hbar(\GL_m(q^2))$
	and $N_{\tG_{m,1}^{tw}}(\tR_{m,1}^{tw})=C_{\tG_{m,1}^{tw}}(\tR_{m,1}^{tw})\grp{v_{m,1}}=N_{\tG_{m,1}^{tw}}(R_{m,1}^{tw})$,
	thus $\tR_{m,1}$ is special when $\nu(m)=0$.
	
	(2) Assume $\tR=E_{m,\gamma}^\pm$($\gamma>1$ if the type is plus, $\gamma>0$ if the type is minus).
	Then it is special since $\det E_{m,\gamma}^\pm = 1$ when $\gamma>1$ if the type is plus and $\gamma>0$ if the type is minus.
	%	Assume the type is plus and $\nu(m)=0$, $\gamma=1$.
	%	Then $R = E_{m,1}^+ \cap G_{m,0,1} = \Grp{I_m\otimes\begin{bmatrix}0&1\\-1&0\end{bmatrix}}$.
	%	By proof of Proposition \ref{prop:tN-2-unitary-E-m,gamma}, $N_{\tG}(E_{m,1}) = Z(\tG)\Grp{E_{m,1},n_{m,1}^{(i)}}$.
	%	But by direct calculation, $I_m \otimes \begin{bmatrix} a&b\\-b&a \end{bmatrix}(a^2+b^2\neq0) \in N_{\tG}(R)$,
	%	so $N_{\tG}(\tR) \neq N_{\tG}(R)$, thus $\tR=E_{m,1}^+$ is not special when $\nu(m)=0$.
	
	(3) Assume $\tR=\tS_{m,1,\gamma-1}$.
	If $\nu(m)>0$ or $\gamma>1$, $\det\tR=1$, then $\tR$ is special.
	Thus assume $\nu(m)=0$ and $\gamma=1$, then $\tR$ is isomorphic to the semi-dihedral group of order $2^{a+2}$.
	Use the notation in \S\ref{subsect:gst-2-unitary}, $\tS_{m,1}^{tw}=\grp{z_{m,1}^{tw},\tau_{m,1}^{tw}}$ with
	$z_{m,1}^{tw}=I_m\otimes\diag\{\zeta_{2^{a+1}},\zeta_{2^{a+1}}^{\eta q}\}$
	and $\tau_{m,1}^{tw}=I_m\otimes\begin{bmatrix}0&1\\-1&0\end{bmatrix}$.
	By Proposition \ref{prop:tN-2-unitary-S-m,gamma},
	$C_{\tG_{m,1}^{tw}}(\tS_{m,1}^{tw}) = \GL_m(\eta q) \otimes I_2$ and
	$N_{\tG_{m,1}^{tw}}(\tS_{m,1}^{tw}) = C_{\tG_{m,1}^{tw}}(\tS_{m,1}^{tw}) \tS_{m,1}^{tw}$.
	Note that $R=\grp{(z_{m,1}^{tw})^2,\tau_{m,1}^{tw}}$ is isomorphic to the generalized quaternion group of order $2^{a+1}$.
	First, assume $a>2$.
	Thus $\Aut(R) = \grp{\sigma} \rtimes \Aut(\grp{z_{m,1}^{tw})^2})$,
	where $\sigma: (z_{m,1}^{tw})^2 \mapsto (z_{m,1}^{tw})^2, \tau_{m,1}^{tw} \mapsto \tau_{m,1}^{tw}(z_{m,1}^{tw})^2$
	and $\Aut(\grp{z_{m,1}^{tw})^2})$ fixes $\tau_{m,1}^{tw}$.
	For any $n \in N_{\tG_{m,1}^{tw}}(R)$, $n(z_{m,1}^{tw})^2n^{-1}$ has the same eigenvalues as $(z_{m,1}^{tw})^2$,
	thus $n(z_{m,1}^{tw})^2n^{-1}$ equals $(z_{m,1}^{tw})^2$ or $(z_{m,1}^{tw})^{-2}$.
	So $N_{\tG_{m,1}^{tw}}(R)$ induces a subgroup of $\grp{\sigma,\tau}\leq\Aut(R)$,
	where $\tau: (z_{m,1}^{tw})^2 \mapsto (z_{m,1}^{tw})^{-2}, \tau_{m,1}^{tw} \mapsto \tau_{m,1}^{tw}$.
	But $S_{m,1}^{tw}$ already induces $\grp{\sigma,\nu}$ and
	direct calculation shows that $C_{\tG_{m,1}^{tw}}(R)=C_{\tG_{m,1}^{tw}}(\tS_{m,1}^{tw})$,
	thus $N_{\tG_{m,1}^{tw}}(R) = C_{\tG_{m,1}^{tw}}(\tS_{m,1}^{tw}) \tS_{m,1}^{tw} = N_{\tG_{m,1}^{tw}}(\tS_{m,1}^{tw})$,
	i.e. $\tS_{m,1}^{tw}$ is special when $a>2$.
	When $a=2$, $R=E_{m,1}^-$ and $E_{m,1}^-$ is special.
	
	(4) This case follows from the definition of $\tR_m$.
\end{proof}

\begin{prop}\label{prop:cc-R-m,alpha,gamma-2-unitary}
	Assume $\ell=2$, $4 \mid q+\eta$, and $\tR=\tR_{m,\alpha,\gamma}$ or
	$\tR=\tR_{m,0,\gamma}^\pm$($\gamma>1$ if the type is plus, $\gamma>0$ if the type is minus)
	is a subgroup of symplectic type of $\tG_{m,\alpha,\gamma}$.
	Denote the centralizer and normalizer of $\tR_{m,\alpha,\gamma}$ (or $\tR_{m,0,\gamma}^\pm$)
	as $\tC_{m,\alpha,\gamma}$  (or $\tC_{m,0,\gamma}^\pm$) and $\tN_{m,\alpha,\gamma}$  (or $\tN_{m,0,\gamma}^\pm$).
	\vspace{-1ex}
	\begin{enumerate}[\rm(1)]\setlength{\itemsep}{-0.5ex}
		\item Assume $\tR=\tR_{m,\alpha,\gamma}$($\alpha>0$ or $\alpha=\gamma=0$).
		Then $\det\tR_{m,\alpha,\gamma}=\fZ_2^{2^{\nu(m)+\gamma}}$,
		$\det\tC_{m,\alpha,\gamma}=\det\tN_{m,\alpha,\gamma}=\fZ_{q-\eta}^{2^\gamma}$.
		If $\tR$ is special, the number of $G_{m,\alpha,\gamma}$-conjugacy classes of radical subgroups contained in the set
		$\{ G_{m,\alpha,\gamma}\cap{^{\tg}\tR_{m,\alpha,\gamma}} \mid \tg\in\tG_{m,\alpha,\gamma} \}$ is $2^{\min\{1,\gamma\}}$.
		\item Assume $\tR=\tR_{m,0,\gamma}=\tS_{m,1,\gamma-1}(\gamma\geq1)$.
		Then $\det\tR_{m,0,\gamma}=1$ unless $\nu(m)=0$ and $\gamma=1$, in which case, $\det\tR_{m,0,\gamma}=\fZ_2$;
		$\det\tC_{m,0,\gamma}=\fZ_{q-\eta}^{2^\gamma}=\mrO_{2'}(\fZ_{q-\eta})$,
		$\det\tN_{m,0,\gamma}= \Grp{ \mrO_{2'}(\fZ_{q-\eta}),\det\tR_{m,0,\gamma} }$.
		If $\tR$ is special, the number of $G_{m,0,\gamma}$-conjugacy classes of radical subgroups contained in the set
		$\{ G_{m,0,\gamma}\cap{^{\tg}\tR_{m,0,\gamma}} \mid \tg\in\tG_{m,0,\gamma} \}$ is $2$,
		unless $\nu(m)=0$ and $\gamma=1$, in which case, the number is $1$.	
		\item Assume $\tR=\tR_{m,0,\gamma}^\pm=E_{m,\gamma}^\pm$~($\gamma>1$ if the type is plus, $\gamma>0$ if the type is minus).
		Then $\det\tR_{m,0,\gamma}^\pm=1$, $\det\tC_{m,0,\gamma}^\pm=\fZ_{q-\eta}^{2^\gamma}=\mrO_{2'}(\fZ_{q-\eta})$,
		$\det\tN_{m,0,\gamma}^\pm=\det\tC_{m,0,\gamma}^\pm=\mrO_{2'}(\fZ_{q-\eta})$
		unless (\ref{special-case-2-uni-1}), in which case, $\det\tN_{m,0,\gamma}^\pm=\fZ_{q-\eta}$.
		If $\tR$ is special, the number of $G_{m,0,\gamma}$-conjugacy classes of radical subgroups contained in the set
		$\{ G_{m,0,\gamma}\cap{^{\tg}\tR_{m,0,\gamma}^\pm} \mid \tg\in\tG_{m,0,\gamma} \}$ is $2$,
		unless (\ref{special-case-2-uni-1}), in which case, the number is $1$.
	\end{enumerate}
\end{prop}
\begin{proof}
	This follows from calculations case by case using the results in \S\ref{subsect:gst-2-unitary}.
\end{proof}

\begin{prop}\label{prop:CN_m,alpha,gamma-2-unitary}
	Assume $\ell=2$, $4 \mid q+\eta$,
	and $\tR=\tR_{m,\alpha,\gamma}$ or $\tR_{m,0,\gamma}^\pm$ is a special subgroup of symplectic type of $\tG_{m,\alpha,\gamma}$.
	Set $R=R_{m,\alpha,\gamma} = \tR_{m,\alpha,\gamma} \cap G_{m,\alpha,\gamma}$
	or $R_{m,0,\gamma}^\pm = \tR_{m,0,\gamma}^\pm \cap G_{m,0,\gamma}$.
	Denote by $C_{m,\alpha,\gamma}$, $N_{m,\alpha,\gamma}$ (or $C_{m,0,\gamma}^\pm$, $N_{m,0,\gamma}^\pm$)
	the centralizer and normalizer of $R_{m,\alpha,\gamma}$ (or $\tR_{m,0,\gamma}^\pm$) in $G_{m,\alpha,\gamma}$.
	If the results are stated using twisted version, the superscript ``$tw$'' will be used.
	\vspace{-1ex}
	\begin{enumerate}[\rm(1)]\setlength{\itemsep}{-0.5ex}
		\item Assume $\tR=\tR_{m,\alpha,\gamma}(\alpha>0)$ with $\nu(m)+\gamma>0$ or $\nu(m)=\gamma=0$ and $\alpha=1$.
		Let $\tM_{m,\alpha,\gamma}^0,\hC_{m,\alpha,\gamma}^0, \hN_{m,\alpha,\gamma}^0=\hC_{m,\alpha,\gamma}^0\tM_{m,\alpha,\gamma}^0,
		\tM_{m,\alpha,\gamma}^{tw}$ be defined as in \S\ref{subsect:srs-odd}.
		Then $C_{m,\alpha,\gamma}^{tw}=\hbar(\hC_{m,\alpha,\gamma}^0)$,
		$\hbar(\tN_{m,\alpha,\gamma}^0) \cap G_{m,\alpha,\gamma}^{tw}
		= \hbar(\hN_{m,\alpha,\gamma}^0) = C_{m,\alpha,\gamma}^{tw}\tM_{m,\alpha,\gamma}^{tw}$,
		and $\hbar(\hN_{m,\alpha,\gamma}^0)/R_{m,\alpha,\gamma}^{tw}
		\cong C_{m,\alpha,\gamma}^{tw}/Z(R_{m,\alpha,\gamma}^{tw}) \times \Sp_{2\gamma}(\ell)$.
		Finally, $N_{m,\alpha,\gamma}^{tw}=\hbar(\hN_{m,\alpha,\gamma}^0)V_{m,\alpha,\gamma}$
		with $V_{m,\alpha,\gamma} = \grp{v_{m,\alpha,\gamma}}$.
		\item Assume $\tR=\tR_{m,0,\gamma}^\pm=E_{m,\gamma}^\pm$($\gamma>1$ if the type is plus, $\gamma>0$ if the type is minus).
		Use the notation in \S\ref{subsect:gst-2-unitary}.
		Then $R_{m,0,\gamma}^\pm=\tR_{m,0,\gamma}^\pm$ and
		$C_{m,0,\gamma}^\pm = \set{A \otimes I_{2^\gamma} \mid A \in \GL_m(\eta q), \det(A)=\pm1}$.
		When we are not in the case (\ref{special-case-2-uni-1}),
		$N_{m,0,\gamma}^\pm = C_{m,0,\gamma}^\pm \tM_{m,0,\gamma}^\pm$
		and $N_{m,0,\gamma}^\pm/E_{m,\gamma}^\pm
		\cong C_{m,0,\gamma}^\pm/Z(E_{m,\gamma}^\pm) \times \GO_{2\gamma}^\pm(2)$.
		When the type is plus and $\gamma=2,\nu(m)=0$,
		$N_{m,0,2}^+=C_{m,0,2}^+\tM_{m,0,2}^{+,0}$ with $\tM_{m,0,2}^{+,0} \leq \tM_{m,0,2}^+$
		such that $\tM_{m,0,2}^{+,0}/E_{m,2}^+ \cong \Omega_4^+(2)$ and
		$N_{m,0,2}^+/E_{m,2}^+ \cong C_{m,0,2}^+/Z(E_{m,2}^+) \times \Omega_4^+(2)$.
		When the type is minus and $\gamma=1,\nu(m)=0,a=2$,
		$N_{m,0,1}^-=C_{m,0,1}^+\tM_{m,0,1}^{-,0}$ with $\tM_{m,0,1}^{-,0} \leq \tM_{m,0,1}^-$
		such that $\tM_{m,0,1}^{-,0}/E_{m,1}^- \cong C_3$ and
		$N_{m,0,1}^-/E_{m,1}^- \cong C_{m,0,1}^-/Z(E_{m,1}^-) \times C_3$.
		\item Assume $\tR=\tR_{m,0,\gamma}=\tS_{m,1,\gamma-1}(\gamma\geq1)$ except the case $\nu(m)=0,\gamma=1,a=2$.
		Use the notation in \S\ref{subsect:gst-2-unitary}.
		Then $C_{m,0,\gamma} = \set{A \otimes I_{2^\gamma} \mid A \in \GL_m(\eta q), \det(A)=\pm1}$.
		When $\nu(m)\neq0$ or $\gamma>1$, $R_{m,0,\gamma}=\tR_{m,0,\gamma}$;
		while when $\nu(m)=0$, $\gamma=1$ and $a>2$,
		$R_{m,0,1} \cong Q_{a+1}$ the generalized quaternion group of order $2^{a+1}$.
		In both case, $N_{m,0,\gamma}/R_{m,0,\gamma} \cong C_{m,0,\gamma}/Z(R_{m,0,\gamma}) \times \Sp_{2(\gamma-1)}(2)$.
		\item If $\tR=\tR_m$, $\det\tR_m=\fZ_2^{\nu(m)}$, $C_m=N_m=G_m$.	
	\end{enumerate}
\end{prop}

\begin{proof}
	(1) follows from Proposition \ref{prop:tN-m,alpha,gamma-2-unitary},
	noting that $\det(\tM_{m,\alpha,\gamma}^{tw})=1$ and $\tM_{m,\alpha,\gamma}^{tw}/E_{m,\alpha,\gamma}^{tw} \cong \Sp(2\gamma,\ell)$.
	(2) follows from Proposition \ref{prop:tN-2-unitary-E-m,gamma},
	noting that $\det\tC_{m,0,\gamma}^\pm=\mrO_{2'}(\fZ_{q-\eta})$,
	$\det \tM_{m,\gamma}^\pm \leq \mrO_2(\fZ_{q-\eta})$,
	and $\tM_{m,\gamma}^\pm/E_{m,\gamma}^\pm \cong \GO_{2\gamma}^\pm(2)$;
	while for the case (\ref{special-case-2-uni-1}),
	it can be calculated from the explicit construction in the proof of Proposition \ref{prop:tN-2-unitary-E-m,gamma}.
	(3) follows from Proposition \ref{prop:tN-2-unitary-S-m,gamma}
	and (4) follows by direct calculation.
\end{proof}

%We write $M_{m,0,\gamma}^\pm:=\tM_{m,0,\gamma}^\pm\cap G_{m,0,\gamma}^\pm$.

\begin{rem}\label{4.39}
	By the above proposition,
	the radical subgroup $R_{m,0,\gamma}^\pm=E_{m,\gamma}^\pm$($\gamma>1$ if the type is plus, $\gamma>0$ if the type is minus)
	will not provide any weight except in the following two cases (see also the Remark in \cite[p.511]{An92}).
	\vspace{-1ex}
	\begin{enumerate}[(i)]\setlength{\itemsep}{-0.5ex}
		\item The type is plus and $\nu(m)=0,\gamma=2$.
		Note that $N_{m,0,2}^+/E_{m,2}^+ \cong C_{m,0,2}^+/Z(E_{m,2}^+) \times \Omega_4^+(2)$
		and $\Omega_4^+(2)$ has a unique character of defect zero.
		\item The type is minus and $\gamma=1$.
		This case splits into two cases.
		When $\nu(m)\neq0$ or $a>2$, note that
		$N_{m,0,1}^-/E_{m,1}^- \cong C_{m,0,1}^-/Z(E_{m,1}^-) \times \GO_2^-(2)$
		and $\GO_2^-(2)\cong\fS_3$ has a unique character of defect zero.
		When $\nu(m)=0$ and $a=2$, note that
		$N_{m,0,1}^-/E_{m,1}^- \cong C_{m,0,1}^-/Z(E_{m,1}^-) \times C_3$
		and $C_3$ provides $3$ characters of defect zero.
	\end{enumerate}
\end{rem}

Now, assume $\bc\neq\zero$ and consider the special radical subgroup of the form
$\tR_{m,\alpha,\gamma,\bc} := \tR_{m,\alpha,\gamma} \wr A_{\bc}$(except $\alpha=\gamma=0, c_1=1$)
or $\tR_{m,0,\gamma,\bc}^\pm := \tR_{m,0,\gamma}^\pm \wr A_{\bc}$($\gamma>1$ if the type is plus, $\gamma>0$ if the type is minus).
Let $R_{m,\alpha,\gamma,\bc} := \tR_{m,\alpha,\gamma,\bc} \cap G_{m,\alpha,\gamma,\bc}$
and $R_{m,0,\gamma,\bc}^\pm := \tR_{m,0,\gamma,\bc}^\pm \cap G_{m,0,\gamma,\bc}$.
Denote by $\tB_{m,\alpha,\gamma,\bc}$ ($\tB_{m,0,\gamma,\bc}^\pm$)
the base subgroup of $\tR_{m,\alpha,\gamma,\bc}$ ($\tR_{m,0,\gamma,\bc}^\pm$).
Set $B_{m,\alpha,\gamma,\bc} = \tB_{m,\alpha,\gamma,\bc} \cap G_{m,\alpha,\gamma,\bc}$
and $B_{m,0,\gamma,\bc}^\pm = \tB_{m,0,\gamma,\bc}^\pm \cap G_{m,0,\gamma,\bc}$.
Since $\tR_{m,0,0,\bc}$ with $c_1=1$ is not a basic subgroup,
we have $\det A_{\bc} = 1$ for all other cases,
thus $R_{m,\alpha,\gamma,\bc}=B_{m,\alpha,\gamma,\bc}A_{\bc}$
and $R_{m,0,\gamma,\bc}^\pm=B_{m,0,\gamma,\bc}^\pm A_{\bc}$.
We denote $\two=(2)$.

\begin{prop}\label{prop:special-2-unitary-2}
	Assume $\ell=2$, $4 \mid q+\eta$ and $\bc\neq\zero$.
	Let $\tR=\tR_{m,\alpha,\gamma,\bc}$(except $\alpha=\gamma=0, c_1=1$) or
	$\tR=\tR_{m,0,\gamma,\bc}^\pm$($\gamma>1$ if the type is plus, $\gamma>0$ if the type is minus)
	be a basic subgroup of $\tG_{m,\alpha,\gamma,\bc}$.
	Then $\tR$ is special except for $\tR=\tR_{m,0,0,\two}$.
\end{prop}
\begin{proof}
	For $\tR=\tR_{m,\alpha,\gamma,\bc}$ with $\alpha>0$, this assertion can be proved similarly as Proposition \ref{prop:special-2-linear-2}.
	
	Assume $\tR=\tR_{m,0,\gamma,\bc}=\tS_{m,1,\gamma-1,\bc}$ with $\gamma\geq1$.
	If $\nu(m)>0$ or $\gamma>1$, then $\det\tR=1$, thus $\tR$ is obviously special.
	Now, assume $\nu(m)=0$ and $\gamma=1$.
	Let $z_{m,1}^{tw}$ and $\iota$ be as in \S\ref{subsect:gst-2-unitary},
	and set $z_{m,1}=\iota(z_{m,1}^{tw})$, $\tZ_{m,1}=\grp{z_{m,1}}$.
	The $\tZ_{m,1}$ is cyclic of order $2^{a+1}$.
	Let $\tilde{A}=(\tZ_{m,1})^{2^{c_1+\cdots+c_r}}$, where $\bc=(c_1,\ldots,c_r)$,
	and $A = \tilde{A} \cap G_{m,0,1,\bc}$.
	Since $\tS_{m,1}$ has a non-abelian subgroup with determinant $1$,
	arguments in Lemma \ref{lem:normal-abelian-2} show that
	any normal abelian subgroup of $R_{m,0,1,\bc}$ is contained in the base subgroup $B_{m,0,1,\bc}$.
	Then $A$ is the maximal normal abelian subgroup of $R_{m,0,1,\bc}$,
	thus $A$ is characteristic in $R_{m,0,1,\bc}$.
	So $N_{\tG_{m,0,1,\bc}}(R_{m,0,1,\bc})$ normalizes $A$.
	It follows by straight forward calculation that $N_{\tG_{m,0,1,\bc}}(A)=N_{\tG_{m,0,1,\bc}}(\tilde{A})$.
	Since $\tR_{m,0,1,\bc}=R_{m,0,1,\bc}\tilde{A}$, $N_{\tG_{m,0,1,\bc}}(R_{m,0,1,\bc}) \leq N_{\tG_{m,0,1,\bc}}(\tR_{m,0,1,\bc})$.
	The reverse inclusion follows by construction and thus the equality holds, which means that $\tR_{m,0,1,\bc}$ is special.
	
	Assume $\tR=\tR_{m,0,0,\bc}$ with $c_1\neq1$.
	When $\nu(m)\neq0$, $\det\tR=1$, hence $\tR=R_{m,0,0,\bc}$ is special.
	Assume then $\nu(m)=0$ and $\bc=(c_1,\ldots,c_r)$.
	Set $\tR_1=\tR_{m,0,0,(c_1)}$, $\bc'=(c_2,\ldots,c_r)$ and $n'=2^{c_2+\cdots+c_r}$.
	Then $\tR = \tR_1\wr A_{\bc'}$.
	Let $\tB_1$ be the base subgroup of the wreath product $\tR_1\wr A_{\bc'}$
	and set $B_1 = \tB_1 \cap G_{m,0,0,\bc}$, then $R_{m,0,0,\bc} = B_1 \rtimes A_{\bc'}$.
	Thus arguments in Lemma \ref{lem:normal-abelian-2}
	show that any normal abelian subgroup of $R_{m,0,0,\bc}$ is contained in $B_1$.
	Since $\tR_1$ is generated by its normal abelian subgroups by \cite[(2A)(b)]{An92},
	and if $W$ is a normal abelian subgroup of $\tR_1$,
	then $W^{n'} \cap B_1$ is a normal subgroup of $R_{m,0,0,\bc}$.
	Thus $B_1$ is generated by all normal abelian subgroups of $R_{m,0,0,\bc}$ and is a characteristic subgroup of $R_{m,0,0,\bc}$.
	Note that the base subgroup $B_{m,0,0,\bc}$ can also be viewed as the base subgroup of $B_1$.
	Assume $\bc\neq\two$.
	Arguments in \cite[(2A)(b)]{An92} with some obvious modifications show that $B_{m,0,0,\bc}=C_{B_1}([B_1,B_1])$.
	Then $B_{m,0,0,\bc}$ is a characteristic subgroup of $R_{m,0,0,\bc}$.
	So $N_{\tG_{m,0,0,\bc}}(R_{m,0,0,\bc})$ normalizes $B_{m,0,0,\bc}$.
	Easy considerations show that the normalizers of $B_{m,0,0,\bc}$ and $\tB_{m,0,0,\bc}$ in $\tG_{m,0,0,\bc}$ coincide.
	Since $\tR_{m,0,0,\bc}=R_{m,0,0,\bc}\tB_{m,0,0,\bc}$, $\tR_{m,0,0,\bc}$ is special.
	When $\bc=\two$, note that $R_{m,0,0,\two} = \tR_{m,0,0,\two} \cap G_{m,0,0,\two} = E_{m,2}^+$
	and $E_{m,2}^+$ is special.
	
	Finally, if $\tR=\tR_{m,0,\gamma,\bc}^\pm$($\gamma>1$ if the type is plus, $\gamma>0$ if the type is minus),
	then $\det\tR=1$, and $\tR$ is obviously special.
\end{proof}

We use the notation $\tC_{m,\alpha,\gamma,\bc}$, $\tC_{m,0,\gamma,\bc}^\pm$,
$\tN_{m,\alpha,\gamma,\bc}$, $\tN_{m,0,\gamma.\bc}^\pm$,
$C_{m,\alpha,\gamma,\bc}$, $C_{m,0,\gamma,\bc}^\pm$,
$N_{m,\alpha,\gamma,\bc}$, $N_{m,0,\gamma.\bc}^\pm$ in the obvious sense.

\begin{prop}\label{prop:cc-R-m,alpha,gamma,c-2-unitary}
	Assume $\ell=2$, $4 \mid q+\eta$ and $\bc\neq\zero$.
	Let $\tR=\tR_{m,\alpha,\gamma,\bc}$(except $\alpha=\gamma=0, c_1=1$) or
	$\tR=\tR_{m,0,\gamma,\bc}^\pm$($\gamma>1$ if the type is plus, $\gamma>0$ if the type is minus)
	be a basic subgroup of $\tG_{m,\alpha,\gamma,\bc}$.
	\vspace{-1ex}
	\begin{enumerate}[\rm(1)]\setlength{\itemsep}{-0.5ex}
		\item Assume $\tR=\tR_{m,\alpha,\gamma,\bc}$($\alpha>0$ or $\alpha=\gamma=0$ and $c_1\neq1$).
		Then $\det\tR_{m,\alpha,\gamma,\bc}=\det\tR_{m,\alpha,\gamma}=\fZ_2^{2^{\nu(m)+\gamma}}$,
		$\det\tC_{m,\alpha,\gamma,\bc}=(\det\tC_{m,\alpha,\gamma})^{2^{|\bc|}}=\fZ_{q-\eta}^{2^{\gamma+|\bc|}}=\mrO_{2'}(\fZ_{q-\eta})$,
		$\det\tN_{m,\alpha,\gamma,\bc}=\Grp{\mrO_{2'}(\fZ_{q-\eta}),\fZ_2^{2^{\nu(m)+\gamma}}}$.
		If $\tR$ is special, the number of $G_{m,\alpha,\gamma,\bc}$-conjugacy classes of radical subgroups contained in the set
		$\{ G_{m,\alpha,\gamma,\bc}\cap{^{\tg}\tR_{m,\alpha,\gamma,\bc}} \mid \tg\in\tG_{m,\alpha,\gamma,\bc} \}$
		is $2^{\min\{1,\nu(m)+\gamma\}}$.
		\item Assume $\tR=\tR_{m,0,\gamma,\bc}=\tS_{m,1,\gamma-1,\bc}(\gamma\geq1)$.
		Then $\det\tR_{m,0,\gamma,\bc}=\det\tR_{m,0,\gamma}=1$
		unless $\nu(m)=0,\gamma=1$, in which case, $\det\tR_{m,0,\gamma,\bc}=\fZ_2$;
		$\det\tC_{m,\alpha,\gamma,\bc}=(\det\tC_{m,\alpha,\gamma})^{2^{|\bc|}}=\fZ_{q-\eta}^{2^{\gamma+|\bc|}}=\mrO_{2'}(\fZ_{q-\eta})$,
		$\det\tN_{m,0,\gamma,\bc}=\Grp{\mrO_{2'}(\fZ_{q-\eta}),\det\tR_{m,0,\gamma,\bc}}$.
		If $\tR$ is special, the number of $G_{m,0,\gamma,\bc}$-conjugacy classes of radical subgroups contained in the set
		$\{ G_{m,0,\gamma,\bc}\cap{^{\tg}\tR_{m,0,\gamma,\bc}} \mid \tg\in\tG_{m,0,\gamma,\bc} \}$ is $2$,
		unless $\nu(m)=0,\gamma=1$, in which case, the number is $1$.	
		\item Assume $\tR=\tR_{m,0,\gamma,\bc}^\pm=E_{m,\gamma,\bc}^\pm$
		($\gamma>1$ if the type is plus, $\gamma>0$ if the type is minus).
		Then $\det\tR_{m,0,\gamma,\bc}^\pm=\det\tR_{m,0,\gamma}^\pm=1$,
		$\det\tC_{m,0,\gamma,\bc}^\pm=(\det\tC_{m,0,\gamma}^\pm)^{2^{|\bc|}}=\fZ_{q-\eta}^{2^{\gamma+|\bc|}}=\mrO_{2'}(\fZ_{q-\eta})$,
		$\det\tN_{m,0,\gamma,\bc}^\pm=\mrO_{2'}(\fZ_{q-\eta})$.
		If $\tR$ is special, the number of $G_{m,0,\gamma}$-conjugacy classes of radical subgroups contained in the set
		$\{ G_{m,0,\gamma}\cap{^{\tg}\tR_{m,0,\gamma}^\pm} \mid \tg\in\tG_{m,0,\gamma} \}$ is $2$.
	\end{enumerate}
\end{prop}

\begin{proof}
	This can be shown by the same arguments as in Proposition \ref{prop:cc-R-m,alpha,gamma,c-odd} case by case,
	noting that for (3), a slight modification is needed.
\end{proof}

\begin{rem}
	Assume $\ell=2$, $4 \mid q+\eta$ and $\bc\neq\zero$,
	and $\tR=\tR_{m,0,\gamma,\bc}^\pm=E_{m,\gamma,\bc}^\pm$($\gamma>1$ if the type is plus, $\gamma>0$ if the type is minus).
	Then $R=R_{m,0,\gamma,\bc}^\pm=\tR_{m,0,\gamma,\bc}^\pm$.
	Recall that
	$$\tN_{m,0,\gamma,\bc}^\pm = (\tN_{m,0,\gamma}^\pm/\tR_{m,0,\gamma}^\pm) \otimes N_{\fS(2^{|\bc|})}(A_\bc),$$
	where $\otimes$ is defined as \cite[(1.3) (1.5)]{AF90}.
	By Proposition \ref{prop:tN-2-unitary-E-m,gamma},
	$\tN_{m,0,\gamma}^\pm = \tC_{m,0,\gamma}^\pm \tM_{m,\gamma}^\pm$ with $\det \tM_{m,\gamma}^\pm \leq \fZ_2$.
	By the meaning of $\otimes$ in the above equation,
	the contribution of $\tM_{m,\gamma}^\pm$ in $\tN_{m,0,\gamma,\bc}^\pm$ will have determinant $1$;
	see the proof of Proposition \ref{prop:cc-R-m,alpha,gamma,c-odd}.
	Thus
	\[ N_{m,0,\gamma,\bc}^\pm/C_{m,0,\gamma,\bc}^\pm R_{m,0,\gamma,\bc}^\pm \cong
	\GO_{2\gamma}^\pm(2) \times N_{\fS(2^{|\bc|})}(A_\bc)/A_\bc. \]
	So $R_{m,0,\gamma,\bc}^\pm$ can not provide any weight except for $\tR_{m,0,1,\bc}^-$; see the Remark on  \cite[p.511]{An92}.
\end{rem}

Finally, we consider the special radical subgroups of the form $\tR=\tR_1\times\cdots\times\tR_u$
with $u>1$ and $\tR_i$ being a basic subgroup.
By Proposition \ref{prop:cc-R-m,alpha,gamma-2-unitary} and  \ref{prop:cc-R-m,alpha,gamma,c-2-unitary}, and similar as Lemma \ref{det-radical-N-C-2-linear}, we have

\begin{lem}\label{det-radical-N-C-2-unitary}
	Assume $\ell=2$ and $4 \mid q+\eta$.
	Let $\tR$ be a radical subgroup of $\tG$ which can be expressed as 
	$\tR=\tR_1\times\cdots\times \tR_u$, a direct product of basic subgroups.
	Denote $\tC=C_{\tG}(\tR)$, $\tN=N_{\tG}(\tR)$.
	\vspace{-0.5ex}
	\begin{enumerate}[\rm(1)]\setlength{\itemsep}{-0.5ex}
		\item $\mrO_{2'}(\det(\tN))=\mrO_{2'}(\det(\tR\tC))=\mrO_{2'}(\fZ_{q-\eta})$.
		In particular, $|\tG:GN_{\tG}(\tR)|$ is a $2$-power.
		\item $\det(\tN)\neq\det(\tR\tC)$ if and only if $\det(\tR\tC)=\mrO_{2'}(\fZ_{q-\eta})$ and $\det(\tN)=\fZ_{q-\eta}$,
		which happens exactly when the following hold
		\vspace{-0.5ex}		
		\begin{equation} \label{special-case-wei-ell2-uni}
		\begin{aligned}
		\textrm{\rm(i)}\ \ & \textrm{there is a component}\ \tR_{m,0,2}^+ \ \textrm{with}\  \nu(m)=0 \ \textrm{or a component} \\ & \tR_{m,0,1}^- \ \textrm{with}\ \nu(m)=0 \ \textrm{when}\ a=2;\\
		\textrm{\rm(ii)}\ \ & \textrm{there is no component of the form}\ \tR_{m,\alpha,0,\bc} ~(\alpha>0\ \textrm{or}\ \alpha=0 \\ & \textrm{and} \ c_1\neq1)\ \textrm{with}\ \nu(m)=0;\\
		\textrm{\rm(iii)}\ \ & \textrm{there is no component of the form}\ \tR_{m,0,1,\bc}=\tS_{m,1,0,\bc} \ \textrm{with}\\ & \nu(m)=0.
		\end{aligned}
		\addtocounter{thm}{1}\tag{\thethm}	
		\end{equation}
	\end{enumerate}
\end{lem}

\begin{prop}\label{prop:special-2-unitary-3}
	Assume $\ell=2$ and $4 \mid q+\eta$.
	Let $\tR$ be a radical subgroup of $\tG$ which can be expressed as 
	$\tR=\tR_1\times\cdots\times \tR_u$ with $u>1$, $\tR_i$ being a basic subgroup and $\det\tR_1 \geq \cdots \geq \det\tR_u$.
	$\tR$ is special if and only if one of the following holds:
	\vspace{-0.5ex}
	\begin{enumerate}[\rm(1)]\setlength{\itemsep}{0ex}
		\item $\det\tR=1$, which happens exactly when
		there is no component of the form $\tR_{m,\alpha,0,\bc}$($\alpha>0$ or $\alpha=0$ and $c_1\neq1$) with $\nu(m)=0$
		and there is no component of the form $\tR_{m,0,1,\bc}=\tS_{m,1,0,\bc}$ with $\nu(m)=0$.
		\item $\sum\limits_{i:\det\tR_i\neq1} 2^{|\bc_i|} \geqslant2$ except when
		\vspace{-0.5ex}
		\begin{enumerate}[\rm(2a)]\setlength{\itemsep}{0.5ex}
			\item $\tR_1 = \tR_{m_1}, \tR_2 = \tR_{m_2}$, $v(m_1)=v(m_2)=0$, $\det (\tR_3 \times\cdots\times \tR_u)=1$;
			\item $\tR_1=\tR_{m_1,0,0,\bc_1}$ with $v(m_1)=0$, $\bc_1=\one$ or $\two$ and $\det\tR_i=1$ for $i>1$.
		\end{enumerate}
		\item $\tR_1=\tR_{m_1,1}$ or $\tR_{m_1}$ with $v(m_1)=0$ and $\det\tR_i=1$ for $i>1$.
		\item $\tR_1=\tR_{m_1,0,1}=\tS_{m,1}$ with $v(m_1)=0$, $a>2$ and $\det\tR_i=1$ for $i>1$.
	\end{enumerate}
	\vspace{-0.5ex}
	Finally, if $\tR$ is special, then $C_{\tG}(R)=\tC$.
\end{prop}
\begin{proof}
	Similar as Proposition \ref{prop:special-odd-3},
	using arguments from Proposition \ref{prop:special-2-unitary-1} and Proposition \ref{prop:special-2-unitary-2} with some slight modifications.
\end{proof}

%%%%%%%%%%%%%%%%%%%%%%%%%%%%%%%%%%%%%%%%%%%%%%%%%%%%%%%%%%%%
\section{Weights for special linear and unitary groups}
\label{subsect:wos}

The weights of $G=\SL_n(\eta q)$ can be obtained  in two ways.
The first one is to consider every  special radical subgroup $\tR$ of $\tG$ with respect to $G$
and determine which of the irreducible constituents of the restrictions of irreducible characters of $N_{\tG}(\tR)/\tR$ to $N_G(\tR)/(\tR\cap G)$
are of defect zero.
Thus using Lemma \ref{lem:motivation-special-radical} we can obtain all weights of $G$.
The second way is to find all the weights of $G$ covered by a given weight of $\tG$.
In this section, we will classify the weights for $G$ via both ways;
we start from the first one and then give a classification for the weights of $G$ in \S  \ref{sec-const-wei}, while we give the results for the second way in \S \ref{Para-of-wei}.

\vspace{2ex}

We will make use of the following lemmas.

\begin{lem}\label{def-clifford}
Let $Y\unlhd X$ be arbitrary finite groups, $\theta\in\Irr(Y)$ and $\chi\in\Irr(X\mid\theta)$.
If $\theta$ is of defect $d$ and $\chi$ is of defect $d'$, then  $d\le d' \le d+\nu(X/Y) $.
\end{lem}

\begin{rem}\label{rem:5.2}
Let $(R,\varphi)$ be a weight of $G$, and $\tR$ a special radical subgroup of $\tG$ such that $R=\tR\cap G$.
We view $\varphi$ as a character of $N_G(R)/R$ and let $\tvarphi\in \Irr(N_{\tG}(\tR)/\tR\mid \varphi)$.
Let $d=\nu(N_{\tG}(\tR)/\tR N_G(R))$ and $\tilde\theta\in \Irr(\tR C_{\tG}(\tR)\mid \tilde\varphi)$.
Then by Lemma \ref{def-clifford}, $\tilde\varphi$ is of defect $\le d$ as a character of $N_{\tG}(\tR)/\tR$ and
then $\ttheta$ is of defect $\le d$ when viewed as a character of $\tR C_{\tG}(\tR)/\tR$.
Moreover, if $\ttheta$ is of defect $d$ and
 $\tpsi\in\Irr(N_{\tG}(\tR)_{\ttheta}\mid \tilde \theta)$ such that $\tilde\varphi=\Ind_{N_{\tG}(\tR)_{\tilde\theta}}^{N_{\tG}(\tR)}(\tilde\psi)$,
then by Clifford theory,  $|N_{\tG}(\tR)_{\ttheta}/\tR C_{\tG}(\tR)|_\ell=\tilde\psi(1)_\ell/\tilde\theta(1)_\ell$.
\end{rem}

\subsection{Defect of characters of general linear and unitary groups.}\label{basic-result}

We first  give some elementary lemmas,  which will be needed in the sequel.

\begin{lem}\label{valuation}	
Recall that $\ell$ is a prime not dividing $q$,  a power of some prime.
Let $e$, $a$ be defined as before.
\begin{enumerate}[\rm(1)]
	\item If $\ell$ is odd, then $\nu((\eta q)^k-1)>0$ if and only if $e\mid k$, and if so, $\nu((\eta q)^k-1)=a+\nu(k)$.
	\item Let $\ell=2$.  Then the following hold.
	\begin{itemize}
		\item If $4\mid q-\eta$, then $\nu((\eta q)^k-1)=a+\nu(k)$ and $\nu((\eta q)^k+1)=1$.
		\item Let $4\mid q+\eta$. If $k$ is even, then $\nu((\eta q)^k-1)=a+\nu(k)$ and $\nu((\eta q)^k+1)=1$. If $k$ is odd, then $\nu((\eta q)^k+1)=a+\nu(k)$ and $\nu((\eta q)^k-1)=1$.
	\end{itemize}
\end{enumerate}
\end{lem}

\begin{proof}
	This lemma is elementary.
For example, (1)  is just \cite[(3A)]{FS82}, while (2) is
\cite[Lemma 2.1]{FLLMZ19a}. 
\end{proof}

Recall that the definition of ``$a$" for $\ell=2$ is different from the cases that $\ell$ is odd;
one has $a=\nu(q-\eta)$ if $\ell$ is odd or $\ell=2$ and $4\mid q-\eta$, while when $\ell=2$ and $4\mid q+\eta$,
one has $\nu(q-\eta)=1$ and $a=\nu(q+\eta)$.

\begin{lem}\label{degree}
Let $\Gamma$, $\Delta \in \cF$. Assume that $\Gamma$ has a root of order $k$  and $\Delta$ has a root of order $k\ell^\alpha$ with $\ell\nmid k$ and $\alpha>0$.
Then $d_\Delta=e_\Gamma d_\Gamma \ell^{\alpha'}$, where 
\begin{itemize}
	\item $\alpha'=\max\{ 0, \alpha-a-\alpha_\Gamma \}$ if ``$\ell$ is odd" or ``$\ell=2$, $4\mid q-\eta$ or $\alpha_\Gamma>0$", 
	while
	\item for $\ell=2$, $4\mid q+\eta$ and $\alpha_\Gamma=0$,
	one has $\alpha'=0$ if $\alpha= 1$, $\alpha'=1$ if $1< \alpha\le a$, and $\alpha'=\alpha-a$ if $\alpha>a$. 
\end{itemize}
\end{lem}

\begin{proof}
The case $\ell=2$ is just a slight generalization of \cite[Lemma 2.2]{FLLMZ19a} and the proof is similar.
The case that $\ell$ is odd is entirely analogous to the case $\ell=2$ and $4\mid q-\eta$.	
\end{proof}

\begin{lem}\label{def-b}
	Let $\tG=\GL_{n}(\eta q)$ and $\tR=\mrO_\ell(Z(\tG))$. 
	Suppose that $s$ is a semisimple $\ell'$-element of $\tG$ such that $s$ has at least two elementary divisors.
	Let $\tilde\chi\in\mathcal E_\ell(\tG,s)$. 
	Then $\nu(\tG/\tR)-\nu(\tilde\chi(1))\ge \nu(q-\eta)$.
\end{lem}

\begin{proof}
	We first compute the defect of a character of $\tG$.
	Let $t$ be the semisimple $\ell$-element of $C_{\tG}(s)$ such that $\tchi\in\cE(\tG,st)$.
	Then $\tilde\chi=\tilde\chi_{st,\mu}$ where  $\mu=\prod_\Delta \mu_\Delta$ with $\mu_\Delta\vdash m_\Delta(st)$.	
	Let $\prod\limits_{\Gamma} s_{\Gamma}$ be the primary decomposition of $s$.
	Then $C_{\tG}(s)$ has a corresponding decompostion, so does $t$.
	We write $t=\prod\limits_\Gamma t_\Gamma$,  where $t_\Gamma$ and $s_\Gamma$ commute.
	Let $\tilde\bL=C_{\tilde\bG}(st)$. Then $\tilde\chi=R_{\tilde\bL}^{\tilde\bG}(\widehat{st} \psi_{\mu})$,
	where $\psi_{\mu}$ is the unipotent 	character of ${\tilde\bL}^F$ corresponding to $\mu$.

	Let $\Delta$ be an elementary divisor of $s_\Gamma t_\Gamma$, then by Lemma 
	\ref{degree}, either $\Delta=\Gamma$ or $d_\Delta=e_\Gamma d_\Gamma \ell^{c}$ for some $c\ge 0$.
	Also, $\alpha_\Delta=\alpha_\Gamma+c$.
	Let $\mathcal H_{\mu_{\Delta}}$ be the set of hook lengths (with multiplicity) of $\mu_{\Delta}$
	for every $\Delta\in\cF$ with $m_\Delta(st)\ne 0$.
	Using \cite[(1.3)]{FS82} and
	the degree formulas of unipotent characters of groups of type $\mathsf A$ (see the Appendix of \cite{Lu84}), we have
	\begin{equation}\label{equ-5.6}
	\nu(\tG)-\nu(\tilde \chi(1))=\sum\limits_{\Gamma:m_\Gamma(s)>0}\ \ \sum\limits_{\Delta:m_\Delta(s_\Gamma t_\Gamma)>0}\ \
	\sum\limits_{h\in\mathcal H_{\mu_\Delta}} \nu((\eta q)^{d_{\Delta}h}-1).
	\addtocounter{thm}{1}\tag{\thethm}
	\end{equation}

	For this lemma,	we need only consider the case $\ell\mid q-\eta$ and then $e=e_\Gamma=1$ for every $\Gamma\in\cF$.
	If $s$ has at least two elementary divisor, then according to (\ref{equ-5.6}),
	we have $\nu(\tG)-\nu(\tilde\chi(1))\ge 2\nu(q-\eta)$ by Lemma \ref{valuation}. 
	Thus $\nu(\tG/\tR)-\nu(\tilde\chi(1))\ge \nu(q-\eta)$ since $\nu(\tR)=\nu(q-\eta)$.
	Then the assertion holds.
\end{proof}

\begin{rem}\label{def-aa-C}
	We give more information for Lemma	\ref{def-b}.
	Suppose that the equality in
	$\nu(\tG/\tR)-\nu(\tilde\chi(1))\ge \nu(q-\eta)$ holds.
	Then $st$ has two elementary divisors and so does $s$.
	Let $\Gamma_1,\Gamma_2$ (or $\Delta_1,\Delta_2$, resp.) be the elementary divisors of $s$ (or $st$, resp.).
	Thus we have further $m_{\Delta_i}(st)=1$ and $\alpha_{\Gamma_i}=\alpha_{\Delta_i}=0$ for $i=1,2$.
	We also write $\tilde\chi_{st}$ for such $\tchi$.
	By Lemma \ref{degree},
	we can assume that $d_{\Delta_i}=d_{\Gamma_i}$ for $i=1,2$.
	From this, $o(t)\le \ell^a$ if $\ell$ is odd or ``$\ell=2$ and $4\mid q-\eta$"
	and $o(t)\le 2$ if $\ell=2$ and $4\mid q+\eta$.
	In addition, similar to Remark \ref{def-a-C},
	we have $\tR\subseteq \ker(\hat t)$ if and only if $t\in\tR$.
\end{rem}

\begin{lem}\label{def-a}
	Let $\tG=\GL_{n}(\eta q)$ and $\tR=\mrO_\ell(Z(\tG))$.  Assume further $\ell\mid q-\eta$.
	Suppose that $s$ is a semisimple $\ell'$-element of $\tG$ such that $s$ has only one elementary divisor. 
	Let $\tilde\chi\in\mathcal E_\ell(\tG,s)$. 
	Then $\nu(\tG/\tR)-\nu(\tilde\chi(1))\ge \nu(n)$.
\end{lem}

\begin{proof}	
	We write $\tilde\chi=\tilde\chi_{st,\mu}$ as in Lemma \ref{def-b}.
	Suppose that $\Gamma$ is the elementary divisor  of $s$.
	Let $m_\Gamma$ be defined as before. Then $d_\Gamma=m_\Gamma \ell^{\alpha_\Gamma}$ since $e=e_\Gamma=1$.
	Let $\Delta_1, \Delta_2, \ldots, \Delta_k$ be the elementary divisors of $st$.
	For $1\le i\le k$, by Lemma \ref{degree}, $d_{\Delta_i}=d_\Gamma \ell^{c_i}$
	and then $\alpha_{\Delta_i}=\alpha_\Gamma+c_i$ for some $c_i\ge 0$.
	Also $n=\sum_{i=1}^{k}m_{\Delta_i}(st)d_\Gamma \ell^{c_i}=m_\Gamma  \sum_{i=1}^{k} m_{\Delta_i}(st) \ell^{\alpha_\Gamma+c_i}$.
	From this  $\nu(n)\le \sum\limits_{i:\alpha_\Gamma+c_i=0} m_{\Delta_i}(st) +
	\sum\limits_{i:\alpha_{\Gamma}+c_i>0} m_{\Delta_i}(st) (\alpha_\Gamma+c_i)$ since $\ell\nmid m_\Gamma$.
	By (\ref{equ-5.6}), we have
	\begin{equation}\label{equ-5.3}
	\nu(\tG)-\nu(\tilde \chi(1))=
	\sum\limits_{i=1}^k\sum\limits_{h\in\mathcal H_{\mu_{\Delta_i}}} \nu((\eta q)^{d_{\Delta_i}h}-1)=
	\sum\limits_{i=1}^k\sum\limits_{h\in\mathcal H_{\mu_{\Delta_i}}} \nu((\eta q)^{d_{\Gamma}\ell^{c_i} h}-1).
	\addtocounter{thm}{1}\tag{\thethm}
	\end{equation}

	We first assume that $\ell$ is odd or ``$\ell=2$ and $4\mid q-\eta$".
	Then $\nu(\tG)-\nu(\tilde \chi(1))\ge
	\sum_{i=1}^{k} m_{\Delta_i}(st) (a+\alpha_\Gamma+c_i)\ge a+\sum_{i=1}^{k} m_{\Delta_i}(st) (\alpha_\Gamma+c_i)\ge
	a+\nu(n)$ by Lemma \ref{valuation}.
	Since $\nu(\tR)=a$, we have  $\nu(\tG/\tR)-\nu(\tilde\chi(1))\ge \nu(n)$.

	Now assume that $\ell=2$ and $4\mid q+\eta$. 
	Then $\nu(R)=1$.
	If $\alpha_\Gamma>0$, then similar as above we also have $\nu(\tG)-\nu(\tilde \chi(1))\ge a+\nu(n)$.
	So $\nu(\tG/\tR)-\nu(\tilde\chi(1))\ge\nu(n)$.
	Let $\alpha_\Gamma=0$. If $c_i>0$ for some $i$, then 
	$\nu(\tG)-\nu(\tilde \chi(1))\ge \sum\limits_{i:c_i=0} m_{\Delta_i}(st) + \sum\limits_{i:c_i>0} m_{\Delta_i}(st) (a+c_i)\ge a+\nu(n)>1+\nu(n)$
	by Lemma \ref{valuation} and then $\nu(\tG/\tR)-\nu(\tilde\chi(1))> \nu(n)$.
	If $c_i=0$ for  $1\le i\le k$, then $n=m_\Gamma  \sum_{i=1}^{k} m_{\Delta_i}(st)$ and then $\nu(n)\le \nu( \sum_{i=1}^{k} m_{\Delta_i}(st))$.
	Also by (\ref{equ-5.3}) and Lemma  \ref{valuation},  $\nu(\tG)-\nu(\tilde \chi(1))\ge \sum_{i=1}^{k} m_{\Delta_i}(st)$.
	Thus $\nu(\tG)-\nu(\tilde \chi(1))\ge \nu(n)+1$ since $m\ge \nu(m)+1$ for every positive integer $m$.
	This gives $\nu(\tG/\tR)-\nu(\tilde\chi(1))\ge \nu(n)$ and completes the proof.
\end{proof}

\begin{rem}\label{def-a-C}
	We give more information for Lemma	\ref{def-a}.
	First we assume that $\ell$ is odd or ``$\ell=2$ and $4\mid q-\eta$".
	By its proof, the equality in $\nu(\tG/\tR)-\nu(\tilde \chi(1))\ge \nu(n)$ holds only when $k=1$ and $m_{\Delta_1}(st)=1$.
	This implies that $st$ has only one  elementary divisor $\Delta_1$ with multiplicity $1$,
	then $s$ has only one elementary divisor $\Gamma$ and we set $m=m_\Gamma(s)$.
	So $n=d_{\Delta_1}$ and $\tilde{\bL}=C_{\tbG}(st)$ is a Coxeter torus of $\tG$.
	Now we assume further
	$\tR\subseteq \ker (\tilde\chi)$.
	
	Suppose that $m>1$.
	It is easy to check that $o(t)=\ell^{a+\alpha_{\Delta_1}}=\ell^{a+\nu(n)}$ and
	$t$ is a generator of $\mrO_\ell(Z(C_{\tilde G}(st)))$, then $\tR\nsubseteq \ker(\hat t)$.
	Let $z\in \tR$. Now $\tilde\chi=\tilde\chi_{st,\mu}$ with $\mu\vdash 1$.
	We also abbreviate $\tilde\chi_{st}$ for $\tilde\chi_{st,\mu}$ when no confusion can arise.
	According to the character formula~(see, for example, \cite[Prop.~12.2]{DM91}),
	\begin{align*}
	\tilde\chi(z)=&R_{\tilde\bL}^{\tilde\bG}(\widehat{st})(z)=|\tilde G|^{-1}\sum\limits_{h\in\tilde\bG}Q^{\tilde\bG}_{\tilde\bL}(1)\widehat{st}(z)
	=Q^{\tilde\bG}_{\tilde\bL}(1)\widehat{st}(z) \\
	=&R^{\tilde\bG}_{\tilde\bL}(1_{\bL^F})\widehat{st}(z)=R^{\tilde\bG}_{\tilde\bL}(\widehat{st})(1)\widehat{st}(z)=\hat{t}(z)\tilde\chi(1).
	\end{align*}
	Thus $z\in \ker(\tilde\chi)$ if and only if $z\in\ker(\hat{t})$.
	So $\tR\nsubseteq\ker (\tilde\chi)$.
	This is a contradiction.
	
	Hence $m=1$ and then $d_\Gamma=n$.
	Then $C_{\tG}(s)\cong \GL_1((\eta q)^{n})$ and $t\in C_{\tG}(s)$. 
	Thus $\tR\subseteq\ker (\hat t)$ if and only if $o(t)\le \ell^{\alpha_\Gamma}=n_\ell$.
	
	For the case $\ell=2$ and $4\mid q+\eta$, by (\ref{equ-5.3}),
	we have $\nu(\tG/\tR)-\nu(\tilde \chi(1))=\nu(n)$ holds only when $\nu(n)=0$.
	In this case, $st$ has only one  elementary divisor $\Delta_1$ with multiplicity $1$ and
	then $s$ has only one  elementary divisor $\Gamma$ with $d_{\Delta_1}=d_\Gamma$ and $m_\Gamma(s)=1$.
	So $t=1$, $\alpha_\Gamma=0$ and $m=1$.
	In addition, this case is entirely analogous with the case that  $\ell$ is odd or ``$\ell=2$ and $4\mid q-\eta$"
	and we also use the notation defined above for this case.
\end{rem}

\begin{prop}\label{def-aa}
	Let $\tG=\GL_{n}(\eta q)$ and $\tR=\mrO_\ell(Z(\tG))$.
	Then $\nu(\tG/\tR)-\nu(\tilde\chi(1))\ge \mathrm{min}\{\nu(q-\eta), \nu(n) \} $ for every $\tilde\chi\in\Irr(\tG)$.
\end{prop}	

\begin{proof}
If $\ell\nmid q-\eta$, then $\nu(q-\eta)=0$ and this assertion holds,
while the case $\ell\mid q-\eta$ follows from Lemmas  \ref{def-b} and \ref{def-a} immediately.
\end{proof}

\begin{rem}
	Let $\tG=\GL_{n}(\eta q)$ and $\tR=\mrO_\ell(Z(\tG))$. Suppose that $\tG/\tR$ has a character $\tchi$ of defect zero.
	Then by Proposition \ref{def-aa}, $\nu(q-\eta)=0$ or $\nu(n)=0$.
	 Assume further $\ell\mid q-\eta$. Then $\nu(n)=0$.
	We regard $\tchi$ as a character of $\tG$ and let $s$ be a semisimple $\ell'$-element of $\tG$ such that $\tchi\in\cE_\ell(\tG,s)$.
	By Lemma \ref{def-b}, $s$ has only one elementary divisor, say $\Gamma$.
	Then by Remark \ref{def-a-C},  the multiplicity of $\Gamma$ in $s$ is 1, which implies that $C_{\tG}(s)$ is a Coxeter torus of $\tG$.
	We mention that this is \cite[(4B)]{FS82} when $\ell$ is odd.
\end{rem}

\subsection{Characters of $\tC$.}\label{char-tC-def}

The weights of $\SL_n(q)$ and $\SU_n(q)$ are described in \cite[\S 5]{Feng19} if $\ell\nmid\gcd(n,q-\eta)$.
So  we always assume $\ell\mid q-\eta$ in \S \ref{char-tC-def} and \S \ref{sec-const-wei}.
In this way, $e=1$ and $e_\Gamma=1$ for every $\Gamma\in\cF$.

\begin{lem}\label{can-char-1}
	Let $\tR_{m,\alpha}$ be  special. 
\begin{enumerate}[\rm(1)]
	\item No irreducible character of $\tC_{m,\alpha}/\tR_{m,\alpha}$ is of defect $<\nu(\tN_{m,\alpha}/\tR_{m,\alpha} N_{m,\alpha})$.
	\item If $\tC_{m,\alpha}/\tR_{m,\alpha}$ has an irreducible character of defect $\nu(\tN_{m,\alpha}/\tR_{m,\alpha} N_{m,\alpha})$,
	then either $\nu(m)\le \nu(q-\eta)$ or $\alpha=0$. 
	\item $\nu(\tN_{m,\alpha}/\tR_{m,\alpha} N_{m,\alpha})=0$ if and only if $\nu(m)=0$.
\end{enumerate}	
\end{lem}

\begin{proof}
By Proposition \ref{prop:cc-R-m,alpha,gamma-odd}, \ref{prop:cc-R-m,alpha,gamma-2-linear} and
\ref{prop:cc-R-m,alpha,gamma-2-unitary},
$\nu(\tN_{m,\alpha}/\tR_{m,\alpha} N_{m,\alpha})=\min\{\nu(q-\eta),\nu(m)\}$.
Note that $\tC_{m,\alpha}\cong \GL_m((\eta q)^{\ell^\alpha})$. 
By Lemma \ref{valuation}, $\nu((\eta q)^{\ell^\alpha}-1)=a+\alpha$ unless $\ell=2$, $4\mid q+\eta$ and $\alpha=0$,
in which case $\nu(\eta q-1)=1$.
Then this lemma 
follows from Proposition  \ref{def-aa} immediately.
\end{proof}

\begin{rem}\label{can-char-1-const}
Using Remark \ref{def-aa-C} and \ref{def-a-C}, we give more information for Lemma \ref{can-char-1}.
If $\tC_{m,\alpha}/\tR_{m,\alpha}$ has an irreducible  character of defect $\nu(\tN_{m,\alpha}/\tR_{m,\alpha} N_{m,\alpha})$,
then we have four cases as follows.
\begin{enumerate}
	\item[(i)] $\alpha>0$ and $\nu(m)\le \nu(q-\eta)$.
	\item[(ii)] $\alpha=0$ and $\nu(m)<\nu(q-\eta)$.
\end{enumerate}
In cases (i) and (ii), we have $\nu(m)<\nu((\eta q)^{\ell^\alpha}-1)$,
and then an irreducible character of $\tC_{m,\alpha}/\tR_{m,\alpha}$ is of defect $\nu(\tN_{m,\alpha}/\tR_{m,\alpha} N_{m,\alpha})=\nu(m)$
if and only if it has the form $\tilde\chi_{st}$~(as in Remark  \ref{def-a-C}),
where $s$ is a semisimple $\ell'$-element of $\tC_{m,\alpha}$ and $t\in \tR_{m,\alpha}$ such that 
as an element of $\GL_m((\eta q)^{\ell^\alpha})$, $s$ has only one elementary divisor $\Delta$ with $d_\Delta=m$ and $o(t)\le m_\ell$.
Thus $s$, when viewed as an element of $\GL_{m\ell^\alpha}(\eta q)$,
has a unique elementary divisor $\Gamma=\Phi(\Delta)$, $d_\Delta=m\ell^\alpha$ and $m_\Gamma(s)=1$.
\begin{enumerate}
	\item[(iii)] $\alpha=0$ and $\nu(m)>\nu(q-\eta)$. 
	An irreducible  character of $\tC_{m,\alpha}\tR_{m,\alpha}/\tR_{m,\alpha}$ is of defect $\nu(\tN_{m,\alpha}/\tR_{m,\alpha} N_{m,\alpha})$
	if and only if it has the form $\tilde\chi_{st}$~(as in Remark  \ref{def-aa-C}),
	where $s$ is a semisimple $\ell'$-element of $\tC_{m,\alpha}$ and $t\in \tR_{m,\alpha}$ such that 
	$s$ has exactly two elementary divisors $\Gamma_1$ and $\Gamma_2$ with
	$\alpha_{\Gamma_1}=\alpha_{\Gamma_2}=0$,
	$m_{\Gamma_1}(s)=m_{\Gamma_2}(s)=1$,
	 $d_{\Gamma_1}+d_{\Gamma_2}=m$.
	\item[(iv)] $\alpha=0$ and $\nu(m)=\nu(q-\eta)$.
	An irreducible  character of $\tC_{m,\alpha}\tR_{m,\alpha}/\tR_{m,\alpha}$ is of defect $\nu(\tN_{m,\alpha}/\tR_{m,\alpha} N_{m,\alpha})$
	when viewed as a characters of $\tC_{m,\alpha}$,
	if and only if it has the form in case (ii) or (iii).
\end{enumerate}
\end{rem}

\begin{lem}\label{can-char-2}
Let $\gamma>0$ and $\tR_{m,\alpha,\gamma}$ be special.
Assume that 
\begin{equation*}
``\ell\ \textrm{is odd}" \ \textrm{or} \ ``\ell=2\ \textrm{and}\ 4\mid q-\eta" \ \textrm{or}  \  ``\ell=2,\ 4\mid q+\eta\ \textrm{and}\ \alpha\ge 1". %\addtocounter{thm}{1}\tag{\thethm}
\end{equation*}
and we are neither in the cases  (\ref{eq:special-case-3-1}) nor in the case (\ref{eq:special-case-2-linear-1}). 
\begin{enumerate}[\rm(1)]
\item No irreducible  character of $\tC_{m,\alpha,\gamma}\tR_{m,\alpha,\gamma}/\tR_{m,\alpha,\gamma}$
	is of defect  $<\nu(\tN_{m,\alpha,\gamma}/\tR_{m,\alpha,\gamma} N_{m,\alpha,\gamma})$.
\item Assume $\tC_{m,\alpha,\gamma}\tR_{m,\alpha,\gamma}/\tR_{m,\alpha,\gamma}$ has an irreducible character
	of defect $\nu(\tN_{m,\alpha,\gamma}/\tR_{m,\alpha,\gamma} N_{m,\alpha,\gamma})$.
	Then $\nu(m)+\gamma\le \nu(q-\eta)$,
	and thus $\nu(\tN_{m,\alpha,\gamma}/\tR_{m,\alpha,\gamma} N_{m,\alpha,\gamma})=0$ if and only if $\nu(m)=0$.
\end{enumerate}	
\end{lem}

\begin{proof}
Under the assumption,
$\nu(\tN_{m,\alpha,\gamma}/\tR_{m,\alpha,\gamma} N_{m,\alpha,\gamma})<\nu(q-\eta)$ and 
$\nu(\tN_{m,\alpha,\gamma}/\tR_{m,\alpha,\gamma} N_{m,\alpha,\gamma})\le\nu(m)$ 
by Proposition \ref{prop:cc-R-m,alpha,gamma-odd}, \ref{prop:cc-R-m,alpha,gamma-2-linear} and \ref{prop:cc-R-m,alpha,gamma-2-unitary}.
In addition,
$\nu(\tN_{m,\alpha,\gamma}/\tR_{m,\alpha,\gamma} N_{m,\alpha,\gamma})=\nu(m)$ if and only if $\nu(m)+\gamma\le \nu(q-\eta)$.
Therefore (1) follows by Proposition \ref{def-aa} since $\tC_{m,\alpha,\gamma}\cong\GL_m((\eta q)^{\ell^\alpha})$.
If $\nu(m)+\gamma> \nu(q-\eta)$,
then $\nu(\tN_{m,\alpha,\gamma}/\tR_{m,\alpha,\gamma} N_{m,\alpha,\gamma})<\min\{\nu(q-\eta),  \nu(m)\}$,
and thus by Proposition \ref{def-aa} again, 
no character of $\tC_{m,\alpha,\gamma}\tR_{m,\alpha,\gamma}/\tR_{m,\alpha,\gamma}$
is of defect $\nu(\tN_{m,\alpha,\gamma}/\tR_{m,\alpha,\gamma} N_{m,\alpha,\gamma})$.
So $\nu(m)+\gamma\le \nu(q-\eta)$.
Using Proposition \ref{prop:cc-R-m,alpha,gamma-odd}, \ref{prop:cc-R-m,alpha,gamma-2-linear} and \ref{prop:cc-R-m,alpha,gamma-2-unitary} again, we have $\nu(\tN_{m,\alpha,\gamma}/\tR_{m,\alpha,\gamma} N_{m,\alpha,\gamma})=0$ if and only if $\nu(m)=0$.
\end{proof}

\begin{rem}
In Lemma \ref{can-char-2},
if some irreducible character of $\tC_{m,\alpha,\gamma}\tR_{m,\alpha,\gamma}/\tR_{m,\alpha,\gamma}$
is of defect $\nu(\tN_{m,\alpha,\gamma}/\tR_{m,\alpha,\gamma} N_{m,\alpha,\gamma})$,
then it has the form similar to the case (i) or (ii) of Remark \ref{can-char-1-const}.
\end{rem}

\begin{rem}\label{special-case-ell3}
	Now we assume that  (\ref{eq:special-case-3-1}) holds for $\tR_{m,\alpha,\gamma}$,
	which means $\ell=3$, $a=1$, $\nu(m)=\alpha=0$, $\gamma=1$.
	By Proposition \ref{prop:special-odd-1}, $\tR_{m,0,1}$ is special.
	Then we have 
	$\tR_{m,0,1}=R_{m,0,1}$,
	$\tC_{m,0,1}/C_{m,0,1}\cong\mrO_{3'}(\fZ_{q-\eta})$, 
	$|\tM_{m,0,1}/M_{m,0,1}|=3$ and
	$\nu(\tN_{m,0,1}/\tR_{m,0,1}N_{m,0,1})=1$
	by \S \ref{subsect:srs-odd}.
	Moreover, 
	$\tM_{m,0,1}\tR_{m,0,1}/\tR_{m,0,1}\cong \SL_{2}(3)$,  $\tN_{m,0,1}/\tR_{m,0,1}\cong\tR_{m,0,1}\tC_{m,0,1}/\tR_{m,0,1}\times \SL_{2}(3)$,
	and $M_{m,0,1}R_{m,0,1}/R_{m,0,1}\cong Q_8$,
	$N_{m,0,1}/R_{m,0,1}\cong R_{m,0,1}C_{m,0,1}/R_{m,0,1}\times Q_8$.
	
	For the character table of $\SL_2(3)$  we refer to the Appendix B of \cite[p.~306]{We16}.
	Using the notation there, the group $\SL_2(3)=Q_8\rtimes C_3$ has $7$ irreducible characters
	$\chi_1$, $\chi_{1a}$, $\chi_{1b}$, $\chi_{2a}$, $\chi_{2b}$, $\chi_{2c}$, $\chi_{3}$ and we denote $\chi_{1c}=\chi_1$ here for convenience.
	Note that $\chi_3$ is the Steinberg character of $\SL_2(3)$.
	Compare with the character table of $Q_8$~(see e.g. \cite[p.~299]{We16}), we have that
	for any $j\in\{ a,b,c\}$, $\Res^{\SL_2(3)}_{Q_8}(\chi_{1j})=1_{Q_8}$,
	$\Res^{\SL_2(3)}_{Q_8}(\chi_{2j})$ are the (unique) irreducible character of $Q_8$ of degree $2$,
	and $\Res^{\SL_2(3)}_{Q_8}(\chi_{3})$ is a sum of the three non-trivial irreducible characters of $Q_8$ of degree $1$.
	Let $\tilde\psi\in\Irr(\tN_{m,0,1}/\tR_{m,0,1})$ be of defect $\le \nu(\tN_{m,0,1}/\tR_{m,0,1}N_{m,0,1})=1$.
	
	(I). 
	If $\tilde\psi$ is of defect $0$, then
	$\tilde\psi=\tilde\theta\times \chi_3$, where $\tilde\theta\in\dz(\tR_{m,0,1}\tC_{m,0,1}/\tR_{m,0,1})$.
	Then every $\psi\in\Irr(N_{m,0,1}/R_{m,0,1}\mid \tilde\psi)$ has defect $0$.
	Moreover, $\nu(\kappa^{\tN_{m,0,1}}_{N_{m,0,1}}(\tilde\psi))=1$.
	
	(II). If $\tilde\psi$ is of defect $1$ and $\tilde\psi=\tilde\theta\times \chi_{kj}$,
	where $\tilde\theta\in\dz(\tR_{m,0,1}\tC_{m,0,1}/\tR_{m,0,1})$, and $k\in\{1,2\}$ and $j\in\{a,b,c\}$,
	then every $\psi\in\Irr(N_{m,0,1}/R_{m,0,1}\mid \tilde\psi)$ has defect $0$.
	Moreover, $3\nmid\kappa^{\tN_{m,0,1}}_{N_{m,0,1}}(\tilde\psi)$.
	
	(III). If $\tilde\psi$ is of defect $1$ and $\tilde\psi=\tilde\theta'\times \chi_{3}$,
	where $\tilde\theta'\in\Irr(\tR_{m,0,1}\tC_{m,0,1}/\tR_{m,0,1})$ is of defect $1$,
	then every $\psi\in\Irr(N_{m,0,1}/R_{m,0,1}\mid \tilde\psi)$ has defect $1$.
	\end{rem}

\begin{rem}\label{special-case-ell2-linear}
	Now we  assume  that (\ref{eq:special-case-2-linear-1}) holds for $\tR_{m,\alpha,\gamma}$,
	which means $\ell=2$, $4 \mid q-\eta$,
	$a=2$, $\nu(m)=\alpha=0$ and $\gamma\in\{1, 2\}$.
	By Proposition \ref{prop:special-2-linear-1},
	$\tR_{m,0,\gamma}$ is special.

	We first assume that $\gamma=1$.
	Then by \S \ref{subsect:srs-2-linear},
	$\nu(\det(\tC_{m,0,1}\tR_{m,0,1}))=1$,
	$\nu(\det(\tN_{m,0,1}))=2$ and
	$\nu(\tN_{m,0,1}/\tR_{m,0,1}N_{m,0,1})=1$.
	Moreover, 
	$\tM_{m,0,1}\tR_{m,0,1}/\tR_{m,0,1}\cong \Sp_{2}(2)\cong\fS_3$,
	$\tN_{m,0,1}/\tR_{m,0,1}\cong\tR_{m,0,1}\tC_{m,0,1}/\tR_{m,0,1}\times \fS_3$ and $M_{m,0,1}R_{m,0,1}/R_{m,0,1}\cong C_3$,
	$N_{m,0,1}/R_{m,0,1}\cong R_{m,0,1}C_{m,0,1}/R_{m,0,1}\times C_3$.
	For the character table of $\fS_3$ we refer to \cite[p.~298]{We16}.
	Using the notation there, the group $\fS_3$ have three irreducible characters $\chi_1$, $\chi_{\textrm{sign}}$, $\chi_2$,
	and for the unique subgroup $C_3$ of $\fS_3$ with index $2$,
	the restrictions of the characters $\chi_1$, $\chi_{\textrm{sign}}$ to $C_3$ are the trivial character of $C_3$,
	while the restriction of the character $\chi_2$ to $C_3$ splits to the two non-trivial characters of $C_3$.
	Let  $\tilde\psi\in\Irr(\tN_{m,0,1}/\tR_{m,0,1})$ be of defect $\le \nu(\tN_{m,0,1}/\tR_{m,0,1}N_{m,0,1})=1$.
	If $\tilde\psi$ is of defect $0$,
	then $\tilde\psi=\tilde\theta\times \chi_2$, where $\tilde\theta\in\dz(\tR_{m,0,1}\tC_{m,0,1}/\tR_{m,0,1})$,
	and then every $\psi\in\Irr(N_{m,0,1}/R_{m,0,1}\mid \tilde\psi)$ has defect $0$.
	Moreover, $\nu(\kappa^{\tN_{m,0,1}}_{N_{m,0,1}}(\tilde\psi))=1$.
	If $\tilde\psi$ is of defect $1$ and
	$\tilde\psi=\tilde\theta\times \chi_i$, where $\tilde\theta\in\dz(\tR_{m,0,1}\tC_{m,0,1}/\tR_{m,0,1})$ and $i\in\{1,\textrm{sign}\}$,
	then every $\psi\in\Irr(N_{m,0,1}/R_{m,0,1}\mid \tilde\psi)$ has defect $0$.
	Moreover, $2\nmid\kappa^{\tN_{m,0,1}}_{N_{m,0,1}}(\tilde\psi)$.
	In addition, the restrictions of $\tilde\theta\times \chi_1$ and $\tilde\theta\times \chi_{\textrm{sign}}$ to $N_{m,0,1}$ coincide.
	If $\tilde\psi$ is of defect $1$ and $\tilde\psi=\tilde\theta'\times \chi_2$,
	where $\tilde\theta'\in\Irr(\tR_{m,0,1}\tC_{m,0,1}/\tR_{m,0,1})$ is of defect $1$,
	then every $\psi\in\Irr(N_{m,0,1}/R_{m,0,1}\mid \tilde\psi)$ has defect $1$.

	Now let $\gamma=2$.
	Then we know 
	$\tR_{m,0,2}=R_{m,0,2}$,
	$\tC_{m,0,2}/C_{m,0,2}\cong\mrO_{2'}(\tG_{m,0,2}/G_{m,0,2})$, 
	$|\tM_{m,0,2}/M_{m,0,2}|=2$ and
	$\nu(\tN_{m,0,2}/\tR_{m,0,2}N_{m,0,2})=1$
	from \S \ref{subsect:srs-2-linear}.
	Moreover, 
	$\tM_{m,0,2}\tR_{m,0,2}/\tR_{m,0,2}\cong \fS_6$,  $\tN_{m,0,2}/\tR_{m,0,2}\cong\tR_{m,0,2}\tC_{m,0,2}/\tR_{m,0,2}\times \fS_6$,
	and 
	$M_{m,0,2}R_{m,0,2}/R_{m,0,2}\cong \fA_6$,
	$N_{m,0,2}/R_{m,0,2}\cong R_{m,0,2}C_{m,0,2}/R_{m,0,2}\times \fA_6$.	
	For the character table of $\fS_6$  we refer to \cite[p.~205]{JL93}.
	Using the notation there, the group
	$\fS_6$ has a unique irreducible character $\chi_{11}$ of defect $0$ and has no character of defect $1$.	
	Compared with the character table of $\fA_6$~(see e.g. \cite[p.~424]{JL93}),
	we know that the group $\fA_6$  has exactly two irreducible characters of defect $0$
	and they are the two irreducible constituents of $\Res^{\fS_6}_{\fA_6}\chi_{11}$.
	Let  $\tilde\psi\in\Irr(\tN_{m,0,2}/\tR_{m,0,2})$ be of defect $\le \nu(\tN_{m,0,2}/\tR_{m,0,2}N_{m,0,2})=1$.
	If $\tilde\psi$ is of defect $0$, then
	$\tilde\psi=\tilde\theta\times \chi_{11}$, where $\tilde\theta\in\dz(\tR_{m,0,2}\tC_{m,0,2}/\tR_{m,0,2})$,
	then every $\psi\in\Irr(N_{m,0,2}/R_{m,0,2}\mid \tilde\psi)$ has defect $0$.
	Moreover, $\nu(\kappa^{\tN_{m,0,2}}_{N_{m,0,2}}(\tilde\psi))=1$.
	If $\tilde\psi$ is of defect $1$, then  $\tilde\psi=\tilde\theta'\times \chi_{11}$,
	where $\tilde\theta'\in\Irr(\tR_{m,0,2}\tC_{m,0,2}/\tR_{m,0,2})$ is of defect $1$,
	then every $\psi\in\Irr(N_{m,0,2}/R_{m,0,2}\mid \tilde\psi)$ has defect $1$.
\end{rem}

\begin{lem}\label{can-char-2'}
Let $\gamma>0$. Assume that 
$\ell=2$, $4\mid q+\eta$ and $\alpha=0$. 
\begin{enumerate}[\rm(1)]
\item Let $\tR_{m,0,\gamma}=\tS_{m,1,\gamma-1}$ be special. Then $\nu(\tN_{m,0,\gamma}/\tR_{m,0,\gamma} N_{m,0,\gamma})=0$. Furthermore,
if $\tC_{m,0,\gamma}\tR_{m,0,\gamma}/\tR_{m,0,\gamma}$ has an irreducible  character of defect $0$, then $\nu(m)=0$.
\item Let $\tR^\pm_{m,0,\gamma}=E^\pm_{m,\gamma}$~($\gamma>1$ if the type is plus, $\gamma>0$ if the type is minus) be special. Assume further we are not in the case (\ref{special-case-2-uni-1}).
Then $\nu(\tN^\pm_{m,0,\gamma}/\tR^\pm_{m,0,\gamma} N^\pm_{m,0,\gamma})=0$. In addition,
if $\tC^\pm_{m,0,\gamma}\tR^\pm_{m,0,\gamma}/\tR^\pm_{m,0,\gamma}$ has an irreducible  character of defect $0$, then $\nu(m)=0$.
\end{enumerate}
\end{lem}

\begin{proof}
This follows from Proposition \ref{prop:cc-R-m,alpha,gamma-2-unitary} and \ref{def-aa}.
\end{proof}

\begin{rem}\label{exc-case-2uni}
Now we assume that
 (\ref{special-case-2-uni-1})  holds for
 $\tR^\pm_{m,0,\gamma}=E^\pm_{m,\gamma}$, which means $\nu(m)=0$ and $\gamma=2$  if the type is plus and $\gamma=1$ and $a=2$
if the type is minus.
By Proposition \ref{prop:special-2-unitary-1}, $\tR^\pm_{m,0,\gamma}$ is special.

First we assume that  the type is plus, $\nu(m)=0$ and $\gamma=2$. 	Then we know 
$\tR^+_{m,0,2}=R^+_{m,0,2}$,
$\tC^+_{m,0,2}/C^+_{m,0,2}\cong\mrO_{2'}(\tG^+_{m,0,2}/G^+_{m,0,2})$, 
$|\tM^+_{m,0,2}/M^+_{m,0,2}|=2$ and
$\nu(\tN^+_{m,0,2}/\tR^+_{m,0,2}N^+_{m,0,2})=1$ by 
\S \ref{subsect:srs-2-unitary}.
Moreover, 
$\tM_{m,0,2}\tR_{m,0,2}/\tR_{m,0,2}\cong \GO^+_4(2)$,  $\tN_{m,0,2}/\tR_{m,0,2}\cong\tR_{m,0,2}\tC_{m,0,2}/\tR_{m,0,2}\times \GO^+_4(2)$,
and 
$M_{m,0,2}R_{m,0,2}/R_{m,0,2}\cong \Omega^+_4(2)$,
$N_{m,0,2}/R_{m,0,2}\cong R_{m,0,2}C_{m,0,2}/R_{m,0,2}\times \Omega^+_4(2)$.
We recall the remark in \cite[p.511]{An92};
the group $\Omega^+_4(2)$ has a unique irreducible character of defect 0, the Steinberg character, which has two extensions $\chi_1$, $\chi_2$ to $\GO^+_4(2)$.
Then the characters $\chi_1$ and $\chi_2$ are of defect $1$.
In addition, the group $\GO^+_4(2)$ has no irreducible character of defect $0$ and if $\chi\notin\{\chi_1,\chi_2\}$ is an irreducible character of $\GO^+_4(2)$ with defect $1$, then every irreducible constituent of $\Res^{\GO^+_4(2)}_{\Omega^+_4(2)}(\chi)$ has defect $1$.
Let  $\tilde\psi\in\Irr(\tN^+_{m,0,2}/\tR^+_{m,0,2})$ be of defect $\le \nu(\tN^+_{m,0,2}/\tR^+_{m,0,2}N^+_{m,0,2})=1$.
If some irreducible character of $N^+_{m,0,2}/R^+_{m,0,2}$ lying below $\tilde{\psi}$ is of defect $0$, then by Remark \ref{4.39},
$\tilde\psi=\tilde\theta\times \chi_{i}$, where $\tilde\theta\in\dz(\tR^+_{m,0,2}\tC^+_{m,0,2}/\tR^+_{m,0,2})$ and $i=1$ or $2$.
Moreover, the restrictions of $\tilde\theta\times \chi_{1}$ and $\tilde\theta\times \chi_{2}$ to $N^+_{m,0,2}$ coincide.

Now let the type be minus, $a=2$, $\nu(m)=0$ and $\gamma=1$.
Then we know 
$\tR^-_{m,0,1}=R^-_{m,0,1}$,
$\tC^-_{m,0,1}/C^-_{m,0,1}\cong\mrO_{2'}(\fZ_{q-\eta})$, 
$|\tM^-_{m,0,1}/M^-_{m,0,1}|=2$ and
$\nu(\tN^-_{m,0,1}/\tR^-_{m,0,1}N^-_{m,0,1})=1$ by 
\S \ref{subsect:srs-2-unitary}.
Moreover, 
$\tM^-_{m,0,1}\tR_{m,0,1}/\tR^-_{m,0,1}\cong \GO^-_2(2)\cong\fS_3$,  $\tN^-_{m,0,1}/\tR^-_{m,0,1}\cong\tR^-_{m,0,1}\tC^-_{m,0,1}/\tR^-_{m,0,1}\times \fS_3$,
and 
$M^-_{m,0,1}R^-_{m,0,1}/R^-_{m,0,1}\cong C_3$,
$N^-_{m,0,1}/R^-_{m,0,1}\cong R^-_{m,0,1}C^-_{m,0,1}/R^-_{m,0,1}\times C_3$.
We use  the notation of irreducible characters of $\fS_3$ as in
the case $\gamma=1$ in
Remark \ref{special-case-ell2-linear}.
Let  $\tilde\psi\in\Irr(\tN^-_{m,0,1}/\tR^-_{m,0,1})$ be of defect $\le \nu(\tN^-_{m,0,1}/\tR^-_{m,0,1}N^-_{m,0,1})=1$.
If $\tilde\psi$ is of defect $0$, then
$\tilde\psi=\tilde\theta\times \chi_2$, where $\tilde\theta\in\dz(\tR^-_{m,0,1}\tC^-_{m,0,1}/\tR^-_{m,0,1})$,
and then every $\psi\in\Irr(N^-_{m,0,1}/R^-_{m,0,1}\mid \tilde\psi)$ has defect $0$.
Moreover, $\nu(\kappa^{\tN^-_{m,0,1}}_{N^-_{m,0,1}}(\tilde\psi))=1$.
If $\tilde\psi$ is of defect $1$ and
$\tilde\psi=\tilde\theta\times \chi_i$, where where $\tilde\theta\in\dz(\tR^-_{m,0,1}\tC^-_{m,0,1}/\tR^-_{m,0,1})$ and $i\in\{1,\textrm{sign}\}$,
then every $\psi\in\Irr(N^-_{m,0,1}/R^-_{m,0,1}\mid \tilde\psi)$ has defect $0$.
Moreover, $2\nmid\kappa^{\tN^-_{m,0,1}}_{N^-_{m,0,1}}(\tilde\psi)$.
In addition, the restrictions of $\tilde\theta\times \chi_1$ and $\tilde\theta\times \chi_{\textrm{sign}}$ to $N^-_{m,0,1}$ coincide.
If $\tilde\psi$ is of defect $1$ and $\tilde\psi=\tilde\theta'\times \chi_2$, where $\tilde\theta'\in\Irr(\tR^-_{m,0,1}\tC^-_{m,0,1}/\tR^-_{m,0,1})$ is of defect $1$,
then every $\psi\in\Irr(N^-_{m,0,1}/R^-_{m,0,1}\mid \tilde\psi)$ has defect $1$.
\end{rem}

\begin{lem}\label{can-char-3}
Let $\bc\ne\zero$ and $\tR_{m,\alpha,\gamma,\bc}$ be special.
\begin{enumerate}[\rm(1)]
\item No irreducible character of $\tC_{m,\alpha,\gamma,\bc}\tR_{m,\alpha,\gamma,\bc}/\tR_{m,\alpha,\gamma,\bc}$
	is of defect $<\nu(\tN_{m,\alpha,\gamma,\bc}/\tR_{m,\alpha,\gamma,\bc} N_{m,\alpha,\gamma,\bc})$.
\item If $\tC_{m,\alpha,\gamma,\bc}\tR_{m,\alpha,\gamma,\bc}/\tR_{m,\alpha,\gamma,\bc}$
	has an irreducible character of defect $\nu(\tN_{m,\alpha,\gamma,\bc}/\tR_{m,\alpha,\gamma,\bc} N_{m,\alpha,\gamma,\bc})$,
	then $\nu(m)=0$ and $\nu(\tN_{m,\alpha,\gamma,\bc}/\tR_{m,\alpha,\gamma,\bc} N_{m,\alpha,\gamma,\bc})=0$. 
\end{enumerate}	
\end{lem}

\begin{proof}
If $\nu(m)>0$, then by	
Proposition \ref{prop:cc-R-m,alpha,gamma,c-odd},
\ref{prop:cc-R-m,alpha,gamma,c-2-linear} and \ref{prop:cc-R-m,alpha,gamma,c-2-unitary},
$\nu(\tN_{m,\alpha,\gamma,\bc}/\tR_{m,\alpha,\gamma,\bc} N_{m,\alpha,\gamma,\bc})<\min\{\nu(q-\eta),\nu(m)\}$.
Therefore no character of $\tC_{m,\alpha,\gamma,\bc}\tR_{m,\alpha,\gamma,\bc}/\tR_{m,\alpha,\gamma,\bc}$ is of defect $\le\nu(\tN_{m,\alpha,\gamma,\bc}/\tR_{m,\alpha,\gamma,\bc} N_{m,\alpha,\gamma,\bc})$ by Proposition \ref{def-aa}.
If $\nu(m)=0$, then $\nu(\tN_{m,\alpha,\gamma,\bc}/\tR_{m,\alpha,\gamma,\bc} N_{m,\alpha,\gamma,\bc})=0$, and the lemma follows.	
\end{proof}

\begin{lem}\label{can-char-3'}
Let $\gamma>0$.
Assume that $\ell=2$, $4\mid q+\eta$ and $\alpha=0$. 
Let $\tR^\pm_{m,0,\gamma,\bc}$ be a special basic subgroup such that $\gamma>1$ if the type is plus and $\gamma>0$ if the type is minus.
Then we have $\nu(\tN^\pm_{m,0,\gamma,\bc}/\tR^\pm_{m,0,\gamma,\bc} N^\pm_{m,0,\gamma,\bc})=0$,
and if $\tC^\pm_{m,0,\gamma,\bc}\tR^\pm_{m,0,\gamma,\bc}/\tR^\pm_{m,0,\gamma,\bc}$ has an irreducible character of defect $0$,
then $\nu(m)=0$.		
\end{lem}

\begin{proof}
Using Proposition \ref{prop:cc-R-m,alpha,gamma,c-2-unitary},
this can be shown by the same arguments  in Lemma \ref{can-char-3}.
\end{proof}

\begin{rem}
In Lemmas \ref{can-char-3} and \ref{can-char-3'},
if $\nu(m)=0$, 
then the  irreducible characters of $\tC_{m,\alpha,\gamma,\bc}\tR_{m,\alpha,\gamma,\bc}/\tR_{m,\alpha,\gamma,\bc}$ of defect zero are constructed in \cite{AF90,An92, An93,An94}.
In fact, every such character is of the form $\ttheta_\Gamma\otimes I_{\ell^\gamma}\otimes I_\bc$ as in \S \ref{subsect:weights-general}, for some $\Gamma\in\cF'$ with $m=m_\Gamma$ and $\alpha=\alpha_\Gamma$.
\end{rem}

Summarizing the results above, we have

\begin{cor}\label{can-char-4}
	For any special basic subgroup $\tR=\tR_{m,\alpha,\gamma,\bc}$ (or $\tR^\pm_{m,\alpha,\gamma,\bc}$),
	we let $\tN=\tN_{m,\alpha,\gamma,\bc}$ (or $\tN^\pm_{m,\alpha,\gamma,\bc}$) and $\tC=\tC_{m,\alpha,\gamma,\bc}$ (or $\tC^\pm_{m,\alpha,\gamma,\bc}$).	

	\begin{enumerate}[\rm(1)]
		\item $\tC\tR/\tR$ has an irreducible  character of defect zero only when $\nu(m)=0$.
		\item No irreducible  character of $\tC\tR/\tR$ is of defect $<\nu(\tN/\tR N)$ unless when we are in one of the cases (\ref{eq:special-case-3-1}), (\ref{eq:special-case-2-linear-1}) and (\ref{special-case-2-uni-1}),
		in which case, $\nu(\tN/\tR N)=1$.	
	\end{enumerate}	
\end{cor}

Now we consider the special radical subgroups of the form $\tR=\tR_1\times\cdots\times\tR_u$  with $\tR_i$ being a basic subgroup.
Assume that $\det(\tR_1) \geq\cdots\geq \det(\tR_u)$.

\begin{lem}\label{lem:char-def-small}
Keep the notation of Proposition \ref{prop:special-odd-3}, \ref{prop:special-2-linear-3} and \ref{prop:special-2-unitary-3}.
Assume that $\tR$ is special.
Then no irreducible  character of $\tC\tR/\tR$ is of defect $<\nu(\tN/\tR N)$ unless we are in one of the cases (\ref{special-case-wei-ell3}), (\ref{special-case-wei-ell2-linear}) and (\ref{special-case-wei-ell2-uni}),
in which case, $\nu(\tN/\tR N)=1$.
\end{lem}

\begin{proof}
As in the proof of Lemma  \ref{det-radical-N-C-2-linear}, $\det(\tN)=\grp{\det(\tN_i)\mid 1\le i\le u}$.
Note that $\tC\tR/\tR=\prod_{i=1}^u \tC_i\tR_i/\tR_i$.
Let $\tilde\theta=\prod_{i=1}^u \tilde\theta_i$ be an irreducible character of $\tC\tR/\tR$ with $\tilde\theta_i\in \Irr(\tC_i\tR_i/\tR_i)$.	
Assume that  $\tilde{\theta}$ is of defect $<\nu(\tN/\tR N)$.
If $\nu(\tN/\tR N)< \nu(\tN_i/\tR_i N_i)$ for some $i$, then by 
Corollary \ref{can-char-4}, $\tC\tR/\tR$ has no 
irreducible character of defect
$<\nu(\tN/\tR N)$.
Then $\nu(\tN/\tR N)\ge\nu(\tN_i/\tR_i N_i)$ for any $i$.
Let $i_0$ satisfy $\det(\tN)=\det(\tN_{i_0})$.
Then $\nu(\tN/\tR N)= \nu(\det(\tN_{i_0}))- \nu(\det(\tR_{1}))\le\nu(\tN_{i_0}/\tR_{i_0} N_{i_0})$.
So $\det(\tR_{i_0})=\det(\tR_1)$ and 
we assume $i_0=1$ without loss of generality. 
Therefore, $\tilde{\theta}_1$ has defect $<\nu(\tN_{1}/\tR_{1} N_{1})$.
By Corollary \ref{can-char-4}, $R_1$ is one of the cases (\ref{eq:special-case-3-1}), (\ref{eq:special-case-2-linear-1}) or (\ref{special-case-2-uni-1}) and $\nu(\tN_{1}/\tR_{1} N_{1})=1$, $\tilde{\theta}_1$ is of defect $0$.
In addition, $\tilde{\theta}$ is of defect $0$. From this, $\nu(m_i)=0$ if $i>1$ and one can check that we are in one of the cases (\ref{special-case-wei-ell3}), (\ref{special-case-wei-ell2-linear}) and (\ref{special-case-wei-ell2-uni}).
\end{proof}

\begin{prop}\label{sepcial-defect}
Keep the notation of Proposition \ref{prop:special-odd-3}, \ref{prop:special-2-linear-3} and \ref{prop:special-2-unitary-3}.
Assume that $\tR$ is special,
$\det(\tR_1) \geq\cdots\geq \det(\tR_u)$
 and $\tR\tC/\tR$ has an irreducible character of defect
$\le\nu(\tN/\tR N)$.
Write $\tR_i=\tR_{m_i,\alpha_i,\gamma_i,\bc_i}$ or $\tR^\pm_{m_i,\alpha_i,\gamma_i,\bc_i}$~~(note that $\tR^\pm_{m_i,\alpha_i,\gamma_i,\bc_i}$ only occurs when $\ell=2$, $4\mid q+\eta$ and $\alpha_i=0$).

\begin{enumerate}[\rm(1)]
	\item Suppose that we are not in one of the cases (\ref{special-case-wei-ell3}), (\ref{special-case-wei-ell2-linear}) and (\ref{special-case-wei-ell2-uni}).
	Then there eixsts some $1\le i_0\le u$ such that
	$\det(\tN)=\det(\tN_{i_0})$ and $\det(\tR)=\det(\tR_{i_0})$.
	We assume $i_0=1$ without loss of generality. 
	Then $\nu(m_{i})=0$ if $i>1$.
	\begin{enumerate}[\rm(a)]
\item $\tR\tC/\tR$ has no irreducible character of defect
$<\nu(\tN/\tR N)$.
     \item $\nu(\tN/\tR N)=0$ if and only if $\nu(m_1)=0$. 
		\item If $\bc_1\ne\zero$, then $\nu(m_{1})=0$.  Moreover, none of the cases
		``$\ell=3$, $a=1$, $\nu(m_1)=\alpha_1=\gamma_1=0$, $\bc_1=\one$ and $a(R_i)>0$ for $i>1$",
		``$\ell=2$, $4\mid q-\eta$, $a=2$, $\nu(m_1)=\alpha_1=\gamma_1=0$, $\bc_1=\one$ and $a(R_i)>a(R_1)$ for $i>1$" and 
		``$\ell=2$, $4\mid q+\eta$,  $\nu(m_1)=\alpha_1=\gamma_1=0$, $\bc_1=\one$ or $\two$ and $\det(R_i)=1$ for $i>1$"
		occurs.
		\item Let $\bc_1=\zero$ and $\gamma_{1}\ne 0$. Moreover, 
		\begin{enumerate}[\rm(i)]
			\item if ``$\ell$ is odd" or ``$\ell=2$ and $4\mid q-\eta$", then $\nu(m_{1})+\gamma_{1}\le \nu(q-\eta)$ and
			one of the following holds,
			\begin{itemize}
				\item ``$\nu(m_{1})+\gamma_{1}=\nu(q-\eta)$" or ``$\nu(m_{1})+\gamma_{1}=\nu(m_{2})+\gamma_{2}$",
				\item
				$\alpha_1=0$ and $\nu(m_{1})+\gamma_{1}<\min\{ \nu(q-\eta),\nu(m_{i})+\gamma_{i}\}$ for $i>1$,
			\end{itemize}	
			\item if $\ell=2$ and $4\mid q+\eta$, then  $\nu(m_1)=0$ and
			one of the following holds,
			\begin{itemize}
				\item 
				$\det(\tR_1)=1$, i.e., $\tR_1=\tR_{m_1,\alpha_1,\gamma_1}$~ ($\alpha_1>0$),
				 $\tR_1=\tR_{m_1,0,\gamma_1}=\tS_{m_1,1,\gamma_1-1}$~ ($\gamma_1>1$),
			or $\tR_1=\tR^\pm_{m_1,0,\gamma_1}$,
			\item $\tR_1=\tR_{m_1,0,1}=\tS_{m_1,1}$  and $\det(\tR_2)>1$,
			\item $\tR_1=\tR_{m_1,0,1}=\tS_{m_1,1}$, $a>2$, and $\det(\tR_i)=1$  for $i>1$.
			\end{itemize}
		\end{enumerate}
		\item Let $\bc_{1}=\zero$ and $\gamma_{1}= 0$. Then $\alpha_1=0$ or $\nu(m_1)\le \nu(q-\eta)$. Moreover,
			\begin{enumerate}[\rm(i)]
				\item  if ``$\ell$ is odd" or ``$\ell=2$ and $4\mid q-\eta$", then 
				one of the following holds,
		\begin{itemize}
			\item $\nu(m_1)=\nu(q-\eta)$,
			\item $\alpha_1=0$, $\nu(m_1)< \nu(q-\eta)$ except
			``$\nu(m_1)=\nu(m_2)=0$, $\alpha_2=\gamma_2=|\bc_2|=0$, $\nu(m_1+m_2)\ge \nu(q-\eta)$, $a(\tR_3)\ge \nu(q-\eta)$",
			\item  $\alpha_1>0$, $\nu(m_1)<\nu(q-\eta)$ and $\nu(m_2)+\gamma_2=\nu(m_1)$,
		\end{itemize}
		\item if $\ell=2$ and $4\mid q+\eta$, then 
		one of the following holds,
		\begin{itemize}
			\item  $\nu(m_1)=1$, 
			\item $\nu(m_1)=0$ and $\alpha_1=1$,
			\item $\nu(m_1)=\alpha_1=0$ except ``$\nu(m_1)=\nu(m_2)=0$, $\alpha_2=\gamma_2=|\bc_2|=0$, $\det(\tR_i)=1$  for $i>2$".
			\end{itemize}
		\end{enumerate}
	\end{enumerate}	
	\item In cases (\ref{special-case-wei-ell3}), (\ref{special-case-wei-ell2-linear}) and (\ref{special-case-wei-ell2-uni}), there exists at most one $i_0$ such that $\nu(m_{i_0})\ne0$. 
\end{enumerate}
\end{prop}

\begin{proof}
As in  the proof of Lemma 	\ref{lem:char-def-small}, we
let $\tilde\theta=\prod_{i=1}^u \tilde\theta_i$ be an irreducible character of $\tC\tR/\tR$. Assume that $\tilde{\theta}$ is of defect $\le\nu(\tN/\tR N)$, where $\tilde\theta_i\in \Irr(\tC_i\tR_i/\tR_i)$.	
	
First suppose that we are not in one of the cases (\ref{special-case-wei-ell3}), (\ref{special-case-wei-ell2-linear}) and (\ref{special-case-wei-ell2-uni}).
Then by Lemma \ref{lem:char-def-small}, (1) (a) holds.
Using a similar argument as in the proof of Lemma \ref{lem:char-def-small}, we may assume that $\det(\tN)=\det(\tN_1)$ and $\det(\tR)=\det(\tR_1)$,
which imply that $\tN/\tR N\cong \tN_1/\tR_1 N_1$. Also $\nu(m_i)=0$ if $i>1$.
By Corollary \ref{can-char-4} again, $\tilde\theta_1$ has defect $\nu(\tN_1/\tR_1 N_1)$ and 
$\tilde\theta_i$ has defect $0$ for $i>1$.
Thus (c)-(e) of (1) follows by Lemma \ref{can-char-1} -- \ref{can-char-3'} and Proposition \ref{prop:special-odd-3}, \ref{prop:special-2-linear-3}, \ref{prop:special-2-unitary-3}.

Now suppose that we are in one of the cases (\ref{special-case-wei-ell3}), (\ref{special-case-wei-ell2-linear}) and (\ref{special-case-wei-ell2-uni}). 
We have  $\nu(\tN/\tR N)=1$.
By Proposition \ref{prop:special-odd-3}, \ref{prop:special-2-linear-3} and \ref{prop:special-2-unitary-3}, $\tR$ is special.
If there exist two $i_1\ne i_2$ such that $\nu(m_{i_1})\ne 0$ and $\nu(m_{i_2})\ne 0$,
then by Corollary \ref{can-char-4} again,
$\tC\tR/\tR$ has no irreducible character of defect $0$, which contradicts the assumption.
 This completes the proof.
\end{proof}

Let $(R,\varphi)$ be a weight of $G$, and $\tR$ a special radical subgroup of $\tG$ such that $R=\tR\cap G$.
Then by	Remark \ref{rem:5.2}, we may assume that $\tR$ is  listed in Proposition \ref{sepcial-defect}.

\subsection{Construction of weights.} \label{sec-const-wei}

Recall that $\cF'$ denotes the subset of $\cF$ consisting of those polynomials whose roots are of $\ell'$-order.
For $\Gamma\in\cF'$,  by \cite[(3.2)]{Brou86},
there is a unique block $\tB_\Gamma$ of $\tG_\Gamma= \tG
_{m_\Gamma,\alpha_\Gamma}$
with $\tR_\Gamma= \tR
_{m_\Gamma,\alpha_\Gamma}$ being a defect group. 
This block $\tB_\Gamma$ is $\cE_\ell(\tG_\Gamma,(\Gamma))$.
Let $\tR_\Gamma^{(k)}= \tR
_{m_\Gamma\ell^k,\alpha_\Gamma-k}$ for $0\le k\le k_\Gamma$, where $k_\Gamma= \min\{\nu(q-\eta),\alpha_\Gamma\}$ if ``$\ell$ is odd" or ``$\ell=2$ and $4\mid q-\eta$",
$k_\Gamma=1$ if $\ell=2$, $4\mid q+\eta$ and $\alpha_\Gamma>1$, and $k_\Gamma=0$ if $\ell=2$, $4\mid q+\eta$ and $\alpha_\Gamma\le 1$.
Then $\tR_\Gamma^{(k)}\le_{\tG_\Gamma} \tR_\Gamma^{(k-1)}\le_{\tG_\Gamma} \cdots \le_{\tG_\Gamma} \tR_\Gamma^{(0)}=\tR_\Gamma$  for $1\le k\le k_\Gamma$.
Let $\tfb_\Gamma$ be a block of $\tC_\Gamma=C_{\tG_\Gamma}(\tR_\Gamma)$
with defect group $\tR_\Gamma$ which is a root block of $\tB_\Gamma$.
Then there exists a unique block $\tfb_\Gamma^{(k)}$ of 
$\tC_\Gamma^{(k)}=C_{\tG_\Gamma}(\tR_\Gamma^{(k)})$
such that 
for some $g\in\tG_\Gamma$, we have
$(1,\tB_\Gamma)\le (\tR_\Gamma^{(k)},\tfb_\Gamma^{(k)})^g \le (\tR_\Gamma,\tfb_\Gamma)$ 
for $0\le k\le k_\Gamma$.

It can be checked that:
(i) the characters of the block $\tfb_\Gamma^{(k)}$ whose kernels contain $\tR_\Gamma^{(k)}$
are of defect $k$ when viewed as characters of $\tC_\Gamma^{(k)}\tR_\Gamma^{(k)}/\tR_\Gamma^{(k)}$;
(ii) the characters of the block $\tfb_\Gamma^{(k)}$ containing $\tR_\Gamma^{(k)}$ in their kernels are of the form $\tilde\chi_{st}$
(as in \S \ref{basic-result}), where $s=(\Gamma)$ and $t\in\tR^{(k)}_\Gamma$ such that $o(t)\le \ell^k$;
(iii) the restrictions of these characters to $C^{(k)}_\Gamma=\tC^{(k)}_\Gamma\cap G_\Gamma$ are the same
(where $G_\Gamma=G_{m_\Gamma,\alpha_\Gamma}$).
We denote by $\tilde\theta_\Gamma^{(k)}:=\tilde\chi_{s}$
and $\tilde\theta_\Gamma^{(k)(t)}:=\tilde\chi_{st}$.
Note that since $\tC_\Gamma^{(k)}\cong\GL_{m_\Gamma \ell^k}((\eta q)^{\ell^{\alpha_\Gamma-k}})$,
we also regard $st$ as a semisimple element of the group $\GL_{m_\Gamma \ell^k}((\eta q)^{\ell^{\alpha_\Gamma-k}})$,
and the elementary divisor of $st$ means one of the inverse images under $\Phi_{\alpha_\Gamma-k}$
(i.e., polynomials in $\cF_{\alpha_\Gamma-k}$), where $\Phi_{\alpha_\Gamma-k}$ is defined as in (\ref{eq:Phi_alpha}).
Then $\tilde\theta_\Gamma:=\tilde\theta_\Gamma^{(0)}$ is the canonical character of $\tfb_\Gamma$.
Also $\tilde\theta_\Gamma^{(k)(t)}=\hat t \tilde\theta_\Gamma^{(k)}$.
Let $\tN_\Gamma^{(k)}=N_{\tG_\Gamma}(\tR_\Gamma^{(k)})$.
Then $\tN_\Gamma^{(k)}(\tilde\theta_\Gamma^{(k)(t)})=\tC_\Gamma^{(k)}$.
Note that the notation $\tfb_\Gamma$ and $\tilde\theta_\Gamma$ coincide with those in \S\ref{subsect:weights-general}.
We let $\tsC^{(k)(t)}_{\Gamma,0}=\{ \tilde\theta_\Gamma^{(k)(t)} \}$   
for $0\le k\le k_\Gamma$ and $o(t)\le \ell^{k}$.
Note that $\tilde\theta_\Gamma^{(k)(t)}$ is of defect $k$ as an irreducible character of $\tN_\Gamma^{(k)}(\tilde\theta_\Gamma^{(k)(t)})/\tR_\Gamma^{(k)}$.
If $t=1$, we abbreviate $\tilde\theta_\Gamma^{(k)(t)}$ and $\tsC^{(k)(t)}_{\Gamma,0}$ as $\tilde\theta_\Gamma^{(k)}$ and $\tsC^{(k)}_{\Gamma,0}$ respectively.

Let $\Gamma_1,\Gamma_2\in\cF'$ satisfy that
$\alpha_{\Gamma_1}=\alpha_{\Gamma_2}=0$ and $\ell^{\nu(q-\eta)}\mid m_{\Gamma_1,\Gamma_2}:=d_{\Gamma_1}+d_{\Gamma_2}$.
Then there exists a block $\tB_{\Gamma_1,\Gamma_2}$ of $\tG_{\Gamma_1,\Gamma_2}=\GL_{m_{\Gamma_1,\Gamma_2}}(\eta q)$ with defect group $\tR_{\Gamma_1}\times \tR_{\Gamma_2}=\tR_{d_{\Gamma_1},0}\times \tR_{d_{\Gamma_2},0}$.
This block $\tB_{\Gamma_1,\Gamma_2}$ is $\cE_\ell(\tG_{\Gamma_1,\Gamma_2},s)$,
where $s=\diag((\Gamma_1),(\Gamma_2))$.
Let $\tR_{\Gamma_1,\Gamma_2}=\tR_{m_{\Gamma_1,\Gamma_2},0}=Z(\tG_{\Gamma_1,\Gamma_2})$.
So $C_{\tG_{\Gamma_1,\Gamma_2}}(\tR_{\Gamma_1,\Gamma_2})=\tG_{\Gamma_1,\Gamma_2}$.
The characters of $\tB_{\Gamma_1,\Gamma_2}$ whose kernels contain $\tR_{\Gamma_1,\Gamma_2}$
is of defect $\nu(q-\eta)$ when viewed as characters of $\tG_{\Gamma_1,\Gamma_2}/\tR_{\Gamma_1,\Gamma_2}$.
On the other hand,
the characters of $\tB_{\Gamma_1,\Gamma_2}$ whose kernels contain $\tR_{\Gamma_1,\Gamma_2}$ are of form $\tilde\chi_{st}$,
where $t\in\tR_{\Gamma_1,\Gamma_2}$.
Denote $\tilde\theta_{\Gamma_1,\Gamma_2}^{(t)}=\tilde\chi_{st}$.
Also, the restrictions of these characters to $G_{\Gamma_1,\Gamma_2}=\SL_{m_{\Gamma_1,\Gamma_2}}(\eta q)$ are the same.
We let 
$\tsC^{(t)}_{\Gamma_1,\Gamma_2}=\{\tilde\theta^{(t)}_{\Gamma_1,\Gamma_2}\}$ for $t\in \tR_{\Gamma_1,\Gamma_2}$.
Note that $\ttheta^{(t)}_{\Gamma_1,\Gamma_2}$ is of defect $a$ as an irreducible character of $\tG_{\Gamma_1,\Gamma_2}/\tR_{\Gamma_1,\Gamma_2}$.
If $t=1$, we abbreviate $\tilde\theta_{\Gamma_1,\Gamma_2}^{(t)}$ and $\tsC^{(t)}_{\Gamma_1,\Gamma_2}$
as $\tilde\theta_{\Gamma_1,\Gamma_2}$ and $\tsC_{\Gamma_1,\Gamma_2}$ respectively.

\vspace{2ex}

Let $\gamma>0$. 
First assume that 
``$\ell$ is odd" or  ``$\ell=2$ and $4\mid q-\eta$" or ``$\ell=2$, $4\mid q+\eta$ and $\alpha_\Gamma>0$".
We denote $\tG_{\Gamma,\gamma}:=\tG_{m_\Gamma,\alpha_\Gamma,\gamma}$,
$G_{\Gamma,\gamma}:=G_{m_\Gamma,\alpha_\Gamma,\gamma}$,
$\tR_{\Gamma,\gamma}^{(k)}:=\tR_{m_\Gamma\ell^k,\alpha_\Gamma-k,\gamma}$,
$\tC_{\Gamma,\gamma}^{(k)}:=C_{\tG_{\Gamma,\gamma}}(\tR_{\Gamma,\gamma}^{(k)})$ and
$\tN_{\Gamma,\gamma}^{(k)}:=N_{\tG_{\Gamma,\gamma}}(\tR_{\Gamma,\gamma}^{(k)})$ for $0\le k\le k_\Gamma$, where
$k_{\Gamma,\gamma}=\min\{\alpha_\Gamma,\max\{ \nu(q-\eta)-\gamma,0\}\}$. Note that $k_\Gamma=0$ if $\ell=2$ and $4\mid q+\eta$.

Then $\tC_{\Gamma,\gamma}^{(k)}=\tC_{\Gamma}^{(k)}\otimes I_{\ell^\gamma}$.
Let $\tN_{\Gamma,\gamma}^{0,(k)}=\tR_{\Gamma,\gamma}^{(k)}\tC_{\Gamma,\gamma}^{(k)}\tM_{\Gamma,\gamma}^{(k)}$, where $\tM_{\Gamma,\gamma}^{(k)}=\tM_{m_\Gamma\ell^k,\alpha_\Gamma-k,\gamma}$ as in \S \ref{sect:weights-general-gp}.
Then $\tN_{\Gamma,\gamma}^{0,(k)}\unlhd \tN_{\Gamma,\gamma}^{(k)}$ and $\tM_{\Gamma,\gamma}^{(k)}\tR_{\Gamma,\gamma}^{(k)}/\tR_{\Gamma,\gamma}^{(k)}\cong \Sp_{2\gamma}(\ell)$. Also $\tN_{\Gamma,\gamma}^{0,(k)}/\tR_{\Gamma,\gamma}^{(k)}\cong \tR_{\Gamma,\gamma}^{(k)}\tC_{\Gamma,\gamma}^{(k)}/\tR_{\Gamma,\gamma}^{(k)}\times \Sp_{2\gamma}(\ell)$.
Note that we take $\tM_{m_\Gamma\ell^k,\alpha_\Gamma-k,\gamma}$ as in Remark \ref{not-cent} and \ref{M-for-not-central-product} if  (\ref{eq:special-case-2-linear-0}) holds for $\tR_{\Gamma,\gamma}^{(k)}$. 
We identify $t\in Z(\tC_{\Gamma}^{(k)})$ with its image in $Z(\tC_{\Gamma,\gamma}^{(k)})$.
We let $\tilde\theta_{\Gamma,\gamma}^{(k)(t)}:=\tilde\theta_\Gamma^{(k)(t)}\otimes I_{\ell^\gamma}$.
Then  $\tN_{\Gamma,\gamma}^{0,(k)}=\tN_{\Gamma,\gamma}^{(k)}(\tilde\theta_{\Gamma,\gamma}^{(k)(t)})$ if $o(t)\le\ell^k$ since $e_\Gamma=1$ (as we assume that $\ell\mid q-\eta$).
If $k=0$, then we abbreviate $\tR_{\Gamma,\gamma}^{(k)}$,
$\tC_{\Gamma,\gamma}^{(k)}$,
$\tN_{\Gamma,\gamma}^{(k)}$,
$\tN_{\Gamma,\gamma}^{0,(k)}$,
$\tM_{\Gamma,\gamma}^{(k)}$
as
$\tR_{\Gamma,\gamma}$,
$\tC_{\Gamma,\gamma}$,
$\tN_{\Gamma,\gamma}$, $\tN_{\Gamma,\gamma}^{0}$,
$\tM_{\Gamma,\gamma}$ respectively.

Assume that neither (\ref{eq:special-case-3-1}) nor (\ref{eq:special-case-2-linear-1}) holds for $\tR_{m_\Gamma,\alpha_\Gamma,\gamma}$. 
Let $\tilde\psi_{\Gamma,\gamma}^{(k)(t)}:=\tilde\theta_{\Gamma,\gamma}^{(k)(t)}\times \mathrm{St}_{2\gamma}$, where $\mathrm{St}_{2\gamma}$ denotes the Steinberg character of $\Sp_{2\gamma}(\ell)$.
Then $\tilde\psi_{\Gamma,\gamma}^{(k)(t)}$ is of defect $k$ when viewed as a character of $\tN_{\Gamma,\gamma}^{0,(k)}/\tR_{\Gamma,\gamma}^{(k)}$.
We let $\tsC^{(k)(t)}_{\Gamma,\gamma}=\{\tilde\psi_{\Gamma,\gamma}^{(k)(t)}\}$
for $0\le k\le k_\Gamma$ and $o(t)\le \ell^{k}$. Then $\tsC^{(k)(t)}_{\Gamma,\gamma}$ is the set of irreducible characters of $\tN_{\Gamma,\gamma}^{(k)}(\tilde\theta_{\Gamma,\gamma}^{(k)(t)})$ lying over $\tilde\theta_{\Gamma,\gamma}^{(k)(t)}$ of defect $k$.
When $t=1$, we abbreviate $\tilde\theta_{\Gamma,\gamma}^{(k)(t)}$, $\tilde\psi_{\Gamma,\gamma}^{(k)(t)}$ and
$\tsC^{(k)(t)}_{\Gamma,\gamma}$ as 
$\tilde\theta_{\Gamma,\gamma}^{(k)}$, $\tilde\psi_{\Gamma,\gamma}^{(k)}$ and
$\tsC^{(k)}_{\Gamma,\gamma}$
respectively.

Now assume that 
(\ref{eq:special-case-3-1}) holds for $\tR_{m_\Gamma,\alpha_\Gamma,\gamma}$, which means $\ell=3$, $a=1$, $\alpha_\Gamma=0$, $\gamma=1$.
Then  $\tN_{\Gamma,1}=\tN_{\Gamma,1}^{0}=\tR_{\Gamma,1}\tC_{\Gamma,1}\tM_{\Gamma,1}$ and $\tilde{\theta}_{\Gamma,1}$ is $\tN_{\Gamma,1}$-invariant.
Let 
$N_{\Gamma,1}=\tN_{\Gamma,1}\cap G_{\Gamma,1}$,
$C_{\Gamma,1}=\tC_{\Gamma,1}\cap G_{\Gamma,1}$,
$R_{\Gamma,1}=\tR_{\Gamma,1}\cap G_{\Gamma,1}$,
$M_{\Gamma,1}=\tM_{\Gamma,1}\cap N_{\Gamma,1}$.
Then we have 
$\tR_{\Gamma,1}=R_{\Gamma,1}$,
$\tC_{\Gamma,1}/C_{\Gamma,1}\cong\mrO_{3'}(\tG_{\Gamma,1}/G_{\Gamma,1})$, 
$|\tM_{\Gamma,1}/M_{\Gamma,1}|=3$ and
$\nu(\tN_{\Gamma,1}/\tR_{\Gamma,1}N_{\Gamma,1})=1$
by \S \ref{subsect:srs-odd}.
Moreover, 
$\tM_{\Gamma,1}\tR_{\Gamma,1}/\tR_{\Gamma,1}\cong \SL_2(3)$.
We let $\tilde\psi_{\Gamma,1}=\tilde\theta_{\Gamma,1}\times \chi_3$, and 
$\tilde\psi_{\Gamma,1}^{(kj)}=\tilde\theta_{\Gamma,1}\times \chi_{kj}$, where $k=1,2$ and $j\in\{a,b,c\}$ (for the characters $\chi_{kj}$ see Remark \ref{special-case-ell3}).
Let $z\in Z(\tG_{\Gamma,1})$ be of order $3$. Then by Remark \ref{special-case-ell3}, 
$\hat z\tilde\psi_{\Gamma,1}=\tilde\psi_{\Gamma,1}$ and the group $\langle \hat z\rangle$
acts faithfully on the sets $\{\tilde\psi_{\Gamma,1}^{(1a)},\tilde\psi_{\Gamma,1}^{(1b)},\tilde\psi_{\Gamma,1}^{(1c)}\}$ and $\{\tilde\psi_{\Gamma,1}^{(2a)},\tilde\psi_{\Gamma,1}^{(2b)},\tilde\psi_{\Gamma,1}^{(2c)}\}$.
We let 
$\tsC^{(0)}_{\Gamma,1}=\{\tilde\psi_{\Gamma,1}\}$ and
$\tsC^{(1)}_{\Gamma,1}=\{\tilde\psi_{\Gamma,1}^{(kj)}\mid k\in\{1,2\}, j\in\{ a,b,c\}\}$.
Then $\tsC^{(k)}_{\Gamma,1}$  is the set of irreducible character of $\tN_{\Gamma,1}/\tR_{\Gamma,1}$
over $\tilde\theta_{\Gamma,1}$
of defect $k$, for $k=0,1$.

Assume that (\ref{eq:special-case-2-linear-1}) holds for $\tR_{m_\Gamma,\alpha_\Gamma,\gamma}$, which means $\ell=2$, $4\mid q-\eta$, $a=2$, $\alpha_\Gamma=0$ and $\gamma\in\{1,2\}$.
This is similar with the case (\ref{eq:special-case-3-1}) above.
Then $\tN_{\Gamma,\gamma}=\tN_{\Gamma,\gamma}^{0}=\tR_{\Gamma,\gamma}\tC_{\Gamma,\gamma}\tM_{\Gamma,\gamma}$
with $|\det(\tM_{\Gamma,\gamma})|=2^{\gamma-1}$ and $\tilde{\theta}_{\Gamma,\gamma}$ is $\tN_{\Gamma,\gamma}$-invariant.
Let 
$N_{\Gamma,\gamma}=\tN_{\Gamma,\gamma}\cap G_{\Gamma,\gamma}$.
Then $\nu(\tN_{\Gamma,\gamma}/\tR_{\Gamma,\gamma}N_{\Gamma,\gamma})=1$.
Moreover,
$\tM_{\Gamma,\gamma}\tR_{\Gamma,\gamma}/\tR_{\Gamma,\gamma}\cong \fS_{3\gamma}$.
We use the notation in Remark \ref{special-case-ell2-linear} for the characters of $\fS_{3\gamma}$ for $\gamma=1$ or $2$.
Let $\gamma=1$.
We let $\tilde\psi_{\Gamma,1}=\tilde\theta_{\Gamma,1}\times \chi_2$,
$\tilde\psi_{\Gamma,1}^{(1)}=\tilde\theta_{\Gamma,1}\times \chi_{1}$, and 
$\tilde\psi_{\Gamma,1}^{(-1)}=\tilde\theta_{\Gamma,1}\times \chi_{\mathrm{sign}}$.
Let $z\in Z(\tG_{\Gamma,1})$ be of order $2$. Then by Remark \ref{special-case-ell2-linear}, 
$\hat z \tilde\psi_{\Gamma,1}=\tilde\psi_{\Gamma,1}$ and 
$\hat z \tilde\psi_{\Gamma,1}^{(1)}=\tilde\psi_{\Gamma,1}^{(-1)}$.
We let 
$\tsC^{(0)}_{\Gamma,1}=\{\tilde\psi_{\Gamma,1}\}$ and
$\tsC^{(1)}_{\Gamma,1}=\{\tilde\psi_{\Gamma,1}^{(1)},\tilde\psi_{\Gamma,1}^{(-1)}\}$.
Then $\tsC^{(k)}_{\Gamma,1}$  is the set of irreducible characters of $\tN_{\Gamma,1}/\tR_{\Gamma,1}$
lying over $\tilde\theta_{\Gamma,1}$ of defect $k$, for $k=0,1$. 
Now let $\gamma=2$.
We let $\tilde\psi_{\Gamma,2}=\tilde\theta_{\Gamma,2}\times \chi_{11}$.
Let $z\in Z(\tG_{\Gamma,2})$ be of order $2$. Then by Remark \ref{special-case-ell2-linear}, 
$\hat z \tilde\psi_{\Gamma,2}=\tilde\psi_{\Gamma,2}$.
Then $\tsC_{\Gamma,2}:=\{\tilde\psi_{\Gamma,2}\}$  is the set of irreducible character of $\tN_{\Gamma,2}/\tR_{\Gamma,2}$
lying over $\tilde\theta_{\Gamma,2}$ of defect $0$.
On the other hand, $\tN_{\Gamma,2}/\tR_{\Gamma,2}$ does not have a irreducible character of defect $1$.

Now let $\ell=2$, $4\mid q+\eta$, $\gamma>0$ and $\alpha_\Gamma=0$.
For basic subgroups $\tS_{m_\Gamma,1,\gamma}$ and $\tR^\pm_{m_\Gamma,0,\gamma}$, by Lemma \ref{can-char-2'},  irreducible characters with positive defect only occur in the case (\ref{special-case-2-uni-1}). 
We use the notation in Remark \ref{exc-case-2uni}.

First let $a=2$ and consider the basic subgroup
$\tR^-_{m_\Gamma,0,1}$.
Let $\tilde{\theta}^-_{\Gamma,1}:=\tilde{\theta}_\Gamma\otimes I_{2}$ be defined similar with the character $\tilde{\theta}_\Gamma$ as above.
Then $\tN^-_{m_\Gamma,0,1}=\tR^-_{m_\Gamma,0,1}\tC^-_{m_\Gamma,0,1}\tM^-_{m_\Gamma,0,1}$ and $\tilde{\theta}^-_{\Gamma,1}$ is $\tN^-_{m_\Gamma,0,1}$-invariant.
Moreover, $\nu(\tN^-_{m_\Gamma,0,1}/\tR^-_{m_\Gamma,0,1}N^-_{m_\Gamma,0,1})=1$,
$\tM^-_{m_\Gamma,0,1}\tR^-_{m_\Gamma,0,1}/\tR^-_{m_\Gamma,0,1}\cong \fS_{3}$ and $$\tN^-_{m_\Gamma,0,1}/\tR^-_{m_\Gamma,0,1}\cong \tC^-_{m_\Gamma,0,1}\tR^-_{m_\Gamma,0,1}/\tR^-_{m_\Gamma,0,1}\times \fS_3.$$
We let $\tilde\psi^-_{\Gamma,1}=\tilde\theta^-_{\Gamma,1}\times \chi_2$,
$\tilde\psi_{\Gamma,1}^{-,(1)}=\tilde\theta^-_{\Gamma,1}\times \chi_{1}$, and 
$\tilde\psi_{\Gamma,1}^{-,(-1)}=\tilde\theta_{\Gamma,1}^-\times \chi_{\mathrm{sign}}$.
Let $z\in Z(\tG_{\Gamma,1})$ be of order $2$. Then by Remark \ref{exc-case-2uni}, 
$\hat z \tilde\psi^-_{\Gamma,1}=\tilde\psi^-_{\Gamma,1}$ and 
$\hat z \tilde\psi_{\Gamma,1}^{-,(1)}=\tilde\psi_{\Gamma,1}^{-,(-1)}$
We let 
$\tsC^{-,(0)}_{\Gamma,1}=\{\tilde\psi^-_{\Gamma,1}\}$ and
$\tsC^{-,(1)}_{\Gamma,1}=\{\tilde\psi_{\Gamma,1}^{-,(1)},\tilde\psi_{\Gamma,1}^{-,(-1)}\}$.
Then $\tsC^{-,(k)}_{\Gamma,1}$  is the set of irreducible characters of $\tN^-_{m_\Gamma,0,1}/\tR^-_{m_\Gamma,0,1}$
lying over $\tilde\theta^-_{\Gamma,1}$ of defect $k$, for $k=0,1$.

Now consider $R^+_{m_\Gamma,0,2}$.
As above, let $\tilde{\theta}^+_{\Gamma,2}:=\tilde{\theta}_\Gamma\otimes I_{2^2}$.
Then $\tN^+_{m_\Gamma,0,2}=\tR^+_{m_\Gamma,0,2}\tC^+_{m_\Gamma,0,2}\tM^+_{m_\Gamma,0,2}$ and $\tilde{\theta}^+_{\Gamma,2}$ is $\tN^+_{m_\Gamma,0,2}$+invariant.
Moreover, $\nu(\tN^+_{m_\Gamma,0,2}/\tR^+_{m_\Gamma,0,2}N^+_{m_\Gamma,0,2})=1$,
$\tM^+_{m_\Gamma,0,2}\tR^+_{m_\Gamma,0,2}/\tR^+_{m_\Gamma,0,2}\cong \GO^+_4(2)$ and $\tN^+_{m_\Gamma,0,2}/\tR^+_{m_\Gamma,0,2}\cong \tC^+_{m_\Gamma,0,2}\tR^+_{m_\Gamma,0,2}/\tR^+_{m_\Gamma,0,2}\times \GO^+_4(2)$.
For the characters of $\GO^+_4(2)$, we use the notation in Remark \ref{exc-case-2uni}.
We let $\tilde\psi^{+,(1)}_{\Gamma,2}=\tilde\theta^+_{\Gamma,2}\times \chi_1$ and 
$\tilde\psi^{+,(2)}_{\Gamma,2}=\tilde\theta^+_{\Gamma,2}\times \chi_2$.
Let $z\in Z(\tG_{\Gamma,1})$ be of order $2$. Then by Remark \ref{exc-case-2uni},
$\hat z \tilde\psi^{+,(1)}_{\Gamma,2}=\tilde\psi^{+,(2)}_{\Gamma,2}$.
We let 
$\tsC^{+,(1)}_{\Gamma,2}=\{\tilde\psi_{\Gamma,2}^{+,(1)},\tilde\psi_{\Gamma,2}^{+,(2)}\}$.
Then $\tsC^{+,(1)}_{\Gamma,2}$  is the set of irreducible characters of $\tN^+_{m_\Gamma,0,2}/\tR^+_{m_\Gamma,0,2}$
lying over $\tilde\theta^+_{\Gamma,2}$ of defect $1$.
Note that $\tN^+_{m_\Gamma,0,2}/\tR^+_{m_\Gamma,0,2}$ does not have defect $0$ characters.

\vspace{2ex}

We recall the set
$\tsC_{\Gamma,\gamma,\bc}:=\dz(\tN_{\Gamma,\gamma,\bc}(\ttheta_{\Gamma,\gamma,\bc})/\tR_{\Gamma,\gamma,\bc}\mid\ttheta_{\Gamma,\gamma,\bc})$ 
or 
$\dz(\tN_{\Gamma,\gamma,\bc}(\ttheta_{\Gamma,\gamma,\bc})/\tD_{\Gamma,\gamma,\bc}\mid\ttheta_{\Gamma,\gamma,\bc})$ 
as  in \S \ref{subsect:weights-general},
which was constructed in \cite{AF90,An92,An93, An94}.

\vspace{2ex}

Now we keep the notation and assumptions of Proposition \ref{sepcial-defect}.
Rewrite $\tR=\tR_1^{t_1}\times\cdots\times\tR_u^{t_1}$ with $\tR_i$'s being basic subgroups such that $\tR_i$ and $\tR_j$ are not conjugate if $i\ne j$.
Recall that we assume $\det(\tR_1) \geq\cdots\geq \det(\tR_u)$.
Now we suppose that $R= \tR \cap G$ provides a weight of $G$, say $(R,\varphi)$.
Let $\tilde\varphi\in\Irr(\tN/\tR\mid \varphi)$.
Then  $\tilde\varphi$ is of defect $\le \nu(\tN/\tR N)$ by Lemma \ref{def-clifford}.
If $\nu(\tN/\tR N)=0$, then $(\tR,\tilde\varphi)$ is a weight of $\tG$ which was constructed in \S \ref{subsect:weights-general}.
So we assume that $d:=\nu(\tN/\tR N)>0$.
We also view $\tilde\varphi$ as a character of $\tN$.
Let $\tilde\theta\in\Irr(\tR\tC\mid \tilde\varphi)$.
Then $\tilde\theta$ is of defect $\le \nu(\tN/\tR N)$ when viewed as a character of $\tR\tC/\tR$.
Note that $\tC\tR=(\tC_1\tR_1)^{t_1}\times\cdots\times(\tC_u\tR_u)^{t_u}$.
We let 
\begin{equation}\label{exp-theta}
\tilde\theta=\prod_{i=1}^{u}\prod_{j=1}^{v_i} \tilde\theta_{ij}^{t_{ij}},
\addtocounter{thm}{1}\tag{\thethm}
\end{equation}
where $\tilde\theta_{ij}\in\Irr(\tC_i\tR_i/\tR_i)$ and $t_i=\sum_{j=1}^{v_i} t_{ij}$
such that
$\tilde\theta_{ij_1}\ne \tilde\theta_{ij_2}$ if $j_1\ne j_2$.
Then $\tN(\tilde\theta)/\tR=\prod_{i=1}^{u}\prod_{j=1}^{v_i} (\tN_i(\tilde\theta_{ij})/\tR_i)\wr \fS(t_{ij})$ since 
$(\tN_i(\tilde\theta_{ij})\wr \fS(t_{ij}))/\tR_i^{t_{ij}}\cong 
(\tN_i(\tilde\theta_{ij})/\tR_i)\wr \fS(t_{ij})$.
The irreducible character $\tilde\psi$ of $\tN(\tilde\theta)$ covering $\theta$
and satifies that $\tilde\varphi=\Ind_{\tN(\tilde\theta)}^{\tN}\tilde\psi$
 is then expressible as 
\begin{equation}\label{exp-psi}
\tilde\psi=\prod_{i=1}^{u}\prod_{j=1}^{v_i} \tilde\psi_{ij},
\addtocounter{thm}{1}\tag{\thethm}
\end{equation}
where $\tilde\psi_{ij}$ is an irreducible character of $(\tN_i(\tilde\theta_{ij})/\tR_i)\wr \fS(t_{ij})$ covering the character $\tilde\theta_{ij}^{t_{ij}}$
of the base subgroup.

Suppose that we are not in one of the cases (\ref{special-case-wei-ell3}), (\ref{special-case-wei-ell2-linear}) and (\ref{special-case-wei-ell2-uni}).
Then by Proposition \ref{sepcial-defect} (1), we may assume that $t_1=1$ and $\nu(m_i)=0$ for any $i>1$ without loss of generality.
Also $\bc_1=\zero$ since we assume $d:=\nu(\tN/\tR N)>0$.
We abbreviate $\tilde\theta_{1j}$,  $\tilde\psi_{1j}$
to 
$\tilde\theta_1$, $\tilde\psi_{1}$ respectively.
Also, $\tilde\theta$ is of defect $\nu(\tN/\tR N)$ when viewed as a character of $\tR\tC/\tR$ and then
$\tilde\varphi$ is of defect $\nu(\tN/\tR N)$.
These imply that $\tilde\theta_1$ is of defect $\nu(\tN_1/\tR_1 N_1)$ when viewed as a character of $\tR_1\tC_1/\tR_1$, and 
$\tilde\theta_{ij}$ is of defect $0$ when viewed as a character of $\tR_i\tC_i/\tR_i$ for $i>1$ and $1\le j\le v_i$.
Since $\tilde\varphi$ is of defect $\nu(\tN/\tR N)$, we know
$\tilde\psi$ is of defect $\nu(\tN/\tR N)$ when viewed as a character of $\tN(\tilde\theta)/\tR$ by Lemma \ref{def-clifford}.
From this, $\tilde\psi_1$ is of defect $\nu(\tN_1/\tR_1 N_1)$ and 
$\tilde\psi_{ij}$ is of defect $0$ 
as a character of $(\tN_i(\tilde\theta_{ij})/\tR_i)\wr \fS(t_{ij})$
for $i>1$, $1\le j\le v_i$.

For $i>1$,
let $\tilde\zeta_{ijk}$ be the irreducible characters of $\tN_i(\tilde\theta_{ij})$ covering $\tilde\theta_{ij}$, and having defect $0$ as characters of $\tN_i(\tilde\theta_{ij})/\tN_i$, where $k$ ranges over a set depending on $i$, $j$.
Then the construction of $\tilde\psi_{ij}$ can be found in the proof of \cite[(4C)]{AF90}, \cite[(3B)]{An92}, \cite[(3B)]{An93} and \cite[(4D)]{An94}
and we state as follows:
$\tilde\psi_{ij}=\Ind_{\tN_i(\tilde\theta_{ij})\wr \prod_k\fS(t_{ijk})}^{\tN_i(\tilde\theta_{ij})\wr \fS(t_{ij})} (\overline{\prod_k \tilde\zeta_{ijk}^{t_{ijk}}}\cdot \prod_{k=1} \phi_{\lambda_{ijk}})$,
where $t_{ij}=\sum_k t_{ijk}$, $\overline{\prod_k \tilde\zeta_{ijk}^{t_{ijk}}}$ is the extension of $\prod_k\tilde\zeta_{ijk}^{t_{ijk}}$ to 
$\tN_i(\tilde\theta_{ij})\wr \fS(t_{ijk})$, 
$\lambda_{ijk}\vdash t_{ijk}$ is an $\ell$-core,
and $\phi_{\lambda_{ijk}}$
is the character of $\fS(t_{ijk})$ corresponding to $\lambda_{ijk}$.

The pair $(\tR_1,\tilde\theta_1)$ is of type
$(\tR_{\Gamma}^{(d)},\tilde\theta_{\Gamma}^{(d)(t)})$, $(\tR_{\Gamma,\gamma}^{(d)},\tilde\theta_{\Gamma,\gamma}^{(d)(t)})$, or $(\tR_{\Gamma_1,\Gamma_2},\tilde\theta_{\Gamma_1,\Gamma_2}^{(t)})$.
Note that $(\tR_{\Gamma_1,\Gamma_2},\tilde\theta_{\Gamma_1,\Gamma_2}^{(t)})$ occurs only when $d=\nu(q-\eta)$.
Then $\tilde\psi_1$ is in 
$\tsC^{(d)(t)}_{\Gamma}$,
$\tsC^{(d)(t)}_{\Gamma,
	\gamma}$ or
$\tsC^{(t)}_{\Gamma_1,\Gamma_2}$ respectively.
If we change $t$ and keep other parametrizations, then we obtain a new $\tilde\theta$ and then 
a new $\tilde\psi$ and 
a new $\tilde\varphi$.
It can be checked that the restriction of $\tilde\varphi$ to $N$
remain unchanged.
Thus in order to classify the weights of $G$, it suffices to let $t=1$.

Now suppose that we are in the case (\ref{special-case-wei-ell3}). Then $\ell=3$ and $a=1$.
By Proposition \ref{sepcial-defect} (2),
we may assume that there exists at most one $i_0$ such that $\nu(m_{i_0})\ne0$. 
Let $I$ be the set of $i$ such that $\tR_i=\tR_{m,0,1}$ for some $3\nmid m$.
First let $\nu(m_{i})=0$ for all $i$.
Then $3\nmid |\det(\tN_i)|$ if $i\notin I$.
If $\tilde\varphi$ is of defect $0$, i.e., $(\tR,\tilde\varphi)$ is a weight of $\tG$,  then its construction is as in \S \ref{subsect:weights-general}.
We assume that $\tilde\varphi$ is of defect $1$, then $\tilde\psi$ is of defect $1$ too.
We write $\tilde\psi=\tilde\psi^{(1)}\times\tilde\psi^{(2)}$, where
$\tilde\psi^{(1)}=\prod_{i\in I}\prod_{j=1}^{v_i} \tilde\psi_{ij}$ and 
$\tilde\psi^{(2)}=\prod_{i\notin I}\prod_{j=1}^{v_i} \tilde\psi_{ij}$.

Suppose that $\tilde\psi^{(2)}$ is of defect $1$ as a character of
$\prod_{i\notin I}(\tN_i(\tilde\theta_{ij})/\tR_i)\wr \fS(t_{ij})$. Then $\tilde\psi^{(1)}$ is of defect $0$ as a character of
$\prod_{i\in I}(\tN_i(\tilde\theta_{ij})/\tR_i)\wr \fS(t_{ij})$ and its construction is as in \S \ref{subsect:weights-general}.
If $z\in\mrO_3(\fZ_{q-\eta})$, then by Remark \ref{special-case-ell3}, $\hat z \tilde\psi^{(1)}=\tilde\psi^{(1)}$.
Therefore,
$\hat z \tilde\varphi=\tilde\varphi$ since $3\nmid\det(N_i)$ for all $i$. 
Then $3\mid \kappa^{\tN}_{N}(\tilde\varphi)$ and then $\varphi$ is of defect $1$ as a character of $N/R$.
This is a contradiction.

Thus $\tilde\psi^{(2)}$ is of defect $0$ as a character of
$\prod_{i\notin I}(\tN_i(\tilde\theta_{ij})/\tR_i)\wr \fS(t_{ij})$,
then $\tilde\psi_{ij}$ is of defect $0$ 
as a character of $(\tN_i(\tilde\theta_{ij})/\tR_i)\wr \fS(t_{ij})$
for $i\notin I$, $1\le j\le v_i$.
On the other hand, $\tilde\psi^{(1)}$ is of defect $1$ as a character of
$\prod_{i\in I}(\tN_i(\tilde\theta_{ij})/\tR_i)\wr \fS(t_{ij})$.
This implies that there exist a unique $k$ such that 
$t_k=1$, and
$\tilde\psi_{k}:=\tilde\psi_{k1}$ is of defect $1$ 
as a character of $\tN_k/\tR_k$. 
The pair $(\tR_k,\tilde\theta_k)$ is of type $(\tR_{\Gamma,1},\tilde\theta_{\Gamma,1})$
and $\tpsi_k \in \tsC^{(1)}_{{\Gamma,1}}$.
In addition,
$\tilde\psi_{ij}$ is of defect $0$ 
as a character of $(\tN_i(\tilde\theta_{ij})/\tR_i)\wr \fS(t_{ij})$ for $i\notin \set{k}$, $1\le j\le v_i$.
For $1\ne z\in\mrO_3(\fZ_{q-\eta})$,  by Remark \ref{special-case-ell3}, we have
$\hat z \tilde\varphi\ne \tilde\varphi$.
Thus $3\nmid \kappa^{\tN}_{N}(\tilde\varphi)$ and then $\varphi$ is of defect $0$ as a character of $N/R$.

Now assume that
there exists a unique $i_0$ such that $\nu(m_{i_0})\ne0$.
Then by Corollary \ref{can-char-4}, $\psi^{(2)}$ is of defect $1$ and then $\psi^{(1)}$ is of defect $0$. 
This implies $\tilde\psi_{ij}$ is of defect $0$ 
as a character of $(\tN_i(\tilde\theta_{ij})/\tR_i)\wr \fS(t_{ij})$
for $i\ne i_0$, $1\le j\le v_i$.
Also, $t_{i_0}=1$ and the pair $(\tR_{i_0},\tilde\theta_{i_0})$ is of type
$(\tR_{\Gamma}^{(1)},\tilde\theta_{\Gamma}^{(1)(t)})$, $(\tR_{\Gamma,\gamma}^{(1)},\tilde\theta_{\Gamma,\gamma}^{(1)(t)})$, or $(\tR_{\Gamma_1,\Gamma_2},\tilde\theta_{\Gamma_1,\Gamma_2}^{(t)})$.
Then $\tilde\psi_{i_0}$ is in 
$\tsC^{(1)(t)}_{\Gamma}$,
$\tsC^{(1)(t)}_{\Gamma,
	\gamma}$ or
$\tsC^{(t)}_{\Gamma_1,\Gamma_2}$ respectively.
Similar as the argument above, it suffices to let $t=1$.
For $1\ne z\in\mrO_3(\fZ_{q-\eta})$,  by Remark \ref{special-case-ell3}, we have
$\hat z \tilde\varphi\ne \tilde\varphi$.
Thus $3\nmid \kappa^{\tN}_{N}(\tilde\varphi)$ and then $\varphi$ is of defect $0$ as a character of $N/R$.

The  cases (\ref{special-case-wei-ell2-linear}) and (\ref{special-case-wei-ell2-uni}) are entirely similar with the case (\ref{special-case-wei-ell3}) above.

\subsection{Classification of weights of $\SL_n(q)$ and $\SU_n(q)$.}\label{Para-of-wei}

Now we give a classification of weights of $G=\SL_n(\eta q)$.
We consider not only $\ell\mid q-\eta$ but all prime $\ell$ from now on.
First we define a set $\mathrm{Alp}'(\tG)$,
with the labeling set  $i\mathrm{Alp}(\tG)$ (defined as in (\S \ref{subsect:weights-general})), 
consisting of some $\tG$-conjugacy classes of
pairs $(\tR,\tilde\varphi)$, where $\tR$ is a special radical subgroup of $\tG$,
$\tilde\varphi\in\Irr(N_{\tG}(\tR)/\tR)$ such that 
$(R,\varphi)$ is a weight of $G$
for $R=\tR\cap G$ and
$\varphi\in\Irr(N_G(R)\mid \tilde\varphi)$. 

We start from $\Alp(\tG)$, the set of $\tG$-conjugacy classes weights of $\tG$ with the labeling set  $i\mathrm{Alp}(\tG)$.
For $(\tR,\tilde\varphi)\in \Alp(\tG)$, if $\tR$ is not special (then $\ell\mid \gcd(n,q-\eta)$ by \cite[Lemma~5.3]{Feng19}),
then we make the change as follows.

First assume ``$\ell$ is odd" or ``$\ell=2$ and $4\mid q-\eta$". 
Suppose that $\tR=\tR_1^{t_1}\times \tR_2^{t_2} \times \cdots\times \tR_u^{t_u}$ is a direct product of basic subgroups with 
$\tR_i=\tR_{m_{i},\alpha_i,\gamma_i,\bc_i}$
and $a(\tR_1)\le\cdots\le a(\tR_u)$.
Of course, $\nu(m_i)=0$ and $a(\tR_i)=\gamma_i$ for every $1\le i\le u$ since $\tR$ provides some weight of $\tG$ (see \S \ref{subsect:weights-general}).
We will use the notation in \S \ref{subsect:weights-general} and \S \ref{sec-const-wei}.
Note that $\tR=\tR_+$, $\tN=\tN_+$, $\tC=\tC_+$, $\tilde\varphi=\tilde\varphi_+$, $\tilde\theta=\tilde\theta_+$ and $\tilde\psi=\tilde\psi_+$.
Also, 
 $\tilde\theta=\prod_{i}\prod_{j} \tilde\theta_{ij}^{t_{ij}}$  similar as  (\ref{exp-theta}) and
$\tilde\psi=\prod_{i}\prod_{j} \tilde\psi_{ij}$ similar as  (\ref{exp-psi}), where $ \tilde\psi_{ij}$
is a character of form (\ref{equation:tpsi}).
Since $\tR$ is not special,
by Proposition \ref{prop:special-odd-3} and \ref{prop:special-2-linear-3}, exactly one of the following cases occurs.

\emph{Case (i)}.~ $t_1=1$, $\bc_1=\zero$, $\gamma_1<a$, $\alpha_1>0$ and $\gamma_i>\gamma_1$ (i.e., $a(\tR_i)>a(\tR_1)$) if $i>1$.
Then we replace $\tR$ by $\tR'=\tR'_1\times \prod_{i>1} R_i^{t_i}$ with
$\tR'_1=\tR_{m'_1,\alpha'_1,\gamma_1}$, where $m'_1=m_1\ell^{\min\{ \alpha_1,a-\gamma_1, \gamma_2-\gamma_1\}}$ and $\alpha'_1=\alpha_1-\min\{ \alpha_1,a-\gamma_1, \gamma_2-\gamma_1\}$.
We abbreviate $\tilde\theta_{11}$, $\tilde\psi_{11}$
as
$\tilde\theta_{1}$, $\tilde\psi_{1}$ and then 
$\tilde\theta=\tilde\theta_{1}\times \prod_{i>1}\prod_{j} \tilde\theta_{ij}^{t_{ij}}$ and
$\tilde\psi= \tilde\psi_1\times \prod_{i>2}\prod_{j} \tilde\psi_{ij}$.
Suppose that 
$\tilde\theta_1=\tilde\theta_{\Gamma,\gamma_1}$,
$\tilde\psi_1=\tilde\psi_{\Gamma,\gamma_1}$.
Then we replace $\tilde\theta$, $\tilde\psi$, $\tilde\varphi$
by $\tilde\theta'=\tilde\theta'_{1}\times \prod_{i>1}\prod_{j} \tilde\theta_{ij}^{t_{ij}}$,
$\tilde\psi'= \tilde\psi'_1\times \prod_{i>1}\prod_{j} \tilde\psi_{ij}$,
$\tilde\varphi'=\Ind^{\tN}_{\tN(\tilde\theta')}\tilde\psi'$ 
respectively,
where $\tilde\theta'_1=\tilde\theta_{\Gamma,\gamma}^{(\min\{ \alpha_1,a-\gamma_1, \gamma_2-\gamma_1\})}$,
$\tilde\psi'_1=\tilde\psi_{\Gamma,\gamma}^{(\min\{ \alpha_1,a-\gamma_1, \gamma_2-\gamma_1\})}$.

\emph{Case (ii)}.~ $t_1=t_2=1$, $\tR_1=\tR_{m_1,0}$, $\tR_2=\tR_{m_2,0}$, $\nu(m_1+m_2)\ge a$ and
$\gamma_i\ge a$ if $i>2$.
Then we replace $\tR$ by $\tR'=\tR'_1\times \prod_{i>2} R_i^{t_i}$ with $\tR'_1=\tR_{m_1+m_2,0}$.
We abbreviate  $\tilde\theta_{i1}$, $\tilde\psi_{i1}$
as
$\tilde\theta_{i}$, $\tilde\psi_{i}$ for $i=1,2$,
then 
$\tilde\theta=\tilde\theta_{1}\times \tilde\theta_{2}\times \prod_{i>2}\prod_{j} \tilde\theta_{ij}^{t_{ij}}$ and
$\tilde\psi= \tilde\psi_1\times \tilde\psi_2\times \prod_{i>2}\prod_{j} \tilde\psi_{ij}$.
Suppose that 
$\tilde\theta_i=\tilde\theta_{\Gamma_i}$,
$\tilde\psi_i=\tilde\psi_{\Gamma_i}$  for $i=1,2$.
Then we replace $\tilde\theta$, $\tilde\psi$, $\tilde\varphi$
by $\tilde\theta'=\tilde\theta'_{1}\times\prod_{i>2}\prod_{j} \tilde\theta_{ij}^{t_{ij}}$,
$\tilde\psi'= \tilde\psi'_1\times \prod_{i>2}\prod_{j} \tilde\psi_{ij}$,
$\tilde\varphi'=\Ind^{\tN}_{\tN(\tilde\theta')}\tilde\psi'$ 
respectively,
where $\tilde\psi'_1=\tilde\theta'_1=\tilde\theta_{\Gamma_1,\Gamma_2}$.

\emph{Case (iii)}.~ $\ell=3$, $a=1$, $t_1=1$, $\tR_1=\tR_{m_1,0,0,\one}$ and $\gamma_i>0$ if $i>1$.
Then we replace $\tR$ by $\tR'=\tR'_1\times \prod_{i>1} R_i^{t_i}$ with
$\tR'_1=\tR_{m_1,0,1}$.
We abbreviate  $\tilde\theta_{11}$, $\tilde\psi_{11}$
as
$\tilde\theta_{1}$, $\tilde\psi_{1}$,
then 
$\tilde\theta=\tilde\theta_{1}\times \prod_{i>1}\prod_{j} \tilde\theta_{ij}^{t_{ij}}$ and
$\tilde\psi= \tilde\psi_1\times \prod_{i>1}\prod_{j} \tilde\psi_{ij}$.
Let
$\tilde\theta_1=\tilde\theta_{\Gamma}\otimes I_3$.
Assume that $\zeta_1$ and $\zeta_2$ are the irreducible characters of $\GL_1(3).$
Then $\tilde\psi_1=\tilde\theta_{1}\times\zeta_k$ for $k\in\{1,2\}$.
We fix $\tilde\theta$, and
replace  $\tilde\psi$, $\tilde\varphi$
by $\tilde\psi'= \tilde\psi'_1\times \prod_{i>1}\prod_{j} \tilde\psi_{ij}$,
$\tilde\varphi'=\Ind^{\tN}_{\tN(\tilde\theta)}\tilde\psi'$ 
respectively,
where $\tilde\psi'_1=\tilde\psi_{\Gamma,1}^{(ka)}$.

\emph{Case (iv)}.~
$\ell=2$, $a=2$, $t_1=1$, $\tR_1=\tR_{m_1,0,0,\one}$ and $\gamma_i>0$ if $i>1$.
Then we replace $\tR$ by $\tR'=\tR'_1\times \prod_{i>1} R_i^{t_i}$ with
$\tR'_1=\tR_{m_1,0,1}$.
We abbreviate  $\tilde\theta_{11}$, $\tilde\psi_{11}$
as 
$\tilde\theta_{1}$, $\tilde\psi_{1}$ and then 
$\tilde\theta=\tilde\theta_{1}\times \prod_{i>1}\prod_{j} \tilde\theta_{ij}^{t_{ij}}$ and
$\tilde\psi= \tilde\psi_1\times \prod_{i>1}\prod_{j} \tilde\psi_{ij}$.
Let
$\tilde\theta_1=\tilde\theta_{\Gamma}\otimes I_2$.
We fix $\tilde\theta$, and
replace  $\tilde\psi$, $\tilde\varphi$
by $\tilde\psi'= \tilde\psi'_1\times \prod_{i>1}\prod_{j} \tilde\psi_{ij}$,
$\tilde\varphi'=\Ind^{\tN}_{\tN(\tilde\theta)}\tilde\psi'$ 
respectively,
where $\tilde\psi'_1=\tilde\psi_{\Gamma,1}^{(1)}$.

Now let $\ell=2$ and $4\mid q+\eta$.
Suppose that $\tR=\tR_1^{t_1}\times \tR_2^{t_2} \times \cdots\times \tR_u^{t_u}$ is a direct product of basic subgroups with 
$\tR_i=\tD_{m_{i},\alpha_i,\gamma_i,\bc_i}$
and $\det(\tR_1)\ge\cdots\ge \det(\tR_u)$.
Of course, $\nu(m_i)=0$ for every $1\le i\le u$.
We also use the notation in \S \ref{subsect:weights-general} and \S \ref{sec-const-wei}.
Note that $\tR=\tR_+$, $\tN=\tN_+$, $\tC=\tC_+$, $\tilde\varphi=\tilde\varphi_+$, $\tilde\theta=\tilde\theta_+$ and $\tilde\psi=\tilde\psi_+$.
Also, 
$\tilde\theta=\prod_{i}\prod_{j} \tilde\theta_{ij}^{t_{ij}}$  and
$\tilde\psi=\prod_{i}\prod_{j} \tilde\psi_{ij}$ similar as  above.
since $\tR$ is not special, by Proposition \ref{prop:special-2-unitary-3}, exactly one of the following cases occurs.

\emph{Case (v)}.~
$t_1=1$, 
$\tR_1=\tD_{m_1,\alpha_1,0,\zero}=\tR_{m_1,\alpha_1}$ (with $\alpha_1>1$) and $\det(\tR_i)=1$ for $i>1$.
Then we replace $\tR$ by $\tR'=\tR'_1\times \prod_{i>1} R_i^{t_i}$ with
$\tR'_1=\tR_{m'_1,\alpha'_1}$, where $m'_1=2m_1$ and $\alpha'_1=\alpha_1-1$.
We abbreviate  $\tilde\theta_{11}$, $\tilde\psi_{11}$
as
$\tilde\theta_{1}$, $\tilde\psi_{1}$,
then 
$\tilde\theta=\tilde\theta_{1}\times \prod_{i>1}\prod_{j} \tilde\theta_{ij}^{t_{ij}}$ and
$\tilde\psi= \tilde\psi_1\times \prod_{i>2}\prod_{j} \tilde\psi_{ij}$.
Suppose that 
$\tilde\theta_1=\tilde\theta_{\Gamma}$,
$\tilde\psi_1=\tilde\psi_{\Gamma}$.
Then we replace $\tilde\theta$, $\tilde\psi$, $\tilde\varphi$
by $\tilde\theta'=\tilde\theta'_{1}\times \prod_{i>1}\prod_{j} \tilde\theta_{ij}^{t_{ij}}$,
$\tilde\psi'= \tilde\psi'_1\times \prod_{i>1}\prod_{j} \tilde\psi_{ij}$,
$\tilde\varphi'=\Ind^{\tN}_{\tN(\tilde\theta')}\tilde\psi'$ 
respectively,
where $\tilde\theta'_1=\tilde\theta_{\Gamma}^{(1)}$,
$\tilde\psi'_1=\tilde\psi_{\Gamma}^{(1)}$.

\emph{Case (vi)}.~
$t_1=t_2=1$, $\tR_1=\tD_{m_1,0,0,\zero}=\tR_{m_1,0}$, $\tR_2=\tD_{m_2,0,0,\zero}=\tR_{m_2,0}$  and
$\det(\tR_i)=1$ for $i>2$.
Then we replace $\tR$ by $\tR'=\tR'_1\times \prod_{i>2} R_i^{t_i}$ with $\tR'_1=\tR_{m_1+m_2,0}$.
We abbreviate  $\tilde\theta_{i1}$, $\tilde\psi_{i1}$
as
$\tilde\theta_{i}$, $\tilde\psi_{i}$ for $i=1,2$,
then 
$\tilde\theta=\tilde\theta_{1}\times \tilde\theta_{2}\times \prod_{i>2}\prod_{j} \tilde\theta_{ij}^{t_{ij}}$ and
$\tilde\psi= \tilde\psi_1\times \tilde\psi_2\times \prod_{i>2}\prod_{j} \tilde\psi_{ij}$.
Suppose that 
$\tilde\theta_i=\tilde\theta_{\Gamma_i}$,
$\tilde\psi_i=\tilde\psi_{\Gamma_i}$  for $i=1,2$.
Then we replace $\tilde\theta$, $\tilde\psi$, $\tilde\varphi$
by $\tilde\theta'=\tilde\theta'_{1}\times\prod_{i>2}\prod_{j} \tilde\theta_{ij}^{t_{ij}}$,
$\tilde\psi'= \tilde\psi'_1\times \prod_{i>2}\prod_{j} \tilde\psi_{ij}$,
$\tilde\varphi'=\Ind^{\tN}_{\tN(\tilde\theta')}\tilde\psi'$ 
respectively,
where $\tilde\psi'_1=\tilde\theta'_1=\tilde\theta_{\Gamma_1,\Gamma_2}$.

\emph{Case (vii)}.~
$t_1=2$, $\tR_1=\tD_{m_1,0,0,\zero}=\tR_{m_1,0}$, and
$\det(\tR_i)=1$ for $i>1$.
Then we replace $\tR$ by $\tR'=\tR'_1\times \prod_{i>1} R_i^{t_i}$ with $\tR'_1=\tR_{2m_1,0}$.
Note that
$\tilde\theta=\tilde\theta_{11}\times \tilde\theta_{12}\times \prod_{i>2}\prod_{j} \tilde\theta_{ij}^{t_{ij}}$ and
$\tilde\psi= \tilde\psi_{11}\times \tilde\psi_{12}\times \prod_{i>1}\prod_{j} \tilde\psi_{ij}$.
Suppose that 
$\tilde\theta_{1i}=\tilde\theta_{\Gamma_i}$,
$\tilde\psi_{1i}=\tilde\psi_{\Gamma_i}$  for $i=1,2$.
Then we replace $\tilde\theta$, $\tilde\psi$, $\tilde\varphi$
by $\tilde\theta'=\tilde\theta'_{1}\times\prod_{i>1}\prod_{j} \tilde\theta_{ij}^{t_{ij}}$,
$\tilde\psi'= \tilde\psi'_1\times \prod_{i>1}\prod_{j} \tilde\psi_{ij}$,
$\tilde\varphi'=\Ind^{\tN}_{\tN(\tilde\theta')}\tilde\psi'$ 
respectively,
where $\tilde\psi'_1=\tilde\theta'_1=\tilde\theta_{\Gamma_1,\Gamma_2}$.

\emph{Case (viii)}.~
$t_1=1$, $\tR_1=\tD_{m_1,0,0,\two}=\tR_{m_1,0,0,\two}$ and
$\det(\tR_i)=1$ for $i>1$.
Then we replace $\tR$ by $\tR'=\tR'_1\times \prod_{i>1} R_i^{t_i}$ with
$\tR'_1=\tR^+_{m_1,0,2}$.
We abbreviate  $\tilde\theta_{11}$, $\tilde\psi_{11}$
as 
$\tilde\theta_{1}$, $\tilde\psi_{1}$,
then 
$\tilde\theta=\tilde\theta_{1}\times \prod_{i>1}\prod_{j} \tilde\theta_{ij}^{t_{ij}}$ and
$\tilde\psi= \tilde\psi_1\times \prod_{i>1}\prod_{j} \tilde\psi_{ij}$.
Let
$\tilde\theta_1=\tilde\theta_{\Gamma}\otimes I_{2^2}$.
We fix $\tilde\theta$, and
replace  $\tilde\psi$, $\tilde\varphi$
by $\tilde\psi'= \tilde\psi'_1\times \prod_{i>1}\prod_{j} \tilde\psi_{ij}$,
$\tilde\varphi'=\Ind^{\tN}_{\tN(\tilde\theta)}\tilde\psi'$ 
respectively,
where $\tilde\psi'_1=\tilde\psi_{\Gamma,2}^{+,(1)}$.

\emph{Case (ix)}.~
$t_1=1$, $a=2$, $\tR_1=\tD_{m_1,0,0,\one}=\tS_{m_1,1}$
and
$\det(\tR_i)=1$ for $i>1$.
Then we replace $\tR$ by $\tR'=\tR'_1\times \prod_{i>1} R_i^{t_i}$ with
$\tR'_1=\tR^-_{m_1,0,1}$.
We abbreviate  $\tilde\theta_{11}$, $\tilde\psi_{11}$
as
$\tilde\theta_{1}$, $\tilde\psi_{1}$,
then 
$\tilde\theta=\tilde\theta_{1}\times \prod_{i>1}\prod_{j} \tilde\theta_{ij}^{t_{ij}}$ and
$\tilde\psi= \tilde\psi_1\times \prod_{i>1}\prod_{j} \tilde\psi_{ij}$.
Let
$\tilde\theta_1=\tilde\theta_{\Gamma}\otimes I_{2}$.
We fix $\tilde\theta$, and
replace  $\tilde\psi$, $\tilde\varphi$
by $\tilde\psi'= \tilde\psi'_1\times \prod_{i>1}\prod_{j} \tilde\psi_{ij}$,
$\tilde\varphi'=\Ind^{\tN}_{\tN(\tilde\theta)}\tilde\psi'$ 
respectively,
where $\tilde\psi'_1=\tilde\psi_{\Gamma,1}^{-,(1)}$.

\vspace{2ex}

Let $(s,\lambda,K)\in i\mathrm{Alp}(\tG)$ be the parametrization for $(\tR,\tilde\varphi)\in\Alp(\tG)$.
If $\tR$ is special, then we take $(\tR',\tilde\varphi')=(\tR,\tilde\varphi)$. If $\tR$ is not special, then we take $(\tR',\tilde\varphi')$ as in the above cases~(i.e., Cases (i)--(ix)).
 Also, we take $(s,\lambda,K)$ to be the parametrization for $(\tR',\tilde\varphi')$. 
Now we obtain a set $\mathrm{Alp}'(\tG)$ containing the pairs $(\tR',\tilde\varphi')$ and a bijection 
\begin{equation}\label{bij-wei}
\Xi:\mathrm{Alp}(\tG)\to\mathrm{Alp}'(\tG), \quad (\tR,\tilde\varphi)\mapsto (\tR',\tilde\varphi').
\addtocounter{thm}{1}\tag{\thethm}
\end{equation}
Then $\mathrm{Alp}'(\tG)$ has the labeling set  $i\mathrm{Alp}(\tG)$ and from this we classify the weights of $G$.

\begin{lem}
Let $(\tR,\tilde\varphi)\in\Alp(\tG)$ and $(\tR',\tilde\varphi')=\Xi(\tR,\tilde\varphi)$.
Then $\tR\cap G$ and $\tR'\cap G$ are $\tG$-conjugate.
\end{lem}

\begin{proof}
This follows from the arguments above case by case directly.
\end{proof}

\begin{prop}\label{classify-wei-slsu}
\begin{enumerate}[\rm(1)]
\item For $(\tR',\tilde\varphi')\in \mathrm{Alp}'(\tG)$, if we let $R=\tR'\cap G$ and $\varphi\in\Irr(N_G(R)\mid \tvarphi')$, then $(R,\varphi)$ is a weight of $G$.
\item For $(R,\varphi)\in \mathrm{Alp}(G) $, there exists $g\in\tG$ and $(\tR',\tvarphi')\in \mathrm{Alp}'(\tG)$ such that $R^g=\tR'\cap G$ and $\varphi^g\in\Irr(N_G(R)\mid \tvarphi')$
\end{enumerate}
\end{prop}

\begin{proof}
	This follows from the construction in \S \ref{sec-const-wei} when $\ell\mid q-\eta$ and follows by \cite[Rmk.~5.6]{Feng19} when $\ell\nmid \gcd(n,q-\eta)$.
\end{proof}

\begin{rmk}
Using the methods of \cite{LZ18} and \cite{Feng19}, we can consider the actions of  groups $D$ and $\Lin_{\ell'}(\tG/G)$ on $\Alp'(\tG)$.
As a consequence, $\Xi$ is $\Lin_{\ell'}(\tG/G)\rtimes D$-equivariant.
\end{rmk}

We mention that for Proposition \ref{classify-wei-slsu}, if none of (\ref{special-case-wei-ell3}) 
 (\ref{special-case-wei-ell2-linear}) and (\ref{special-case-wei-ell2-uni})
holds  for $\tR$ and $t\in\mrO_\ell(\tN/\tR N)\le \mrO_\ell(\tG/G)$, then $\hat t \tilde\varphi=\tilde\varphi$ by the construction in \S \ref{sec-const-wei}.
Thus using Lemma \ref{det-radical-N-C-odd}, \ref{det-radical-N-C-2-linear}, \ref{det-radical-N-C-2-unitary}, we have the following result.

\begin{cor}\label{num-ell-ell'}
Suppose that $(\tR',\tvarphi')\in \mathrm{Alp}'(\tG)$.
Then 
\begin{enumerate}[\rm(1)]
\item $|\tG:GN_{\tG}(\tR')|$ is an $\ell$-number, and
\item $\kappa^{N_{\tG}(\tR')}_{N_{G}(\tR')}(\tvarphi')$ is an $\ell'$-number if none of (\ref{special-case-wei-ell3}), (\ref{special-case-wei-ell2-linear}) and (\ref{special-case-wei-ell2-uni}) holds for $\tR'$.
\end{enumerate}
\end{cor}

Now we determine the weights of $G$ covered by a given weight of $\tG$.

\begin{prop}\label{cover-case3}
Let $(\tR,\tvarphi)\in\Alp(\tG)$, $(\tR',\tvarphi')=\Xi (\tR,\tvarphi)$, $R=\tR'\cap G$ and $\varphi\in\Irr(N_G(R)\mid \tvarphi')$.
Then the weights of $G$ covered by $(\tR,\tvarphi)$ is just the $\tG$-conjugacy class of $(R,\varphi)$.
\end{prop}

\begin{proof}
If $\tR$ is special, then this proposition follows by Lemma \ref{spec-rad-cover} easily.
Now we assume that $\tR$ is not special which means one of Cases (i)--(ix) occurs.
In particular, $\ell\mid q-\eta$.

Assume that it is Case (i).
 We consider a special case first.
Let $\tR=\tR_{\Gamma}$ and $\tR'=\tR_{\Gamma}^{(k)}$ for suitable $k$ such that the hypothesis of Case (i) holds.
Let $\ttheta=\ttheta_\Gamma$ and $\ttheta'=\ttheta_\Gamma^{(k)}$.
For convenience, we replace $\tR$ by some conjugate group so that $\tR'\le \tR$. 
We also denote 
$R=\tR\cap G$,
$\tN'=N_{\tG}(\tR')$,
$\tC'=C_{\tG}(\tR')$,
$\tN=N_{\tG}(\tR)$,
$\tC=C_{\tG}(\tR)$,
$N=N_{G}(R)$.
Then $\tN'_{\ttheta'}=\tC'$ and $\tN_{\ttheta}=\tC$.
We consider the weight $(\tR,\tvarphi)$ with $\tvarphi=\Ind^{\tN}_{\tC}\tilde\theta$ and $(\tR',\tvarphi')\in\Alp'(\tG)$ with $\tvarphi'=\Ind^{\tN}_{\tC}\tilde\theta'$.
Then $(\tR',\tvarphi')=\Xi(\tR,\tvarphi)$.
By  \cite[(3.2)]{Brou86}, $(\tR',\ttheta')$ is a $\tB_\Gamma$-pair.
On the other hand, the block $\tB_\Gamma$ has  a unique weight $(\tR,\tvarphi)$. 
Thus a weight of $G$ is covered by $(\tR,\tvarphi)$ if and only if it belongs to a block of $G$ covered by $\tB_\Gamma$.
Let $(R,\varphi)$ be a weight of $G$ covered by $(\tR,\tvarphi)$ and $\tvarphi''\in\Irr(N_{\tG}(R)/\tR'\mid \varphi)$.
Then $\varphi''$ belongs to a block of $N$ covered by $\tB_\Gamma$.
According to  \cite[Lemma 2.3]{KS15} and the construction in \S \ref{sec-const-wei},  $\tvarphi''$ lies over $\ttheta'$ and we may assume that
$\tvarphi''=\tvarphi'$. 
So $\varphi\in\Irr(N\mid \tvarphi')$ and this proves the assertion in this case.

For general case of  Case (i), we write
$\tR=\tR_1\times \prod_{i=1}^u R_i^{t_i}$ and $\tR'=\tR'_1\times \prod_{i=1}^u  R_i^{t_i}$.
Then the normalizers have the form $N_{\tG}(\tR)=\tN_1\times \tN_{\ge 2}$,
$N_{\tG}(\tR')=\tN'_1\times \tN_{\ge 2}$ with 
$\tN_1\le \tN_1'$.
Let $N_1$ and $N_1'$ be the subgroups of  $\tN_1$ and $\tN_1'$ consisting of matrices whose determinant is in $\det( \tN_{\ge 2})$.
Then $N_1\times \tN_{\ge 2} \le N_{G}(\tR)$ and $N'_1\times \tN_{\ge 2} \le N_{G}(\tR')$.
From this, we replace $N_G(\tR)$ and $N_G(\tR')$ by  $N_1\times \tN_{\ge 2}$ and $N'_1\times \tN_{\ge 2}$ respectively, then the difference only comes from the first component. Therefore we may assume $u=1$.
In addition, when $\gamma>0$, we assume $\ttheta_{\Gamma,\gamma}\in \Irr(C_{\tG}(\tR)\mid \tvarphi)$,
and $\ttheta_{\Gamma,\gamma}^{(k)}\in \Irr(C_{\tG}(\tR')\mid \tvarphi')$ for a suitable $k$.
Then
$N_{\tG}(\tR)_{\ttheta}/\tR=\tC_{\Gamma,\gamma}\tR_{\Gamma,\gamma} /\tR_{\Gamma,\gamma}  \times \tM$ and $N_{\tG}(\tR')_{\ttheta'}/\tR'=\tC_{\Gamma,\gamma}^{(k)}\tR_{\Gamma,\gamma}^{(k)} /\tR_{\Gamma,\gamma}^{(k)} \times \tM$ with $\tM\cong \Sp_{2\gamma}(\ell)$.
So the case now follows by reduction to the preceding case.
In fact, if $\tR=\tR_{\Gamma,\gamma}$ and $(\tR,\tvarphi)$ is a weight of $\tG$, then $(\tR,\tvarphi)$ is the unique weight in the block of $G$ containing it and then the proof is similar as the above paragraph.
This proposition can be deduced by an entirely analogous method for Case (ii).
Moreover, Case (v) is similar with Case (i), while Case (vi) and Case (vii) are analogous with Case (ii).

Now assume that  it is Case (iii). Then $\ell=3$, $a=1$.
As in the above paragraph, we may assume that $\tR=\tR_{\Gamma,0,\one}$ and $\tR'=\tR_{\Gamma,1}$
with $\alpha_\Gamma=0$. Let $R=\tR'\cap G$.
For convenience, we replace $\tR$ by some conjugate group such that $\tR'\le \tR$. 
Also we assume $\tilde\theta_{\Gamma,0,\one}\in \Irr(\tR_{\Gamma,0,\one}\tC_{\Gamma,0,\one}\mid \tvarphi)$
and $\tilde\theta_{\Gamma,1}\in \Irr(\tR_{\Gamma,1}\tC_{\Gamma,1}\mid \tvarphi')$.
By direct calculation, there exist a group $\tilde{\mathcal C}$ such that $N_{\tG}(\tR)/\tR \cong \tilde{\mathcal C}\times \GL_1(3)$ and
$N_{\tG}(\tR')/\tR' \cong \tilde{\mathcal C}\times \SL_2(3)$.
Also $N_{G}(\tR)/R \cong {\mathcal C}\times \GL_1(3)$ and
$N_{G}(\tR')/R \cong {\mathcal C}\times Q_8$.
Here $\tilde{\mathcal C}$ can be regarded as a subgroup of $\tC_{\Gamma,0,\one}$ or $ \tC_{\Gamma,1}$ and $\tilde\theta_{\Gamma,0,\one}$ coincides with $\tilde\theta_{\Gamma,1}$ as a character of $\tilde{\mathcal C}$ and $\mathcal C\le \tilde{\mathcal C}$.
From this we may write $\tilde\varphi=\tilde\theta\times \xi$ and $\tilde\varphi'=\tilde\theta\times \tilde\zeta$.
Here, $\tilde\zeta$ is $\chi_{1a}$ or $\chi_{2a}$ and $\xi\in\Irr(\GL_1(3))$.
Let $\varphi\in\Irr(N_G(\tR)\mid\tvarphi)$ and 
$\varphi'\in\Irr(N_G(\tR')\mid\tvarphi')$.
We can take $\varphi=\theta\times \xi$ and 
$\varphi'=\theta\times \zeta$, where $\theta\in\Irr(\mathcal C\mid \ttheta)$, $\zeta\in\Irr(Q_8\mid \tilde\zeta)$.
By the construction, $\zeta$ has degree 2 or is the trivial character.
On the other hand, when considering the Dade--Glauberman--Nagao correspondence between characters of $N_G(R)/R$ and $N_G(\tR)/\tR$,
we only need to consider the characters of $Q_8$ and $\GL_1(3)$.
By  \cite[Thm.~5.2]{NS14}, the Dade--Glauberman--Nagao correspondence can be given by considering the irreducible constituents of restrictions of characters from $Q_8$ to $\GL_1(3)$.
By comparing the degrees of irreducible characters of $Q_8$ and $\GL_1(3)$, one knows that
the images of the irreducible character of degree $2$ of $Q_8$ and $1_{Q_8}$ are the non-trivial irreducible character
and $1_{\GL_1(3)}$ of $\GL_1(3)$ respectively.
From this, we can choose suitable $\varphi$ and $\varphi'$ such that $\varphi'$ corresponds to $\varphi$ via the Dade--Glauberman--Nagao correspondence.
Case (iv), (viii) and (ix) are similar to Case (iii) and
this completes the proof.
\end{proof}

%%%%%%%%%%%%%%%%%%%%%%%%%%%%%%%%%%%%%%%%%%%%%%%%%%%%%%%%%%%%%%%%%

\section{An equivariant bijection}\label{An-equivariant-bijection}

For a partition $\mu=(m_1,m_2,\ldots,m_t)$, denote $|\mu|=m_1+m_2+\cdots+m_t$ and write $\mu'$ for the transposed partition.
Set $\Delta(\mu)=\mathrm{gcd}(m_1,m_2,\ldots,m_t)$.
Let $d$ be a positive integer, then $\mu$ is uniquely determined by its $d$-core and $d$-quotient~(for the definitions, see for instance \cite[\S 2.7]{JK81}).

Recall that we always assume $\ell$ is a prime different from $p$.
Note that we do not always assume that $\ell\mid q-\eta$ in this section.
By \cite[Thm. 14.4]{CE04}, $\cE(\tG,\ell'):=\cup_{s}\cE(\tG,s)$ is a basic set of $\tG$, where 
$s$ runs through the  semisimple $\ell'$-elements of $\tG$.
In addition, by \cite{Ge91}, the decomposition matrix with respect to $\cE(\tG,\ell')$ is unitriangular. 
Thus the irreducible Brauer characters of $\tG$ can be parametrized by the subset $i\IBr(\tG)$ of $i\Irr(\tG)$,
containing $\tG$-conjugacy classes of pairs $(s,\mu)$ with $s$ a semisimple $\ell'$-element; see \cite{LZ18}.
Then the group $\mrO_{\ell'}(\fZ_{q-\eta})$ acts on $i\IBr(\tG)$.
Denote the irreducible Brauer character corresponding to $(s,\mu)^{\tG}$ by  $\tpsi_{s,\mu}$ if no confusion exists.
We also denote by $\Delta(\mu')$ the greatest common divisor of the $\Delta(\mu_\Gamma')$.

\begin{thm}[Kleshchev--Tiep \cite{KT09}, Denoncin \cite{De17}]\label{thm:branch-IBr}
\begin{enumerate}[\rm(1)]
\item $\hat\zeta \tpsi_{s,\mu} = \tpsi_{\zeta s,\zeta.\mu}$ for $\zeta\in\mrO_{\ell'}(\fZ_{q-\eta})$.
\item $\IBr(G)$ is the disjoint union of $\IBr(G\mid\tpsi_{s,\mu})$, where $(s,\mu)^{\tG}$ runs over a set of representatives of $\mrO_{\ell'}(\fZ_{q-\eta})$-orbits on $i\IBr(\tG)$ and
	$$(\kappa_G^{\tG}\tpsi_{s,\mu})_{\ell'} = |\{ \zeta\in\mrO_{\ell'}(\fZ_{q-\eta}) \mid \zeta.(s,\mu)^{\tG}=(s,\mu)^{\tG} \}|$$
	and $(\kappa_G^{\tG}\tpsi_{s,\mu})_\ell = \gcd(q-\eta,\Delta(\mu'))_\ell$.
\end{enumerate}	
\end{thm}

\begin{rem}\label{basic-set-sl}
In \cite{De17}, the author constructs an $\Aut(G)$-stable basic set $\cE$ of $G$.
Using the notation of \cite[\S 4.4]{De17}, we have a basic set  $\tilde{\cE}'$, obtained from $\mathcal E(\tG,\ell')$ by replacing its character $\tilde\chi$ by a suitable $\tilde\chi'$ through a process described there.
This basic set $\tilde{\cE}'$ satisfies that the decomposition matrix with respect to $\tilde{\cE}'$ of $\tG$ is unitriangular and the irreducible constituents of the restrictions of characters in $\tilde{\cE}'$ to $G$ form a unitriangular basic set for $G$.
It can be checked directly from the construction of $\tilde\cE'$ in \cite[\S 4.4]{De17}
 that the bijection $\mathcal E(\tG,\ell')\to \tilde{\cE}'$, $\tilde\chi\mapsto \tilde\chi'$ is 
 $\Lin_{\ell'}(\tG/G)\rtimes D$-equivariant.
Therefore,  there is a $\Lin_{\ell'}(\tG/G)\rtimes D$-equivariant bijection between $\tilde{\cE}'$ and $\IBr(\tG)$ 
  by \cite[Lemma~7.5]{CS13} or 
\cite[Lemma~2.3]{De17}.
\end{rem}

We will recall the bijection between $\IBr(\tG)$ and $\Alp(\tG)$ constructed in  \cite{AF90,An92,An93,An94}.
This construction  depends on \cite[(1A)]{AF90} and the notion of the $\ell$-core tower is used there.
In order to make the bijection satisfy condition (ii) (b) of Theorem \ref{thm:criterion},
we strengthen \cite[(1A)]{AF90}.

For a  positive integer $m$, we denote by $\mathcal U_m$ the set of partitions of $m$.
Let $\mathcal U_m(\gamma)$ be the set of partitions $\mu\in\mathcal U_m$ such that $\nu( \Delta(\mu'))=\gamma $ for $\gamma\ge 0$.
Then $\mathcal U_m=\coprod\limits_{\gamma\ge 0}\mathcal U_m(\gamma)$.
It is well-known that the irreducible characters of the symmetric group $\mathfrak S_m$ are in bijection with $\mathcal U_m$ (cf. \cite[Thm.~2.1.11]{JK81}).
Denote $p(m):=|\mathcal U_m|$.

Let $\delta,\gamma$ be non-negative integers.
Denote $\mathcal M_\delta:=\{\ (\delta,j)\ |\ j\in\mathbb Z, 1\le j\le \ell^\delta\ \}$
and $\mathcal M_\delta(\gamma):=\{\ (\delta,j)\in \mathcal M_\delta\ |\  \ell^{\delta-\gamma-1}< j\le \ell^{\delta-\gamma}\ \}$ for $0\le\gamma\le \delta$.
Then $\mathcal M_\delta=\coprod\limits_{0\le\gamma\le \delta}\mathcal M_\delta(\gamma)$.
Let  $\mathcal M=\coprod\limits_{\delta\ge 0}\mathcal M_\delta$.

Let $\mathcal D_m$ be the set of assignments $f:\mathcal M \to \{ \ell\text{-cores} \}$ such that $\sum\limits_{\delta,j}\ell^\delta |f(\delta,j)|=m$.
%The elements of $\mathcal D_m$ are in fact the $\ell$-core towers.
For $f\in \mathcal D_m$, we call the \emph{degree}  of $f$ the minimal integer $\gamma$ satisfying that there exists $(\delta,j)\in\mathcal M_\delta(\gamma)$ such that $f(\delta,j)$ is not empty.
Write $\deg(f)$ for the degree of $f\in \mathcal D_m$.
Let $\mathcal D_m(\gamma)$ denote the set of all $f\in\mathcal D_m$ such that $\mathrm{deg}(f)=\gamma$ for $\gamma\ge 0$.

\begin{lem}\label{forthebijexction}
	With the notation above, there exists a bijection $\pi_m:\mathcal U_m\to \mathcal D_m$ such that $\pi_m(\mathcal U_m(\gamma))=\mathcal D_m(\gamma)$ for every $\gamma\ge 0$.
	In particular, for $\mu\in \mathcal U_m$, we have $\nu( \Delta(\mu'))=\deg(\pi_m(\mu))$.
\end{lem}

\begin{proof}
	First note that	it suffices to show that $|\mathcal U_m(\gamma)|=|\mathcal D_m(\gamma)|$ for every $\gamma\ge 0$.
	Let $a=\nu( m)$. 
%	We may assume that $a>0$.
	If $c\ge 0$ satisfies that $\ell^c\mid m$,
	then $\coprod\limits_{\gamma\ge c}\mathcal U_m(\gamma)$ is in bijection with $\mathcal U_{m/\ell^c}$.
	This implies $\sum\limits_{\gamma\ge c}|\mathcal U_m(\gamma)|=p(m/\ell^c)$.
	Also note that the set $\mathcal U_m(\gamma)$ is empty when $\ell^\gamma\nmid m$.
	Hence $|\mathcal U_m(\gamma)|=p(m/\ell^\gamma)-p(m/\ell^{\gamma+1})$ if $0\le\gamma<a$, $|\mathcal U_m(a)|=p(m/\ell^a)$,
	and $|\mathcal U_m(\gamma)|=0$ if $\gamma>a$.

	Every $\mu\in\mathcal U_m(\gamma)$ is determined by its tower of $\ell$-cores $\{\lambda^\delta_j\}$ where the indices $(\delta,j)$ are such that $\delta\ge 0$, and for fixed $\delta$, $0< j\le \ell^\delta$.
	We give the definition of $\ell$-core towers as follows.
	The $\ell$-core $\lambda^\delta_j$ are recursively defined as follows.
	For $\delta=0$, let $\mu^0_1=\mu$, and let $\lambda^0_1$ and $(\mu^1_1,\mu^1_2,\dots,\mu^1_\ell)$ be respectively the $\ell$-core and the $\ell$-quotient of $\mu^0_1$.
	More generally, if $\lambda^{\delta-1}_j$ and $(\mu^\delta_{(j-1)\ell+1},\mu^\delta_{(j-1)\ell+2},\dots,\mu^\delta_{j\ell})$ are already defined for $1\le j\le \ell^{\delta-1}$, let $\lambda_j^\delta$ and $(\mu^{\delta+1}_{(j-1)\ell+1},\mu^{\delta+1}_{(j-1)\ell+2},\dots,\mu^{\delta+1}_{j\ell})$ be respectively  the $\ell$-core and the $\ell$-quotient of $\mu^\delta_j$ for $1\le j\le \ell^\delta$.
	Note that our definition of $\ell$-core towers here is slightly different from that in the proof of \cite[(1A)]{AF90} since for every $\delta$ the order of $\lambda_j^\delta$ ($1\le j\le \ell^\delta$) is changed.	
	Then the mapping $f:(\delta,j)\mapsto \lambda^\delta_j$ is an assignment in $\mathcal D_m$.
	The mapping $\tau:\mu\mapsto f$ is a bijection between $\mathcal U_m$ and $\mathcal D_m$
	since every partition is uniquely determined by its $\ell$-core and $\ell$-quotient, then $|\mathcal U_m|=|\mathcal D_m|$.
	
	Assume that integer $c\ge 0$ satisfies that $\ell^c\mid m$. Let $f\in \coprod\limits_{\gamma\ge c}\mathcal D_m(\gamma)$, then 
	$|f(\delta,j)|=0$ if $(\delta,j)\in\mathcal M_\delta(h)$ with $h<c$, i.e. $|f(\delta,j)|=0$ if $j>\ell^{\delta-c}$ or
	 $\delta< c$.
	Let $\mu=\pi_m^{-1}(f)$. Then $\lambda^\delta_j$ is empty if $j>\ell^{\delta-c}$ or $\delta< c$.
	%Hence $\mu^\delta_j$ is empty if $j>\ell^{\delta-c}$ or $\delta< c$. 
	So $\mu^\delta_j$ is empty with $j>\ell^{\delta-c}$ and $\delta\ge c$.
	From this $\mu$ is uniquely determined by $\mu^c_1$.
	Therefore $\sum\limits_{\gamma\ge c}|\mathcal D_m(\gamma)|=p(m/\ell^c)$ since $\mu^c_1$ is a partition of $m/\ell^c$.
	Also $\mathcal D_m(\gamma)$ is empty when $\ell^\gamma\nmid m$.
	Hence $|\mathcal D_m(\gamma)|=p(m/\ell^\gamma)-p(m/\ell^{\gamma+1})$ if $0\le\gamma<a$,  $|\mathcal D_m(\gamma)|=p(m/\ell^a)$ if $\gamma=a$, and $|\mathcal D_m(\gamma)|=0$ if $\gamma>a$.	
	So $|\mathcal U_m(\gamma)|=|\mathcal D_m(\gamma)|$ for every $\gamma\ge 0$.
From this there exists a bijection $\pi_m:\mathcal U_m\to \mathcal D_m$ such that $\pi_m(\mathcal U_m(\gamma))=\mathcal D_m(\gamma)$ for every $\gamma\ge 0$.
\end{proof}

Let $\tsC_{\Gamma,\delta}$ be defined as in (\ref{def-tsC-Gamma,delta}). Noting that $|\tsC_{\Gamma,\delta}|=e_\Gamma \ell^\delta$ by the proofs of \cite[(4C)]{AF90}, \cite[(3C), (3D)]{An92}, \cite[(3C), (3D)]{An93} and 
\cite[(4A), (4B)]{An94}.
Now we give some modification for the set $\tsC_{\Gamma,\delta}$ when $\ell\mid q-\eta$.
Let $$\tsC_{\Gamma,\delta}(\gamma)=\bigcup\limits_{\bc:\gamma+|\bc|=\delta} \dz(\tN_{\Gamma,\gamma,\bc}(\ttheta_{\Gamma,\gamma,\bc})/\tR_{\Gamma,\gamma,\bc}\mid\ttheta_{\Gamma,\gamma,\bc})$$
if $\ell$ is odd and $\ell\mid q-\eta$ or $\ell=2$, $4\mid q-\eta$, and
$$\tsC_{\Gamma,\delta}(\gamma)=\bigcup\limits_{\bc:\gamma+|\bc|=\delta} \dz(\tN_{\Gamma,\gamma,\bc}(\ttheta_{\Gamma,\gamma,\bc})/\tD_{\Gamma,\gamma,\bc}\mid\ttheta_{\Gamma,\gamma,\bc})$$
if $\ell=2$ and $4\mid q+\eta$.
Then $\tsC_{\Gamma,\delta}=\coprod\limits_{0\le \gamma\le \delta} \tsC_{\Gamma,\delta}(\gamma)$ and
$|\tsC_{\Gamma,\delta}(\gamma)|=\ell^{\delta-\gamma}-\ell^{\delta-\gamma-1}$ for $0\le \gamma\le \delta-1$ and $|\tsC_{\Gamma,\delta}(\delta)|=1$.
When  labelling the elements of $\tsC_{\Gamma,\delta}=\{\tilde\psi_{\Gamma,\delta,i}\mid 1\le i\le \ell^\delta\}$,
we set that $\tilde\psi=\tilde\psi_{\Gamma,\delta,i}$ for
$\ell^{\delta-\gamma-1}< i \le \ell^{\delta-\gamma}$ if and only if $\tilde\psi\in \tsC_{\Gamma,\delta}(\gamma)$.

If $\ell\mid q-\eta$, then we regard the sets
$\tsC_{\Gamma,\delta}$,
$\tsC_{\Gamma,\delta}(\gamma)$,
$\cup_\delta \tsC_{\Gamma,\delta}$ as $\mathcal M_\delta$,  $\mathcal M_\delta(\gamma)$, $\mathcal{M}$ respectively.
Thus for every $(s,\lambda,K)\in i \Alp(\tG)$ with $K=K_\Gamma$,
we can define $\mathrm{deg}(K_\Gamma)$ to be the degree of $K_\Gamma$ as an element of $\mathcal D_{m_\Gamma(s)}$.
Define $\mathrm{deg}(K):=\min_\Gamma\{ \mathrm{deg}(K_\Gamma) \}$, where $\Gamma$ runs over all elementary divisors of $s$.

\begin{lem}\label{num-conj-class}
Assume that $\ell\mid q-\eta$.
Let $(\tR,\tvarphi)\in\Alp(\tG)$ with label $(s,\lambda,K)$, $(\tR',\tvarphi')=\Xi(\tR,\tvarphi)$ and $R=\tR'\cap G$, where $\Xi$  is defined in (\ref{bij-wei}).
If it is not one of the cases (\ref{special-case-wei-ell3}), (\ref{special-case-wei-ell2-linear}) and (\ref{special-case-wei-ell2-uni}) for $\tR'$, then $|\tG:GN_{\tG}(R)|=\gcd( q-\eta,\ell^{\mathrm{def}(K)})$.
\end{lem}

\begin{proof}
First note that $\tG/GN_{\tG}(R)\cong \fZ_{q-\eta}/\det(N_{\tG}(R))$.		
Up to conjugacy, we can suppose that $\tR$  can be expressible as 
$\tR=\tR_1\times\cdots\times \tR_u$, where $\tR_i=\tR_{m_i,\alpha_i,\gamma_i,\bc_i}$ when $\ell$ is odd or $\ell=2$, $4\mid q-\eta$ and $\tR_i=\tD_{m_i,\alpha_i,\gamma_i,\bc_i}$ when $\ell=2$ and $4\mid q+\eta$.
Then $\deg(K)=\min_i\{\gamma_i \}$ since it is not one of the cases (\ref{special-case-wei-ell3}), (\ref{special-case-wei-ell2-linear}) and (\ref{special-case-wei-ell2-uni}).
Thus direct calculation of $\det(N_{\tG}(R))$ case by case shows this lemma.
\end{proof}

The bijection between $\IBr(\tG)$ and $\Alp(\tG)$ constructed in  \cite{AF90,An92,An93,An94} comes from a bijection between $i\IBr(\tG)$ and $i\Alp(\tG)$; see also the proof of \cite[Thm.~1.1]{LZ18}.
Now we recall that construction as follows.

Given $(s,\mu)$ in $i\IBr(G)$, we let $\lambda=\prod_\Gamma\lambda_\Gamma$ where $\lambda_\Gamma$ is the $e_\Gamma$-core of $\mu_\Gamma$.
And we assume that $(\mu_{\Gamma,1},\cdots,\mu_{\Gamma,e_\Gamma})$ is the $e_\Gamma$-quotient of $\mu_\Gamma$.
Then $|\mu_{\Gamma,1}|+\cdots+|\mu_{\Gamma,e_\Gamma}|=w_\Gamma$ where $w_\Gamma$ is determined by $m_{\Gamma}(s)=|\lambda_\Gamma|+e_\Gamma w_\Gamma$.
Since $|\sC_{\Gamma,\delta}|=e_\Gamma \ell^\delta$, we can use \cite[(1A)]{AF90} to define $K_\Gamma$.
Thus, we get $(s,\lambda,K)$ in $i\Alp(G)$.
The map $(s,\mu)^G \mapsto (s,\lambda,K)^G$ is a bijection by  \cite[(1A)]{AF90}.

If $\ell\nmid q-\eta$, we use the bijections as above.
If $\ell\mid q-\eta$, we assume further that when using \cite[(1A)]{AF90} in the construction above,
the bijections between $\mathcal U_m$ and $\mathcal D_m$ are chosen as in Lemma \ref{forthebijexction}.
From this we have chosen a bijection $\tOmega$ between $\IBr(\tG)$ and $\Alp(\tG)$ for both cases.

It turns out that with above combinatorial technique
 this bijection can be taken as the bijection $\tOmega$ required in part (ii) of Theorem \ref{thm:criterion}.
Recall that we  identify $\IBr(\tG/G)$ with $\Lin_{\ell'}(\tG/G)$, which is in bijection with $\mrO_{\ell'}(Z(\tG))$.

\begin{thm}\label{thm-equ-bijection}
	The bijection $\tOmega$ between $\IBr(\tG)$ and $\Alp(\tG)$ as above is equivariant under the action of $\Lin_{\ell'}(\tG/G) \rtimes D$ and  can be chosen  to satisfy 
		\begin{itemize}\setlength{\itemsep}{0pt}
			\item $\tOmega(\IBr(\tG\mid\nu)) = \Alp(\tG\mid\nu)$
			for every $\nu \in \Lin_{\ell'}(Z(\tG))$,
			\item $\J_G(\tpsi) = \J_G(\tOmega(\tpsi))$ for every $\tpsi \in \IBr(\tG)$.
		\end{itemize}
In particular, $\tOmega$ satisfies the condition (ii) of Theorem \ref{thm:criterion}.
\end{thm}

\begin{proof}
The equivariance under the action of $D$ is established in \cite[Thm.~1.1]{LZ18} already, while the equivariance under the action of $\Lin_{\ell'}(\tG/G)\cong \mrO_{\ell'}(\fZ_{q-\eta})$ are given in \cite[\S 5]{Feng19}. In fact, for a  weights $\overline{(\tR,\tvarphi)}$ of $\tG$ with label $(s,\lambda,K)$ and $\zeta\in\mrO_{\ell'}(\fZ_{q-\eta})$, by Theorem \ref{action-z-onwei}, $\overline{(\tR,\hat \zeta\tvarphi)}$ has label $(\zeta s,\zeta.\lambda,\zeta.K)$.
	By above observations, $\overline{(\tR,\tvarphi)}$ corresponds via $\tOmega^{-1}$ to an irreducible Brauer character $\tpsi_{s,\mu}$.
By Theorem \ref{thm:branch-IBr} (a), $\hat \zeta\tpsi_{s,\mu} = \tpsi_{\zeta s,\zeta.\mu}$.
	Then it is easy to see $\tOmega$ is equivariant under $\Lin_{\ell'}(\tG/G)$.
	Finally, by the construction of weights it follows that $\Res^{N_{\tG}(\tR)}_{Z(\tG)} \tvarphi = \Res_{Z(\tG)}\hs = \Res^{\tG}_{Z(\tG)} \tpsi_{s,\mu}$.
From this, it remains to prove that $\J_G(\tpsi) = \J_G(\tOmega(\tpsi))$ for every $\tpsi \in \IBr(\tG)$. Obviously, we can assume that $\ell\mid q-\eta$.

Since $\tG/G$ is cyclic,
it suffices to show $|\tG:\J_G(\tpsi)| = |\tG:\J_G(\tOmega(\tpsi))|$.
Note that $|\tG:\J_G(\tpsi)| =\kappa^{\tG}_{G}(\tilde\psi)_\ell$ for $\tpsi \in \IBr(\tG)$.
If $(R,\varphi)$ is a weight covered by $\tOmega(\tpsi)=(\tR,\tvarphi)$, then $|\tG:\J_G(\tOmega(\tpsi))|=|\tG:GN_{\tG}(R)|\cdot |N_{\tG}(R):N_{\tG}(R)_{\varphi}|_\ell$ 
since $|\tG:GN_{\tG}(R)|$ is an $\ell$-number 
by Lemma \ref{det-radical-N-C-odd}, \ref{det-radical-N-C-2-linear} and \ref{det-radical-N-C-2-unitary}.
Also note that $\tG/GN_{\tG}(R)\cong \fZ_{q-\eta}/\det(N_{\tG}(R))$.
We let $(\tR',\tvarphi')=\Xi(\tOmega(\tpsi))$.

First suppose that  it is not  one of the cases (\ref{special-case-wei-ell3}), (\ref{special-case-wei-ell2-linear}) and (\ref{special-case-wei-ell2-uni}) for $\tR'$.
Then by Corollary \ref{num-ell-ell'}, $|\tG:\J_G(\tOmega(\tpsi))|=|\tG:GN_{\tG}(\tR')|$.
Then this assertion follows from Theorem \ref{thm:branch-IBr} and
Lemma \ref{forthebijexction} and \ref{num-conj-class}.

Now suppose that  we are in the case (\ref{special-case-wei-ell3}) for $\tR'$.
Then $\ell=3$ and $a=1$ and $R=\tR'$.
We have either $(\tR',\tvarphi')=(\tR,\tvarphi)$
or we are in the Case (iii) of \S \ref{Para-of-wei} for $(\tR,\tvarphi)$.
Combine Lemma \ref{det-radical-N-C-odd}, \ref{spec-rad-cover} and Proposition \ref{cover-case3}, $|\tG:GN_{\tG}(\tR')|=1$ 
 and $|\tG:\J_G(\tOmega(\tpsi))|=\kappa^{N_{\tG}(\tR')}_{N_{G}(\tR')}(\tvarphi')_3$.
Also by \S \ref{Para-of-wei}, 
$\kappa^{N_{\tG}(\tR')}_{N_{G}(\tR')}(\tvarphi')_3$ is $0$ or $1$ if 
$\tvarphi'$ is of defect $1$ or $0$	when viewed as a character of $N_{\tG}(\tR')/\tR'$.

Assume that $\tvarphi'$ is of defect  $0$ first, which implies that  $(\tR',\tvarphi')=(\tR,\tvarphi)$.
Then it suffices to prove $\nu(\kappa^{\tG}_{G}(\tilde\psi))=1$, which
 follows from Theorem \ref{thm:branch-IBr} (b) and
the fact that $\nu(\Delta(\mu'))=\mathrm{def}(K)=1$
(by Lemma \ref{forthebijexction}).
Now let  $\tvarphi'$ is of defect  $1$.
Then we are in the Case (iii) of \S \ref{Para-of-wei} and the proof is completed  by showing that $3\nmid \kappa^{\tG}_{G}(\tilde\psi)$.
This holds because $\nu(\Delta(\mu'))=\mathrm{def}(K)=0$.

The proofs for the cases  (\ref{special-case-wei-ell2-linear}) and (\ref{special-case-wei-ell2-uni}) are entirely analogous with for the case (\ref{special-case-wei-ell3}) as above, using the construction in Cases (iv), (viii) and (ix) of \S \ref{Para-of-wei}.
\end{proof}

%%%%%%%%%%%%%%%%%%%%%%%%%%%%%%%%%%%%%%%%%%%%%%%%%%%%%%%%%%%%%

\section{Stabilizers and extensions of weight characters}
\label{sec-Stab-ext-wei}

In this section, we prove the following theorem, towards the condition (iv) of Theorem \ref{thm:criterion}.

\begin{thm}\label{local-condition}
	Let $(R',\varphi')$ be an weight of $G$.
	Then there exists some $x\in\tG$ with 
	\begin{enumerate}[\rm(1)]
		\item $(\tG D)_{R',\varphi'}=\tG_{R',\varphi'}(GD^x)_{R',\varphi'}$, and
		\item $\varphi'$ extends to $(GD^x)_{R',\varphi'}$.
	\end{enumerate}
\end{thm}

\subsection{General results on extensions of characters.}\label{gen-re-extension}

In order to consider the stabilizers and extensions of weight characters as in Theorem \ref{local-condition},
we first give some general results about extensions of characters.
Throughout this subsection \S \ref{gen-re-extension}, all groups considered are arbitrary finite groups.

\begin{lem}\label{pre-dir}
	Let $X$ be a finite group and $Y,H$ be two normal subgroups of $X$ such that $Y\leq H$.
	Assume $\chi\in\Irr(X)$ and $\xi\in\Irr(Y\mid\chi)$.
	If $\Res^X_Y\chi$ is multiplicity free, then $|\Irr(H\mid\chi)\cap\Irr(H\mid\xi)|=1$.
\end{lem}

\begin{proof}
	This follows from Clifford theory.
\end{proof}

The following lemma is a generalisation of \cite[Lemma 5.8 (a)]{CS17}.

\begin{lem}\label{ext-char}
	Suppose that $X=YU$ is a finite group and $Y\unlhd X$, $U\le  X$ such that $Y\cap U\unlhd Y$. 
	Let $\chi\in\Irr(Y)$, $\xi\in\Irr(Y\cap U\mid \chi)$ and $\phi\in\Irr(Y_\xi\mid\xi)$  such that $\chi=\Ind_{Y_\xi}^{Y}\phi$.
	Assume further that $U_\chi\le U_\xi$ and $\xi=\Res^{Y_\xi}_{Y\cap U}\phi$.
	Then 
	\begin{enumerate}[\rm(1)]
		\item $U_\chi=U_\phi\le U_\xi$,
		\item The following are equivalent.
		\begin{enumerate}[\rm(a)]
			\item $\chi$ extends to $X_\chi$,
			\item $\xi$ extends to $U_\chi$,
			\item $\phi$ extends to $U_\chi Y_\xi$.
		\end{enumerate}
	\end{enumerate}
\end{lem}

\begin{proof}
%	Since $\Res^Y_{Y\cap U}(\chi)$ is multiplicity-free,	$\phi$ is an extension of $\xi$.
First, $U_\phi\le U_\xi$ and $U_\phi\le U_\chi$.
%	Also one have $U_\chi\le U_\xi$ from	$X_\xi=Y_\xi U_\xi$.
	By Clifford theory, $\Irr(Y_\xi\mid\xi)\cap \Irr(Y_\xi\mid\chi)=\{\phi\}$.
	Then $U_\chi$ stabilizes $\phi$ and then $U_\chi=U_\phi\le U_\xi$, which gives (1).
	Thus by \cite[Cor. 4.2 and 4.3]{Is84}, 
	the restriction $\Res^{Y_\xi U_\chi}_{U_\chi}:\Irr(Y_\xi U_\chi\mid \phi)\to \Irr(U_\chi\mid \xi)$
	and induction
	$\Ind^{X_\chi}_{Y_\xi U_\chi}:\Irr(Y_\xi U_\chi\mid\phi)\to \Irr(X_\chi\mid\chi)$
	are well-defined bijections.
	So (2) holds.
\end{proof}

\begin{rem}\label{ker-ext}
	\begin{enumerate}[(a)]
		\item We consider a special case. If $X=Y\rtimes U$ with abelian $Y$, then by Lemma \ref{ext-char},
		every character $\chi\in\Irr(Y)$ extends to $X_\chi$. This also follows from \cite[Prop.~19.12 (b)]{Hu98}.
		\item In Lemma \ref{ext-char}, if $\xi$ has an extension $\tilde\xi\in\Irr(U_\chi)$ with $u\in\ker(\tilde\xi)$ and $\grp{u}\unlhd X_\chi$,
		then $\chi$ has an extension $\tilde\chi$ to $\Irr(X_\chi)$ such that $u\in\ker(\tilde\chi)$. 
		\item We will make use of the following special situation.
			Suppose that $X=YU$ is a finite group, where $Y\unlhd X$, $U\le  X$,  $Y\cap U\unlhd Y$ such that $Y/Y\cap U$ is abelian. 
		Let $\chi\in\Irr(Y)$, $\xi\in\Irr(Y\cap U\mid \chi)$.
		Assume further that $X_\xi=Y_\xi U_\xi$ and $\Res^Y_{Y\cap U}\chi$ is multiplicity-free.
		Then $U_\chi\le U_\xi$ and
		$\chi$ extends to $X_\chi$ if and only if $\xi$ extends to $U_\chi$.
	\end{enumerate}
\end{rem}

The extensions for central products was considered in \cite[2.1]{Sp12} (see for example \cite[Lemma 3.9]{FLZ19a} for more details).
Now we give a generalisation to non-commutative case.

\begin{lem} \label{ext-semi}
	Let  $A=XY$ be a finite group where $X\unlhd A$ and $Z=X\cap Y\subseteq Z(X)$.
	Suppose that we have characters $\chi\in\Irr(X)$ and $\zeta\in\Irr(Y)$ such that $A_\chi=A$, $\zeta(1)=1$ and $\Irr(Z\mid\chi)=\Irr(Z\mid\zeta)$.
	If $\tilde\chi\in\Irr(X\rtimes (Y/Z))$ is an extension of $\chi$, then   $\tilde\chi\cdot \zeta\in\Irr(A)$ is an extension of $\chi$, where $\tilde\chi\cdot\zeta(xy)=\tilde\chi(x\cdot yZ)\zeta(y)$ for $x\in X$ and $y\in Y$.
\end{lem}

\begin{proof}
Take $\tilde Q$ a representation of $X\rtimes (Y/Z)$ yielding $\tchi$ and define a representation $\mathcal D$ of $A=XY$ as $\mathcal D(xy)=\zeta(y)\tilde Q(x,\bar y)$ for $x\in X$ and $y\in Y$.
Here we abbreviate $\tilde Q((x,\bar y))$ to $\tilde Q(x,\bar y)$.
Then this map $\mathcal D$ is well-defined. Indeed, if $xy=x'y'$ in $A$, with $x,x'\in X$ and $y,y'\in Y$, then take $z\in Z$ such that $x=x'z$ and $y=z^{-1}y'$.
Thus 
\begin{align*}
\mathcal D(xy)=&\zeta(y)\tilde Q(x,\bar y)=\zeta(z^{-1}y')\tilde Q(x'z,\bar{y'})\\
=&\zeta(z^{-1}y')\tilde Q(x',1)\tilde Q(z,1)\tilde Q(1,\bar{y'})\\
=&\zeta(z^{-1}y')\tilde Q(x',1)\tilde Q(1,\bar{y'})\tilde Q(z,1)\\
=&\zeta(y')\tilde Q(x',\bar{y'})=\mathcal D(x'y').
\end{align*}
We note that $\tilde Q(z,1)$ above must be a scalar matrix by Schur's lemma, since $z\in Z(X)$.
Also by assumption, $Q(z)=\zeta(z)\cdot I_{\chi(1)}$ as $Q$ is the restriction of $\tilde Q$ to $X$ and is irreducible.

The proof that it is a homomorphism is similar.
For $x,x'\in X$ and $y,y'\in Y$, one have $xyx'y'=xx'^yyy'=xx'^{\bar y}yy'$ and so
\begin{align*}
\mathcal D(xyx'y')=&\zeta(yy')\tilde Q(xx'^y,\overline{yy'})=\zeta(yy')\tilde Q(xx'^{\bar y},\overline{y}\bar{y'})\\=&\zeta(y)\zeta(y')\tilde Q(x,\bar y)\tilde Q(x',\bar{y'})=\zeta(y)
\tilde Q(x,\bar y)\zeta(y')\tilde Q(x',\bar{y'})\\=&\mathcal D(xy)\mathcal D(x'y').
\end{align*}

Therefore $\mathcal D$ defined as above is a representation of $A$ and affords the character $\tchi\cdot\zeta$ of $A$, which is an extension of $\chi$.
\end{proof}

Next, we give a lemma about extensions of characters of direct products, which is a generalisation of \cite[Lemma 2.2]{LZ19}.

\begin{lem}\label{ext-central-prod}
	Let  $A=XE$ be a finite group, where $X=X_1\times X_2$ is the direct product of two subgroups $X_1,X_2$ and $E$ stabilizes $X_1$ and $X_2$. Then $X\unlhd A$.
	Assume further  $X_1\cap E\unlhd X_1$, $X_2\cap E\unlhd X_2$ and
	$(E\cap X_1)(E\cap X_2)=E\cap X_1X_2$.
	Let $\chi_i\in\Irr(X_i)$ satisfy $E=E_{\chi_1}=E_{\chi_2}$, and let $\chi=\chi_1\times\chi_2\in\Irr(X)$.
	
	Let $v_i\in X_i$, $\sigma_0\in E$ and assume that there exists an extension $\tilde\chi_i$ of $\chi_i$ to $X_i E$ satisfying that $v_i\sigma_0\in\ker(\tilde\chi_i)$
	for $i=1,2$.
	Then there is an extension $\tilde\chi$ of
	$\chi$  to $A=X E$ such that $v_1v_2\sigma_0\in\ker(\tilde\chi)$.
\end{lem}

\begin{proof}
We regard $X_iE$ as a quotient group of $X_i\rtimes E$ for $i=1,2$, \emph{i.e.}, $X_iE=(X_i\rtimes E)/Z_i$, where $Z_i=\{ (z,z^{-1})\mid z\in X_i\cap E \}$. 
	So we may view $\tilde \chi_i$ as a character of $X_i\rtimes E$.
Similarly, the group $A=XE$ can be regarded as a quotient group of $(X_1\times X_2)\rtimes E$, \emph{i.e.}, $A=((X_1\times X_2)\rtimes E)/Z_1Z_2$.

	Note that we can view $(X_1\times X_2)\rtimes E$ as a subgroup of $(X_1\rtimes E)\times (X_2\rtimes E)$ via $(x_1,x_2)\sigma\mapsto (x_1\sigma,x_2\sigma)$.
	Let $\tilde\chi_i$ be an extension of $\chi_i$ to $X_i\rtimes E$ for $i=1,2$.
	Then $\tilde\chi'=\Res^{(X_1\rtimes E)\times (X_2\rtimes E)}_{(X_1\times X_2)\rtimes E}(\tilde\chi_1\times\tilde\chi_2)$ contains $Z_1Z_2$ in its kernel.
	Thus $\tchi'$ induces an extension $\tchi$ of $\chi$ to $A$, as required.
\end{proof}

Finally, we consider the extension of characters of groups closely related to wreath products.

\begin{lem}\label{ext-wre}
	Let $Y\unlhd X$ be finite groups.
	Let $X^t=X_1\times\cdots\times X_t$ with $X_i\cong X$ for $i=1,\ldots,t$.
	Suppose that  $A$ is a finite group such that $A=X^tH$ with $X^t\unlhd A$ and
	$H$ acts on $X^t$ by permuting $X_i$'s
	via an epimorphism $\pi:H\to\mathfrak S_t$ with $\ker(\pi)=X^t\cap H\subseteq Z(X^t)\cap Y$.
	Suppose that a finite abelian group $E$ acts on $A$ via automorphisms such that $E$ stabilizes $X$ and $Y$ and $E$	
	fixes every element of $H$.
	
	Let $\xi=\prod_i \xi_i^{t_i}$ with  $\xi_i\in\Irr(Y)$, $t=\sum_{i}t_i$ and $\xi_i\ne \xi_j$ if $i\ne j$. Suppose that $(A\rtimes E)_\xi=(X^t\rtimes E)_\xi H_\xi$.
	Assume further  the character $\zeta\in \Irr(\ker(\pi)\mid\xi)$ extends to $\tilde\zeta\in\Irr(H_\xi)$.
	Let $v_i\in X_i$ and $\tau\in E$ such that $v_i\tau\in (X_i\rtimes E)_{\xi_i}$ and $v=(v_1,\ldots,v_t)$.
	If every $\xi_i$ has an extension $\tilde\xi_i$  to $(X\rtimes E)_{\xi_i}$ with $v_i\tau\in \ker(\tilde \xi_i)$,
	then $\xi$ has an extension $\tilde\xi$ to $(A\rtimes E)_\xi$ with $v\tau\in \ker(\tilde\xi)$.
\end{lem}

\begin{proof}
	By the assumption
	one have  $\pi(H_\xi)=\prod_i \mathfrak{S}(t_i)$.	
	We transfer to wreath products.
			By \cite[Lemma~25.5]{Hu98}, $\xi$ extends to $\prod_i(X\rtimes E)_{\xi_i}\wr \mathfrak S(t_i)$.
			Note that we can view $(\prod_i X\wr \mathfrak S(t_i))\rtimes E$ as a subgroup of $\prod_i(X\rtimes E)\wr \mathfrak S(t_i)$.
Then $\xi$ extends to its stabilizer in $(\prod_i X\wr \mathfrak S(t_i))\rtimes E$,
	which implies that $\xi$ extends to its stabilizer in $(X^t H)\rtimes E$ 
by Lemma \ref{ext-semi} since $(X^t\rtimes H/\ker(\pi))\rtimes E\cong (X\wr \fS_t)\rtimes E$,
	and we denote by $\tilde \xi$ one extension.
	The property $v\tau\in \ker(\tilde\xi)$ holds from the construction above.
\end{proof}

\subsection{Transfer to twisted groups.}

We will use the methods and  technique in  the verification of the inductive McKay condition of some simple groups  in \cite[\S 5]{CS17} and \cite[\S 3]{MS16}.
We first transfer to twisted groups, then give a parametrization for certain weight  characters of $G$, and finally prove our Theorem \ref{local-condition}. Note that we don't always assume that $\ell\mid q-\eta$  any more.

Let $(R',\varphi')$ be a weight of $G=\SL_n(\eta q)$ as in Theorem \ref{local-condition}.
Up to $\tG$-conjugacy, we may assume that there is a special radical subgroup $\tR'$ of $\tG$ of the form
$\tR'=\tR'_0\times\prod_{i=1}^s (\tR'_i)^{t_i}$ such that $R'=\tR'\cap G$,
where $\tR_0'$ is the identity group and $\tR'_i$'s ($i>0$) are different twisted basic subgroups, and we let
 $\tR'_i=\tR_{m_i,\alpha_i,\gamma_i,\bc_i}$ or $\tR^\pm_{m_i,\alpha_i,\gamma_i,\bc_i}$.
Note that $\tR^\pm_{m_i,\alpha_i,\gamma_i,\bc_i}$  occurs only when $\ell=2$, $4\mid q+\eta$ and  $\alpha_i=0$. Recall that if $\ell=2$ and $4\mid q+\eta$, then $\tR_{m,0,\gamma,\bc}=\tS_{m,1,\gamma-1,\bc}$.

Let $v_{m,\alpha,\gamma,\bc}=v_{m,\alpha,\gamma}\otimes I_{\ell^{|\bc|}}$, where $v_{m,\alpha,\gamma}$ is defined as in \S \ref{sect:weights-general-gp}.
Let $v=I_0\times\prod_i v_{m_i,\alpha_i,\gamma_i,\bc_i}\otimes I_{t_i}$, where
$I_0$ is the identity of $\tR_0$.
By Lang--Steinberg theorem (cf. \cite[Thm.~21.7]{MT11}), there exists an element $g\in\bG$ such that $g^{-1}F(g)=v$. 
In fact, we can take $g=I_0\times\prod_i g_{m_i,\alpha_i,\gamma_i}\otimes I_{\bc_i}\otimes I_{t_i}$, where $g_{m_i,\alpha_i,\gamma_i}$ is defined as in \S \ref{sect:weights-general-gp}.
Denote by $\iota$ the map $x\mapsto gxg^{-1}$. Then 
similar as the proof of \cite[Prop. 5.3]{CS17},
$\iota$ induces an isomorphism
$$\iota:\quad \tbG^{vF}\rtimes E\ \to\ \tbG^F\rtimes E,$$
where $E$ is the group of graph and field automorphisms. We denote the field automorphism in $E$ as $\hat F_p$ and $\hat F=\hat F_p^f$.
Let
$\tG^{tw}=\tbG^{vF}$ and
$G^{tw}=\bG^{vF}$.

Let $\tR=\iota^{-1}(\tR')$.
Then $\tR=\tR_0\times\prod_{i=0}^s (\tR_i)^{t_i}$ with $\tR_0=\tR_0'$ and $\tR_i=\tR^{tw}_{m_i,\alpha_i,\gamma_i,\bc_i}$ or $\tR^{\pm,tw}_{m_i,\alpha_i,\gamma_i,\bc_i}$. 
Recall that $\tR^{tw}_{m_i,\alpha_i,\gamma_i,\bc_i}=\tR^{tw}_{m_i,\alpha_i,\gamma_i}\wr A_{\bc_i}$ and $\tR^{\pm,tw}_{m_i,\alpha_i,\gamma_i,\bc_i}=\tR^{\pm,tw}_{m_i,\alpha_i,\gamma_i}\wr A_{\bc_i}$.
Let $R=\tR\cap G^{tw}$.
In this section, we write $N=N_{G^{tw}}(R)$,
$C=C_{G^{tw}}(R)$,
$\tN=N_{\tG^{tw}}(R)$ and  $\tC=C_{\tG^{tw}}(R)$.
Then 
$\tC=\tC_0\times\prod_{i=1}^s \tC_i$ and
\begin{equation}\label{stru-tn}
\tN=\tN_0\times\prod_{i=1}^s\tN_i\wr \mathfrak S(t_i),
\addtocounter{thm}{1}\tag{\thethm}
\end{equation}
where
$\tC_i=\tC^{\pm,tw}_{m_i,\alpha_i,\gamma_i,\bc_i}$ 
and 
$\tN_i=\tN^{\pm,tw}_{m_i,\alpha_i,\gamma_i,\bc_i}$ for $1\le i\le s$.

Let $\tilde M=\tR_0\times\prod_{i=1}^s (\tM_i)^{t_i}$ and $M=\tM\cap G^{tw}$,
where $\tM_i=(\tM^{tw}_{m_i,\alpha_i,\gamma_i})^{\ell^{|\bc_i|}}$ or $(\tM^{\pm,tw}_{m_i,\alpha_i,\gamma_i})^{\ell^{|\bc_i|}}$.
Then $\det(\tM)=1$ 
and $\tM=M$
unless when 
\begin{equation}\label{special-case-7-1}
\begin{aligned}
&\textrm{there exists a component}\
\tR_{m_i,\alpha_i,\gamma_i,\bc_i}\  (\textrm{or}\ \tR^\pm_{m_i,\alpha_i,\gamma_i,\bc_i})\ \textrm{of}\ \tR'\ \textrm{which satisfies that} \
\bc_i=\zero\  \\ & \textrm{and}\
\textrm{one of the cases}\ 
(\ref{eq:special-case-3-1}), (\ref{eq:special-case-2-linear-1})\ \textrm{and}\ (\ref{special-case-2-uni-1})\ \textrm{holds for}\ \tR_{m_i,\alpha_i,\gamma_i}\ (\textrm{or}\ \tR^\pm_{m_i,\alpha_i,\gamma_i}).
\end{aligned}
\addtocounter{thm}{1}\tag{\thethm}
\end{equation}
Also, $\det(\tN_i)=\det(\tR_i\tC_i \tM_i)$, and $\tM_i\cap\tR_i\tC_i\subseteq \tR_i$ and $\tR_i\tC_i \tM_i/\tR_i\cong\tR_i\tC_i/\tR_i\times \tM_i/(\tM_i\cap\tR_i)$.
Note that we take $\tM_{m_i,\alpha_i,\gamma_i}$ as in Remark  \ref{not-cent} and \ref{M-for-not-central-product}
if we are in the case (\ref{eq:special-case-2-linear-0})
for $\tR_{m_i,\alpha_i,\gamma_i}$.
 In addition, $\tM_i/(\tM_i\cap\tR_i)\cong\Sp_{2\gamma_i}(\ell)$.

Let
$V=\tR_0\times\prod_{i=1}^s V_i^{t_i}$,
where $V_i=\langle v_{m_i,\alpha_i,\gamma_i,\bc_i} \rangle$. 
Then $\det(V)=1$.
Let $\tZ=\tR_0 \times\prod_{i=1}^{s} (\tZ_i)^{t_i}$ such that $\tZ_i\subseteq Z(\tC_i)$ and $|\tZ_i|=2$ for $i>0$.\label{def-Z}
Then we can check that $V\cap \tR\tC\tM=V\cap \tR\tC\subseteq \tZ$.
Precisely,
$V_i\cap \tC_i \tR_i\subseteq \tZ_i$ and $V_i/(V_i\cap \tC_i \tR_i)\cong V_{m_i,\alpha_i,\gamma_i}$.

Note that
$\tN_i/\tR_i\cong \tC_i\tR_i\tM_i V_i/\tR_i\times T_i$, where $T_i=N_{\mathfrak{S}(\ell^{\bc_i|})}(A_{\bc_i})/A_{\bc_i}$.	
Here, the group $\mathfrak{S}(\ell^{\bc_i|})$ consists of permutation matrices.
If $\bc_i=(c_{i1},\ldots,c_{ih})$, then 
$T_i\cong \prod_{j=1}^h \GL_{c_{ij}}(\ell)$.
Let $T=\tR_0\times\prod_{i=1}^s T_i^{t_i}$.

For a positive integer $d$, we let $I^d_{i,j}$ be the permutation matrix obtained from the identity matrix of $\GL_d(q)$ by exchanging the $i$-th row and $j$-th row, for $1\le i, j\le d$.
Let $I'^d_{1,j}$ be the matrix obtaining from $I^d_{1,j}$ by changing the $1$ in the first row to $-1$.

For $1\le i\le s$, the symmetric groups $\mathfrak S(t_i)$ in (\ref{stru-tn})
is generated by $I_{m_ie\ell^{\alpha_i+\gamma_i+|\bc_i|}}\otimes I_{1,l}^{t_i}$
where $l$ runs through $2,\ldots, t_i$.
Now for $1\le i\le s$, we let $\mathcal S_i$ be the identity matrix 
if $t_i=1$. If $\ell$ is odd and $t_i>1$, let
$\mathcal S_i$ be the group generated by
$I_{m_ie\ell^{\alpha_i+\gamma_i+|\bc_i|}}\otimes {I}_{1,l}^{t_i}$ if $e$ is even,
and generated by 
$I_{m_ie\ell^{\alpha_i+\gamma_i+|\bc_i|}}\otimes {I'}_{1,l}^{t_i}$ if $e$ is odd, where $l$ runs through $2,\ldots, t_i$.
If $\ell=2$ and $t_i>1$, let
$\mathcal S_i$ be the group generated by
$I_{m_i2^{\alpha_i+\gamma_i+|\bc_i|}}\otimes {I}_{1,l}^{t_i}$ if $\nu(m_i)+\alpha_i+\gamma_i+|\bc_i|>0$, and generated by
$I_{m_i2^{\alpha_i+\gamma_i+|\bc_i|}}\otimes {I'}_{1,l}^{t_i}$ if $\nu(m_i)=\alpha_i=\gamma_i=|\bc_i|=0$, where $l$ runs through $2,\ldots, t_i$.
Then $\mathcal S_i\cap (\tC_i\tR_i\tM_i)=\mathcal S_i\cap (\tC_i\tR_i)\subseteq \tZ_i$
and $\mathcal S_i/\mathcal S_i\cap (\tC_i\tR_i)\cong \mathfrak S(t_i)$.
Let 
$\mathcal S=\tR_0\times\prod_{i=1}^{u}\mathcal S_i$.
Then $\det(\mathcal S)=1$ and $\mathcal S\cap \tR\tC \tM V=\mathcal S\cap \tR\tC\subseteq \tZ$.

Therefore,
$\tN/\tR=((\tR\tC /\tR)(\tR\tM/\tR)V\times T)\mathcal S$ and $N/R=(((\tR\tC\tM\cap N)/R)V\times T)\mathcal S$.
By Remark \ref{M-is-stable-odd}, \ref{M-is-stable-2-1}, \ref{M-is-stable-2-2} and \ref{M-is-stable-2-3}, $\tM\unlhd \tN$ and the groups
$\tR$, $\tM$, $\tC$, $\tN$ are $E$-stable.
Further, $E$ commutes with $\tZ$, $V$, $T$ and $\mathcal S$.

\begin{prop}\label{twist-thm}
	Suppose that for every $\tilde\varphi\in\Irr(\tN/\tR\mid \mathrm{dz}(N/R))$
	there exists some $\varphi\in\Irr(N\mid \tilde\varphi)$ such that 
	\begin{enumerate}[\rm(1)]
		\item $(\tN\rtimes E)_{\varphi}=\tN_{\varphi}\rtimes E_{\varphi}$,
		\item $\varphi$ has an extension $\hat{\varphi}\in\Irr(N\rtimes E_{\varphi})$ with $v\hat F\in\ker(\hat \varphi)$.
	\end{enumerate}
	Then Theorem \ref{local-condition} holds.
\end{prop}

\begin{proof}
	The proof is entirely analogous with \cite[Prop.~5.13]{CS17}.
\end{proof}

\subsection{Characters of $C$.}

For $0\le i\le s$, let $R_i$, $C_i$, $M_i$ and $N_i$ be the subgroup of $\tR_i$, $\tC_i$, $\tM_i$ and $\tN_i$ respectively, consisting of matrices of determinant $1$.
We denote $\mathcal R=R_0\times\prod_{i=1}^s R_i^{t_i}$,
$\mathcal C=C_0\times\prod_{i=1}^s C_i^{t_i}$, 
$\mathcal M=R_0\times\prod_{i=1}^s M_i$
and
$\mathcal N=(N_0\times\prod_{i=1}^s N_i^{t_i})\mathcal S$.
Then $\mathcal R\le R$, $\mathcal C\le C$, $\mathcal N\le N$.
In addition, $\mathcal N/\mathcal R=(\mathcal R\mathcal C\cM V/\mathcal R\times T)\mathcal S$ and $(\tR\tC\tM\cap \mathcal N)/\tR=\mathcal R\mathcal C\mathcal M/\mathcal R$.

Let $\Irr'(\tC\tR/\tR)$ be the subset of $\Irr(\tC\tR/\tR)$ consisting of characters of defect $\le \nu(\tN/\tR N)$.
Then Proposition \ref{sepcial-defect} gives the possiblity for $\tR$ if $\Irr'(\tC\tR/\tR)$  is non-empty.
Let $\tilde{\mathcal G}$ be a representative system  of $(\tR\tC \tM VT\rtimes E)$-orbits on $\Irr'(\tC\tR/\tR)$ such that for every 
$\tilde\theta=\tilde\theta_0\times\prod_{i=1}^s \prod_{j=1}^{t_{i}}\tilde\theta_{ij}
\in\tilde{\mathcal G}$ with $\tilde\theta_{ij}\in\Irr(\tR_i\tC_i/\tR_i)$,
we have either $\tilde\theta_{ij_1}=\tilde\theta_{ij_2}$
or $\tilde\theta_{ij_1}$ and $\tilde\theta_{ij_2}$ are not $V_iT_iE$-conjugate for any $1\le i\le s$ and $1\le j_1\ne j_2\le t_i$.

\begin{lem}\label{act-ext-C0}
	Let $\vartheta\in\dz(\cR \cM/\cR)$.
	There exists a representative system $\mathcal G_0$ of $(\tN\rtimes E)$-orbits on $\Irr(\mathcal R\mathcal C/\mathcal R\mid\tilde{\mathcal G})$ such that every $\xi\in\mathcal G_0$ satisfies
	$(\tN/\tR\rtimes E )_{\xi}=(\tR\tC_{\xi}\tM/\tR) (VT\rtimes E)_{\xi}\mathcal S_\xi$
	and $\xi\times\vartheta$ 
	has an extension $\widetilde{\xi\times\vartheta}$
	to $(\mathcal R\mathcal C\cM/\mathcal R) (VT\mathcal S\rtimes E)_{\xi\times\vartheta}$ such that $v\hat F\in \ker(\widetilde{\xi\times\vartheta})$.
\end{lem}

\begin{proof}
	Note that 
	$\langle V_i,E \rangle$	acts on $\tC_i$ via field and graph automorphisms.	
	Let $\theta\in\tilde{\mathcal G}$ with $\tilde\theta=\tilde\theta_0\times\prod_{i=1}^s \prod_{j=1}^{t_{i}}\tilde\theta_{ij}$.
	By \cite[Rmk.~4.7]{CS17}, 
	there exists $\xi_{ij}\in\Irr(C_i\mid \tilde\theta_{ij})$	such that 
	$((\tR_i\tC_i \tM_i/\tR_i)V_i T_iE)_{\xi_{ij}}=(\tR_i\tC_i \tM_i/\tR_i)_{\xi_{ij}}(V_i T_i E)_{\xi_{ij}}$ and $\xi_{ij}$	extends to $(R_iC_i/R_i)(V_i E)_{\xi_{ij}}$.
	Also, there is $\xi_0\in\Irr(C_0\mid\tilde\theta_0)$ such that $(\tC_0\rtimes E)_{\xi_0}=\tC_0\rtimes E_{\xi_0}$ and $\xi_0$ extends to $C_0\rtimes E_{\xi_0}$.
	Let $\xi=\xi_0\times \prod_{i=1}^s\prod_{j=1}^{t_i}\xi_{ij}$.
	So $\xi_{ij_1}=\xi_{ij_2}$ or $\xi_{ij_1}$ and $\xi_{ij_2}$ are not $M_iV_iT_iE$-conjugate for any $1\le i\le s$ and $1\le j_1,j_2\le t_i$.
	Assume further if $\Res^{\tC_i}_{C_i}(\tilde\theta_{ij_1})=\Res^{\tC_i}_{C_i}(\tilde\theta_{ij_2})$, then we take $\xi_{ij_1}=\xi_{ij_2}$.
	Then it is easy to check that $\xi$  satisfies	
	$(\tN/\tR\rtimes E )_{\xi}=(\tR\tC_{\xi}\tM/\tR) (VT\rtimes E)_{\xi}\mathcal S_\xi$.	
	
	Now we prove the extendibility of $\xi\times\vartheta$.	
	First assume that we are not in the cases (\ref{special-case-7-1}).		If we write $\vartheta=\prod_{i=1}^{s} \vartheta_i^{t_i}$ with 
	$\vartheta_i\in\Irr( R_iM_i/R_i)$,
	then $\vartheta_i$ is the (unique) Steinberg character of $R_iM_i/ R_i\cong\Sp_{2\gamma_i}(\ell)$.	
	Thus $(\tN/\tR\rtimes E )_{\xi}=(\tN/\tR\rtimes E )_{\xi\times\vartheta}$.	
	On the other hand, $(\mathcal R\mathcal C\cM/\mathcal R) (VT\mathcal S\rtimes E)_{\xi\times\vartheta}=(\mathcal R\mathcal C\cM/\mathcal R) (VTE)_{\xi}\mathcal S_{\xi}$.
	Then by Lemma \ref{ext-wre},
	it suffices to show the following statements 
	\begin{enumerate}[(a)]	
		\item $\mu\in\Irr(Z\mid \xi)$ extends to $\mathcal S_\xi$ for $Z=\mathcal S\cap \mathcal R\mathcal C$ and
		\item $\theta_{ij}\times\vartheta_i$ has an extension $\widetilde{\theta_{ij}\times\vartheta_i}$
		to $( R_i C_iM_i/ R_i)(V_iT_i\rtimes E)_{\theta_{ij}}$ with $v_i\hat F\in\ker(\widetilde{\theta_{ij}\times\vartheta_i})$.
	\end{enumerate}
	
	We first prove (a). 
	It is trivial if $Z=1$ and
	so we assume that $Z\ne 1$.
	According to Lemma \ref{ext-semi}, when considering the extension of characters, we pay attention to the extendibility of $\mu$ to $Z\rtimes \prod_{i=1}^s \mathfrak{S}(t_i)$.
	Let $\tilde{\mathcal S'}= C_2\wr \prod_{i=1}^s \mathfrak{S}(t_i)$.
	Then $Z\rtimes \prod_{i=1}^s \mathfrak{S}(t_i)$ is the kernel of the linear character
	$\tilde{\mathcal S}'\to \mathbb{C}^\times$,
	$(x,\sigma)\mapsto \beta(x)\mathrm{sgn}(\sigma)$, where $\beta:(C_2)^{\sum_{i=1}^st_i}\to \mathbb{C}^\times$ is a linear character and faithful on each summand $C_2$. 
	Here we write $C_2$ for the cyclic group of order $2$.
	From this, the proof of \cite[Prop.~5.5~(b)]{CS17} applies here.

	Now we prove (b).
	Note that
	$( R_i C_iM_i/ R_i)(V_iT_i\rtimes E)_{\theta_{ij}}=( R_i C_iM_i/R_i)(V_iE)_{\theta_{ij}}\times T_i$.
	Then we may assume that $T_i=1$.
	By the first paragraph, $\theta_{ij}$ has an extension $\tilde{\theta}_{ij}$ to 	$(R_i C_i/ R_i)(V_iE)_{\theta_{ij}}$.
	In particular, we may assume $v_i\hat F\in \ker(\tilde{\theta}_{ij})$ since $v_i\hat F$ acts trivially on $ R_iC_i$.
	So it suffices to show that $\vartheta_i$ extends to $( R_iM_i/ R_i)V_iE/\langle v_i\hat F \rangle$.
	Now	$ R_iM_i/R_i\cong \Sp_{2\gamma_i}(\ell)$.
	Since the outer automorphism group of $\Sp_{2\gamma_i}(\ell)$ is cyclic,
	we have the extendibility of $\vartheta_i$ and then (b) follows by Lemma \ref{ext-central-prod}.
	The property $v\hat F\in\ker(\widetilde{\theta_{ij}\times\vartheta_i})$ also follows by the construction of extensions in the proof of Lemma \ref{ext-central-prod}.

	Now assume that  (\ref{special-case-7-1}) holds.
	Then there exists some $i$ such that $\bc_i=\zero$ and one of
	(\ref{eq:special-case-3-1}), (\ref{eq:special-case-2-linear-1}) and (\ref{special-case-2-uni-1}) holds for $\tR_{m_i,\alpha_i,\gamma_i}$ (or $\tR^\pm_{m_i,\alpha_i,\gamma_i}$).
	By the argument above, it also suffices to show (a) and (b).  
	The property (a) also holds analogously here and we only need to prove (b).	 In addition, the extendibility of $\theta_{ij}$ can also be obtained as above.
	
First we let $\ell=3$ and $a=1$ and case (\ref{eq:special-case-3-1}) occurs.	
The only case that should be considered here is that 
	$\vartheta_i$ extends to $( R_iM_i/ R_i)V_iE$ when $\tR_i=\tR_{m_i,0,1}$ with $3\nmid m_i$.
Then $\nu(\det( R_i C_i))=0$ and $|\det(\tM_i)|=3$.
	We have
	$V_i=1$
	and $M_i/(M_i\cap R_i)\cong Q_8$ by Proposition \ref{prop:CN_m,alpha,gamma-odd}.
	By direct calculation, $F_p$ acts trivially on $M_i/(M_i\cap R_i)$ 
	if $p\equiv 1\ (\mathrm{mod}\ 3)$
	and $F_p$ acts as graph automorphism on $M_i/(M_i\cap R_i)$ otherwise.
	Indeed, using an easy argument, which can be found in the proof of \cite[Prop.~5.10]{FLL17a},
	we know that $E$ acts trivially on $M_i/(M_i\cap R_i)$.
	Hence $\vartheta_i$ 
	has an extension $\tilde\vartheta_i$ to $( R_iM_i/ R_i)E$ such that $E\subseteq \ker(\tilde\vartheta_i)$.
	
Suppose that $\ell=2$ and   (\ref{eq:special-case-2-linear-1}) or (\ref{special-case-2-uni-1}) holds
for some $\tR_{m_i,\alpha_i,\gamma_i}$.
Then $V_i=1$ and $M_i/(M_i\cap R_i)$ is isomorphic to one of the groups $C_3$, $\fA_6$ and $\Omega_4^+(2)$ by Proposition \ref{prop:CN_m,alpha,gamma-2-linear}, Remark \ref{M-for-not-central-product} and Proposition \ref{prop:CN_m,alpha,gamma-2-unitary}.
Thus, as above, it suffices to show that for the group $P=C_3$, $\fA_6$ or $\Omega_4^+(2)$, if a group $Q\unrhd P$ induces all automorphisms on $P$ via conjugation, then every character $\vartheta\in\Irr(P)$ extends to the stabilizer $Q_\vartheta$.
If $P=C_3$ or $\Omega_4^+(2)$, then $\Out(P)$ is cyclic and then the assertion above holds immediately. Therefore, it remains to consider $P=\fA_6$, which case follows by \cite[Lemma 15.2]{IMN07} since	
$\fA_6\cong \PSL_2(9)$. The property $v\hat F\in\ker(\widetilde{\theta_{ij}\times\vartheta_i})$ also follows by the construction of extensions in the proof of Lemma \ref{ext-central-prod}.
This completes the proof.
\end{proof}

\subsection{Parametrization of certain characters of $N/R$.}

Let $Y\unlhd X$ be finite groups and $\mathcal Y\subseteq \Irr(Y)$. Following	\cite[Def.~3.5]{MS16}, we say that \emph{maximal extendibility holds for $\mathcal Y$ with respect to $Y\unlhd X$} if every $\chi\in\mathcal Y$ extends (as irreducible characters) to $X_\chi$.
Then, an \emph{extension map for $\mathcal Y$ with respect to $Y\unlhd X$} is a map $$\Lambda:\mathcal Y\to \bigcup\limits_{Y\le I\le X}\Irr(I),$$
such that for every $\chi\in\mathcal Y$ the character $\Lambda(\chi)\in\Irr(X_\chi\mid \chi)$ is an extension of $\chi$.

\vspace{2ex}

Let $\mathcal{G}_0'=\{\xi\times \vartheta \in \Irr(\mathcal R\mathcal C\mathcal M/\mathcal R)\mid \xi\in\mathcal G_0, \vartheta\in \dz(\mathcal R\mathcal M/\mathcal R)\}$ and $\tilde{\mathcal G}'=\{\ttheta\times\tilde\vartheta\in \Irr(\tR\tC\tM/\tR)  \mid \ttheta\in\tilde{\mathcal{G}}, \tilde\vartheta\in \Irr(\tR\tM/\tR\mid \dz(\mathcal R\mathcal M/\mathcal R))\}$.
Using Lemma \ref{pre-dir}, we have the following lemma immediately.

\begin{lem}\label{irr-c}
	Let $\xi'\in\mathcal G'_0$ and $\tilde\theta'\in\Irr(\tR\tC\tM/\tR\mid \xi')$.
	Then there is a unique character $\theta\in \Irr((\tR\tC\tM\cap N)/R\mid \xi')\cap\Irr((\tR\tC\tM\cap N)/R\mid \tilde\theta')$.
\end{lem}

Let $\mathcal G=\Irr((\tR\tC\tM\cap N)/R\mid \mathcal G'_0)\cap \Irr((\tR\tC\tM\cap N)/R\mid \tilde{\mathcal G}')$
and let $\mathfrak{C} = \mathcal{G} \times \Irr(T)$.

\begin{lem}\label{sta-psi}
	Let $\psi\in\mathfrak C$.
	Then $((\tR\tC\tM V/\tR\times T)\mathcal{S}\rtimes E)_\psi = (\tR\tC_\psi \tM/\tR \times T) (V\mathcal{S}\rtimes E)_\psi$.
\end{lem}

\begin{proof}
	Write $\psi=\theta'\times \lambda$,
	where $\theta'\in \Irr((\tR\tC\tM\cap N)/R)$
	and $\lambda\in\Irr(T)$.
	We claim that $((\tR\tC\tM V/\tR\times T\mathcal S)\rtimes E)_{\theta'}=\tR\tC_{\theta'} \tM/\tR (VT\mathcal S\rtimes E)_{\theta'}$.
	Assume that $t\in\tC$, $g\in VT\mathcal S\rtimes E$ satisfies that $\theta'^t=\theta'^g\ne\theta'$.
	Let $\xi\times\vartheta\in\Irr(\mathcal R\mathcal C\cM/\mathcal R\mid\theta')$ with $\xi\in \mathcal G_0$ and $\vartheta\in \dz(\mathcal R\mathcal M/\mathcal R)$.
	Then $\xi^g=\xi^{t'}$ for some $t'\in\tC$.
	By Lemma \ref{pre-dir},
	$\Irr(\mathcal R\mathcal C\cM/\mathcal R\mid \theta'^t)$ and $\Irr(\mathcal R\mathcal C\cM/\mathcal R\mid \theta')$ are disjoint,
	and then $\xi^g=\xi^{t'}\ne \xi$, which contradicts with Lemma \ref{act-ext-C0}.
	Thus the claim holds.
	Then the assertion follows by the facts that $((\tR\tC\tM V/\tR\times T\mathcal S)\rtimes E)_\psi\le ((\tR\tC\tM V/\tR\times T\mathcal S)\rtimes E)_{\theta'}$ and 
	$\tC_{\theta'}=\tC_\psi$.
\end{proof}

We verify Proposition \ref{twist-thm} in steps  mimicking the strategy applied in \cite[\S 5]{CS17}. In fact, the groups $(\tR\tC \tM\cap N)/R\times T$ and $N/R$ in this section play the same roles as $C$ and $N$ in that paper. We first have

\begin{prop}\label{extension-map}
	Let $\mathfrak C$ be the set defined as above.
	There exists an extension map $\Lambda$ for $\mathfrak C$ with respect to $(\tR\tC \tM\cap N)/R\times T\unlhd N/R$ such that:
	\begin{enumerate}[\rm(1)]
		\item $\Lambda$ is $N\rtimes E$-equivariant,
		\item for every $\psi\in\mathfrak C$ the character $\Lambda(\psi)$ has an extension $\hat \psi\in\Irr((N/R\rtimes E)_\psi)$ with the property that $v\hat F\in \ker(\hat\psi)$,
		\item for every $\psi\in\mathfrak C$ and $t\in\tC_\psi\tR\tM/\tR$ there exists a linear character $\nu\in\Irr(N_\psi/R)$ with the property that $\Lambda(\psi)^t=\Lambda(\psi)\nu$, and $\nu$ is the lift of a faithful character of $N_\psi/N_{\tilde\psi}$ for any $\tilde\psi\in\Irr(\langle (\tR\tC \tM\cap N)/R\times T,t \rangle\mid\psi)$.
	\end{enumerate}
\end{prop}

\begin{proof}
	According to \cite[Lemma~5.8(b)]{CS17}, part (3) is satisfied for any extension map $\Lambda$ for $\mathfrak C$ with respect to $(\tR\tC \tM\cap N)T\unlhd N$, since $\tR\tC\tM$ acts trivially on $(N/R)/((\tR\tC \tM\cap N)/R\times T)\cong (\tN/\tR)/(\tR\tC MT/\tR)$.
	
	For the proof of (1) and (2), it suffices to show that every $\psi\in \mathfrak C$ extends to  $(N\rtimes E)_\psi/\langle v\hat F_1 \rangle$ by \cite[Lemma~3.6]{MS16}.
	We apply Lemma \ref{ext-char}
	(see Remark \ref{ker-ext} (c)) with $X:=(N/R)\rtimes E$,
	$Y:=(\tR\tC \tM\cap N)/R\times T$,
	$U:=(\mathcal N/\mathcal R)\rtimes E$,
	$\chi:=\psi$,
	$\xi\in \mathcal G_0$.
	Note that  $Y\cap U=\mathcal C\mathcal R\cM/\mathcal R\times T$, $Y/Y\cap U$ is abelian and $\Res^Y_{Y\cap U}(\chi)$ is multiplicity-free.
	From this, it suffices to show that every $\xi\times \vartheta\times \lambda\in \Irr(\mathcal C\mathcal R/\mathcal R\times \cM\mathcal R/\mathcal R \times T)$ extends to $((\mathcal N/\mathcal R)\rtimes E)_{\xi}/\langle v\hat F_1 \rangle$, where $\vartheta\in\dz(\cM\cR/\cR)$ and $\lambda\in\Irr(T)$.
	By Lemma \ref{act-ext-C0} and Lemma \ref{ext-wre},
	it suffices to show that $\lambda_i\in\Irr(T_i)$ extends to $T_iV_iE$,
	which is obvious since 
	$T_i$ commutes with $V_iE$.
	The property $v\hat F\in\ker(\hat \theta)$
	follows from Lemma \ref{act-ext-C0} and Remark \ref{ker-ext} (b).
\end{proof}

Let
$\mathfrak N=\Irr(N/R\mid\mathfrak C)$.
Then we have a parametrization of $\mathfrak N$ by standard Clifford theory.

\begin{prop}\label{clifford-N}
	Let $\Lambda$ be the extension map from Proposition \ref{extension-map} for $\mathfrak C$ with respect to $(\tR\tC\tM\cap N)/R\times T\unlhd N/R$.
	Let
	$\mathcal P$ be the set of pairs $(\psi,\zeta)$ with $\psi\in\mathfrak C$ and $\zeta\in\Irr(W_\psi)$, where $W_\psi:=(N_\psi/R)/((\tR\tC\tM\cap N)/R\times T)$.
	Then the map $$\Pi:\mathcal P\to \mathfrak N,\quad (\psi,\zeta)\mapsto \Ind_{N_\psi}^{N}(\Lambda(\psi)\zeta)$$
	is surjective and satisfies
	\begin{enumerate}[\rm(1)]
		\item $\Pi(\psi,\zeta)=\Pi(\psi',\zeta')$ if and only if there exists some $n\in N$ such that $\psi^n=\psi'$ and $\zeta^n=\zeta'$.
		\item $\Pi(\psi,\zeta)^\sigma=\Pi(\psi^\sigma,\zeta^\sigma)$ for every $\sigma\in E$.
		\item Let $t\in \tC_\psi\tR\tM$.
		Then $\Pi(\psi,\zeta)^t=\Pi(\psi,\zeta\nu)$, where $\nu\in\Irr(N_\psi/R)$ is given by $\Lambda(\psi)^t=\Lambda(\psi)\nu$ and is a linear character with $N_{\tilde\psi}/R=\ker(\nu)$ for any extension $\tilde\psi\in\Irr(\langle (\tR \tC\tM\cap N)/R\times T,t \rangle)$ of $\psi$.
	\end{enumerate}
\end{prop}

\begin{proof}
	This follows from Clifford theory and Proposition \ref{extension-map} and the proof is entirely analogous with \cite[Prop.~5.10]{CS17} or \cite[Prop. 3.15]{MS16}.
\end{proof}

\begin{rem}\label{conj-cla}
	The property (3) of Proposition \ref{clifford-N} implies that for $\psi\in\mathfrak C$, $\tilde\psi\in\Irr(\tR\tC\tM/\tR\times T\mid \psi)$ and $\zeta_0\in W_{\tilde\psi}$ (where $W_{\tilde\psi}$ is defined as in Lemma \ref{quo-irr}), the set 
	$\{(\psi,\zeta')\mid\zeta'\in\Irr(W_\psi\mid\zeta_0)  \}$
	corresponds via $\Pi$ to an $\tR\tC_{\psi}\tM$-orbit.
	Because of $\tN_\psi=N\tR\tC_\psi \tM$ this is then also an $\tN_\psi$-orbit.
\end{rem}

\subsection{Stabilizers and extendibility of characters of $N$.}

\begin{lem}\label{quo-irr}
	Let $\psi\in\mathfrak C$, $\tilde \psi\in\Irr(\tR\tilde C \tilde MT/\tR\mid \psi)$,
	\begin{align*}
	&W:=(N/R)/((\tR\tC\tM\cap N)/R\times T),\\
	&W_\psi:=(N_\psi/R)/((\tR\tC\tM\cap N)/R\times T),\\
	&W_{\tilde\psi}=(N_{\tilde \psi}/R)/((\tR\tC\tM\cap N)/R\times T)
	\end{align*}
	and $\zeta_0\in\Irr(W_{\tilde\psi})$. Let
	$$K:=N_W(W_\psi)\cap N_W(W_{\tilde\psi}).$$
	Then there exists a character $\zeta\in\Irr(W_\psi\mid\zeta_0)$ such that the following hold:
	\begin{enumerate}[\rm(1)]
		\item $\{ \zeta^w \mid w\in K\}\cap \Irr(W_\psi\mid\zeta_0)=\{\zeta\}$,
		\item $\zeta$ extend to $K_\zeta$,
		\item $\zeta$ has an extension $\hat \zeta\in\Irr(K_\zeta\times E)$ with $v\hat F\in\ker(\hat\zeta)$.
	\end{enumerate}
\end{lem}

\begin{proof}
	The property claimed in (3) follows from the fact that the order of $v$ divides the order of $\hat F$ and both are central in the group $K\times E$.	
	
	By the structure of $N$, we have
	$W=V\mathcal S/(V\mathcal S\cap \tZ)$ and then $W\cong\prod_{i=1}^s C_{e\ell^{\alpha_i}}\wr \mathfrak S(t_i)$.	
	Here 	$C_{e\ell^{\alpha_i}}$ means the cyclic group of order $e\ell^{\alpha_i}$.
	In this way we only need to consider the case $s=1$ and then $W$ is a wreath product and
	the proof is entirely analogous with \cite[Prop.~5.12]{CS17}.
\end{proof}

Finally we prove the assumption made in Proposition \ref{twist-thm} holds.

\begin{prop}\label{tranfer-ver}
	Suppose that $\tilde\varphi\in\Irr(\tN/\tR\mid \mathrm{dz}(N/R))$.
	Then there exists some $\varphi\in\Irr(N\mid \tilde\varphi)$ such that 
	\begin{enumerate}[\rm(1)]
		\item $(\tN\rtimes E)_{\varphi}=\tN_{\varphi}\rtimes E_{\varphi}$,
		\item $\varphi$ has an extension $\hat{\varphi}\in\Irr(N\rtimes E_{\varphi})$ with $v\hat F\in\ker(\hat \varphi)$.
	\end{enumerate}
\end{prop}

\begin{proof}
	It suffices to prove the conclusion for all $\tilde\varphi$ in a given $E$-transversal in $\Irr(\tN/\tR\mid \mathrm{dz}(N/R))$.
	Then we may assume that there exists $\tilde\psi\in\Irr(\tR\tC\tM/\tR\times T\mid \tilde\varphi)$ such that
	$\tilde\psi=\tilde\theta\times\tilde\vartheta\times\lambda$ where $\tilde\theta\in\tilde{\mathcal G}$,
	$\tilde\vartheta\in\Irr(\tR\tM/\tR)$ and $\lambda\in\dz(T)$ and
	there is a character $\xi'\in\Irr(\mathcal C\mathcal R\mathcal M/\mathcal R\mid \ttheta\times \tilde\vartheta)\cap \mathcal G'_0$.
	Then by Lemma \ref{irr-c},
	there exists a character $\theta'\in\Irr((\tR\tC\tM\cap N)/R\mid \tilde\theta')\cap\Irr((\tR\tC\tM\cap N)/R\mid \xi')$,
	where $\ttheta' = \ttheta \times \tilde{\vartheta}$.
	Let $\psi=\theta'\times \lambda\in\mathfrak C$.
	
	For any $\varphi''\in\Irr(N\mid\tilde\varphi)$,
	there is a character $\varphi'$ of $N$ which is $\tN$-conjugate to $\varphi''$
	such that $\psi\in\Irr((\tR\tC\tM\cap N)/R\times T\mid \varphi')$ by Clifford theory.	
	Keep the notation of Proposition \ref{clifford-N} and Lemma \ref{quo-irr}.
	With the map $\Pi$,
	we have $\varphi'=\Pi(\psi,\zeta')$ where $\zeta'\in\Irr(W_\psi)$.
	Let $\tilde\psi\in\Irr(\tR\tC\tM/\tR\times T\mid \psi)$ and $\zeta_0\in\Irr(W_{\tilde\psi}\mid\zeta')$.
	Then by Lemma \ref{quo-irr},
	there exists some $\zeta\in\Irr(W_\psi\mid \zeta_0)$ such that $\{ \zeta^w\mid w\in K  \}\cap \Irr(W_\psi\mid\zeta_0)=\{\zeta\}$
	and $\zeta$ extends to some character $\hat\zeta\in\Irr(K_\zeta\times E)$ with $v\hat F\in\ker(\hat\zeta)$.
	Let $\varphi=\Pi(\psi,\zeta)$.
	Then Remark \ref{conj-cla} implies that $\varphi$ and $\varphi'$ are $\tN$-conjugate.
	
	Now we verify that $\varphi$ satisfies the  equation $(\tN\rtimes E)_{\varphi}=\tN_{\varphi}\rtimes E_{\varphi}$.
	Let $x\in(\tN\rtimes E)_{\varphi}$.
	After suitable $N$-multiplication we can assume that $\psi^x=\psi$.
	By Lemma \ref{sta-psi},
	we can write $x=nte$, where $n\in N$, $t\in\tR\tC_\psi\tM$ and $e\in E$.
	Denote by $w$ the image of $n$ in $W$.
	According to Proposition \ref{clifford-N},
	$\varphi^x=\Pi(\psi^{ne},\zeta^w\nu)$, where $\nu$ is a linear character of $W_{\psi^n}$ such that $W_{\tilde\psi}\in\ker(\nu)$.
	
	The equation $\varphi^x=\varphi$ is equivalent to $\zeta=\zeta^w\nu$ since $\psi^x=\psi$.
	We claim that $w\in K$, \emph{i.e.}, $w$ normalizes $W_\psi$ and $W_{\tilde\psi}$.
	First we have $W_\psi^x=W_{\psi^x}=W_{\psi}$ from $\psi^x=\psi$.
	Since $t$ and $e$ acts trivially on $W_\psi$, this implies that $W_\psi^w=W_{\psi}$.
	On the other hand for $W_{\tilde\psi}$ we have $W_{\tilde\psi}^w=W_{\tilde\psi^x}$.
	Since $\tpsi^x,\tpsi\in\Irr(\tR\tC\tM/\tR\times T\mid\psi)$,
	there exists some linear character $\delta\in\Irr(\tR\tC\tM/\tR\times T)$
	such that $(\tR\tC\tM\cap N)/R\times T\subseteq \ker(\delta)$.
	Since $[\tN/\tR,\tR\tC\tM/\tR\times T]\le (\tR\tC\tM\cap N)/R \times T$,
	the character $\delta$ extends to some character of $\tN$, especially $\delta$ is $\tN$-invariant.
	Thus $W_{\tilde\psi}=W_{\tilde\psi\delta}$ and this proves $w\in N_{W}(W_{\tilde\psi})$. 
	So $w\in K$ and the claim holds.
	
	Since $\zeta\nu^{-1}\in\Irr(W_\psi\mid\zeta_0)$ and $\{ \zeta^w\mid w\in K  \}\cap \Irr(W_\psi\mid\zeta_0)=\{\zeta\}$,
	this implies $\zeta^w=\zeta=\zeta\nu$.
	By Proposition \ref{clifford-N} we see
	$\varphi^t=\Pi(\psi,\zeta\nu)=\Pi(\psi,\zeta)=\varphi$.
	Then $t\in \tN_\varphi$ and hence $(\tN\rtimes E)_{\varphi}=\tN_{\varphi}\rtimes E_{\varphi}$.
	
	Based on the observation on $\varphi$,
	the construction of the proof of \cite[Prop.~5.13~(ii)]{CS17}  
	(or \cite[Thm.~3.18~(ii)]{MS16}) 
	gives an extension of $\varphi$ to $N\rtimes E_\varphi/\langle v\hat F\rangle$, and
	this completes the proof.
\end{proof}

Thanks to Proposition \ref{tranfer-ver}, we can apply Proposition \ref{twist-thm} now and this
implies Theorem \ref{local-condition}.

%%%%%%%%%%%%%%%%%%%%%%%%%%%%%%%%%%%%%%%%%%%%%%%%%%%%%%%%%%%%

\section{Proof of the main theorem}
\label{proof-main-thm}

We first prove the following.

\begin{thm}\label{split-brauer}
	Let $\tilde\psi\in\IBr(\tG)$, then there exists $\psi\in\IBr(G\mid \tilde\psi)$ such that
	\begin{enumerate}[\rm(1)]
		\item $(\tG\rtimes D)_\psi=\tG_\psi\rtimes D_\psi$,
		\item $\psi$ extends to $G\rtimes D_\psi$.
	\end{enumerate}
\end{thm}

\begin{proof}
	By \cite[Thm.~4.1]{CS17}, for any $\tilde\chi\in\Irr(\tG)$, there is a $\chi\in\Irr(G\mid\tilde\chi)$ such that $(\tG\rtimes D)_\chi=\tG_\chi\rtimes D_\chi$ and $\chi$ extends to $G\rtimes D_\chi$.
	
	By \cite{De17} (see Remark \ref{basic-set-sl}), there is a unitriangular basic set $\tilde{\cE}'$ of $\tG$ such that its set $\cE$ of irreducible constituents upon restriction to $G$ is a unitriangular basic set for $G$.
	Then by \cite[Lemma~2.3]{De17}, there is a $\Lin_{\ell'}(\tG/G)\rtimes D$-equivariant bijection between $\tilde{\cE}'$ and $\IBr(\tG)$ and there is a $\tilde G\rtimes D$-equivariant bijection between $\cE$ and $\IBr(G)$.
	Also by \cite[Lemma~2.9]{FLZ19a}, the extendibility of irreducible Brauer characters can be deduced from the extendibility of irreducible ordinary characters.
	Thus this assertion holds by the above paragraph.
\end{proof}

Let $S=\PSL_n(\eta q)$ be a simple group of type $\mathsf A$ and $G$ be the universal covering group of $S$.
We  consider the exceptional covering cases, namely, the universal covering group of $S$ is bigger than $\SL_n(\eta q)$.  
Explicitly, according to \cite[Table~6.1.3]{GLS98}, $S$ is one of the groups $\PSL_2(4)\cong \mathfrak A_5$, $\PSL_2(9)\cong\mathfrak A_6$, $\PSL_3(2)$,  $\PSL_3(4)$, $\PSL_4(2)\cong \mathfrak A_8$, $\PSU_4(2)\cong\PSp_4(3)$, $\PSU_4(3)$ and $\PSU_6(2)$.
Thanks to \cite[Thm.~1.1]{Ma14},
the alternating groups $\mathfrak A_5$, $\mathfrak A_6$ and $\mathfrak A_8$ satisfy the (iBAW) condition.
The simple group $\PSp_4(3)$ was checked in \cite{BSF19}  or \cite{LL19}.
If $S=\PSL_3(2)$, then $|G|=2^4\cdot 3\cdot 7$, and then the only prime $\ell$ such that the Sylow $\ell$-subgroups of $G$ are non-cyclic is $\ell=2$, which is the defining characteristic.
Note that the case for cyclic  Sylow subgroups was settled in \cite{KS16a,KS16b}, while the defining characteristic case was checked in \cite[Thm.~C]{Sp13}.
Analogously, except the cyclic  Sylow $\ell$-subgroups cases and the defining characteristic cases,
the only prime we need to consider for the  simple group $\PSL_3(4)$, $\PSU_4(3)$, $\PSU_6(2)$ is just $3$, $2$, $3$, respectively.
Thus by \cite{Du19, Du20},
the exceptional cases are all solved.
Note that the paper \cite{Du20} dealt with the irreducible Brauer characters and weight characters of the universal covering groups of $\PSU_4(3)$ and $\PSU_6(2)$ which are not the inflation of the irreducible Brauer characters and weight characters of $\SU_4(3)$ and $\SU_6(2)$, while the irreducible Brauer characters and weight characters of special unitary groups are considered in this paper.

Now we are able to combine our observations and results obtained to verify our main assertion Theorem \ref{mainthm-1}.

\begin{proof}[Proof of Theorem \ref{mainthm-1}]
By \cite[Thm.~C]{NT11}, we only need to consider the non-defining characteristic.	
According to the above arguments, we may assume that the group $G=\SL_n(\eta q)$ is the universal covering group of the simple group $S=\PSL_n(\eta q)$.
Let $\tG$, $D$ be as in \S \ref{subsec:Linear and unitary groups}.
We check the conditions of Theorem
	\ref{thm:criterion}.
For (i), conditions (a) and (b) are obvious and (c) and (d) follow by the fact that $\tG/G$ is cyclic.
Then this theorem follows by combining
Theorem \ref{thm-equ-bijection}, \ref{local-condition} and  \ref{split-brauer}.
\end{proof}

%%%%%%%%%%%%%%%%%%%%%%%%%%%%%%%%%%%%%%%%%%%%%%%%%%%%%%%%%%%%

\subsection*{Acknowledgement}
The authors would like to express their deep thanks to Gunter Malle for useful remarks and suggestions on the manuscript of this paper.
They are also indebted to the anonymous referee for many helpful comments that have helped improve this paper.

%%%%%%%%%%%%%%%%%%%%%%%%%%%%%%%%%%%%%%%%%%%%%%%%%%%%%%%%%%%%%%%%%%%%%%%%%%%%%%%%%%%%%%%%%%%%%%%%%%%%%%%%%%%%%%%%%%%%%%%%


\addcontentsline{toc}{section}{References}	

\begin{thebibliography}{99}\setlength{\itemsep}{-3pt}
	

	
\small

\bibitem{Al87}
J. L. Alperin, Weights for finite groups. In: \emph{The Arcata Conference on Representations of Finite Groups, Arcata, Calif. (1986), Part I}. 
Proc. Sympos. Pure Math., {vol. \bf 47}, Amer. Math. Soc., Providence, 1987, pp.369--379.

\bibitem{AF90}
J. L. Alperin, P. Fong, Weights for symmetric and general linear groups.
\emph{J. Algebra \bf 131} (1990), 2--22.

\bibitem{An92}
J. An, 2-weights for general linear groups.
\emph{J. Algebra \bf 149} (1992), 500--527.

\bibitem{An93b}
J. An, $2$-weights for classical groups.
\emph{J. Reine Angew. Math. \bf 439} (1993), 159-204.


\bibitem{An93}
J. An, 2-weights for unitary groups.
\emph{Trans. Amer. Math. Soc. \bf 339} (1993), 251--278.

\bibitem{An94}
J. An, Weights for classical groups.
\emph{Trans. Amer. Math. Soc. \bf 342} (1994), 1--42.

\bibitem{AD12}
J. An, H. Dietrich, The AWC-goodness and essential rank of sporadic simple groups. 
\emph{J. Algebra \bf 356} (2012), 325--354.

\bibitem{Bon99b}
C. Bonnaf\'e, Produits en couronne de groupes lin\'eaires.
\emph{J. Algebra \bf 211}~(1999), 57--98.

\bibitem{Bon00}
C. Bonnaf\'e,  Mackey formula in type $\mathsf A$. \emph{Proc. Lond. Math. Soc. \bf 80}  (2000), 545--574.


\bibitem{Br}
T. Breuer,  Computations for some simple groups. At: \url{http://www.math.rwth-aachen.de/~Thomas.Breuer/ctblocks/doc/overview.html}.

\bibitem{Brou86}
M. Brou\'e, Les $\ell$-blocs des groupes $\GL(n,q)$ et $\U(n,q^2)$ et leurs structures locales.
S\'eminaire Bourbaki,  \emph{Ast\'erisque \bf 640 }(1986), 159--188.

\bibitem{BSF19}
J. Brough, A. A. Schaeffer Fry, Radical subgroups and the inductive blockwise Alperin weight conditions for $\PSp_4(q)$.  \emph{Rocky Mountain J. Math. \bf 50} (2020), 1181--1205.

\bibitem{BS19}
J. Brough, B. Sp\"ath, A criterion for the inductive Alperin weight condition. arXiv: 2009.02074.

\bibitem{BS19b} 
J. Brough, B. Sp\"ath, On the Alperin--McKay conjecture for simple groups of type $\mathsf A$.
\emph{J. Algebra \bf  558} (2020), 221--259.


\bibitem{CE04}
M. Cabanes, M. Enguehard, \emph{Representation Theory of Finite Reductive Groups}.
New Math. Monogr., {vol. \bf 1},
Cambridge University Press, Cambridge, 2004.

\bibitem{CS13}
M. Cabanes, B. Sp\"ath, Equivariance and extendibility in finite reductive groups with connected center. \emph{ Math. Z. \bf 275}~(2013), 689--713.


\bibitem{CS17}
M. Cabanes, B. Sp\"ath, Equivariant character correspondences and inductive McKay condition for type $\mathsf A$.
\emph{J. Reine Angew. Math. \bf 728} (2017), 153--194.



\bibitem{De17}
D. Denoncin, Stable basic sets for finite special linear and unitary groups.
\emph{Adv. Math. \bf 307}~(2017), 344--368.

\bibitem{DM90}
F. Digne, J. Michel, 
On Lusztig's parametrization of characters of finite groups of Lie type. 
\emph{Ast\'erisque \bf 181--182}  (1990),  113--156. 

\bibitem{DM91}
F. Digne, J. Michel, \emph{Representations of Finite Groups of Lie Type}. 
London Math. Soc. Stud. Texts, {vol. \bf 21}, Cambridge University Press, Cambridge, 1991.


\bibitem{Du19}
Y. Du, The inductive blockwise Alperin weight condition for $\PSL(3, 4)$. \emph{Comm. Algebra \bf 49}~(2021), 292--300.

\bibitem{Du20}
Y. Du, The inductive blockwise Alperin weight condition for $\PSU(4,3)$ and $\PSU(6,2)$. Preprint.

\bibitem{Feng19}
Z. Feng, The blocks and weights of finite special linear and unitary groups.
\emph{J. Algebra \bf 523} (2019), 53--92.

\bibitem{FLL17a}
Z. Feng, C. Li, Z. Li, The inductive blockwise Alperin weight condition for $\PSL(3,q)$.
\emph{Algebra Colloq. \bf 24}~(2017), 123--152.



\bibitem{FLLMZ19a}
Z. Feng, C. Li, Y. Liu, G. Malle, J. Zhang, 
Robinson’s conjecture for classical groups.
\emph{J. Group Theory \bf 22}  (2019), 555--578.




\bibitem{FLZ19a}
Z. Feng, Z. Li, J. Zhang, On the inductive blockwise Alperin weight condition for classical groups.
\emph{J. Algebra \bf 537} (2019), 381--434.



\bibitem{FM20}
Z. Feng, G. Malle,
The inductive blockwise Alperin weight condition for type $\mathsf C$ and the prime $2$. \emph{J. Austral. Math. Soc.} (2020), https://doi.org/10.1017/S1446788720000439.



\bibitem{FS82}
P.  Fong, B. Srinivasan, The blocks of finite general linear and unitary groups.
\emph{Invent. Math. \bf 69} (1982), 109--153.

\bibitem{Ge91}
M. Geck, On the decomposition numbers of the finite unitary groups in non-defining characteristic.
\emph{Math. Z. \bf 207} (1991), 83--89.

\bibitem{GLS98}
D. Gorenstein, R. Lyons, R. Solomon, \emph{The Classification of the Finite Simple Groups, Number 3}.
Math. Surveys Monogr., {vol. \bf 40},
American Mathematical Society, Providence, RI, 1998.



\bibitem{Hu98} 
B. Huppert, \emph{Character Theory of Finite Groups}. 
De Gruyter Exp. Math., {vol. \bf 25},
Walter de Gruyter \& Co., Berlin, 1998.



\bibitem{Is84}
I. M. Isaacs, Characters of $\pi$-separable groups. \emph{J. Algebra \bf 86} (1984), 98--128.

\bibitem{IMN07}
I. M. Issacs, G. Malle, G. Navarro, A reduction theorem for the McKay conjecture.
\emph{Invent. Math. \bf 170} (2007), 33--101.

\bibitem{JK81}
G. James, A. Kerber, \emph{The Representation Theory of the Symmetric Group}. 
Encyclopedia Math. Appl., {vol. \bf 16},
Addison--Wesley Publishing Co., Reading, Mass., 1981.

\bibitem{JL93} 
G. James, M. Liebeck, \emph{Representations and Characters of Groups}. 
2nd ed. Cambridge University Press, New York, 2001.


\bibitem{KT09}
A. S. Kleshchev, P. H. Tiep, Representations of finite special linear groups in non-defining characteristic.
\emph{Adv. Math. \bf 220}~(2009), 478--504.

\bibitem{KS15}
S. Koshitani, B. Sp\"ath, Clifford theory of characters in induced blocks. \emph{Proc. Amer. Math. Soc. \bf 143} (2015), 3687--3702.

\bibitem{KS16a}
S. Koshitani, B. Sp\"ath,
The inductive Alperin--McKay and blockwise Alperin weight conditions for blocks with cyclic defect groups and odd primes.
\emph{J. Group Theory \bf 19} (2016), 777--813.

\bibitem{KS16b}
S. Koshitani, B. Sp\"ath, The inductive Alperin--McKay condition for 2-blocks with cyclic defect groups.
\emph{Arch. Math \bf 106} (2016), 107--116.

\bibitem{Li19}
C. Li, An equivariant bijection between irreducible Brauer characters and weights for $\Sp(2n,q)$.
\emph{J. Algebra \bf 539} (2019), 84--117.

\bibitem{Li19b}
C. Li, The inductive blockwise Alperin weight condition for $\PSp_{2n}(q)$ and odd primes. \emph{J. Algebra \bf 567} (2021), 582--612.

\bibitem{LL19}
C. Li, Z. Li, The inductive blockwise Alperin weight condition for $\PSp_4(q)$.
\emph{Algebra Colloq. \bf 26} (2019), 361--386.

\bibitem{LZ18}
C. Li, J. Zhang,
The inductive blockwise Alperin weight condition for $\PSL_n(q)$ and $\PSU_n(q)$ with cyclic outer automorphism groups.
\emph{J. Algebra \bf 495} (2018), 130--149.

\bibitem{LZ19}
C. Li, J. Zhang, The inductive blockwise Alperin weight condition for $\PSL_n(q)$ with $(n,q-1)=1$.
\emph{J. Algebra \bf 558} (2020), 582--594.


\bibitem{Lu84}
G. Lusztig, \emph{Characters of Reductive Groups over a Finite Field}. 
Ann. of Math. Stud., {vol. \bf 107}, Princeton University Press, Princeton, NJ, 1984.


\bibitem{Ma14}
G. Malle, On the inductive Alperin--McKay and Alperin weight conjecture for groups with abelian Sylow subgroups.
\emph{J. Algebra \bf 397} (2014), 190--208.

\bibitem{MS16}
G. Malle, B. Sp\"ath, Characters of odd degree.  \emph{Ann. of Math. (2)  \bf 184}  (2016), 869--908.

\bibitem{MT11}
G. Malle, D. Testerman, \emph{Linear Algebraic Groups and Finite Groups of Lie Type}.
Cambridge Stud. Adv. Math., {vol. \bf 133},
Cambridge University Press, Cambridge, 2011.



\bibitem{NT89}
H. Nagao, Y. Tsushima, \emph{Representations of Finite Groups}.
Academic Press Inc., Boston, 1989.



\bibitem{NS14}
G. Navarro, B. Sp\"ath,  On Brauer's height zero conjecture. \emph{J. Eur. Math. Soc.  \bf 16} (2014),  695--747.

\bibitem{NT11}
G. Navarro, P. H. Tiep, A reduction theorem for the Alperin weight conjecture.
\emph{Invent. Math. \bf 184} (2011), 529--565.


\bibitem{OU95}
J. B. Olsson, K. Uno, Dade's conjecture for symmetric groups.
\emph{J. Algebra \bf 176} (1995), 534--560.

\bibitem{SF14}
A. A. Schaeffer Fry, $\Sp_6(2^a)$ is ``good'' for the McKay, Alperin weight, and related local-global conjectures.
\emph{J. Algebra \bf 401} (2014), 13--47.

\bibitem{Sch16}
E. Schulte, The inductive blockwise Alperin weight condition for $G_2(q)$ and $^3D_4(q)$.
\emph{J. Algebra \bf 466} (2016), 314--369.



\bibitem{Sp12}
B. Sp\"ath, Inductive McKay condition in defining characteristic. 
\emph{Bull. Lond. Math. Soc. \bf 44} (2012), 426--438.

\bibitem{Sp13}
B. Sp\"ath, A reduction theorem for the blockwise Alperin weight conjecture.
\emph{J. Group Theory \bf 16} (2013), 159--220.



\bibitem{We16}
P. Webb,  \emph{A Course in Finite Group Representation Theory}. 
Cambridge Stud. Adv. Math., {vol. \bf 161}, 
Cambridge University Press, Cambridge, 2016.



\bibitem{Win72}
D. L. Winter, The automorphism group of an extraspecial $p$-group.
\emph{Rocky Mountain J. Math. \bf 2} (1972), 159--168.

\end{thebibliography}
\end{document}